\let\iflabels=\iffalse 
\let\iflitnum=\iftrue 
\let\iftxfonts=\iftrue 
\let\ifindexindication=\iffalse
\let\iffilename=\iffalse 

\let\ifhyperref=\iftrue 
\let\ifcom=\iffalse 
 
\iflabels
\documentclass[11pt,reqno]{amsbook}
\else
\documentclass[11pt]{amsbook}
\fi
 
\usepackage{amssymb,amsmath,amsthm}

\usepackage[all]{xy}
\usepackage{graphicx}

\iflabels
\usepackage[notcite,notref]{showkeys}
\fi

\iftxfonts
\usepackage[varg]{txfonts}
\fi

\newcommand\grf[2]{
\includegraphics[width=#1cm]{fig/#2.pdf}}

\usepackage{enumerate}

\usepackage{afterpage}

\numberwithin{equation}{section}

\usepackage{longtable}

\usepackage{color}

\definecolor{blue}{rgb}{0,0,1}
\definecolor{red}{rgb}{1,0,0}
\definecolor{green}{rgb}{0,.6,.2}
\definecolor{purple}{rgb}{1,0,1}
\definecolor{brown}{rgb}{.59,.29,0}

\long\def\red#1\endred{{\color{red}#1}}
\long\def\blue#1\endblue{{\color{blue}#1}}
\long\def\purple#1\endpurple{{\color{purple} #1}}
\long\def\green#1\endgreen{{\color{green}#1}}
\long\def\brown#1\endbrown{{\color{brown}#1}}

\ifhyperref
  \usepackage[colorlinks,breaklinks]{hyperref}
\else
  
\fi

\renewcommand\setminus{\smallsetminus}

\renewcommand\={\,=\,}
\newcommand\isd{:=}
\newcommand\isdd{\,\isd\,}
\newcommand\dis{\stackrel{\hbox{.}}=}
\newcommand\ddis{\,\dis\,}

\def\be#1\ee{\begin{equation}#1\end{equation}}
\def\bad#1\ead{\be\begin{aligned}#1\end{aligned}\ee}
\def\badl#1#2\eadl{\be\label{#1}\begin{aligned}#2\end{aligned}\ee}

\newtheorem{mnthm}{Theorem}

\newtheorem{thm}{Theorem}[section]
\newtheorem{prop}[thm]{Proposition}
\newtheorem{cor}[thm]{Corollary}
\newtheorem{lem}[thm]{Lemma}

\theoremstyle{definition}
\newtheorem{defn}[thm]{Definition}
\newtheorem{remark}[thm]{Remark}

\newcommand\rmrk[1]{\smallskip\par\emph{#1. }}

\usepackage{multicol}
\ifindexindication
 
 \newcommand\il[2]{\label{i-#1}\indexindication{#2}{l}}
 \newcommand\ir[2]{\indexindication{$#2$}{r}}
\else
 \newcommand\il[2]{\label{i-#1}}
 \newcommand\ir[2]{}
\fi

\newcommand\al{\alpha}
\newcommand\bt{\beta}
\newcommand\Gm{\Gamma}
\newcommand\gm{\gamma}
\newcommand\Dt{\Delta}
\newcommand\dt{\delta}
\newcommand\e{\varepsilon}
\newcommand\z{\zeta}
\newcommand\Th{\Theta}
\renewcommand\th{\vartheta}
\renewcommand\k{\kappa}
\newcommand\Ld{\Lambda}
\newcommand\ld{\lambda}
\newcommand\Mu{{\mathsf M}}

\newcommand\s{\sigma}
\newcommand\Ups{\Upsilon}
\newcommand\ups{\upsilon}

\newcommand\ph{\varphi}
\newcommand\ch{\chi}

\newcommand\ps{\psi}
\newcommand\Om{\Omega}
\newcommand\om{\omega}

\newcommand\CC{\mathbb{C}}
\newcommand\QQ{\mathbb{Q}}
\newcommand\RR{\mathbb{R}}
\newcommand\ZZ{\mathbb{Z}}

\newcommand\WW{{\mathbf W}}
\newcommand\CK{{\mathbf C}}
\newcommand\HH{{\mathbf H}}
\newcommand\XX{{\mathbf X}}
\newcommand\Z{{\mathbf Z}}
\newcommand\ii{{\mathbf i}}

\newcommand\re{{\mathrm{Re}\,}}
\newcommand\im{{\mathrm{Im}\,}}
\newcommand\SL{{\mathrm{SL}}}
\newcommand\GL{{\mathrm{GL}}}

\newcommand\SU{{\mathrm{SU}}}
\newcommand\U{{\mathrm{U}}}

\newcommand\glie{{\mathfrak g}}
\newcommand\klie{{\mathfrak k}}
\newcommand\alie{{\mathfrak a}}
\newcommand\nlie{{\mathfrak n}}
\newcommand\mlie{{\mathfrak m}}

\newcommand\hK{{\mathfrak h}}

\newcommand\km{\mathrm{k}}
\newcommand\am{\mathrm{a}}
\newcommand\nm{\mathrm{n}}
\newcommand\mm{\mathrm{m}}
\newcommand\wm{\mathrm{w}}
\newcommand\hm{\mathrm{h}}
\newcommand\X{{\mathcal X}}

\newcommand\G{{\mathbf G}}
\newcommand\Schw{{\mathcal S}}
\newcommand\Aut{{\mathrm{Aut}}}
\newcommand\sign{{\,\mathrm{Sign}\,}}
\newcommand\Four{{\mathbf{F}}}
\newcommand\Ffu{{\mathcal F}}
\newcommand\Nfu{{\mathcal N}}
\newcommand\Mfu{{\mathcal M}}
\newcommand\Vfu{{\mathcal V}}
\newcommand\Lfu{{\mathcal L}}
\newcommand\Wfu{{\mathcal W}}
\newcommand\n{{\mathbf n}}
\newcommand\Ws[1]{S_{\!#1}}
\newcommand\wo{\mathrm{O}_{\mathrm{W}}}
\newcommand\WO{{\mathfrak W}}
\newcommand\WOS{{\WO_{\mathrm{sp}}}}
\newcommand\WOI{{\WO_{\mathrm{ip}}}}
\newcommand\WOM{{\WO_{\mathrm{mp}}}}
\newcommand\WOG{{\WO_{\mathrm{gp}}}}
\newcommand\WOGM{{\WO_{\mathrm{gmp}}}}
\newcommand\II{\mathit{II}}
\newcommand\IF{\mathit{IF}}
\newcommand\FI{\mathit{FI}}
\newcommand\FF{\mathit{FF}}

\newcommand\kk{\mathbf{k}}
\newcommand\oh{{\mathrm{O}}}
\newcommand\Gf{\Gamma}

\newcommand\cmpl{{\mathrm{cmpl}}}
\newcommand\uprs{{\mathrm{ups}}}

\newcommand\aut{\mathrm{Aut}}
\newcommand\sect{\mathrm{Sect}}

\renewcommand\c{{\mathfrak c}}
\newcommand\I{\mathbf{I}}

\newcommand\me{{\!!}}
\newcommand\A{{\mathbf A}}
\newcommand\Au[1]{{\mathbf A}^{\!#1}}

\newcommand\Norm{{\mathrm{Norm}}}
\newcommand\fl{{\mathrm{FL}}}

\newcommand\Kph[4]{\;{}^{#1}\Phi^{#2}_{#3,#4}}
\newcommand\kph[4]{\;{}^{#1\!}\varphi^{#2}_{#3,#4}}
\newcommand\ldph[4]{\;{}^{#1\!}\lambda^{#2}_{#3,#4}}
\newcommand\sh[2]{S^{#1}_{#2}}

\newcommand\question[3]%
{\ifcom\marginpar{\purple\tiny question #1\endpurple}%
\footnote{#1. \purple #3 \endpurple}%
\addtocontents{toc}{\string\contentsline{subsection}
{\tocsubsection{}{}{\blue\quad ${}^{\thefootnote}$ #1, #2  \endblue}}%
{\thepage} {} }\fi }

\newcommand\rem[3]%
{\ifcom\marginpar{\purple\tiny remark #1\endpurple}%
\footnote{#1. \purple #3 \endpurple}%
\addtocontents{toc}{\string\contentsline{subsection}
{\tocsubsection{}{}{\blue\quad ${}^{\thefootnote}$ #1, #2\endblue}} {\thepage}
{}
}\fi }

\newcommand\todo[3]%
{\ifcom\marginpar{\purple\tiny to do #1\endpurple}%
\footnote{#1. To do: \purple #3 \endpurple}%
\addtocontents{toc}{\string\contentsline{subsection}
{\tocsubsection{}{}{\blue\quad ${}^{\thefootnote}$ #1, #2\endblue}} {\thepage}
{}
}\fi }

\iffilename
\makeatletter\def\@oddfoot{\rm{\footnotesize{File name
\tt\flnm.tex}\quad\today\hfil\thepage}}
\let\@evenfoot\@oddfoot 
\makeatother
\fi

\begin{document}



\def\flnm{rFtm-title} 

\title[Representations of SU(2,1) in Fourier term
modules] {Representations of the unitary group SU(2,1) in Fourier term
modules}
\author{Roelof W.~Bruggeman}
\address{Mathematisch In\-sti\-tuut Uni\-ver\-si\-teit Utrecht,
Post\-bus 80010, 3508 TA Utrecht, Ne\-der\-land}
\email{r.w.bruggeman@uu.nl}
\author{Roberto J.~Miatello}
\address{FAMAF-CIEM, Universidad Nacional de C\'or\-do\-ba,
C\'or\-do\-ba~5000, Argentina}
\email{miatello@famaf.unc.edu.ar}
\begin{abstract}
We study Fourier term modules on $\SU(2,1)$, which are the modules
arising in Fourier expansions of automorphic forms. Maximal unipotent
subgroups $N$ of $\SU(2,1)$ are non-abelian, and we consider the
``abelian'' Fourier term modules connected to characters of $N$, and
also the ``non-abelian'' modules described with theta functions.
Poincar\'e series for $\SU(2,1)$ have in general exponential growth. To
deal with such generalized automorphic forms we allow exponential
growth for the functions in Fourier term modules. We give a complete
description of the submodule structure of all Fourier term modules, and
discuss the consequences for Fourier expansions of automorphic forms.
\end{abstract}
\maketitle

\iffilename\thispagestyle{empty}\fi
\tableofcontents

 \newpage
 


\def\flnm{rFtm-intro}

\markboth{1. INTRODUCTION}{1. INTRODUCTION}

\section{Introduction} The real Lie group $\SU(2,1)$ is the smallest
rank one Lie group with a non-abelian unipotent subgroup. This has the
consequence that automorphic forms have expansions in which, besides
Fourier terms based on a character of the unipotent group, also terms
with theta functions occur. To investigate the real-analytic Poincar\'e
series similar to those studied in \cite{MW89}, but associated to
non-abelian representations of the unipotent subgroup, one needs to
understand the Fourier-Jacobi expansion of functions on $\SU(2,1)$ with
exponential growth near the cusps.

Here we study, for a fixed Iwasawa decomposition $G=NAK$, the spaces
$C^\infty( \Ld\backslash G)_K$ of $K$-finite functions that are
invariant under a lattice $\Ld \subset N$. We do not impose conditions
on the growth of $f(nak)$ as $a$ varies in the subgroup~$A$. As a
simplification we restrict the discussion to a collection of lattices
that is invariant under conjugation by the normalizer $NAM$ of $N$.
Elements of $C^\infty(\Ld\backslash N)_K$ can be expanded in a general
type of Fourier expansion, which is the sum of two parts: the abelian
part, with terms parametrized by characters of $\Ld \backslash N$, and
the non-abelian part, with terms parametrized by an orthogonal
collection of realizations of the Stone-von Neumann representation in
$L^2(\Ld\backslash N)$. When applying this to automorphic forms we have
to deal with the subspace
\il{CiLdGps}{$C^\infty(\Ld\backslash G)^\ps_K$}%
$C^\infty(\Ld\backslash G)^\ps_K\subset C^\infty(\Ld\backslash G)_K$ in
which the center of the universal enveloping algebra of the Lie algebra
of $\SU(2,1)$ acts according to a character~$\ps$.

The linear operator that consists of taking a term in this expansion is
an intertwining operator from the Lie algebra module
$C^\infty(\Ld\backslash G)_K^\ps$ to a submodule
$\Ffu_\Nfu^\ps\subset C^\infty(\Ld\backslash G)_K^\ps$, where $\Nfu$
denotes a character of $ N$ or a realization of the Stone-von Neumann
representation.

If $\Nfu$ corresponds to a non-trivial character or an infinite
dimensional representation of $N$, then there are inside
$\Ffu_\Nfu^\ps$ two important submodules. The submodule $\Wfu_\Nfu^\ps$
is characterized by the property of exponential decay in $a\in A$. It
is present in $\Ffu^\ps_\Nfu$ for all $\Nfu$, except for $\Nfu_0$,
corresponding to the trivial character.

We parametrize $A$ by $\am(t)$ with $t\in (0,\infty)$ such that the
exponential decay occurs as $t\uparrow \infty$. The submodule
$\Mfu^\ps_\Nfu\subset \Ffu_\Nfu^\ps$ is defined by the behavior as
$t\downarrow 0$ by prescribing a convergent expansion. For $\ps$ in a
subset of characters of the enveloping algebra, elements of
$\Mfu^\ps_\Nfu$ can be used to define absolutely convergent Poincar\'e
series. For the Fourier expansion of such Poincar\'e series
(and their meromorphic continuation) we will need both $\Mfu^\ps_\Nfu$
and $\Wfu^\ps_\Nfu$.\smallskip

Various authors have studied the modules $\Wfu^\ps_\Nfu$. One says that
an irreducible representation admits a Whittaker module if it can be
realized in $\Wfu^\ps_\Nfu$ for a non-trivial character $\Nfu$.
(For $\SU(2,1)$ the non-trivial characters of $N$ are the generic
characters.) A representation that does not admit a Whittaker model is
called non-generic. Gelbart and Piatetski-Shapiro in their study
(\cite{GPS84}) of lifting from $U_{1,1}$ to $U_{2,1}$ and L-functions,
exhibit by adelic methods many non-generic representations (called by
them hypercuspidal).

The study of non-abelian Fourier terms is important in the theory of
automorphic forms, since for many representations abelian Fourier terms
turn out to be zero. This is connected with the Gelfand-Kirillov
dimension of modules. See \cite{Kos78} and \cite{Vo78}. Explicit
Fourier expansion of automorphic forms are given by Koseki and Oda
\cite{KO95}, and Ishikawa \cite{Ish99}, \cite{Ish00}. We mention also
Bao, Kleinschmidt Nilsson, Persson and Pioline \cite{BKNPP10} with a
study of holomorphic Eisenstein series and Eisenstein series on~$G/K$,
respectively.\smallskip

In this paper we do not impose a restriction of polynomial growth. We
determine the module structure of all Fourier term modules that arise
in this way. In the reducible cases we see considerable differences in
structure between the reducible principal series, the Fourier term
modules arising from a character of $N$, and the modules arising from
infinite-dimensional representations of~$N$. As far as we know, the
results on the last two types of modules are new. On the other hand,
information on the structure of reducible principal series
representations can be found in many places, for instance \cite{Kr72},
\cite{BB83}, \cite{Coll85}.

In this paper we use concepts and notations that seem suitable for the
Lie group $\SU(2,1)$, even when some of these concepts may be
suboptimal for more general semisimple Lie groups.

The group $\SU(2,1)$ is small among general semisimple Lie groups.
Nevertheless, the wish to work explicitly leads at several places to
computations of a size that are hard to carry out by hand and are more
suitable for symbolic computation with help of a computer. We explain
and carry out these computations in the Mathematica notebook
\cite{Math}, which we consider to be a substantial complement to this
paper, and refer to the relevant sections in the text. At many other
places we checked with Mathematica computations carried out by hand;
this can be found in the notebook as well, but is often not indicated
in the text.

\subsection{Overview of the main results} We study $(\glie,K)$-modules
in the space of $K$-finite functions in $C^\infty(\Ld_\s\backslash G)$
for a standard lattice $\Ld_\s$ (Definition~\ref{def-stlatt}) in the
unipotent subgroup $N$. The Fourier expansion of elements of
$C^\infty(\Ld_\s\backslash G)_K$ on which the center of the enveloping
algebra of $\glie=\mathrm{Lie}(G)$ acts by a character $\ps$ are built
from terms in the following $(\glie,K)$-modules:\il{Ftmo}{Fourier term
module}
\begin{itemize}
\item \emph{ Abelian Fourier term modules. }\il{Ftma}{Fourier term
module, abelian}$\Ffu^\ps_\bt$ consists of functions transforming on
the left according to a character $\ch_\bt$ of~$N$, $\bt\in \ZZ[i]$, as
defined in~\eqref{chbt}.
\item \emph{Non-abelian Fourier term modules. }\il{Ftmn}{Fourier term
module, non-abelian}The group $N$ is non-abelian.\\
The Stone-von Neumann representation of $N$ leads to modules
$\Ffu^\ps_\n$ with the abbreviation \il{n-intr}{$\n$}$\n=(\ell,c,d)$.
The elements of $\Ffu^\ps_\n$ are described by use of theta-functions
on~$N$. See \S\ref{sect-lnab}.

The parameter \il{ell-intro}{$\ell$}$\ell\in \frac \s 2\ZZ_{\neq 0}$
determines a character of the center of~$N$, and the parameter
\il{c-intro}{shift parameter}$c\in \ZZ\bmod2\ell$ determines a shift in
the theta functions. The \il{mpprm-i}{metaplectic
parameter}``metaplectic parameter'' \il{d-inftyo}{$d$ metaplectic
parameter}$d\in 1+2\ZZ$ determines a character of the double cover of
the group $M\subset K$ normalizing $NA$. See \S\ref{sect-thfu},
and~\S\ref{sect-lnab}.
\end{itemize}

One can parametrize the characters $\ps$ of the center of the enveloping
algebra $ZU(\glie)$ by characters of $AM$
(with $M\subset K$ normalizing $NA$). The characters $\xi$ of $M$
correspond to integers $j_\xi$, the characters of $A$ to complex
numbers $\nu$. The Weyl group $W$ of type $\mathrm{A}_2$, isomorphic to
the symmetric group $S_{\!3}$, acts on the elements $(j,\nu)\in \CC^2$,
and the orbits of $W$ in~$\CC^2$ correspond bijectively to the
characters $\ps$ of $ZU(\glie)$. For $\SU(2,1)$ only the intersection
of these orbits with $\ZZ\times\CC$ is relevant. This intersection,
denoted \il{Wops}{$\wo(\ps)$}$\wo(\ps)$, can have from zero to six
elements. For our purpose, this intersection $\wo(\ps)$ is relevant
only if it is non-empty. See
Table~\ref{tab-parms}, p~\pageref{tab-parms}, for a further discussion.

\rmrk{$N$-trivial Fourier term modules}\il{Ftm-Ntri}{Fourier term
module, $N$-trivial}The $N$-trivial Fourier term modules $\Ffu_0^\ps$
contain modules \il{HK-intro}{$H^{\xi,\nu}_K$}$H^{\xi,\nu}_K$ in the
\il{pssum}{principal series}principal series, discussed
in~\S\ref{sect-prs}. If $\wo(\ps)$ does not contain elements of the
form $(j,0)$, then by Proposition~\ref{prop-FFu0}
\be\label{psdcp} \Ffu^\ps_0 \= \bigoplus_{(j_\xi,\nu)\in \wo(\ps)}
H^{\xi,\nu}_K\,.\ee

If $\wo(\ps)$ contains elements of the form $(j,0)$ the principal series
modules $H^{\xi,\nu}_K$ with $(j_\xi,\nu)$ in $\wo(\ps)$ do not suffice
to obtain the whole of $\Ffu_0^\ps$. See Propositions \ref{prop-iLog}
and~\ref{prop-logIP}.

\rmrk{Submodules determined by boundary behavior} In the modules
$\Ffu_\Nfu^\ps = \Ffu_\bt^\ps$ with $\bt\neq 0$ and $\Ffu_\n^\ps$, we
define in \S\ref{sect-smbh} two classes of
submodules.\il{FfuNfu-intro}{$\Ffu^\ps_\bt,\; \Ffu^\ps_\n$}
\begin{itemize}
\item $\Wfu^{\xi,\nu}_\Nfu$, consisting of functions for which
$t\mapsto f\bigl( n \am(t)k)$ has exponential decay as
$t\uparrow\infty$. \il{Wfu-intro}{$\Wfu^{\xi,\nu}_\Nfu$} (We use the
Iwasawa decomposition, and the parametrization $t\mapsto \am(t)$ of $A$
by $(0,\infty)$ given in~\eqref{Adef}.)
\item $\Mfu_\Nfu^{\xi,\nu}$, consisting of functions for which
$t\mapsto f\bigl( n \am(t)k\bigr)$ has the form
$t\mapsto t^{2+\nu} \, h(t)$ with $h$ extending holomorphically to
$\CC$. \il{Mfu-intro}{$\Mfu^{\xi,\nu}_\Nfu$}
\end{itemize}

\rmrk{Fourier term modules under generic parametrization}We consider
first a character $\ps$ of $ZU(\glie)$ represented by elements
$(j,\nu)$ such that $\nu \not\equiv j\bmod 2$, or $(j,\nu)=(0,0)$. In
the terminology used in Table~\ref{tab-parms}, p~\pageref{tab-parms},
this is called \emph{generic parametrization}. For such $\ps$ the set
$\wo(\ps)$ often consists of two element $(j,\nu)$ and $(j,-\nu)$. It
may happen that $3\nu\equiv j\bmod 2$. Then $\wo(\ps)$ is a full Weyl
group orbit with six elements. We put\ir{Wops+}{\wo(\ps)^+}
\be\label{Wops+} \wo(\ps)^+ \= \bigl\{ (j,\nu) \in \wo(\ps) \;:\;
\re\nu\geq 0 \bigr\}\,.\ee

The irreducible representations of the maximal compact subgroup $K$ are
the $(p+1)$-dimensional representations
\il{tauhp-intro}{$\tau^h_p$}$\tau^h_p$ discussed in \S\ref{sect-irrpK}.
The parameters satisfy $p\in \ZZ_{\geq 0}$, $h\equiv p \bmod 2$.

\begin{mnthm}\label{mnthm-ab-gp}Let the character $\ps$ of $ZU(\glie)$
correspond to generic parametrization. Let $\bt \in \CC^\ast$.
\begin{enumerate}
\item[i)] For each $(j,\nu) \in \wo(\ps)^+$ the submodules
$\Wfu_\bt^{\xi,\nu}$ and $\Mfu_\bt^{\xi,\nu}$ of $\Ffu^\ps_\bt$ are
irreducible $(\glie,K)$-modules isomorphic to $H^{\xi,\nu}_K$ and to
$H^{\xi,-\nu}_K$.

The $K$-types $\tau^h_p$ in $\Wfu^{\xi,\nu}_\bt$ and in
$\Mfu^{\xi,\nu}_\bt$ satisfy $|h-2j_\xi|\leq 3p$. They occur in both
modules with multiplicity one.

\item[ii)]
$\Ffu^\ps_\bt \= \bigoplus_{(j,\nu)\in \wo(\ps)^+ } \Bigl( \Wfu^{\xi_j,\nu}_\bt \oplus \Mfu_\bt^{\xi_j,\nu} \Bigr)$.
\end{enumerate}
\end{mnthm}

For the non-abelian Fourier term modules $\Ffu^\ps_\n$ we need the
following integral quantity.
\be m_0(j) \= \frac16 \sign(\ell) \bigl( d-2j\bigr) - \frac12\,.\ee
We put\ir{Wops+n}{\wo(\ps)^+_\n}
\be\label{Wops+n} \wo(\ps)^+_\n \= \bigl\{ (j,\nu) \in \wo(\ps)^+\;:\;
m_0(j) \geq 0 \bigr\}\,.\ee

\begin{mnthm}\label{mnthm-nab-gp}Let the character $\ps$ of $ZU(\glie)$
correspond to generic parametrization. Let $\n$ be a parameter triple
$(\ell,c,d)$.
\begin{enumerate}
\item[i)] $\Ffu^\ps_\n$ is non-zero if and only if
$m_0(j)\in \ZZ_{\geq 0}$ for some $(j,\nu)\in \wo(\ps)^+$.
\item[ii)] For each $(j,\nu) \in \wo(\ps)^+_\n$ the submodules
$\Wfu_\n^{\xi,\nu}$ and $\Mfu_\n^{\xi,\nu}$ of $\Ffu_\n^\bt$ are
irreducible $(\glie,K)$-modules isomorphic to $H^{\xi,\nu}_K$ and to
$H^{\xi,-\nu}_K$.

The $K$-types $\tau^h_p$ in these modules have multiplicity one, and
satisfy $|h-2j_\xi|\leq 3p$.

\item[iii)]
$\Ffu^\ps_\n \= \bigoplus_{(j,\nu)\in \wo(\ps)^+ _\n } \Bigl( \Wfu^{\xi_j,\nu}_\n \oplus \Mfu_\n^{\xi_j,\nu} \Bigr)$.
\end{enumerate}
\end{mnthm}

Theorems \ref{mnthm-ab-gp} and~\ref{mnthm-nab-gp} summarize the results
of Sections~\ref{sect-gKm}--\ref{sect-spFtm}. The proof is completed on
p~\pageref{prfAB}.

\rmrk{Integral parametrization}A character $\ps$ of $ZU(\glie)$
corresponds to \emph{integral parametrization} if the elements
$(j,\nu) \in \wo(\ps)$ satisfy $\nu \equiv j \bmod 2$.

Under integral parametrization the structure of the $N$-trivial Fourier
term modules stays as indicated in~\eqref{psdcp}. However, the
principal series modules $H^{\xi,\nu}_K$ become reducible. Much more
generally than only for $\SU(2,1)$, one knows that all irreducible
$(\glie,K)$-modules occur as subquotients of some $H^{\xi,\nu}_K$
(Harish Chandra \cite{HCh54}), and even as submodules
(Casselman and Mili\v ci\'c \cite{CM82}).

For the other Fourier term modules the submodules $\Wfu_\Nfu^\ps$ and
$\Mfu^\ps_\Nfu$ become reducible as well. However, the way they fit
together in $\Ffu^\ps_\Nfu$ differs remarkably from the $N$-trivial
case: ``They coincide wherever they can.'' To formulate this more
precisely, we denote by \il{Ktp-intro}{$V_{h,p}$}$V_{h,p}$ the subspace
of $K$-type $\tau^h_p$ in the $(\glie,K)$-module~$V$.

\begin{mnthm}\label{mnthm-ab-ip}Let the character $\ps$ of $ZU(\glie)$
correspond to integral parametrization. Let $\bt\in \CC^\ast$.
\begin{enumerate}
\item[i)] The $(\glie,K)$-submodules $\Mfu^{\xi,\nu}_\bt$ and
$\Wfu^{\xi,\nu}_\bt$ of $\Ffu^\ps_\bt$ are reducible for each element
$(\xi,\nu)\in \wo(\ps)^+$. The $K$-types $\tau^h_p$ in these modules
have multiplicity one, and satisfy $|h-2j_\xi|\leq 3p$.
\item[ii)] If a $K$-type $\tau^h_p$ occurs in $\Mfu^{\xi,\nu}_\bt $ and
in $\Mfu^{\xi',\nu'}_\bt $ for
$(j_\xi,\nu), (j_{\xi'},\nu\}\in \wo(\ps)^+$, then
$\Mfu^{\xi,\nu}_{\bt;h,p} = \Mfu^{\xi',\nu'}_{\bt;h,p}$; and similarly
for the $\Wfu$-modules.
\item[iii)]
$\Ffu^\ps_\bt \= \Bigl( \sum_{(j,\nu)\in \wo(\ps)^+} \Mfu^{\xi,\nu}_\bt \Bigr) \oplus \Bigl( \sum_{(j,\nu)\in \wo(\ps)^+} \Wfu^{\xi,\nu}_\bt \Bigr)
$.
\item[iv)] The intersection
$\bigcap_{(j,\nu)\in \wo(\ps)^+} \Mfu^{\xi_j,\nu}_\bt$ is the unique
irreducible submodule of
$\sum_{(j,\nu)\in \wo(\ps)^+} \Mfu^{\xi,\nu}_\bt $; and the
intersection $\bigcap_{(j,\nu)\in \wo(\ps)^+} \Wfu^{\xi_j,\nu}_\bt$ is
the unique irreducible submodule of
$\sum_{(j,\nu)\in \wo(\ps)^+} \Wfu^{\xi,\nu}_\bt $.
\end{enumerate}
\end{mnthm}
In the terminology discussed in~\S\ref{sect-list-irr}, the irreducible
modules in iv) are of \il{ldst-intro}{large discrete series type}large
discrete series type.

Section~\ref{sect-abip} concerns the reducible generic abelian Fourier
term modules. Most of the results in Theorem~\ref{mnthm-ab-ip} are
stated in Lemma~\ref{lem-strab}. The proof is completed on
p~\pageref{prfC}.
\medskip

The non-abelian Fourier term modules $\Ffu^\ps_\n$ present a more
complicated structure. The modules $\Mfu^{\xi,\nu}_\n$ and
$\Wfu^{\xi,\nu}_\n$ may coincide in some or in all $K$-types. So a
decomposition as in iii) of Theorem~\ref{mnthm-ab-ip} cannot hold. An
analogous decomposition holds if we define other submodules
\il{Vfu-intro}{$\Vfu^{\xi,\nu}_\n$}$\Vfu^{\xi,\nu}_\n$ to take the role
of $\Mfu$ in the decomposition. See \eqref{upsab}. The definition of
$\Vfu^{\xi,\nu}_\n$ is not intrinsic, but it serves to give us some
hold on the complications.

\begin{mnthm}\label{mnthm-nab-ip}Let the character $\ps$ of $ZU(\glie)$
correspond to integral parametrization. Let $\n=(\ell,c,d)$ be a
non-abelian parameter triple.

\begin{enumerate}
\item[i)] For each $(\xi,\nu)\in \wo(\ps)_\n^+$ the modules
$\Vfu^{\xi,\nu}_\n$, $\Wfu^{\xi,\nu}_\n$ and $\Wfu^{\xi,\nu}_\n$ are
reducible, and are in general non-isomorphic.

These modules contain, with multiplicity one, the $K$-types $\tau^h_p$
satisfying $|h-2j_\xi|\leq 3p$.

\item[ii)] Let $\X$ denote $\Vfu$, $\Wfu$ or $\Mfu$. If a $K$-type
$\tau^h_p$ occurs in $\X^{\xi,\nu}_\n$ and in $\X^{\xi',\nu'}_\n$ for
$(j_\xi,\nu) , (j_{\xi'},\nu') \in \wo(\ps)^+_\n$, then
$\X^{\xi,\nu}_{\n;h,p} = \X^{\xi',\nu'}_{\n;h,p}$.
\item[iii)]
$\Ffu^\ps_\n \= \Bigl( \sum_{(j,\nu)\in \wo(\ps)^+_\n} \Vfu^{\xi,\nu}_\n \Bigr) \oplus \Bigl( \sum_{(j,\nu)\in \wo(\ps)^+_\n} \Wfu^{\xi,\nu}_\n\Bigr)
$.
\end{enumerate}
Denote
$\Vfu^\ps_\n =  \sum_{(j,\nu)\in \wo(\ps)^+_\n} \Vfu^{\xi,\nu}_\n$ and
define $\Wfu^\ps_\n$ and $\Mfu^\ps_\n$ similarly. We put
$j_r= \max\bigl\{j\;:\; (j,\nu) \in \wo(\ps)^+\bigr\}$ and
$j_l=\min\bigl\{ (j,\nu)\in \wo(\ps)^+\bigr\}$.
\begin{enumerate}
\item[iv)] The modules $\Mfu^\ps_\n$ and $\Vfu^\ps_\n$ intersect
non-trivially in the following cases.
\begin{enumerate}
\item[a)] If $\wo(\ps)_\n^+ = \wo(\ps)^+$, then
$\Mfu^\ps_\n = \Vfu^\ps_\n$.
\item[b)] If $\wo(\ps)_\n^+ \neq \wo(\ps)^+$, and $\ell>0$, then
$\Mfu^\ps_{\n;h,p} = \Vfu^\ps_{\n;h,p}$ for all $K$-types $\tau^h_p$
that satisfy the additional condition $\bigl|h-2j_r\bigr|\leq 3p$.
\item[c)] If $\wo(\ps)_\n^+ \neq \wo(\ps)^+$, and $\ell<0$, then
$\Mfu^\ps_{\n;h,p} = \Vfu^\ps_{\n;h,p}$ for all $K$-types $\tau^h_p$
that satisfy the additional condition $\bigl|h-2j_l\bigr|\leq 3p$.
\end{enumerate}
\item[v)] The modules $\Mfu^\ps_\n$ and $\Wfu^\ps_\n$ have a non-trivial
intersection in the following cases.
\begin{enumerate}
\item[a)] If $\ell>0$, $m_0(j_l)\geq 0$ and $m_0(j)<0$ for other
$(j,\nu)\in \wo(\ps)^+$, then $\Mfu^\ps_{\n;h,p} = \Wfu^\ps_{\n;h,p}$
for all $K$-types $\tau^h_p$ that satisfy
$\bigl| h- 2j_l\bigr| \leq 3p$, and $\bigl|h-2j\bigr|>3p$ for
$(j,\nu)\in \wo(\ps)^+$, $j \neq j_l$.
\item[b)] If $\ell<0$, $m_0(j_r)\geq 0$ and $m_0(j)<0$ for other
$(j,\nu)\in \wo(\ps)^+$, then $\Mfu^\ps_{\n;h,p} = \Wfu^\ps_{\n;h,p}$
for all $K$-types $\tau^h_p$ that satisfy
$\bigl| h- 2j_r\bigr| \leq 3p$ and $\bigl|h-2j\bigr|>3p$ for
$(j,\nu)\in \wo(\ps)^+$, $j \neq j_r$.
\end{enumerate}
\end{enumerate}
\end{mnthm}

In the long section~\ref{sect-nab} we study the reducible non-abelian
Fourier term modules. The figures on
pp~\pageref{fig-j3}--\pageref{fig-mj1a} illustrate the variety of
submodule structures that we meet in non-abelian Fourier term modules.
All isomorphism types of $(\glie,K)$-modules that are not in the
irreducible principal series occur in $\Ffu^\ps_\n$ for certain
combinations of the parameters. The proof of Theorem~\ref{mnthm-nab-ip}
is completed on p~\pageref{prfD}.
\bigskip

Our motivation to study Fourier term modules arose from the wish to get
grip on the full Fourier-Jacobi expansion of automorphic forms, with
both abelian and non-abelian terms. The automorphic forms that we
consider are $K$-finite functions in $C^\infty(\Gm\backslash G)$ on
which the center of the universal enveloping algebra acts according to
a character. The discrete subgroup $\Gm$ is cofinite and not cocompact.
The choice to work with standard lattices in $N$ leads to a condition
on the intersections $\Gm\cap N_\k$, where $N_\k$ is the unipotent
subgroup fixing the cusp~$\k$. See \S\ref{sect-cucond}.

In Chapter~\ref{chap-4} we apply the results in the earlier chapters to
automorphic forms, guided by the needs in our paper \cite{BM2}, devoted
to the study of Poincar\'e series on $\SU(2,1)$ and their completeness.
Besides the usual automorphic forms with at most polynomial growth at
the cusps, we consider automorphic forms with moderate exponential
growth. In~\S\ref{sect-famaf} we consider meromorphic and holomorphic
families of automorphic forms. In \S\ref{sect-sqiaf} we give the form
of the Fourier expansion of generators of irreducible modules of square
integrable automorphic forms. Expansions similar to the latter ones are
given by Ishikawa \cite{Ish99}, \cite{Ish00}; see the comparison at the
end of \S\ref{sect-sqiaf}.



\def\flnm{rFtm-chap-I}
\newcounter{tabold}
\newcounter{figold}
\setcounter{tabold}{\arabic{table}}
\setcounter{figold}{\arabic{figure}}

\renewcommand\thechapter{\Roman{chapter}} \newpage
\chapter{The Lie group SU(2,1) and subgroups}\label{chap-1}
\markboth{I. THE LIE GROUP SU(2,1)} {I. THE LIE GROUP SU(2,1)}
\setcounter{section}{1}\setcounter{table}{\arabic{tabold}}
\setcounter{figure}{\arabic{figold}} This is a preparatory chapter. We
choose a realization of the Lie group $\SU(2,1)$ inside $\SL_3(\CC)$.
The Iwasawa decomposition $NAK$ of the group $\SU(2,1)$ involves a
non-abelian unipotent subgroup $N$ and a maximal compact subgroup $K$.
The representations of these groups are considered, in Section
\ref{sect-K} for $K$, and in Section \ref{sect-N} for~$N$. All this is
known; we give a summary, as a preparation for the later chapters.

The aim of this paper is to understand the modules involved in the
Fourier expansions of functions on $\Gamma\backslash \SU(2,1)$ for
discrete subgroups $\Gm$. Section \ref{sect-ds} discusses discrete
subgroups and Fourier expansions.


\def\flnm{rFtm-I-SU}

\section{Realization of the group SU(2,1)}
\label{sect-G}\markright{2. REALIZATION OF  SU(2,1) }
We choose a realization of the group $\SU(2,1)$, and we describe the
Iwasawa decomposition, the Bruhat decomposition, and the symmetric
space associated to~$\SU(2,1)$.\medskip

The unitary group \il{su21}{$\SU(2,1)$}$\SU(2,1)$ is the group of
matrices $g\in \SL_3(\CC)$ that preserve a given hermitian form of
signature $(2,1).$ Different hermitian forms give isomorphic
realizations of $\SU(2,1)$ as a real Lie group. We use the hermitian
form $(x,y) = \bar y^t I_{2,1} x,$ with \ir{I21}{I_{2,1}}
\be \label{I21}I_{2,1}\=
\begin{pmatrix}1&0&0\\
0&1&0\\ 0&0&-1
\end{pmatrix}\,,\ee
which leads to\ir{Gdef}{G=\SU(2,1)}\il{remn}{realization of $\SU(2,1)$}
\be \label{Gdef}G\isdd\SU(2,1) = \bigl\{ g \in \SL_3(\CC)\;:\; \bar g^t
I_{2,1} g = I_{2,1}\bigr\}\,.\ee

This defines $G$ as a semi-simple Lie group of dimension $8$ with real
rank one. We choose the \il{Iwd}{Iwasawa decomposition}Iwasawa
decomposition $G=NAK,$ with the \il{mcs}{maximal compact
subgroup}maximal compact subgroup
\il{Kdef}{$K$}\ir{Kdef}{\km(\eta,\al,\bt)} \=
\badl{Kdef} &K= \bigl\{ \km(\eta,\al,\bt) \;: \; \eta,\al,\bt\in \CC,\,
|\eta|=1,\, |\al|^2+|\bt|^2=1\bigr\}\,,\\
&\km(\eta,\al,\bt) \= \km(-\eta,-\al,-\bt)\=
\begin{pmatrix}
\eta\al&\eta\bt&0\\
-\eta\bar\bt&\eta\bar \al&0\\
0&0& \eta^{-2}
\end{pmatrix}\,; \eadl
the \il{upsg}{unipotent subgroup}unipotent subgroup
\il{Ndef}{$N$}\ir{Ndef}{\nm(x+iy,r)
= \nm(x,y,r)}
\badl{Ndef}
&N= \bigl\{ \nm(b,r)\;:\; b\in \CC,\, r\in \RR\bigr\}\,,\\
&\nm(b,r) = \nm(\re b,\im b,r) =
\begin{pmatrix} 1+ir-\frac{|b|^2}2 & b&
- ir + \frac{|b|^2}2\\
-\bar b& 1 & \bar b\\
ir -\frac{|b|^2}2 &b & 1-ir+\frac{|b|^2}2
\end{pmatrix}\,; \eadl
and the connected component of $1$ in an $\RR$-split \il{sto}{split
torus}torus\il{Adef}{$A$}\ir{Adef}{\am(t)}
\badl{Adef}
&A= \bigl\{ \am(t)\;:\; t>0\bigr\}\,,\quad \am(t)=
\begin{pmatrix}\frac{t+t^{-1}}2&0& \frac{t-t^{-1}}2\\
0&1&0\\
\frac{t-t^{-1}}2&0& \frac{t+t^{-1}}2
\end{pmatrix}\,. \eadl

The group product for $N$ is
\be\label{Nmult} \nm(b,r) \nm(b_1,r_1) = \nm\bigl(b+b_1,r+r_1+\im(\bar b
b_1)
\bigr)\,.\ee
The commutative group $AM$ with\il{mztdef}{$\mm(\z)$}\il{Mdef}{$M$}
\bad M&=\bigl\{ \mm(\z)\;:\;|\z|=1\bigr\} \,\subset\, K\,,\\
\mm(\z)
&=
\begin{pmatrix}\z&0&0\\0&\z^{-2}&0\\0&0&\z
\end{pmatrix} = \km( \z^{-1/2}, \z^{3/2},0)\,, \ead
normalizes $N$:
\be \label{Nnorm}\am(t) \mm(\z) \nm(b,r)
\mm(\z)^{-1}\am(t)^{-1} = \nm\bigl( \z^3 t b,t^2 r)\,. \ee

At some places it is convenient to use $\nm(x,y,r) = \nm(x+i y,r)$, with
three real parameters.

\rmrk{Other realizations}We use the same realization as Ishikawa
\cite{Ish99}. This realization has the advantage that $K$ has the
simple form $\Biggl\{\begin{pmatrix}\ast&\ast&0\\
\ast&\ast&0\\0&0&\ast \end{pmatrix}\Biggr\}.$ Isomorphic realizations
are obtained by replacing the matrix $I_{2,1}$ in~\eqref{I21} by
$J=\bar U^t I_{2,1} U$ or $J=-\bar U^t I_{2,1} U$ with
$U\in \GL_3(\CC).$ Then we obtain the isomorphic Lie group $U^{-1}G U.$

We mention the realization used by Francsics and Lax \cite{FL}. In their
realization the group $N$ has upper triangular form. A closely related
realization is used in~\cite{BKNPP10}.

\rmrk{Rational structure} We can view $G$ as the group $\G_\RR$ of real
points of an algebraic group $\G$ over $\QQ.$ This can be done by
viewing $\SL_3$ as an algebraic group over $\QQ(i)$, and obtaining $\G$
as an algebraic subgroup of the \il{Wr}{Weil restriction}Weil
restriction $R_{\QQ(i)/\QQ} \SL_3$. See \cite[\S1.3]{We61}. The group
of rational points $\G_\QQ$ can be identified with the subgroup of
$g\in G$ that have matrix coefficients in $\QQ(i).$

With other realizations of $G,$ we can follow the same approach. This
may lead to other rational structures. In this paper, we do not
consider $\SL_3$ as an algebraic group over other imaginary quadratic
number fields.
\subsection{Symmetric space} \label{sect-symsp}
The \il{syms}{symmetric space}symmetric space corresponding to
$\SU(2,1)$ is the quotient $G/K.$ We use the realization as the upper
half-plane model: \ir{Xdef}{\X}\il{uhpm}{upper half-plane model}
\be \label{Xdef}\X = \Bigl\{ (z,u)\in \CC^2\;:\; |u|^2<\im
z\Bigr\}\,.\ee
As an analytic variety it is diffeomorphic to $NA,$ which is visible in
the left action
\be \label{na-actX}\nm(b,r)\am(t)\cdot(z,u)
= \bigl( t^2 z
+2tbu+2r+i|b|^2, tu+i\bar b\bigr)\,,\ee
which satisfies $\nm(b,r)\am(t) \cdot (i,0)=
  \bigl( t^2 i +2r + i|b|^2, i\bar b\bigr)
.$ The group $K$ leaves the point $(i,0)$ fixed. The action of general
elements $k\in K$ is complicated, except for elements of the form
$\mm(\z)$ or $\wm$ \il{wmdef}{$\wm$}
\badl{wm-actX} \wm =
&\begin{pmatrix}-1&0&0\\0&-1&0\\0&0&1\end{pmatrix}
= \km(1,-1,0)\,,\\
&\wm\cdot(z,u) = \Bigl( \frac{-1}z, \frac{-iu}z \Bigr)\,,\qquad
\mm(\z)\cdot(z,u)
= \bigl( z,\z^{-3}u\bigr)\,. \eadl

The space $\X$ inherits the complex structure of $\CC^2.$ The action of
all elements of $G$ preserves this complex structure.
\rmrk{Boundary}\il{bdss}{boundary of symmetric space} The boundary of
$\partial \X$ can be described as
\be \partial \X = \{\infty\}\sqcup \bigl\{
(z,u)\in \CC^2\:;\; |u|^2=\im z\bigr\}\,,\ee
where \il{iuhp}{$\infty\in\partial\X$}$\infty$ is the limit of $(z,u)$
as $\im z\rightarrow\infty$ while $\re z$ and $u$ stay bounded.

The point $\infty\in \partial\X$ is fixed by the parabolic subgroup
$N A M\subset G.$ We can write $\partial\X$ as the disjoint union
\be\label{sdcp} \partial\X = \{\infty\bigr\} \sqcup N \cdot\bigl(
0,0)\,.\ee
We have $(0,0)= \wm \cdot \infty.$
\medskip

Relation \eqref{sdcp} leads to the \il{Bd}{Bruhat decomposition}Bruhat
decomposition
\be G = NAM \sqcup N \wm NAM \,.\ee
With the notation\ir{hmdef}{\hm(\cdot)}
\be\label{hmdef} \hm(c) = \am\bigl(|c|\bigr)\, \mm\bigl(c/|c|\bigr)
\qquad c\in \CC^\ast\,,\ee
each element of $G$ can be written uniquely as either $g=n\hm(c),$
$n\in N$ and $c\in \CC^\ast,$ or $g=n_1\wm\, \hm(c) n_2$ with
$n_1,n_2\in N$ and $c\in \CC^\ast.$

To go from $g$ in the \il{bc}{big cell}big cell $N\wm MNA$ of the Bruhat
decomposition to the \il{Id1}{Iwasawa decomposition}Iwasawa
decomposition we use the following lemma.
\begin{lem}\label{lem-BI}
For $b\in \CC,$ $r\in \RR,$ $t>0,$ and $c\in \CC^\ast$:
\begin{align*}
\wm &\, \hm(c) \nm(b,r)\am(t) = \nm(b',r')\am(t')
k'\,,\text{ with}\\
b'&= \frac{-cb}{\bar c^2 D}\,,\qquad r'= \frac{-r}{|c|^2|D|^2}\,,\qquad
t'=\frac t{|c|\,|D|}\,,
\displaybreak[0]\\
D&=2ir+t^2+|b|^2\,,
\displaybreak[0]\\
k'&=\begin{pmatrix}
c(\bar D-2t^2)/|c|\,|D|&
-2ctb/|c|\,|D| & 0\\
2\bar c t \bar b/c\bar D& \bar c(D-2t^2)/c\bar D& 0\\
0&0& c\bar D/|c|\,|D|
\end{pmatrix}
\end{align*}
\end{lem}
\begin{proof}
We find this relation by applying $\wm\, \hm(c)\, \nm(b,r)\,\am(t)$ to
$(i,0)\in \X$. That gives the values of $b'$, $r'$ and $t'$, and then
$k_1$ by a computation. We carried out the computation in
\cite[\S1b]{Math}.
\end{proof}


\def\flnm{rFtm-I-K}


\section{Maximal compact subgroup} \markright{3. MAXIMAL COMPACT
SUBGROUP}\label{sect-K}We discuss the structure of the maximal compact
Lie subgroup $K\subset G$ chosen in~\eqref{Kdef}, and describe its
irreducible representations. We give explicit realizations of these
representations in polynomial functions on~$K$.\medskip

The description of $K$ in \eqref{Kdef} amounts to \il{Kdefa}{K}
\be\label{Kprod} K= \Bigl ( \U(1) \times \SU(2) \Bigr) \bigm/
\bigl\{\pm1 \bigr\}\,,\ee
where ${\pm 1}$ is embedded diagonally in the product. See (1), (2)
on p.~185 of \cite{Wal76}. The presence of $\U(1)$ shows that $K$ is not
simply connected. It has an infinite covering.

We denote by $K_0\subset K$ the group corresponding to $\SU(2)$. We
have, in the notation of \eqref{Kdef}: \ir{K0def}{K_0}
\be\label{K0def} K_0 \= \bigl\{ \km(1,\al,\bt)\;:\; \al,\bt\in\CC\,,\;
|\al|^2+|\bt|^2=1\bigr\}\,.\ee
The subgroup $\bigl\{ \km(\eta,1,0)\;:\; |\eta|=1\bigr\}$ of $K$
corresponding to $\U(1)$ is the center of $K$. The characters are
induced by characters of the center of $K$ that are trivial on $K_0$.
These characters are of the form
\be \label{xidef}\xi_j:\km(\eta,\al,\bt) \mapsto \eta^{-2j}\,,\ee
with $j\in \ZZ$,\il{chM}{characters of $M\subset K$} and extend to
characters of~$K$. On the subgroup $M\subset K$ these characters take
the form $\xi_j: \mm(\z) \mapsto \z^j$. \il{j}{$j = j_\xi$} If a
character $\xi$ of $K$ is given, we use \il{xi}{$\xi_j$}$j_\xi\in \ZZ$
for the corresponding integral parameter.

\subsection{Irreducible representations} \label{sect-irrpK}The
isomorphism classes \il{taup}{$\tau_p$ irreducible representation of
$\SU(2)$}$\tau_p$ of irreducible representations of $K_0$ are
parametrized by their dimension $p+1,$ with $p\in \ZZ_{\geq 0}.$ The
irreducible representations of the center of $K$ are the characters
parametrized by $h\in \ZZ.$ With the product description~\eqref{Kprod}
the isomorphism classes $\tau^h_p$ of irreducible representations of
$K$ are parametrized by $(h,p)\in \ZZ^2,$ $p\geq 0,$
$h\equiv p\bmod 2:$\ir{tauhp}{\tau^h_p \text{ irreducible
representation of $K$}}
\be\label{tauhp} \tau^h_p
\begin{pmatrix}\z^{1/2}u&0\\
0&\z^{-1}\end{pmatrix}
=
(\z^{1/2})^{-h}\, \tau_p(u)\,.\ee
We use here that each $k\in K$ can be written as
$k=\km(\z^{1/2},\al,\bt)= \km(-\z^{1/2},-\al,-\bt)$.

\rmrk{Realizations of irreducible representations of $\SU(2)$} To
realize the representations of $K_0 \cong\SU(2)$ in $C^\infty(K)$
provided with the action by right translation we start with polynomial
functions $\Kph{}prq$ on $ K_0 \cong \SU(2)$ by the identity
\ir{PhiSU}{\Kph{}prq}
\badl{PhiSU}
(ax+&c)^{(p-q)/2}\, (bx+d)^{(p+q)/2} \\
&= \sum_r \Kph{}prq\begin{pmatrix}a&b\\
c&d\end{pmatrix}\, x^{(p-r)/2}\,, \eadl
where $r$ and $q$ are in $p+2\ZZ,$ with absolute value at most $p.$ See
\cite[Chapter 6]{VK91}, where the function $t^{p}_{r,q}$ in (2),
Section 6.3.2 is a multiple of $\Kph{}prq.$

The generating function shows that right translation by each element of
$\SU(2)$ preserves the space $\sum_{r,q} \CC\, \Kph{}prq,$ where we let
both $r$ and $q$ run from $-p$ to $p$ in steps of~$2$. Right
translation
\[ \begin{pmatrix}a&b\\
c&d\end{pmatrix}\mapsto
\begin{pmatrix}a&b\\
c&d\end{pmatrix}\begin{pmatrix}\eta&0\\0&\eta^{-1}
\end{pmatrix}
= \begin{pmatrix} a\eta& b/\eta\\ c \eta
&d/\eta\end{pmatrix}\]
multiplies $\Kph{}prq$ by $\eta^{-q}.$ With use of the generating
function we can also check that left translation by
$\begin{pmatrix}\eta&0\\0&\eta^{-1}
\end{pmatrix}$ multiplies each $\Kph {}p r q$ by $\eta^{-r}.$ Since left
and right translations commute, the action of $\SU(2)$ by right
translations preserves the eigenspaces for left translation. Hence we
get $p$ invariant subspaces $\sum_q \CC \Kph{}{p}rq.$ There are $p+1$
realizations of $\tau_p$ in $C^\infty(K).$ In \S\ref{sect-klie} we give
the Lie algebra action on these spaces, which shows their
irreducibility and the isomorphism between them.

\rmrk{Realizations of irreducible representations of $K$} In view of
\eqref{tauhp} we put\ir{Kphd}{\Kph hp r q}
\be \label{Kphd} \Kph hp r q \begin{pmatrix}
\z^{1/2}u &0\\
0&\z^{-1}
\end{pmatrix} = \z^{-h/2} \, \Kph{}prq(u)\,.\ee
This is well defined for $h\equiv p\bmod 2,$ independently of the choice
of $\z^{1/2}.$ We need $p\equiv r\equiv q\bmod 2$ and $|r|,|q|\leq p.$

Left translation by $\begin{pmatrix} \eta&0&0\\0&\eta^{-1}&0\\0&0&1
\end{pmatrix}$ acts on $\Kph hp r q $ by multiplication by $\eta^{-r}.$
Right translation by elements of $K$ preserve the spaces
$\sum_ q \CC\Kph hp r q .$ This gives $p+1$ realizations
\il{tauhrp}{$\tau^h_{r,p}$ realization of $\tau^h_p$}$\tau^h_{r,p}$ of
$\tau^h_p$ in $C^\infty(K).$ The function $\Kph{2j}000$ is the
character $\xi_j$ of $K$ in~\eqref{xidef}.

\rmrk{Polynomial functions} The functions $\Kph hp r q $ are polynomial
functions in the matrix elements of $k\in K$. For $k=\km(\eta,\al,\bt)$
as in~\eqref{Kdef} we have, with $j\in \ZZ$ and
$h\equiv 1\bmod 2$:\ir{Phi-exmp}{\Kphh{p}{r}{q}}
\badl{Phi-exmp} \Kph {2j}000(k) &\= \eta^{-2j}\,,\\
\Kph h{1}{1}{1}(k)
&\= \eta^{-h}\,\bar \al\,,&\qquad \Kph h{1}{1}{-1}(k)
&\= \eta^{-h}\,\bar \bt\,,\\
\Kph h{1}{-1}{1}(k) &\=\z^{-h}\, \bt\,,& \Kph h{1}{-1}{-1}(k)
&\=\eta^{-h}\, \al\,. \eadl
Manipulations with the generating function in~\eqref{PhiSU} lead to the
multiplication relations \il{mlt}{multiplication relations for
functions on $K$} in Table~\ref{tab-mlt}.
\begin{table}[ht]
\begin{align*} \text{ For }\eta\equiv 1\bmod 2,
&\;\; \e,\z\in \{1,-1\}:\\
\Kph\eta{1}{\e}{\z} & \Kph hp r q = A_{\e,\z}(p,r) \,
\Kph{h+\eta}{p+1}{r+\e}{q+\z}\\
&\qquad\hbox{}
+ B_{\e,\z}(p,r,q) \, \Kph{h+\eta}{p-1}{r+\e}{q+\z}\,,\\
A_{\e,\z}(p,r)
&= \frac{p+\e r+2}{2(p+1)}\,,\\
B_{\e,\z}(p,r,q) &=
\begin{cases}
\frac{\e(\z p-q)}{2(p+1)}
&\text{ if } \e r\leq p-2\,,\\
0 &\text{ if }\e r=p\,.\end{cases}
\end{align*}
\caption{Multiplication relations for polynomial functions
on~$K$.}\label{tab-mlt}
\end{table}
\medskip

The polynomial functions $\Kph h{p}{r}{q}$ form an orthogonal basis of
$L^2(K)$. We normalize the \il{HmK}{Haar measure}Haar measure on $K$
such that \il{dk}{$dk$}$\int_K dk=1$. We will not need an explicit
formula for $\bigl\|\Kph hp r q \bigr\|_K$, but will use the
relation\ir{Kph-norm}{\text{norm in }L^2(K)}
\be\label{Kph-norm}
\|\Kph hp r q \|^2_K = \frac{p! } {\bigl( \frac{p+r}2\bigr)!\; \bigl(
\frac{p-r}2\bigr)!} \|\Kph h{p}{p}{q}\|^2_K\,. \ee
This can be checked using the relation
\[ \bigl ( L(\Z_{21})
\ph_1,\ph_2\bigr)_K
+ \bigl(\ph_1,L(\Z_{12})\ph_2 \bigr)_K =0 \,. \]
See Table~\ref{tab-RLdK} below.

\subsection{Lie algebra}\label{sect-klie}
A basis of the real \il{La}{Lie algebra}Lie algebra \il{klie}{$\klie$
real Lie algebra of $K$}$\klie$ of $K$ is
\bad\CK_i
&=\begin{pmatrix}i&0&0\\
0&i&0\\
0&0&-2i
\end{pmatrix}\,,&
\qquad \WW_0 &=
\begin{pmatrix}i&0&0\\
0&-i&0\\
0&0&0
\end{pmatrix}\,,\\
\WW_1 &=\begin{pmatrix}
0&1&0\\-1&0&0\\
0&0&0\end{pmatrix}\,,& \WW_2
&=
\begin{pmatrix}
0&i&0\\i&0&0\\
0&0&0\end{pmatrix}\,.
\ead\il{CKi}{$\CK_i\in \klie$}
\il{WW}{$\WW_0, \WW_1, \WW_2 \in \klie$} The element $\CK_i$ spans the
Lie algebra of the center of~$K$, the three remaining elements span the
Lie algebra of $K_0$. The exponentials are
\ir{LieKexp}{\text{exponentials of basis elements of $\klie$}}
\badl{LieKexp} \exp(t\CK_i) &= \km(e^{it},1,0)\,,&
\qquad \exp(t\WW_0)&= \km(1,e^{it},0)\,,\\
\exp( t\WW_1)&= \km(1,\cos t,\sin t)\,,& \exp(t\WW_2)&= \km(1,\cos
t,i\sin t)\,. \eadl
The element
\be \label{HHi} \HH_i = \frac 32\WW_0 - \frac12\CK_i=
\begin{pmatrix}i&0&0\\0&-2i&0\\
0&0&i\end{pmatrix} \in \klie\ee
spans the Lie algebra of $M$, and
\be \label{mHHi}\exp(t\HH_i)\= \mm(e^{it})\,. \ee

\rmrk{Actions of the Lie algebra on the polynomial functions} The action
of $\XX\in \klie$ by \il{rd}{right differentiation}right
differentiation of functions in $C^\infty(K)$ is given by
\ir{Rdiff}{\XX f = R(\XX)f}
\be\label{Rdiff} \XX f\,(g)= \partial_t
f\bigl( g\exp(t\XX)
\bigr)|_{t=0}\,.\ee
This is extended $\CC$-linearly to an action
of~\il{kliec}{$\klie_c$}$\klie_c=\CC\otimes_\RR \klie$, for instance to
$\Z_{12}=\WW_1-i\WW_2$ and
$\Z_{21}= \WW_1+i\WW_2$.\il{Z12}{$\Z_{12},\; \Z_{21}$} We write
$R(\XX)f $ for $\XX f$ in discussions where other actions by
differentiation occur as well, for instance the action by left
differentiation.

Left differentiation is the right action of $\klie$ given
by\ir{Ldiff}{L(\XX)}\il{ld}{left differentiation}
\be \label{Ldiff}L(\XX) f(g) = \partial_t
f\bigl( \exp(t\XX)g\bigr)\bigm|_{t=0}\qquad (g\in G,\; \XX\in
\klie)\,.\ee

\begin{table}
\begin{align*} R(\CK_i) \Kph hp r q & \;=\; L(\CK_i) \Kph hp r q \;=\;
- i h \, \Kph hp r q \,,\\
R(\WW_0) \Kph hp r q & \;=\; -i q \;\Kph hp r q \,,\quad L(\WW_0) \Kph
hp r q \;=\; -i r\;\Kph hp r q \,,\\
R(\Z_{21})\Kph hp r q & \;=\; (q-p)\;\Kph h{p}{r}{q+2}\,,\\
L(\Z_{21})
\Kph hp r q & \;=\;
\begin{cases}
(r-p-2) \Kph h{p}{r-2}{q}&\text{ if }r\geq 2-p\,,\\
0&\text{ if }r=-p\,,\end{cases}\\
R(\Z_{12}) \Kph hp r q & \;=\; (q+p)\;\Kph h{p}{r}{q-2}\,,\\
L(\Z_{12}) \Kph hp r q & \;=\;\begin{cases}
(r+p+2) \Kph h{p}{r+2}{q} &\text{ if }r\leq p-2\,,\\
0&\text{ if }r=p\,.
\end{cases}
\end{align*}
\caption{Actions of $\klie_c$ by left and right differentiation.}
\label{tab-RLdK}
\end{table}
Table~\ref{tab-RLdK} gives the left and right actions of the Lie algebra
on basis elements. Some of these relations are easily seen from the
definition, for instance we have seen that $\exp(t\WW_0)$ acts on
$\Kph hp r q $ under left translation by multiplication by $e^{-ir t}$,
and under right translation by multiplication by $e^{-i q t}$. This
gives the actions of~$\WW_0$. In \eqref{Kphd} we see that the center of
$K$ acts by $\km(\eta,1,0) \mapsto \z^{-h}$; this leads to the action
of $L(\CK_i)=R(\CK_i)$. The actions of $\Z_{12}$ and $\Z_{21}$ take
more computations, carried out in \cite[\S4b]{Math}.

In the formulas for $R(\Z_{21})$ and $R(\Z_{12})$ we have the factor
$q\mp p$, which becomes zero if $q$ has the value $\pm p$ for which
$\frac q2\pm 1$ threatens to be outside the range of~$q$. For left
differentiation these values of $q$ have to be treated separately.

\rmrk{Parameters and eigenvalues} We have seen that $h$ is determined by
the eigenvalue $-i h$ of $\CK_i$ in $\sum_r\tau^h_{r,p}$. The parameter
$p$ is determined by the action of the \il{Cas-K}{Casimir
element}Casimir element of the Lie algebra of $K_0$:\ir{CasK}{C_K}
\be\label{CasK}
C_K = \WW_0^2+\WW_1^2+\WW_2^2 = \WW_0^2-2i\WW_0+\Z_{12}\Z_{21} \ee
It acts in $\sum_r \tau^h_{r,p}$ by multiplication by $-p(p+2)$.


\def\flnm{rFtm-I-N}


\section{Unipotent subgroup}
\label{sect-N}\markright{4. UNIPOTENT SUBGROUP}The
group $\SU(2,1)$ is the smallest semisimple Lie group of rank one with
unipotent subgroups that are not commutative.

The representation theory of the unipotent subgroup $N$ is important for
the Fourier expansion of automorphic forms and for Poincar\'e series.
Our main aim is to give an orthonormal basis of
$L^2(\Ld_\s\backslash N)$ for a class of standard lattices
$\{\Ld_\s\;:\; \s\in \ZZ_{\geq 1}\}.$ This requires a discussion of the
Stone-von Neumann representation and its realizations by means of theta
functions. We checked many computations in this section in \cite[\S3c,
\S5]{Math}.
\medskip

The group $N$ is a realization of the \il{Hg}{Heisenberg
group}Heisenberg group. It fits into the exact sequence
\be\label{Nexs} 1\longrightarrow Z(N)
\longrightarrow N \longrightarrow \RR^2 \longrightarrow 0\,, \ee
where \il{ZN}{$Z(N)$}$Z(N) = \bigl\{ \nm(0,r)\;:\; r\in \RR\bigr\}$ is
the center of $N.$ The homomorphism $N \rightarrow \ZZ^2$ is given by
$\nm(x,y,r) \mapsto (x,y).$

The group \il{AutN}{$\Aut(N)$}$\Aut(N)$ of \il{aN}{automorphism group of
$N$}continuous automorphisms of $N$ is isomorphic to the semi-direct
product $\GL_2(\RR) \ltimes \RR^2.$ Here $\RR^2$ corresponds to the
group of interior automorphisms. Furthermore,
$\begin{pmatrix}a&b\\c&d\end{pmatrix}
\in \GL_2(\RR)$ corresponds to the outer automorphism
\be \label{oautN}
\nm(x,y,r) \mapsto \nm(a x + b y, c x + d y,(a d-b c)r)\,. \ee
Conjugation by $\am(t)\in A$ corresponds to the automorphism of~$N$
given by the matrix $\begin{pmatrix} t&0\\0&t\end{pmatrix},$ and
conjugation by $\mm(e^{i	h})$ to the matrix
$\begin{pmatrix}\cos3	h& -\sin 3	h\\   \sin3	h&\cos 3	h\end{pmatrix}.$

\rmrk{Comparison} We use the multiplication relation on $N$ given
in~\eqref{Nmult}. In \cite{Tha98} Thangavelu uses the multiplication
relation
\be [x,y,t]\,[u,v,s]= \bigl[x+u,y+v,t+s+\frac12(uy-vx)
\bigr]\,.\ee
The isomorphism
\be\label{Tdef} T :\nm(x,y,r) \mapsto [x,2y,-r]\ee
relates both realizations of the Heisenberg group.

\rmrk{Lie algebra} A basis of the real Lie algebra \il{Lie1}{Lie
algebra}\il{nlie}{$\nlie$}$\nlie$ of $N$
is\ir{XX012}{\XX_0,\XX_1,\XX_2\in \nlie}
\badl{XX012} \XX_0&=\begin{pmatrix}
\frac i2&0&-\frac i2\\0&0&0\\
\frac i2&0&-\frac i2
\end{pmatrix}\,,\quad \XX_1=\begin{pmatrix}
0&1&0\\-1&0&1\\0&1&0
\end{pmatrix}\,,\\
\XX_2&=\begin{pmatrix}
0&i&0\\i&0&-i\\0&i&0
\end{pmatrix}
\,. \eadl
The sole non-zero commutator of these elements is
\be \bigl[\XX_1,\XX_2] = - [\XX_2,\XX_1] = 4\XX_0\,. \ee

The \il{expnlie}{exponentials for basis elements of $\nlie$}exponential
map gives
\badl{expnlie} \exp(t\XX_0)& = \nm(0,t/2)\,,
\quad \exp(t\XX_1) = \nm(t,0)\,,\\
\exp(t\XX_2)&= \nm(it,0)\,. \eadl
See \cite[\S3c]{Math}.
\subsection{Characters and Stone-von Neumann representation}There are
two types of unitary irreducible representations of $N$, namely the
unitary characters, and the \il{StvNr}{Stone-von Neumann
representation}Stone-von Neumann representations, which are infinite
dimensional.

Characters of $N$ are trivial on the center, so they are characters of
$\RR^2.$ They have the form\ir{chbt}{\chi_\bt}
\be\label{chbt} \chi_\bt: \nm(b,r)
\mapsto e^{2\pi i \re(\bar \bt b)} \ee
with $\bt\in \CC.$ All other irreducible unitary representations of $N$
have infinite dimension. See \cite[Theorem 1.2.4]{Tha98}. The center
$Z(N)$ acts by multiplication by a non-trivial character of $Z(N).$ For each non-trivial central character there is only one isomorphism class of irreducible representations.

\rmrk{Schr\"odinger representation}The \il{Schrep}{Schr\"odinger
representation}Schr\"odinger representation is a realization of the
Stone-von Neumann representation in $L^2(\RR).$ It depends on a
non-trivial character
$$ \nm(0,r) \mapsto e^{i\ld r} $$
of $Z(N),$ parametrized by $\ld\in \RR^\ast.$

The Schr\"odinger representation $\pi_{\ld}\bigl(\nm(x,y,r)\bigr)$
applied to $\ph$ in the space
\il{Schsp}{$ \Schw(\RR)$}$ \Schw(\RR)\subset L^2(\RR)$ of
\il{Sf}{Schwartz function}Schwartz functions on $\RR$ is given by
\ir{pild}{\pi_\ld}
\be\label{pild}
\pi_{\ld}\bigl(\nm(x,y,r)\bigr)\ph
(\xi)
= e^{i\ld ( r-2\xi x
- xy )}\ph(\xi+y)\,.\ee
The Schwartz space is invariant under these transformations. The
operators $\pi_{\ld}(n)$ extend to $L^2(\RR)$, and determine $\pi_\ld$
as a unitary representation of $N$ in $L^2(\RR)$.
\rmrk{Comparison}Thangavelu \cite[(1.2.1)]{Tha98} uses the
representation
\be \pi_\mu^T [x,y,r] \ph (\xi) = e^{i\mu(r+\xi x+xy/2)}\,
\ph(\xi+y)\,.\ee
With $U \ph(\xi) = \ph(\xi/2)$ and the isomorphism $T$ in~\eqref{Tdef}
we have
\be \pi_{-\ld}^T\bigl( T(\nm(x,y,r)\bigr)
U \ph = U \,\pi_\ld\bigl(\nm(x,y,r)
\bigr)
\ph\,.\ee
This shows that both representations are equivalent.
\rmrk{Automorphisms}For each automorphism $A\in \Aut(N)$ the
representation $n\mapsto \pi_\ld(A n)$ is equivalent to some
Schr\"odinger representation $\pi_{\ld'}.$ So there exists a unitary
map $U_A:L^2(\RR)\rightarrow L^2(\RR)$ and a number $\ld'\in \RR^\ast$
such that
\be U_A \pi_{\ld}(An) = \pi_{\ld'}(n) U_A\,.\ee

If the automorphism $A$ corresponds to
$\begin{pmatrix}t&0\\0&t\end{pmatrix}
\in \GL_2(\RR),$ see \eqref{oautN}, then we can take
$U_A\ph(\xi)= t\ph(t\xi)$ and $\ld'=t^2 \ld.$

\rmrk{Lie algebra action} Application of \eqref{expnlie} gives the
derived represen\-ta\-tion:\ir{dpild}{d \pi_\ld}
\badl{dpild} d\pi_\ld
(\XX_0)\ph&=\frac i2\ld\,\ph\,,
&\qquad d\pi_{\ld}
(\XX_1)\ph(\xi) &= -2i\ld\, \xi\, \ph(\xi)\,,\\
d\pi_{\ld} (\XX_2)
\ph&=\ph'\,. \eadl
This derived action is well defined if $\ph$ is a Schwartz function.

\subsection{Theta functions}\label{sect-thfu} The Stone-von Neumann
representation can be realized in spaces generated by theta functions
on $N$ modulo a lattice, i.e., \il{latt}{lattice} a discrete subgroup
such that the quotient $\Ld\backslash N$ is compact.

\begin{defn}\label{def-stlatt} We denote by \il{Ldsg}{$\Ld_\s$}$\Ld_\s$
the lattice generated by
\be \nm(1,0,0)\,,\quad\nm(0,1,0)\,,\quad \nm(0,0,2/\s)\,.\ee
We call the lattices $\Ld_\s$ \il{stl}{standard lattice}standard
lattices. \end{defn}
Since the commutator of the generators $\nm(1,0,0)$ and $\nm(0,1,0)$ is
$\nm(0,0,2),$ we need the restriction that $\s\in \ZZ_{\geq 1}.$

Any lattice is isomorphic to a standard lattice by an element in
$\Aut(N)$.
\smallskip

Let $\s\in \ZZ_{\geq 1}$ and $\ld \in \RR_{\neq 0}.$ The central
character of the Schr\"odinger representation $\pi_\ld$ is
$\nm(0,0,r) \mapsto e^{ i \ld r}.$ To have it trivial on
$\Ld_\s \cap Z(N)$ we take $\ld=2\pi \ell$ with
$\ell \in \frac \s 2\ZZ_{\neq 0}.$\il{ell}{$\ell$ parameter for
character of $N$}

The space of Schwartz functions $\Schw(\RR)$ is dual to the space
$\Schw'(\RR)$ of tempered distributions. Under this duality the
Schr\"odinger $\pi_{2\pi \ell}$ on $\Schw(\RR)$ corresponds to
$\pi_{-2\pi \ell}$ on $\Schw'(\RR)$. The relations
\bad \pi_{-2\pi \ell}\bigl( \nm(1,0,0)\bigr) \dt_a&\= e^{4\pi i \ell c}
\dt_a\\
\pi_{-2\pi \ell}\bigl( \nm(0,1,0)\bigr) \dt_a &\= \dt_{a-1}\\
\pi_{-2\pi \ell}\bigl( \nm(0,0,2/\s) \bigr)\dt_a&\= e^{-4\pi i \ell
/\s}\, \dt_a
\ead
imply that the distribution
\be \mu_{\ell,c} \= \sum_{k\in \ZZ} \dt_{k+c/2\ell} \ee
is invariant under $\pi_{2\pi  \ell}(\Ld_\s)$ if we take
$c\in \ZZ\bmod 2\ell$. This motivates the definition
\be \Th_{\ell,c}(\ph)(n)\= \Th_{\ell,c}(\ph;n) \= \bigl[ \pi_{2\pi
\ell}(n) \ph , \, \mu_{\ell,c}\bigr]\ee
as a function on $\Ld_\s\backslash N$. Writing this out explicitly gives
for $\ph \in \Schw(\RR)$ the following
series\ir{Thdef}{\Th_{\ell,c}}\il{c-theta}{$c$ shift parameter for
theta functions}
\be\label{Thdef}
\Th_{\ell,c}\bigl(\ph;\nm(x,y,r)\bigr)
= \sum_{k\in \ZZ} e^{2\pi i \ell\bigl(r-x(c/\ell+2k+y)\bigr)}\,
\ph\Bigl( \frac c{2\ell}+k
+y\Bigr)\,. \ee
\il{Thf}{theta function}The decay of Schwartz functions ensures absolute
convergence of the series and of all its derivatives with respect to
the coordinates $x,$ $y$ and~$r.$ It transforms via the central
character determined by $\ld = 2\pi \ell,$ and some computations show
that it is left-invariant under multiplication by elements of $\Ld_\s.$
Actually, we have
\badl{ThNshift} \Th_{\ell,c}\bigl( \ph; \nm(1/2\ell,0,0)n\bigr)
&= e^{-\pi i c /\ell}\, \Th_{\ell,c}(\ph;n)\,,\\
\Th_{\ell,c}\bigl( \ph;\nm(0,1/2\ell,0)n\bigr)
&\ = \Th_{\ell,c+1}(\ph;n)\,. \eadl
So we have
$\Th_{\ell,c}(\ph) \in   C^\infty\bigl(\Ld_\s\backslash N\bigr).$

Since $\nm(x,y,r) \nm(t,0,0) = \nm(x+t,y,r-ty)$ we obtain
\badl{Thd1}\XX_1 \Th_{\ell,c}\bigl(\ph;\nm(x,y,r)\bigr)
&= \bigl( \partial_x - y\partial_r\bigr)
\Th_{\ell,c}\bigl(\ph;\nm(x,y,r)\bigr)\\
&=-4\pi i \ell \Th_{\ell,c}\bigl( \ph_1;\nm(x,y,r)\bigr)
\eadl
with $\ph_1 (\xi) = \xi\ph(\xi).$ Proceeding in a similar way we get
\badl{Thd2} \XX_2 \Th_{\ell,c}\bigl( \ph;\nm(x,y,r)\bigr) &\=
\Th_{\ell,c}\bigl(\ph';\nm(x,y,r)
\bigr)\,,\\
\XX_0 \Th_{\ell,c}( \ph;n) &\= \pi i \ell \Th_{\ell,c}(\ph;n)\,.\eadl
Comparison with~\eqref{dpild} shows that $\ph \mapsto \Th_{\ell,c}(\ph)$
induces an intertwining operator between $\pi_{2\pi\ell}$ and the
subspace spanned by theta functions $\Th_{\ell,c}(\ph)$, inducing a
unitary injection $L^2(\RR)\rightarrow L^2(\Ld_\s\backslash N)$. Thus
we have $2\ell$ realizations of the Stone-von Neumann representation in
the space of functions on $\Ld_\s\backslash N.$ See \cite[\S5c]{Math}.
\smallskip

Let us choose the \il{Hm2}{Haar measure}Haar measure \il{dn}{$dn$}$dn$
on $N$ as $dn= dx\, dy\, dr$ for $n=\nm(x,y,r),$ with the Lebesgue
measure on $\RR$ in each of the coordinates. With this choice, $\Ld_\s$
has covolume $\frac \s 2.$ By taking apart the summations in
\eqref{Thdef} we find
\be \bigl( \Th_{\ell,c}(\ph),
\Th_{\ell',c'}(\ps)\Bigr)_{\Ld_\s\backslash N} =
\begin{cases}
\frac 2\s\, (\ph,\ps)_\RR&\text{ if } \ell=\ell',\; c\equiv c'\bmod
2\ell\,,\\
0&\text{ otherwise}\,.
\end{cases}
\ee
So the $2|\ell|$ realizations of the Stone-von Neumann representation
are mutually orthogonal, and we have injective linear maps
$ \bigl( L^2(\RR) \bigr)^{2|\ell|} \rightarrow L^2(\Ld_\s\backslash N)$
induced by
\be
(\ph_c)_{c\bmod 2\ell} \mapsto \sqrt{\frac \s2} \sum_{c\bmod 2\ell}
\Th_{\ell,c}(\ph_c)\,.\ee
The image is contained in the subspace $L^2(\Ld_\s\backslash N)_\ell$
determined by the central character corresponding to $\ell.$ It is
known that this map is a unitary isomorphism; below we indicate an
argument. The conclusion is that the orthogonal complement of
$L^2(\Ld_\s\backslash N)_0$ in $L^2(\Ld_\s\backslash N)$ is described
by theta functions.

\rmrk{Argument for unitarity} Let
$F\in C^\infty(\Ld_\s\backslash N)_\ell$ be orthogonal to
$\Th_{\ell,c}(\ph)$ for all $c$ and $\ph.$ To see that this implies
that $F$ vanishes we consider
\be\label{fF} f(x,y) = e^{2\pi i \ell xy} F\bigl( \nm(x,y,0)\bigr)\,.\ee
Then $f(x+1,y) = f(x,y),$ and $f(x,y+1)=e^{4\pi i \ell x} f(x,y).$ The
assumption implies that
\begin{align*}
0 &= \int_{x=0}^1 \int_{y=0}^1 F\bigl( \nm(x,y,0)
\bigr)
\, \overline{\Th_{\ell,c}(\ph)\bigl(\nm(x,y,0)\bigr)}\, dy\, dx
\displaybreak[0]\\
&\= \int_{x=0}^1\int_{y=0}^1 f(x,y) \, e^{-2\pi i \ell x y} \sum_k
e^{2\pi i \ell x(c/\ell+2k+y)}\, \overline{\ph\bigl( c/2\ell+k+y)}\,
dy\,dx
\displaybreak[0]\\
&= \int_{x=0}^1 \int_{y=0}^1 f(x,y) \, \sum_k e^{2\pi i \ell
x(c/\ell+2k)}\, \overline{\ph\bigl( c/2\ell+k+y)}\, dy\, dx\\
&= \int_{x=0}^1 \sum_k \int_{y=0}^1 f(x,y+k)\, e^{2\pi i c x}
\,\overline{\ph\bigl( c/2\ell+k+y\bigr)}\, dy \, dx\\
&= \int_{x=0}^1 e^{2\pi i c x}\, \int_{y=-\infty}^\infty f(x,y) \,
\overline{\ph(y+c/2\ell)}\, dy\, dx\,.
\end{align*}
This holds for all $c\in \ZZ$ and all Schwartz functions $\ph$. In
particular, replacing $\ph$ by a translate depending on $c$, we get for
all $c\in \ZZ$ and all Schwartz functions $\ph$
\begin{align*}
0&\=\int_{x=0}^1 e^{2\pi i c x} h_{\ph}(x) \, dx\,,\\
h_{\ph}(x) &\= \int_{y=-\infty}^\infty f(x,y) \, \overline{\ph(y)}\,
dy\,.
\end{align*}
The function $h_\ph$ is continuous, $1$-periodic, and all its Fourier
coefficients vanish. So $h_\ph(x)=0$ for all $x\in \RR$. The Schwartz
functions $\ph$ are dense in $L^2(\RR)$ so for each $x\in \RR$ the
function $y\mapsto f(x,y)$ is zero. Hence $F$ vanishes.

\rmrk{Hermite basis}Normalized Hermite functions provide us with a
suitable basis of Schwartz functions to use in the theta functions. See
\cite[\S1.4]{Tha98}.

The \il{Hpol}{Hermite polynomial}Hermite polynomials $H_m$ are
determined by the identity
\[ e^{-\xi^2}\, H_m(\xi)
= (-1)^m \,\partial_\xi^m e^{-\xi^2}\,.\]
The \il{nHf}{normalized Hermite function}normalized Hermite functions
are the following Schwartz functions:\ir{hlmdef}{h_{\ell,m}}
\be\label{hlmdef} h_{\ell,m}(\xi) = 2^{2-m/2}\, |\ell|^{1/4} \,
(m!)^{-1/2}\, H_m\bigl(\sqrt{4\pi |\ell|}\, \xi\bigr) \,
e^{-2\pi|\ell|\xi^2}\ee
for $\ell \in \RR_{\neq 0}$ and $m\in \ZZ_{\geq 0}.$ For any given
$\ell\in \RR_{\neq 0}$
\be \bigl( h_{\ell,m},h_{\ell,m'} \bigr)_\RR = \dt_{m,m'}\,.\ee
\begin{table}[tp]
\[ \begin{array}{|c|c|}
\hline
m& h_{\ell,m}(\xi) \\ \hline
0 & \sqrt 2 |\ell|^{1/4}\, e^{-2\pi|\ell|\xi^2} \\
1& 4\sqrt\pi |\ell|^{3/4}\,\xi\, e^{-2\pi|\ell|\xi^2} \\
2& |\ell|^{1/4}\,\bigl( 8\pi|\ell|\xi^2-1)\, e^{-2\pi|\ell|\xi^2} \\
3&2\sqrt{\frac{2\pi}3} |\ell|^{3/4}\, \bigl(
8\pi|\ell|\xi^2-3\bigr)\,\xi\, e^{-2\pi |\ell|\xi^2}
\\ \hline
\end{array}\]
\caption{Some normalized Hermite functions.}\label{tab-herm}
\end{table}

The Hermite polynomials satisfy the relation
$H_{m+1}=2\xi \, H_m - 2m \, H_{m-1}$. This leads to relations between
Hermite function $h_{\ell,m},$ $h_{\ell,m+1},$ and $h_{\ell,m-1}.$ The
derived Schr\"odinger representation is described in
Table~\ref{tab-hermdiff}.
\begin{table}[ht]
\begin{align*} d\pi_{2\pi \ell}&(\XX_0)
h_{\ell,m} = \pi i \ell \, h_{\ell,m}\,,\\
d\pi_{2\pi\ell} & (\XX_1) h_{\ell,m} =
-4\pi i \ell \, \xi\, h_{\ell,m}\\
&= -2 i \,\sign(\ell)
\sqrt{\pi|\ell|} \,\biggl( \sqrt{\frac m2}\, h_{\ell,m-1} +
\sqrt{\frac{m+1}2}\, h_{\ell,m+1}\biggr)\,,\\
d\pi_{2\pi \ell}&(\XX_2) h_{\ell,m} = h_{\ell,m}' = 2\sqrt{\pi|\ell|}
\biggl( \sqrt{\frac m2} \, h_{\ell,m-1} - \sqrt{\frac{m+1}2}\,
h_{\ell,m+1}\biggr)\,. \end{align*}
\caption{Derived Schr\"odinger representation on Hermite
functions.}\label{tab-hermdiff}
\end{table}

Since $\ph \mapsto \Th_{\ell,c}(\ph)$ is an intertwining operator for
the action of $\nlie$ (and of $N$) we have the corresponding relations
for theta functions built with normalized Hermite functions.
\smallskip

In the direct sum decomposition
\[ L^2(\Ld_\s\backslash N) = \bigoplus_{\ell \in (\s/2)\ZZ}
L^2(\Ld_\s\backslash N)_\ell\]
according to the central character, we have orthonormal bases
\be \label{bas-Th}
\bigl\{ \sqrt{\s/2}\,\Th_{\ell,c}(h_{\ell,m})\;:\; c\bmod 2\ell,\, m\in
\ZZ_{\geq 0}\bigr\}\ee
for each summand with $\ell\neq 0.$ The character $\chi_\bt$ is trivial
on $\Ld_\s$ for $\bt\in \ZZ[i].$ An orthonormal basis of
$L^2(\Ld_\s\backslash N)_0$ is
\be \label{bas0}
\bigl \{ \sqrt{\s/2}\,\chi_\bt \;:\; \bt\in \ZZ[i]\bigr\}\,.\ee

The factor $\sqrt{\s/2}$ is caused by the choice to use
$dn = dx\, dy\,dr$ as the Haar measure on~$\Ld_\s\backslash N$, with
$n=\nm(x,y,r)$. It seems natural to work with characters and theta
functions that do not depend on~$\s$.

\rmrk{Automorphisms of $N$ and theta functions}The map
$\nm(b,r) \mapsto \nm(ib,r)$ is an outer automorphism of $N$ leaving
invariant the lattice $\Ld_\s\subset N$. It is given by
$n \mapsto \mm(-i) n \mm(i)$.
\begin{prop}\label{prop-actmi} Let $m\in \ZZ_{\geq 0}$, $\s\in \ZZ$, and
$\ell\in \frac \s2\ZZ_{\neq 0}$. The automorphism
$\nm(b,r) \mapsto \nm(ib,r)$ of $N$ induces in the
$2|\ell|$-dimensional space of theta functions with basis
$\bigl\{ \Th_{\ell,c}(h_{\ell,m)}\;:\; 0\leq c < 2|\ell|\bigr\}$ the
linear transformation determined by
\be \Th_{\ell,c}(h_{\ell,m}) \bigl( \nm(ib,r) \bigr) \=
\frac{\bigl(-i\sign(\ell)\bigr)^m}{\sqrt{2|\ell|}}\sum_{c'=0}^{2|\ell|-1}
e^{\pi i c c'/\ell} \Th_{\ell,c'} (h_{\ell,m}\bigl( \nm(b,r)
\bigr)\,.\ee
\end{prop}

We give the proof in two lemmas.
\begin{lem}Let $\ell \in \frac \s 2 \ZZ_{\neq 0}$, and
$c\in \{0,1,\ldots,2|\ell|-1\}$. For each $\ph \in \Schw(\RR)$
\be \Th_{\ell,c}(\ph)\bigl( \nm(ib,r) \bigr) \= \frac1{\sqrt{2|\ell|}}
\sum_{c'=0}^{2|\ell|-1} e^{\pi i c c'/\ell} \, \Th_{\ell,c'}\bigl(
f_\ell \ph\bigr)\bigl( \nm(b,r)\bigr)
\,,\ee
where
\be \bigl(f_\ell \ph\bigr)(\xi) \= \sqrt{2|\ell|}\, \hat \ph(2\ell
\xi)\,.\ee
\end{lem}
\begin{proof}We write out
$\th_{\ell,c}(\ph) \bigl( \nm(y- i x,r) \bigr)$ with \eqref{Thdef} and
apply the following consequence of Poisson's summation formula
\be \sum_{k\in \ZZ} e^{2\pi i \bt k} \,\ph(\al+k)
\= e^{-2\pi i \al \bt}\sum_{k\in \ZZ} e^{2\pi i \al k}\,
\hat\ph(k-\bt)\,.\ee
That leads to the relation
\begin{align*} \Th_{\ell,c}(\ph)&\bigl(\nm(y-ix,r) \bigr) \= e^{2\pi i
\ell(r - x y)} \sum_{k\in \ZZ} e^{2\pi i (c/2\ell-x) k }\, \hat \ph(k+2
\ell y)
\displaybreak[0]\\
&\= e^{2\pi i \ell(r-xy)} \sum_{c'=0}^{2|\ell|-1}\sum_{\k\in \ZZ}
  e^{2\pi i (c/2\ell-x)(c'+2\ell k)} \,\hat \ph\bigl( c'+2\ell(k+y)\bigr)
  \displaybreak[0]\\
  &\=\frac1{\sqrt{2|\ell|}} \sum_{c'=0}^{2|\ell|-1} e^{\pi i c c'/\ell}
  \,\Th_{\ell,c'}\bigl( f_\ell \ph\bigr)\bigl( \nm(x+i
  y,r)\bigr)\,.\qedhere
\end{align*}
\end{proof}

\begin{lem}For $\ell \in \frac \s2 \ZZ_{\neq 0}$ and $m\in \ZZ$
\be \label{flrel} f_\ell h_{\ell,m} \= \bigl( -i \sign(\ell) \bigr)^m\,
h_{\ell,m}\,.\ee
\end{lem}
\begin{proof}The case $m=0$ can be checked by an explicit computation,
based on Table~\ref{tab-herm}, p~\pageref{tab-herm}. We start with the
third relation in Table~\ref{tab-hermdiff}, p~\pageref{tab-hermdiff},
and take the Fourier transforms of all terms:
\[ 2\pi i \xi \hat h_{\ell,m}(\xi) \= \sqrt{2\pi|\ell|}\Bigl( \sqrt{m}
\hat h_{\ell,m-1}(\xi) - \sqrt{m-1} \hat h_{\ell,m+1}(\xi)\Bigr)\,. \]
Replacing $\xi$ by $2\ell \xi$ we formulate this in terms of
$f_\ell h_{\ell,m}$:
\[ 4\pi i \ell \sqrt{2|\ell|} f_\ell h_{\ell,m} \= 2|\ell| \sqrt
\pi\Bigl( \sqrt m f_\ell h_{\ell,m-1} - \sqrt{m-1} f_\ell
h_{\ell,m+1}\Bigr)\,. \]
Using \eqref{flrel} for $h_{\ell,m}$ and $h_{\ell,m-1}$, and taking the
second relation in Table~\ref{tab-hermdiff} into account, we see that
\eqref{flrel} is valid for $h_{\ell,m+1}$ as well.
\end{proof}


\def\flnm{rFtm-I-dsg}


\section{Discrete subgroups} \label{sect-ds}\markright{5. DISCRETE
SUBGROUPS} Our purpose in this paper is to provide a background for the
study of functions on $G=\SU(2,1)$ that are invariant under certain
lattices in~$G$.

\subsection{Condition on the cuspidal lattices}\label{sect-cucond}
We use discrete subgroups $\Gm\subset G$ that have finite covolume and
are not cocompact. We will impose below one further
condition.\smallskip

By a \il{cu}{cusp}cusp $\c$ of $\Gm$ we mean here a parabolic subgroup
$P_\c$ such that its unipotent radical $N_\c\subset P_\c$ intersects
$\Gm$ in a lattice in \il{Nc}{$N_\c$}$N_\c.$ We use the name cusp also
for the unique point in the boundary of the symmetric space~$\X$ fixed
by~$P_\c$.\il{Pc}{$P_\c$}

All cusps are conjugate to each other by elements of $G=\SU(2,1)$.
Conjugation by $\gm\in \Gm$ results in finitely many $\Gm$-orbits of
cusps. For each cusp $\c$ we can choose elements $g\in G$ such that
$P_\c = g NAM g^{-1}.$ We impose the following additional condition on
the cofinite discrete subgroups $\Gm$ that we
consider:\il{cuc}{$\ZZ[i]$-condition on the cusps}
\begin{defn}\label{cucond} If for each cusp $\c$ there is
\il{gc}{$g_\c$}$g_\c\in G$ such that
$g_\c \Ld_{\s(\c)} g_\c^{-1} =\Gm \cap N_\c$ for
\il{sgc}{$\s(\c)$}$\s(\c)\in \ZZ_{\geq 1}$, we say that the group $\Gm$
satisfies the $\ZZ[i]$-condition on the cusps.
\end{defn}

This condition needs to be checked only for cusps in a system of
representatives of the $\Gm$-orbits of cusps, since if $\c'=\gm \c$ for
a given $\gm$, we can choose $\gm_{\c'}=\gm g_\c$.
\rmrk{Discussion}All lattices in $N$ are isomorphic to a standard
lattice $\Ld_\s$ (Definition~\ref{def-stlatt}). We discussed that the
automorphism group of $N$ is isomorphic to $\GL_2(\RR)\ltimes \RR^2$.
Interior automorphisms of $N$ give the factor $\RR^2$. Conjugation by
an element of $AM\subset G$ gives an automorphism that acts in
$N/Z(N) \cong \RR^2$ corresponding to a map $b \mapsto \al b$ for some
$\al\in \CC^\ast$ under the identification $\RR^2\cong \CC$ via
$(1,0) \leftrightarrow 1$, $(0,1)\leftrightarrow i$. See \eqref{Nnorm}.
Each standard lattice has image $\ZZ[i]$ in $\CC \cong N/Z(N)$. The
$\ZZ[i]$-condition ensures that we have to deal only with isomorphisms
of $N$ that are natural for the chosen identification $\RR^2\cong \CC$.

\rmrk{Examples}We may take the integral points in the standard
realization:\il{Gmst}{$\Gm_0$}
\be\label{Gmst} \Gm_0= G \cap \SL_3(\ZZ[i])\,.\ee
From \eqref{Ndef} we see that $\Gm_0\cap N$ consists of the elements
$\nm(b,r)$ with $b$ in the ideal $(1+\nobreak i)$ in $\ZZ[i]$ and
$r\in \ZZ.$ So $\infty$ is a cusp of~$\Gm_0$, and
$\Gm\cap N = g_\infty \Ld_4 g_\infty^{-1}$ with $g_\infty=\hm[i+1]
\= \am(\sqrt2)\,\mm(e^{\pi i/4})$. We will show in Appendix
\ref{app-ds} that the cusps form one $\Gm_0$-orbit. Hence $\Gm_0$
satisfies the $\ZZ[i]$-condition on the cusps.

Francsics and Lax \cite{FL} give explicitly a fundamental domain of a
discrete subgroup of $\SU(2,1)$ with one cusp. It is built with another
realization than we use here.
\badl{FL} & \SL_3(\ZZ[i]) \cap U_\fl^{-1} G U_\fl\,,\\
U_\fl &\=
\begin{pmatrix}\frac{-i}{\sqrt2}& 0 &\frac{1}{\sqrt2}\\
0&1&0 \\
\frac{-i}{\sqrt2} & 0 &\frac{-1}{\sqrt2}\end{pmatrix}\,,\quad \bar
U_\fl^T I_{2,1} U_\fl \= \begin{pmatrix}0&0&i\\
0&1&0\\-i&0&0\end{pmatrix}\,.
\eadl
This group, once conjugated to $G$ by $U_{\mathrm{FL}}$, satisfies the
$\ZZ[i]$-condition on the cusps. See Appendix~\ref{app-ds}.

\subsection{Fourier expansions}\label{sect-Fe1}We consider Fourier
expansions of functions on $G$ that are invariant under a standard
lattice in~$N$.\smallskip

Let \il{CLGK}{$C^\infty(\Ld_\s\backslash G)_K$
}$C^\infty(\Ld_\s\backslash G)_K,$ $\s\in \ZZ_{\geq 1},$ denote the
space of smooth functions on $G$ that are left-$\Ld_\s$-invariant on
the left and $K$-finite on the right.

Let $f\in C^\infty(\Ld_\s\backslash G)_K.$ For each $a\in A,$ $k\in K$
the function $n\mapsto f(nak)$ lies in
$C^\infty(\Ld_\s\backslash N) \subset L^2(\Ld_\s\backslash N),$ and can
be expanded in terms of the orthonormal basis in \eqref{bas-Th}
and~\eqref{bas0}:\il{Fe}{Fourier expansion}\il{fbt}{$f_\bt$}
\il{fellcm}{$f_{\ell,c,m}$}
\begin{align}
\nonumber f(nak) &= \sum_{\bt\in \ZZ[i]} \chi_\bt(n) \, f_\bt(ak)\\
\nonumber
&\qquad\hbox{}
+ \sum_{\ell \in (\s/2)\ZZ_{\neq 0}}\sum_{c\bmod 2\ell} \sum_{m\in
\ZZ_{\geq 0}} \Th_{\ell,c}(h_{\ell,m};n)\, f_{\ell,c,m}(ak)\,,
\displaybreak[0]\\
\label{Febt}
\text{where } \quad f_\bt(ak) &= \frac\s 2 \int_{\Ld_\s\backslash
N}\overline{\chi_\bt(n')}\, f(n'ak)\, dn'\,,
\displaybreak[0]\\
\label{Fen}
f_{\ell,c,m}(ak) &= \frac \s 2 \int_{\Ld_\s\backslash N}\overline{
\Th_{\ell,c}(h_{\ell,m};n')}\, f(n'ak)\, dn'\,.
\end{align}

\begin{prop}\label{prop-absconv}The Fourier expansion converges
absolutely for each function $f\in C^\infty(\Ld_\s\backslash N)_K,$ and
for all derivatives $uf$ with $u$ in the universal enveloping algebra
$U(\nlie).$\il{Fsac}{Fourier series, absolute convergence}
\end{prop}
\begin{proof}We know that the Fourier series converges in the sense of
$L^2(\Ld_\s\backslash N).$ In the coordinates
$(x,y,r) \leftrightarrow \nm(x,y,r)$ on $N$ we have the elliptic
operator
\[ L= \XX_0^2+\XX_1^2+\XX_2^2= \bigl(
\partial_x-y\,\partial_r)^2
+ (\partial_y+x\,
\partial r)^2+\frac14\partial_r^2\,.
\]
This operator acts on $f_\bt$ with eigenvalue $-4\pi^2|\bt|^2$ and on
$\Th_{\ell,c}(h_{\ell,m})$ with eigenvalue $-(4+8m+\pi|\ell|)|\ell|.$
Since $f$ is smooth we can apply partial integration as many times as
we want. In this way we obtain better and better estimates, uniform for
$ak$ in compact sets. Now we apply Sobolev theory; see for instance
\cite[Appendix 4]{La74}. The Sobolev inequality \cite[p 393]{La74}
bounds the supremum norm in terms of the second Sobolev norm, and the
basic estimate \cite[p 401]{La74} bounds the second Sobolev norm in
terms of the $L^2$-norm of $f$ and $Lf.$ This gives pointwise
convergence.

We can differentiate $n\mapsto f(nak)$ as many times as we want, and
also interchange differentiation and taking Fourier terms.
\end{proof}

For $f\in C^\infty(\Ld_\s\backslash G)_K$ and $\bt \in \ZZ[i]$ we have
\begin{align*} \chi_\bt(n) \, f_\bt(ak)
&= \frac \s 2 \int_{\Ld_\s\backslash N} \chi_\bt\bigl( n
(n')^{-1} \bigr)
f(n'ak)\, dn'
\displaybreak[0]\\
&= \frac \s 2\int_{\Ld_\s\backslash N} \chi_\bt(n')^{-1} f(n'nak)\,
dn'\,,
\end{align*}
since $\chi_\bt$ is a (unitary) character. So with
\ir{Frbtdef}{\Four_\bt}
\be \label{Frbtdef}\Four_\bt f (g) \isdd \frac \s
2\int_{\Ld_\s\backslash N} \overline{\chi_\bt(n)}\, f(n g)\, dn\,, \ee
we have $\chi_\bt(n) f_\bt(ak) = \Four_\bt f(nak).$ Since we integrate
over a compact set, the action of $\glie$ by right differentiation
commutes with the operator $\Four_\bt.$ We see also that the action of
$K$ by right translation commutes with $\Four_\bt.$ So $\Four_\bt$ is
an intertwining operator of $(\glie,K)$-modules. (See, eg Section 3.3
in \cite{Wal88}.)
\be \Four_\bt : C^\infty(\Ld_\s\backslash G)_K \rightarrow \Ffu_\bt\ee
where \il{Ffubt}{$\Ffu_\bt$}$\Ffu_\bt $ consists of the $K$-finite
elements of $C^\infty(G)$ that transform on the left according to the
character $\chi_\bt$ of~$N.$\il{Fto}{Fourier term operator}

For $\ell \in \frac \s 2 \ZZ_{\neq 0},$ $c\bmod 2\ell$, we define the
operator $\Four_{\ell,c} $ in $ C^\infty(\Ld_\s\backslash G)_K  $
by\ir{Frlcdef}{\Four_{\ell,c}}
\be\label{Frlcdef} \Four_{\ell,c} f (nak)
= \sum_{m\geq 0} \Th_{\ell,c}(h_{\ell,m};n)
\, f_{\ell,c,m}(ak)\,, \ee
with $f_{\ell,c,m}$ as in~\eqref{Fen}. Here we cannot use a simple
integral like in~\eqref{Frbtdef}. Proposition~\ref{prop-Ftit} will show
that the operator $\Four_{\ell,c}$ is an intertwining operator of
$(\glie,K)$-modules as well. There, we will also describe the
$(\glie,K)$-module $\Ffu_{\ell,c}$ such that $ \Four_{\ell,c} 
C_c^\infty(\Ld_\s\backslash G)_K\subset \Ffu_{\ell,c}$.

The Fourier expansion can be rewritten as \ir{Fe1}{\text{Fourier
expansion}}
\be \label{Fe1} f(g) = \sum_{\bt\in \ZZ[i]} \Four_\bt f(g)
+ \sum_{\ell\in (\s/2)\ZZ_{\neq 0}}\sum_{c\bmod 2\ell} \Four_{\ell,c}
f(g)\,. \ee

\rmrk{Fourier expansion of functions on $\Gm\backslash G$} Let
$\Gm\subset G$ be a cofinite discrete subgroup satisfying the
$\ZZ[i]$-condition on the cusps in Definition~\ref{cucond}, and let
$f\in C^\infty(\Gm\backslash G)_K.$ For each cusp $\c$ the function
\il{fc}{$f^\c$}$f^\c:g\mapsto f(g_\c g)$ is in
$C^\infty(\Ld_{\s(\c)}\backslash G)_K,$ with a Fourier expansion as
in~\eqref{Fe1}. The operator $f\mapsto f^\c$ commutes with the action
of $\glie$ by right differentiation. Here the $\ZZ[i]$-condition on the
cusps is essential.


\def\flnm{rFtm-chap-II}

\setcounter{tabold}{\arabic{table}}%
\setcounter{figold}{\arabic{figure}}%


\chapter{Fourier term modules}\label{chap-2}
\setcounter{section}5
\setcounter{table}{\arabic{tabold}}
\setcounter{figure}{\arabic{figold}} \markboth{II. FOURIER TERM
MODULES}{II. FOURIER TERM MODULES} The aim of this chapter is the study
of the spaces in which Fourier term operators have their image. Since
we restrict these operators to $K$-finite invariant functions their
images are in a space in which the Lie algebra $\glie$ and the group $K$
act.

In Section~\ref{sect-gKm} we discuss $(\glie,K)$-modules and in
Section~\ref{sect-exdf} we discuss how to carry out explicit
computations of the Lie algebra action in modules of tensor form with
respect to the Iwasawa decomposition. In the Sections
\ref{sect-lFtm}--\ref{sect-spFtm} we apply this to Fourier term
modules.


\def\flnm{rFtm-II-gKm}


\section{(g,K)-modules} \markright{6.
(g,K)-MODULES}\label{sect-gKm} After fixing notations for the Lie
algebra $\glie$ and its universal enveloping algebra, we turn to
$(\glie,K)$-modules. We take advantage of the relatively simple
structure of $\SU(2,1)$.

An important tool are the shift operators in
Proposition~\ref{prop-shdef}. They allow us to distill the action of
$\glie$ in a set of four operators acting on highest weight elements in
``neighboring $K$-types''. We will use them throughout Chapters
\ref{chap-2} and~\ref{chap-3} in the study of Fourier term modules.
Irreducibility of $(\glie,K)$-modules can be investigated with the use
of shift operators; see Section~\ref{sect-scm}.\medskip

The real Lie algebra \il{glie}{$\glie$}$\glie$ of $\SU(2,1)$ is a real
form of the complex Lie algebra of type $\mathrm{A}_2.$ As a vector
space it is the sum $\glie=\nlie\oplus \alie\oplus \klie$ of the Lie
algebras of the subgroups $N,$ $A$ and $K$.\il{nlie1}{$\nlie$}
\il{alie}{$\alie$}\il{klie1}{$\klie$} We gave a basis
$\{ \XX_0, \XX_1, \XX_2\}$ for $\nlie$ in \eqref{XX012}, and a basis
$\{\CK_i, \WW_0, \WW_1,\WW_2\}$ of $\klie$ in \S\ref{sect-klie}. The
group $A$ in the Iwasawa decomposition has Lie algebra
$\alie=\RR\,\HH_r,$\il{HHr}{$\HH_r$}
\be \HH_r =
\begin{pmatrix}0&0&1\\0&0&0\\1&0&0
\end{pmatrix}\,.\ee

Inside $K$ is the group
\il{M1}{$M$}$M=\bigl\{\mm(\z)\;:\; |\z|=1\bigr\}$ with Lie algebra
$\mlie = \RR \, \HH_i,$\ir{HHi}{\HH_i}
\be \HH_i = \frac 32\WW_0 - \frac12\CK_i=
\begin{pmatrix}i&0&0\\0&-2i&0\\
0&0&i\end{pmatrix}\,.\ee
It satisfies $e^{t\HH_i} = \mm(e^{it})$.\smallskip

There are two \il{Ca}{Cartan subalgebra}Cartan subalgebras of
\il{gliec}{$\glie_c$}$\glie_c= \CC\otimes_\RR \glie$ of interest in the
context of this paper. The Cartan subalgebra $\alie_c\oplus \mlie_c$ is
adapted to the Iwasawa decomposition. We may choose $\XX_1-i \XX_2$ and
$\XX_1+i\XX_2$ in $\nlie_c$ as the root vectors for the choice of
\il{spr}{simple positive roots}simple positive roots, \ir{al12}{\al_1,
\al_2}$\al_1$ and $\al_2$, respectively.
\be\label{al12}
\begin{array}{|c|cc|}\hline
&\al_1 & \al_2\\ \hline
\HH_r & 1 & 1\\
i\HH_i& -3 & 3 \\ \hline
\end{array}
\ee
The element $\XX_0\in \nlie$ is a root vector for $\al_1+\al_2$.

There is also the Cartan subalgebra \il{hK1}{$\hK$}$\hK= \CC\, \CK_i 
\oplus \CC\, \WW_0$ contained in $ \klie_c$, with root vectors as
indicated in Table~\ref{tab-rootv}.
\begin{table}[ht]
\[
\begin{array}{|rcl|c|c|}\hline
\multicolumn{3}{|c|}{\Z}& \bigl[\CK_i,\Z\bigr]& \bigl[ \WW_0,\Z\bigr]\\
\hline
\Z_{12}&=& \WW_1-i\WW_2 & 0 & 2i \Z_{12}\\
\Z_{23}&=& \frac12( - \WW_1-i\WW_2+\XX_1+i\XX_2)& 3i \Z_{23}&-i
\Z_{23}\\
\Z_{13}&=& \frac12\bigl( i\HH_i - 2 i \WW_0 + \HH_r + 2i \XX_0\bigr) &
3i\Z_{13}
& i \Z_{13}\\
\Z_{21}&=&\WW_1+i\WW_2 & 0 &
-2i \Z_{21} \\
\Z_{32}&=& \frac12\bigl(
-\WW_1+i\WW_2+\XX_1-i\XX_2\bigr)&
-3i \Z_{32}& i\Z_{32}
\\
\Z_{31}&=& \frac12\bigl(
-i\HH_i+2i\WW_0+\HH_r-2i\XX_0\bigr)&
-3i\Z_{31}&-i\Z_{31}\\
\hline
\end{array}\]
\caption{Root vectors for $\hK$.}\label{tab-rootv}
\end{table}\il{tab-rootv}{$\Z_{ij} \in \glie_c$}

The \il{uea}{universal enveloping algebra}universal enveloping algebra
\il{Uglie}{$U(\glie$)}$U(\glie)$ is the associative $\CC$-algebra
generated by $\glie.$

The center \il{ZUg}{$ZU(\glie)$}$ZU(\glie)$ of $U(\glie)$ is isomorphic
to a polynomial algebra in two variables, which can be chosen as
indicated in Table~\ref{tab-casdt}. The first generator $C$ of degree
two is the \il{Cas}{Casimir element}Casimir element.\ir{tab-casdt}{C
\in ZU(\glie))} As a second generator we use the element
\il{Dt3}{$\Dt_3$}$\Dt_3$ of degree three given in \cite[Proposition
3.1]{GPT02}.
\begin{table}[ht]
\begin{align*}
C &= \HH_r^2-4\HH_r
-\frac13\HH_i^2+4\XX_0\HH_i
-8\XX_0\WW_0
\\
&\qquad\hbox{}
+4\XX_0^2
-2\XX_1\WW_1+\XX_1^2-2\XX_2\WW_2+\XX_2^2
\\
&= -\frac13\CK_i^2 +2i\CK_i
-\WW_0^2+2i \WW_0 -\Z_{12}\Z_{21}
+4 \Z_{13}\Z_{31}+4\Z_{23}\Z_{32}\,,\\
\Dt_3& =-\frac i9 \CK_i^3 + i \CK_i\WW_0^2 +i \Z_{12}\CK_i\Z_{21}
+2i \Z_{13}\CK_i\Z_{31}
-6i \Z_{13}\WW_0\Z_{31}\\
&\qquad\hbox{}
- 6\Z_{13}\Z_{21}\Z_{32}
+2i \Z_{23}\CK_i\Z_{32}
+ 6i \Z_{23}\WW_0\Z_{32}+6 \Z_{23}\Z_{12}\Z_{31}\\
&\qquad\hbox{}
-2 \CK_i^2 +2\CK_i\WW_0
+24 \Z_{13}\Z_{31}
+24\Z_{23}\Z_{32}+8i\CK_i\,.
\end{align*}
\caption{Generators of the center of the enveloping
algebra.}\label{tab-casdt}
\end{table}

In \cite[\S3]{Math} the Lie algebra relations are defined for symbolic
manipulation. We carry out various checks, among them the centrality of
$C$ and~$\Dt_3$.

\rmrk{(g,K)-modules}A \il{gKm}{$(\glie,K)$-module}$(\glie,K)$-module has
the structure of a $\glie$-module and of a $K$-module in which the
actions of $\glie$ and $K$ are compatible, and in which each vector is
$K$-finite. See, eg, \cite[Section 3.3]{Wal88} for a discussion in a
more general context.

The group $K\subset \SU(2,1)$ is connected, and the $\glie$-action
determines the action of $K$. Nevertheless the action of $K$ gives
additional information; not every $\glie$-module can be made into a
$(\glie,K)$-module. \smallskip

Let $V$ be any $(\glie,K)$-module. For each $v\in V$ the
finite-dimensional representation generated by $Kv$ is a direct sum of
irreducible representations, each isomorphic to some representation
$\tau^h_p$ discussed in~\S\ref{sect-irrpK}. We denote by
\il{Vhp}{$V_{h,p}$}$V_{h,p}$ the submodule of $V$ consisting of a sum
of copies isomorphic to $\tau^h_p.$ This submodule is characterized by
the eigenvalues $-ih$ of $\CK_i\in \klie$ and $-p(p+2)$ of the Casimir
element $C_K$ in~\eqref{CasK}. We can characterize $V_{h,p}$ as the
intersection of the kernels of the elements
\[ \CK_i+ih \text{ and } C_K +2p+p^2 \text{ in } U(\klie)\,.\]

If we know for a given element $v\in V$ a finite set of $(h_i,p_i)$'s
such that $v\in \bigoplus_i V_{h_i,p_i}$ it is possible to give the
projection of $v$ onto a given factor $V_{h,p}$ as the image $uv$ by
an element $u\in U(\klie)$ (actually a polynomial in $\CK_i$ and
$C_K$). However, in general it is impossible to find an element in
$U(\klie)$ that works for all $v\in V.$ The best one can do is to give
a sequence of elements that work for an increasing collection of finite
sets.

Inside the component $V_{h,p}$ we have a further decomposition into
\il{wsp}{weight space} weight spaces \il{Vhpq}{$V_{h,p,q}$}$V_{h,p,q},$
characterized as the kernel of $\WW_0+iq.$ In this way
\be V_{h,p} = \bigoplus_{|q|\leq p,\; q\equiv p\bmod 2} V_{h,p,q}\,.\ee
The actions of $\Z_{21},\Z_{12}\in \klie_c$ shift the weight by $2.$ We
have seen this in the description of the action on basis functions on
$K$ in Table~\ref{tab-RLdK}, p~\pageref{tab-RLdK}; in a general
$(\glie,K)$-module this follows from the commutation relations. The
occurrence in the module $V$ of a $K$-type $\tau^h_p$ is completely
determined by the \il{hwsp}{highest weight space}highest weight space
$V_{h,p,p}.$ We use the term \il{Kwt}{weight}weight for the eigenvalues
of $i \WW_0$ in representations of $\klie$. The highest weight subspace
in $V_{h,p}$ is the kernel of $\Z_{21}$.\smallskip

The commutation relations imply that the action of each of the basis
elements $\Z_{23}, \Z_{32},$ $\Z_{13},$ and $\Z_{31}$ sends elements in
$V_{h,p,p}$ to elements in the sum of spaces $V_{h',p',q'}$ with
$|h'-h|=3,$ $|p'-p|=1,$ $|q'-p|=1.$

The action of elements of $U(\glie)$ enables us to project the result to
one of these spaces. The resulting operators are well known in the
theory of semi-simple Lie groups. Here we call them shift operators.

\begin{prop}\label{prop-shdef}There are \il{sho}{shift operator}shift
operators in each $( \glie,K)$-module $V$\il{sopS}{$\sh{\pm 3}{\pm 1}$}
\[\begin{aligned} \sh 3 1 &: V_{h,p,p}\rightarrow
V_{h+3,p+1,p+1}\,,&\quad \sh {-3}1 &: V_{h,p,p} \rightarrow
V_{h-3,p+1,p+1}\,,\\
\sh3{-1}&: V_{h,p,p} \rightarrow V_{h+3,p-1,p-1}\,,& \sh {-3}{-1} &:
V_{h,p,p} \rightarrow V_{h-3,p-1,p-1}\,,
\end{aligned}\]
defined in terms of the action of $U(\glie)$ in $V$ as indicated in
Table~\ref{tab-sho}.

The operators $\sh 3 1$ and $\sh 3{-1}$ commute, and the operators
$\sh{-3}1$ and $\sh{-3}{-1}$ commute.
\end{prop}
\begin{table}[tp]
\[
\renewcommand\arraystretch{1.3}
\begin{array}{|rl|c|}\hline
\sh 3 1 & : V_{h,p,p}\rightarrow V_{h+3,p+1,p+1} & \Z_{31}\\
\hline
\sh {-3} 1 & : V_{h,p,p}\rightarrow V_{h-3,p+1,p+1} & \Z_{23} \\
\hline
\sh3{-1}& : V_{h,p,p}\rightarrow V_{h+3,p-1,p-1} & \Z_{32}+ \bigl(
2(p+1)\bigr)^{-1} \Z_{12}\Z_{31}
\\
&& = p(p+1)^{-1}\Z_{32}
+(2(p+1))^{-1}\Z_{31}\Z_{12}\\ \hline
\sh{-3}{-1}& : V_{h,p,p}\rightarrow V_{h-3,p-1,p-1} & \Z_{13} - \bigl(
2(p+1)\bigr)^{-1} \Z_{12}\Z_{23}\\
&&= p(p+1)^{-1}\Z_{13} - \bigl( 2(p+1)\bigr)^{-1} \Z_{23}\Z_{12}
\\ \hline
\end{array}
\]
\caption{Description of the shift operators} \label{tab-sho}
\end{table}

\begin{proof}We define the \il{souw}{shift operator,
upward}\il{uwso}{upward shift operator}\emph{upward shift operators}
$\sh 3 1$ and $\sh {-3}1$ by the action of Lie algebra elements:
$\sh 3 1 = \Z_{31}$, and $\sh{-3}1 = \Z_{23}$. Since
$[\Z_{31},\Z_{23}]=\frac12 \Z_{21}$ raises the weight by $2$, and since
the action of $\Z_{12}$ on highest weight spaces is zero, the operators
$\sh31$ and $\sh{-3}1$ commute, provided that we show that the
operators $\sh {\pm 3}1$ send $V_{h,p,p} $ to $V_{h\pm 3,p+1,p+1}.$

To check that $\sh 3 1 V_{h,p,p} \subset V_{h+3,p+1,p+1}$ we use that
for all $K$-types
\be V_{h',p',p'} = \ker\bigl( \CK_i + i h'\bigr)\, \cap\, \ker\bigl(
\WW_0+ip'\bigr)
\cap \,\ker\bigl( C_K+p'(p'+2)\bigr)\,. \ee
We start with $v\in V_{h,p,p},$ which is in this intersection for
$(h',p')=(h,p),$ and also satisfies $\Z_{21}v=0.$ For
$w=\sh 31 v = \Z_{31}v$ we have
\begin{align*}
\CK_i w &= \CK_i \Z_{31}v = \Z_{31} \CK_iv+\bigl[ \CK_i,
\Z_{31}\bigr]v\\
&\quad =-i h w -3i \Z_{31} v = -i(h+3)
w\,,\displaybreak[0]
\\
\WW_0 w &= \Z_{31}\WW_0 v + \bigl[ \WW_0,\Z_{31}\bigr] v = -ip w -i
\Z_{31} v\\
&\quad =-i(p+1)w\,,
\displaybreak[0]\\
C_K w &= \bigl( -2i \WW_0+\WW_0^2\bigr) w
+ \Z_{12}\Z_{21}\Z_{31}v\\
& = -(p^2+4p+3) w + \Z_{12}\Bigl( \Z_{31}\, 0
+\bigl[\Z_{21},\Z_{31}\bigr] v\Bigr)
\displaybreak[0]\\
&= -(p^2+4p+3) w +\Z_{12}\bigl( 0 \bigr)
= -(p+1)(p+3)w\,.
\end{align*}
This shows that $w=\sh31 v \in V_{h+3,p+1,p+1}.$ We check that
$\sh {-3}1 V_{h,p,p}\subset V_{h-3,p+1,p+1}$ in a similar way.
\medskip

The \il{sowd}{shift operator, downward}\il{dso}{downward shift
operator}\emph{downward shift operators} $\sh 3{-1}$ and $\sh{-3}{-1}$
are based on $\Z_{32},$ respectively $\Z_{13},$ but in a more
complicated way than for the upward shift operators.

Let $v\in V_{h,p,p}.$ So it has the properties indicated above. We put
$u= \Z_{32} v.$ Then
\begin{align*}
\CK_i u &= \Z_{32}\CK_i v
+ \bigl[ \CK_i,\Z_{32}\bigr] v = -i(h+3)
u\,,\\
\WW_0 u &= \Z_{32} \WW_0 v
+ \bigl[ \WW_0,\Z_{32}\bigr] v =-i(p-1)u\,
\end{align*}
Preliminary computations suggest that $u=u_{p-1}+ u_{p+1}$ with
$u_{p\pm 1} \in V_{h+3,p\pm 1,p-1}.$ If this is true, then
\[ C_K u = -(p-1)(p+1)\, u_{p-1} -(p+1)(p+3)
\, u_{p+1}\,.\]
It turns out that the two equations for $u_{p+ 1}$ and $u_{p-1}$ have
the unique solution
\bad u_{p-1}
&= \frac1{4(p+1)} \Bigl( C_K u+(p+1)(p+3)
u \Bigr)
= u
+ \frac 1{2(p+1)} \Z_{12}\Z_{31} v\,,\\
u_{p+1}&= \frac{-1}{4(p+1)}\Bigl( C_K u
+(p^2-1) u\Bigr)
=
- \frac1{2(p+1)}\, \Z_{12}\Z_{31} v\,.
\ead
So we take
\be \sh 3{-1} v = u_{p-1} = \Z_{32}v+ \frac 1{2(p+1)} \Z_{12}\Z_{31} v
\,.\ee
We check the computations in \cite[\S7a]{Math}, also for the other case
\be \sh{-3}{-1}v\= \Z_{13}v - \frac1{2(p+1)}\Z_{12}\Z_{23}v\,,\ee
and for the relation
\be\sh{-3}{-1}\sh3{-1} v \= \sh3{-1}\sh{-3}{-1}v\,.
\iflabels\else\qedhere\fi\ee
\end{proof}

The proof has given us explicit descriptions in Table~\ref{tab-sho} of
the shift operators. In the case of the downward shift operators this
description depends on the $K$-type of the element to which the
operator is applied.

\begin{cor}Let $V$ be a $(\glie,K)$-module, and let $v\in V_{h,p,p}.$
The submodule $U(\glie)\, v$ of $V$ is equal to the space
$U(\klie) P  v $ where $P$ runs over all compositions of shift
operators.
\end{cor}
\begin{proof}The statement is clear if we let $P$ run through the
products of the elements $\Z_{31},$ $\Z_{23},$ $\Z_{32}$ and $\Z_{13}.$
The relations in Table~\ref{tab-sho} allow us to rewrite the products
in terms of shift operators.\end{proof}

\begin{defn}Let $V$ be a $(\glie,K)$-module. We call a vector
$v \in V_{h,p,p}$ \il{mv}{minimal vector}\emph{minimal} if $v\neq 0,$
$\sh3{-1}v=0$ and $\sh{-3}{-1}v=0.$\end{defn}

\begin{lem}\label{lem-mv}Let $v\in V_{h,p,p}$ be a minimal vector. Then
we have:
\begin{enumerate}
\item[i)] $2(p+1) \Z_{32} v = - \Z_{12}\Z_{31}v$ and
$2(p+1)\Z_{13}v = \Z_{12}\Z_{23}v.$
\item[ii)] $\sh{-3}{-1}\sh 3 1 v -(p+1) \sh 3{-1}\sh{-3}1 v =
\frac12 (p-h)(p+1)v .$
\item[iii)] For the Casimir element in Table~\ref{tab-casdt},
p~\pageref{tab-casdt}:
\[ C v = \Bigl( \frac 13 h^2+p^2+2h+2p\Bigr)\, v
+ 4\frac{p+2}{p+1}\, \sh{-3}{-1}\sh31 v\,.\]
\item[iv)] For the central element $\Dt_3\in ZU(\glie)$ in
Table~\ref{tab-casdt}:
\begin{align*} \Dt_3 v &= \frac19 h
(h+3p+12)(h-3p+6)\, v \\
&\quad\hbox{}
+ 2 \frac{(p+2)(h-3p+6)}{p+1}\, \sh{-3}{-1}\sh31 v\,.\end{align*}
\end{enumerate}
\end{lem}
\begin{proof}Part~i)
is a direct consequence of the description of the downward shift
operators in Table~\ref{tab-sho}, p~\pageref{tab-sho}.

For the remaining statements we use Mathematica to carry out
computations. See \cite[\S7b]{Math}. For ii)
we write out the description of the shift operators in
Table~\ref{tab-sho}, using that $\sh {\pm 3}1 v \in V_{h\pm 3,p+1}.$
Taking the relations in the universal enveloping algebra into account,
this gives
\begin{align*}\sh{-3}{-1}\sh 3 1 v
&-(p+1)
\sh 3{-1}\sh{-3}1 v \\
&= \frac12(p-h)(p+1) v - \Z_{23} \Bigl( 2(p+1) \Z_{32}
+\Z_{12}\Z_{31} \Bigr)v\,.
\end{align*}
Since $\sh3{-1}$ vanishes on $v$ this gives the formula.

The computations for iii) and iv) are carried out with Mathematica in a
similar way.
\end{proof}

\subsection{Special modules}\label{sect-scm}

\begin{defn}\label{def-scm}
By a \il{spcm}{special module}\emph{special module} with
\il{tp}{parameter set of a special module}\emph{parameter set}
$\bigl[ \mu_2;h_0,p_0;A,B\bigr]$ we mean a $(\glie,K)$-module $V$ with
the following properties:
\begin{enumerate}
\item[a)] The module $V$ is generated by a minimal vector in
$V_{h_0,p_0},$ with the usual conditions $h_0\equiv p_0\bmod 2,$
$p_0\in \ZZ_{\geq 0}.$
\item[b)] The Casimir element $C$ acts in $V$ by multiplication by
$\mu_2\in \CC.$
\item[c)] The parameter set determines the set of $K$-types in~$V$:
\be \label{absum}
V = \bigoplus_{0\leq a \leq A}\bigoplus_{0\leq b\leq B} \, \klie_c\,
\Z_{31}^a\, \Z_{23}^b \, V_{h_0,p_0,p_0}\,.\ee
The parameters $A$ and $B$ are in $\ZZ_{\geq 0} \cup\infty$. If
$A=\infty$ take $\bigoplus_{a\geq 0}$, and similarly for $b$ and~$B$.
\item[d)] All $K$-types in $V$ have multiplicity one.
\end{enumerate}
\end{defn}

In the introduction of Subsection~\ref{sect-iclirr} we quote theorems
implying that all irreducible $(\glie,K)$-modules are isomorphic to
modules found in these notes. All of these are special modules as
defined above.

\begin{prop}\label{prop-scm}Two irreducible special modules are
isomorphic if and only if their parameter sets are equal.
\end{prop}
The proof will take the remainder of this subsection.

Let $V$ be a special module. Within the space $V_{h,p}$ of a $K$-type we
can use the elements $\ZZ_{12}$ and $\Z_{21}$ of $\klie_c$ to move
between the weight spaces $V_{h,p,q}$. This reduces the consideration
to the highest weight spaces $V_{h,p,p}$. Table~\ref{tab-sho},
p~\pageref{tab-sho}, implies that the highest weight vectors in $V$ are
obtained by repeated application of upward shift operators, which
commute. Choosing a minimal vector $v(0,0)\in V_{h_0,p_0,p_0}$ we
define a basis vector \il{vab}{$v(a,b)$ in
Subsection~\ref{sect-scm}}$v(a,b) = 
(\sh31)^a(\sh{-3}1)^b v(0,0)$ in the highest weight space $V_{h,p,p}$
with $h=h_0+3(a-b)$ and $p=p_0+a+b$. Then
$\bigl\{ \Z_{12}^c \, v(a,b) \;:\; 0 \leq c \leq p_0+a+b\bigr\}$ is a
basis of $V_{h,p}.$

The actions of $\klie_c$ and $K$ on the spaces $V_{h_0+3(a-b),p_0a+b}$
is completely determined by the $K$-type. For each $\XX$ in the
complementary set $Q=\bigl\{\Z_{31},\Z_{13},\Z_{32},\Z_{23}\bigr\}$ the
product $\XX \Z_{21}^c$ is a linear combination of terms
$\Z_{21}^{c(j)} \, \XX_j$ with $\XX_j \in Q.$ So the action of $\glie$
is completely determined by the action of the elements of $Q$ on the
vectors $v(a,b).$ This in turn can be expressed 
by
Table~\ref{tab-sho} in the action of the shift operators on the
$v(a,b).$ We have
\begin{align*} \sh 3 1 v(a,b) &= \begin{cases}
v(a+1,b) &\text{ if } a<A\,,\\
0 &\text{ if } a=A\,;
\end{cases}\displaybreak[0] \\
\sh{-3}1 v(a,b) &= \begin{cases}
v(a,b+1) &\text{ if } b<B\,;\\
0 &\text{ if } b=B\,;\end{cases}
\displaybreak[0]\\
\sh3{-1} v(a,b) &= \al_+(a,b) \, v(a,b-1)\,;
\displaybreak[0]\\
\sh{-3}{-1}v(a,b) &= \al_-(a,b)\, v(a-1,b)\,;
\end{align*}
with coefficients $\al_\pm(a,b)\in \CC,$ and the convention
$\al_+(0,b)=\al_-(a,0)=0.$
\be\label{sqrel} \al_-(a,b) \, \al_+(a-1,b)
= \al_+(a,b)\, \al_-(a,b-1)\qquad
(a,b\geq 1)\,.\ee
See Figure~\ref{fig-sq}.
\begin{figure}[tp]
\begin{center}\grf{7}{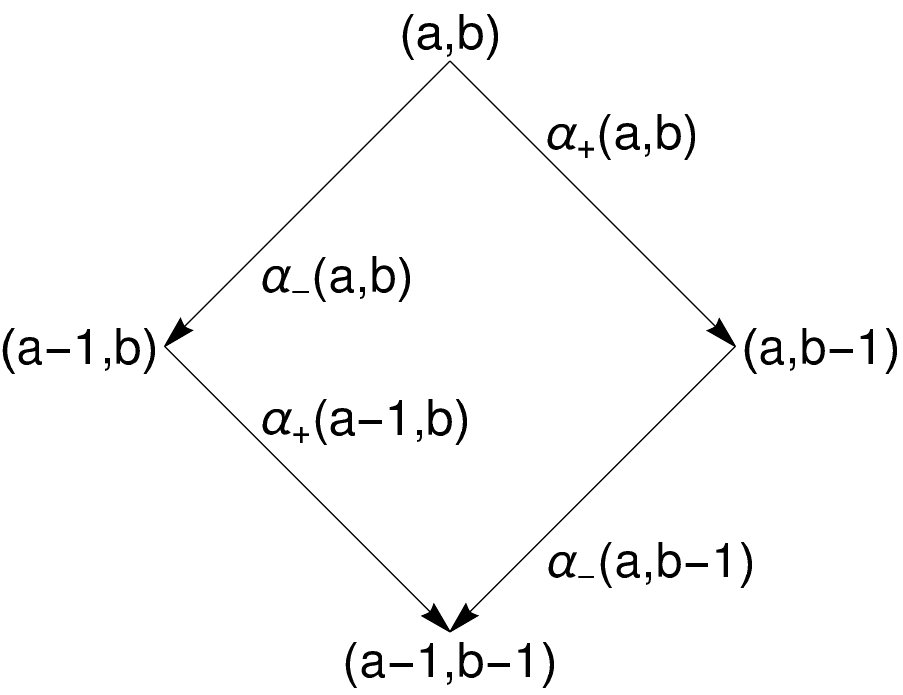}\end{center}
\caption{Square with factors for downward shift operators}
\label{fig-sq}
\end{figure}

\begin{lem}\label{lem-redmv}Suppose that a special module $V$ with
parameter set
\[[\mu_2;h_0,p_0;A,B]\]
has a non-trivial invariant subspace $W.$ Then there are $a\in [0,A+1),$
$b\in [0,B+1),$ $(a,b)\neq (0,0),$ such that $v(a,b)\in W$ is a minimal
vector in~$V.$
\end{lem}
\begin{proof}Any $(\glie,K)$-module is the direct sum of its isotypic
subspaces, in which $K$ acts according to a fixed $K$-type. Consider
$v(a,b) \in W $ with $(a,b)\neq(0,0)$.

As long as one of $\al_-(a,b) \neq 0$ or $\al_+(a,b)\neq 0$ we find a
lower $K$-type in $W.$ This process has to stop before $(0,0)$ is
reached. That gives a minimal vector in $V$ that is an element of $W.$
\end{proof}

Since a minimal vector $v(a,b)$ in $V$ is defined by $\sh3{-1}v(a,b)=0$,
$\sh{-3}{-1}v(a,b)=0,$ we have the following direct consequence.
\begin{lem}\label{lem-alnzir}If a special module has $\al_\pm(a,b)\neq0$
for all $0\leq a<A+1,$ $0\leq b<B+1,$ then it is irreducible.
\end{lem}

\begin{lem}\label{lem-alzmv}Let $V$ be a special module, and suppose
that $\al_+(a,b)=0$ for $b>0$ or $\al_-(a,b)=0$ for $a>0.$ Then $V$ has
a minimal vector $v(a_1,b_1)$ with $(a_1,b_1)\neq 0.$
\end{lem}
\begin{proof}Suppose that $\al_+(a,b)=0.$ Then relation~\eqref{sqrel}
implies that \[\al_-(a,b)\, \al_+(a-1,b)\=0\,.\]
So at least one of the situations in Figure~\ref{fig-kr} occurs.
\begin{figure}[htp]
\begin{center}\grf{3}{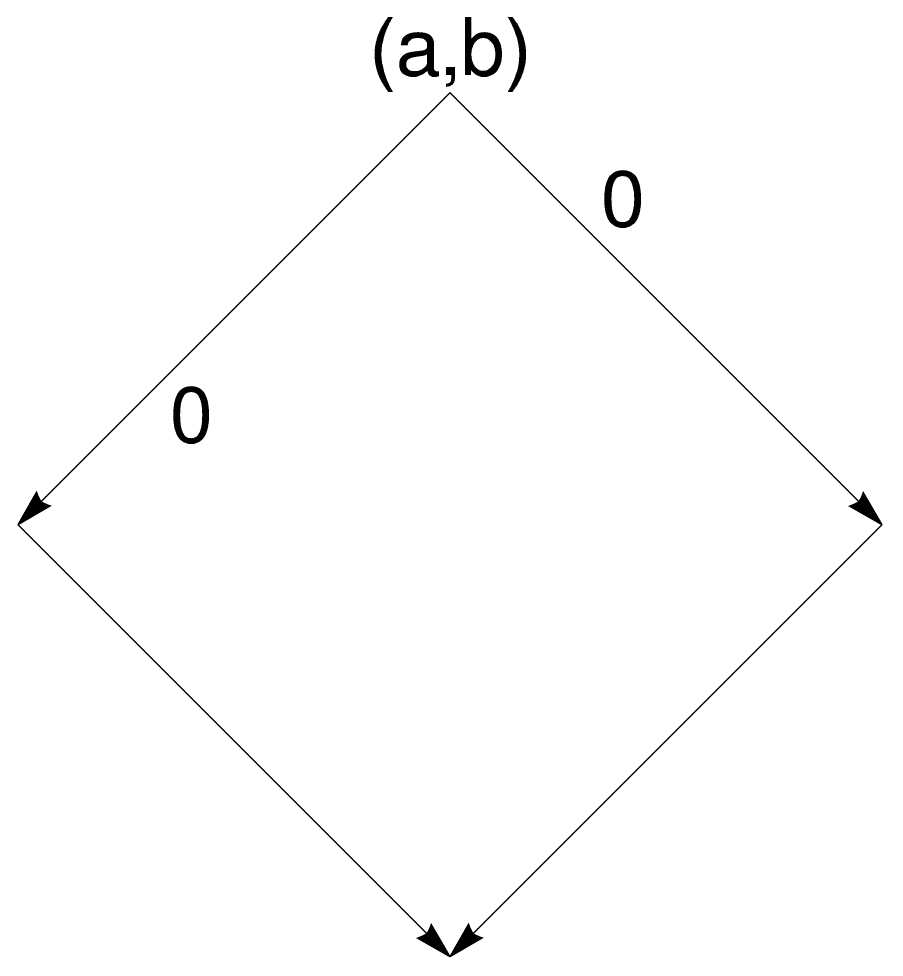}\qquad \grf{3}{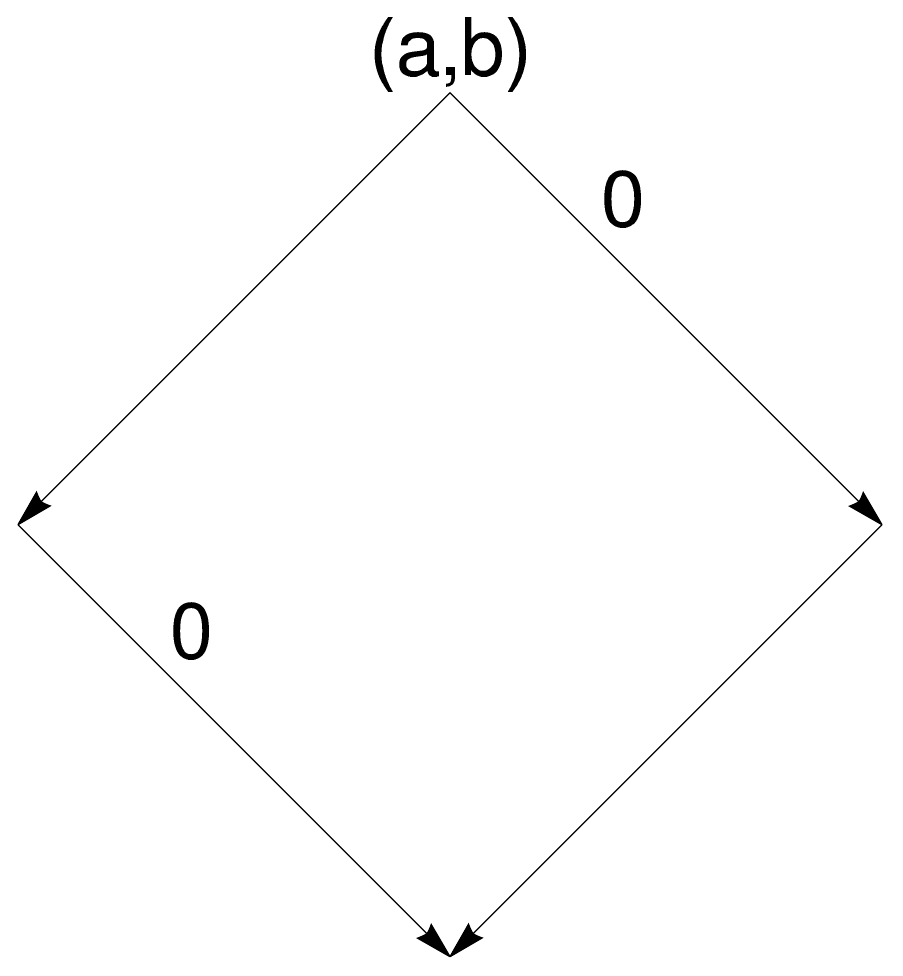}\end{center}
\caption{Squares with factors for downward shift operators.}
\label{fig-kr}
\end{figure}
If $\al_-(a,b)=0$ the vector $v(a,b)$ is a minimal vector. Otherwise,
the square relation \eqref{sqrel} leads to the relation
$\sh 3{-1}v(a-1,b)=0.$ Proceeding in this way the process cannot go on
longer than in the situation of Figure~\ref{fig-kr3}. So a minimal
vector is reached.\smallskip
\begin{figure}[htp]
\begin{center}\grf{6}{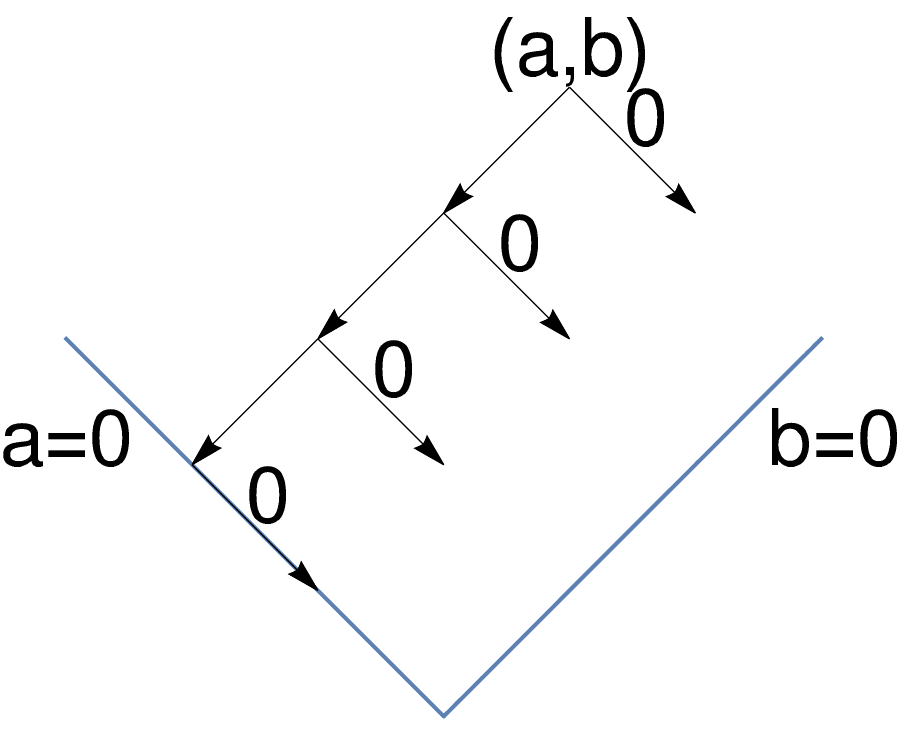}\end{center}
\caption{Propagation of the kernel of $\sh 3{-1}.$} \label{fig-kr3}
\end{figure}

If $\sh{-3}{-1} v(a,b)=0$ we proceed similarly, with reflected figures.
\end{proof}

Let $S(a,b)$ be the square in the $(a,b)$-plane with $(a,b)$ at the top,
and $(a-1,b-1)$ at the bottom, like in Figure~\ref{fig-sq}. We call it
a zero square if at least one of $\al_+(a,b)$ and $\al_-(a,b-1)$ on the
right is zero. Then also at least one of the $\al$'s on the left is
zero, by \eqref{sqrel}.

If $S(a,b)$ is a zero square, then at least one of the adjoining squares
on the right, $S(a+1,b)$ and $S(a,b-1)$ is also a zero square. The same
holds on the left. So zero squares do not come singly, but form
connected regions, like sketched in Figure~\ref{fig-zqs}.
\begin{figure}[htp]
\begin{center}\grf{8}{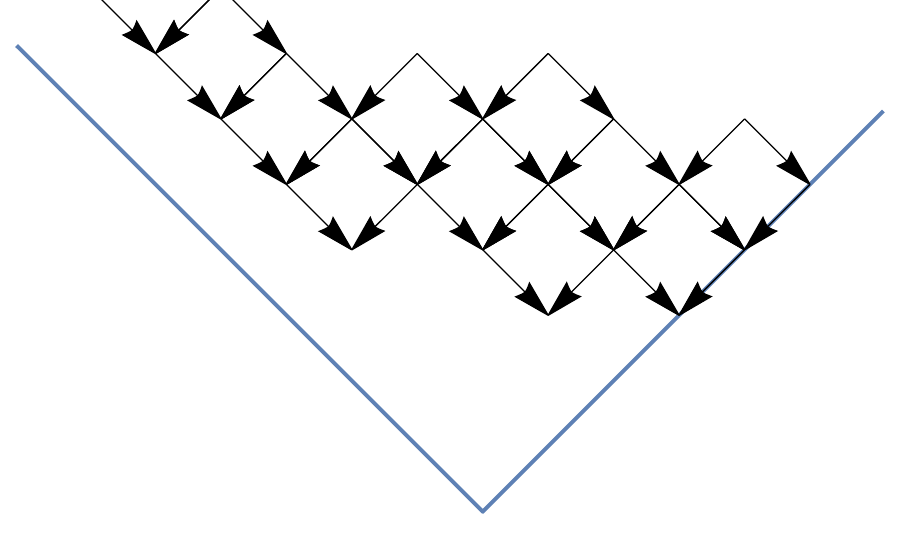}\end{center}
\caption{A connected collection of zero squares.} \label{fig-zqs}
\end{figure}
The region may reach one of the boundary lines $b=0$ or $a=0,$ or it may
extend to infinity, parallel to a boundary line.

The factors $\al_\pm(a',b')$ corresponding to common edges of adjoining
squares are zero.

\begin{lem}\label{lem-mvred} If $v(a,b)$ with $(a,b)\neq(0,0)$ is a
minimal vector in a special module $V,$ then $V$ is reducible.
\end{lem}
\begin{proof}Let $W\subset V$ be the submodule generated by $v(a,b).$ If
we can find a product of shift operators that sends $v(a,b)$ to a
non-zero multiple of $v(0,0),$ then $W=V.$ So we have to show that, if
starting from $v(a,b)$ we cannot reach $v(0,0)$ by shift operators,
then $W$ is a non-trivial invariant subspace.

We have $\sh 3{-1}v(a,b)=0$ and $\sh{-3}{-1}v(a,b)=0.$ So going downward
directly from $(a,b)$ yields zero. A path from $(a,b)$ to $(0,0)$
corresponding to a product of shift operators gives zero if it passes a
downward edge that is common to two adjoining zero squares. Since we
start at the top of a zero square such a path always has to go through
an interior edge of the collection of zero squares.
\end{proof}

By Lemmas \ref{lem-redmv}, \ref{lem-alzmv} and \ref{lem-mvred} we now
have equivalence between reducibility, existence of non-trivial minimal
vectors, and vanishing of at least one~$\al_\pm.$

The next step is to find more relations for the $\al$'s than those given
by the square relation~\eqref{sqrel}.

\begin{lem}\label{lem-iae}In an irreducible special module, all
coefficients $\al_\pm(a,b),$ with $0\leq a<A,$ $0\leq b<B,$ are
determined by the parameter set $[\mu_2;h_0,p_0;A,B].$
\end{lem}
\begin{proof}
We know that $v(a,b)$ satisfies $C v(a,b) = \mu_2 v(a,v).$ We write
\bad C&= - \WW_0^2 - \frac 13 \CK_i^2 - \Z_{12}\Z_{21} + 4
(s_{31}+s_{23})\,,\\
s_{31}&= \Z_{31}\Z_{13} \,,\qquad s_{23}= \Z_{23}\Z_{32}\,.\ead
The terms in $U(\klie)$ all have $v(a,b)$ as an eigenvector. So we know
that $(s_{31}+s_{23}) v(a,b) = \mu_{a,b} v(a,b)$ with an explicitly
known eigenvalue $\mu_{a,b}.$ With Mathematica it is no problem to
compute $\mu_{a,b},$ but we do not need an explicit value. A not too
complicated computation shows that there is $d\in \ZZ$ such that
\bad (s_{31}+s_{23}) v(a,b) &= \sh{-3}{-1} \sh 31v(a,b) + \sh 3{-1}
\sh{-3}1 v(a,b) + d \, v(a,b)
\\
&= \bigl( \al_-(a+1,b) + \al_+(a,b+1) \bigr)
v(a,b) + d \, v(a,b)
\,.
\ead
So $\al_-(a+1,b)+\al_+(a,b+1)$ is equal to a well-defined number
$C_{a,b}.$ With \eqref{sqrel} we now have two relations:
\bad \al_-(a,b) \, \al_+(a-1,b) &= \al_+(a,b)
\,\al_-(a,b-1)\,,\\
\al_+(a-1,b) + \al_-(a,b-1) &= C_{a-1,b-1} \ead
Since $V$ is assumed to be irreducible all factors $\al_\pm$ are
non-zero. The values of $\al_+(a-1,b)$ and $\al_-(a,b-1)$ determine the
values of $\al_+(a,b)$ and $\al_-(a,b)$ completely.

We need only start the induction. To the minimal vector $v(0,0)$ we
apply iii)
in Lemma~\ref{lem-mv} that gives the value of $\al_-(1,0).$ Then ii)
in the same lemma also gives $\al_+(0,1).$ This suffices to start the
induction.
\end{proof}

\rmrk{Proof of Proposition~\ref{prop-scm}}Implied directly by
Lemma~\ref{lem-iae}.


\def\flnm{rFtm-II-exdf}


\section{Explicit differentiation of K-finite
functions}\label{sect-exdf}\markright{7. EXPLICIT DIFFERENTIATION}From
general $(\glie,K)$-modules we turn to $(\glie,K)$-modules contained in
the space \il{CiGK}{$C^\infty(G)_K$}$C^\infty(G)_K$ of smooth
$K$-finite functions on $G.$ It is a $(\glie,K)$-module for the actions
of $\glie$ and $K$ by right differentiation and right translation.

The aim of this section is to establish explicit formulas for the
differentiation of elements of $C^\infty(G)_K$, and to implement these
formulas as Mathematica routines. We use an approach that is known for
general semisimple Lie groups, especially for functions with prescribed
left behavior under $N$. Then only the differentiation on $A$ remains
to be carried out: the radial parts of the differentiation operators.
\medskip

By the Iwasawa decomposition any element $F \in C^\infty(G)_K$ has the
form
\be\label{comps} F(nak) = \sum_{h,p,r,q} F_{h,p,r,q}(na)\, \Kph
h{p}{r}{q}(k)\,,\ee
with \il{comp}{component function}\emph{component functions}
$F_{h,p,r,q}\in C^\infty(NA).$ The summation variables run over
integers such that $h\equiv p \equiv r \equiv 1\bmod 2,$
$|r|, |q|\leq p.$ Only finitely many component functions are non-zero.
The aim in this section is to describe explicitly the action of
$\glie_c.$\smallskip

The action of any $\XX\in \klie_c$ involves only the basis elements in
Table~\ref{tab-RLdK}, p~\pageref{tab-RLdK}. We have to describe the
action of the basis elements $\Z_{13},$ $\Z_{32},$ $\Z_{23}$ and
$\Z_{31}.$ The procedure is known. Our task is just to work it out
explicitly. We will carry out the following steps:
\begin{enumerate}
\item We consider the action of $k\in K$ by conjugation on the basis
elements of $L^2(K).$
\item We relate the action by right differentiation to the following
action of $\XX\in \glie$
\be \label{Mdiff}\bigl( M(\XX) F \bigr)(nak) = \partial_t
F\bigl(na e^{t\XX} k \bigr)
\bigm|_{t=0}\,.\ee
\item We describe $M(\XX)$ on functions of the form
$nak\mapsto h(na) \,\Phi(k)$ in terms of right differentiation of $h$
and left differentiation of $\Phi.$ \end{enumerate}

\rmrk{Conjugation by elements of $K$}A direct computation in
\cite[\S8a]{Math} gives for $ k = \km(\eta,\al,\bt)\in K$
\badl{KLconj} k\Z_{13}k^{-1}&= \al\eta^3\, \Z_{13}
  - \bar\bt\eta^3 \Z_{23}\,,&\qquad k\Z_{31}k^{-1}& = \bar\al\eta^{-3}\,
\Z_{31}
- \bt \eta^{-3}\Z_{32}\,,\\
k\Z_{23}k^{-1}&= \bt\eta^3\,\Z_{13}
+ \bar\al\eta^3\,\Z_{23}\,,& k\Z_{32}k^{-1} & = \bar\bt\eta^{-3}\,
\Z_{31}
+ \al \eta^{-3}\, \Z_{32}\,. \eadl
The factors are polynomial functions on $K.$ From \eqref{Phi-exmp} we
have the following:
\badl{Phi1} \al \eta^{\pm 3} &= \Kph {\mp3}{1}{-1}{-1}(k)\,,&
\qquad \bt \eta^{\pm 3}&= \Kph{\mp3}{1}{-1}{1}(k)\,,\\
\bar \al \eta^{\pm 3} &= \Kph {\mp3}{1}{1}{1}(k)\,,&
\qquad \bar\bt \eta^{\pm 3}&=
-\Kph{\mp3}{1}{1}{-1}(k)\,. \eadl

\rmrk{Right differentiation and interior differentiation} The right
differentiation by $\XX\in \glie$ in $nak\in NAK$\ir{Rdiff1}{R(\XX) F =
\XX F}
\be \label{Rdiff1}
\XX F (nak)= R(\XX) F(nak) \= \partial_t
F\bigl( nak e^{t\XX}\bigr)\bigm|_{t=0}\ee
and the \il{idiff}{interior differentiation}\emph{interior
differentiation}\ir{Mdiff1}{M(\XX)F}
\be \label{Mdiff1}
M(\XX) F(nak) = \partial_t F\bigl( na e^{t\XX} k \bigr)\bigm|_{t=0}\ee
are related by
\be \XX F(nak) = M(k\XX k^{-1})
F(nak)\,.\ee
This relation extends to $\XX\in \glie_c$ by linearity.

We apply this with
$\XX\in \bigl\{ \Z_{31},\Z_{13},\Z_{32},\Z_{23}\bigr\}$, and get
\be\label{RM} R(\XX) F (nak) \= \sum_{ij} \ph_{ij}(k) \,
M(\Z_{ij})F(nak)\ee
with functions $\ph_{ij}$ on $K$ indicated in~\eqref{KLconj}.

By the Iwasawa decomposition we have
\be\label{IdL}
\renewcommand\arraystretch{1.3}
\begin{array}{rlll}
\Z_{13}&=i\XX_0 &+ \frac12\HH_r
&+ \frac i2
(\HH_i-2 \WW_0)\\
\Z_{31}=\overline{\Z_{13}}&=-i\XX_0
&+ \frac12\HH_r &- \frac i2
(\HH_i-2\WW_0)\\
\Z_{23}&=\frac12(\XX_1+i\XX_2)
&&- \frac12(\WW_1+i \WW_2)\\
\Z_{32}=\overline{Z_{23}}&= \frac12(\XX_1-i\XX_2)
&&- \frac12(\WW_1-i \WW_2)
\end{array}
\ee
where the elements in the columns on the right are in $\nlie_c$,
$\alie_c$, and $\klie_c$, respectively. In this way, we get for $F$ of
the form $F(nak) = b(na)\Phi(k)$, with $b\in C^\infty(NA)$ and
$\Phi=\Kph h{p}{r}{q}$ a formula
\be M(\XX) \bigl( b(na)\, \Phi(k) \bigr) \= \bigl(R(\XX_{NA})b\bigr)(na)
\, \Phi(k)
+ b(na) \, \bigl( L(\XX_K) \Phi\bigr)(k)\,.\ee

We have given the left differentiation on $K$ in Table~\ref{tab-RLdK},
p~\pageref{tab-RLdK}. To carry out the multiplication by $\ph_{ij}(k)$
in~\eqref{RM} we use the multiplication formulas in
Table~\ref{tab-mlt}, p~\pageref{tab-mlt}. In this way we have reduced
the action of the elements of
$\bigl\{ \Z_{31},\Z_{13},\Z_{32},\Z_{23}\bigr\}$ to known relations and
right differentiation on $NA$.

\rmrk{Implementation}In this way, the action by right differentiation of
(a basis of) $\glie_c$ on functions in the form \eqref{comps} can be
described in terms of right differentiation on $NA$ of the components
$F_{h,p,r,q}$ in~\eqref{comps}. Carrying out such computations we
gladly leave to a computer, since errors slip in easily into
computations by hand. Of course, it requires great care to write the
routines. The version in Section 8 of \cite{Math} gives results that we
checked in various ways. These routines are the basis for essential
computations in this paper.
In the notebook we give explanations of the way we build the routines.

\rmrk{Example}A computation in \cite[\S8e]{Math} gives for
$F(nak)=b(na)\Kph h{p}{r}{p}(k)$
\begin{align}\label{E31} \Z_{31}&\bigl(b\, \Kph h{p}{r}{q}\bigr)
= \frac1{8(p+1)} \biggl(
(2+p+r) \Bigl( (2 \HH_r - 4 i \XX_0)b \\
\nonumber
&\qquad\qquad\qquad\hbox{}
+ (h+2p-r)b \Bigr) \, \Kph {h+3}{p+1}{r+1}{q+1}
\displaybreak[0]\\
\nonumber
&\quad\hbox{}
-2(2+p-r) \bigl( (\XX_1-i\XX_2) b\bigr)\, \Kph {h+3}{p+1}{r-1}{q+1}
\displaybreak[0]\\
\nonumber
&\quad\hbox{}
-(p-q) \Bigl( (4i \XX_0-2\HH_r)b+(4-h+2p+r)b \Bigr)
\, \Kph{h+3}{p-1}{r+1}{q+1}
\displaybreak[0]\\
\nonumber
&\quad\hbox{}
+2(p-q) \bigl( (\XX_1-i\XX_2) b\bigr)
\, \Kph{h+3}{p-1}{r-1}{q+1}\biggr)\,, \end{align}
with the action of $\nlie_c\oplus \alie_c$ by right differentiation on
the function $b$ on~$NA.$

A similar formula is available for all elements in the Lie algebra. For
the elements in $\klie_c$ the formula is easy, since it involves only
$\Kph h{p}{r}{q}.$ For $\Z_{13},$ $\Z_{32}$ and $\Z_{23}$ the general
structure of the formula is the same as in~\eqref{E31}.

\rmrk{Application to shift operators} If we take $q=p$ in~\eqref{E31}
the last two terms become zero, and we arrive at a highest weight
vector in $V_{h+3,p+1,p+1}.$ Thus, we obtain the description of
$\sh 3 1\bigl( b\, \Kph h{p}{r}{q}\bigr),$ in accordance with
Proposition~\ref{prop-shdef}. The same works for $\sh{-3}1.$

The downward shift operators $\sh{\pm3}{-1}$ are based on $\Z_{32}$ and
$\Z_{13}.$ In these cases we just delete the contributions of the
$K$-type $\tau^{h\pm3}_{p-1}.$ In Proposition~\ref{prop-shdef} the
projection to a given $K$-type was given by an element of $U(\glie)$
depending on the $K$-type of the argument. It would be inefficient to
do that in the actual computations.

\begin{table}[tp]
\begin{align*}
8(p+1)& \, \sh31 \bigl( b\, \Kph h {p}{r}{p}\bigr)
\;=\; \Bigl( (2+p+r) \bigl( 2\HH_r-4i\XX_0+h+2p-r\bigr) b\\
&\quad\qquad \hbox{} \cdot
\Kph{h+3}{p+1}{r+1}{p+1}\\
&\qquad\hbox{}
- 2(2+p-r) \bigl( \XX_1-i \XX_2\bigr) b\;
\Kph{h+3}{p+1}{r-1}{p+1}\Bigr)\,,\\
8(p+1)&\sh{-3}1 \bigl( b\, \Kph h {p}{r}{p}\bigr)
\;=\; 2(2+p+r)\bigl( \XX_1+i \XX_2 \bigr)b\; \Kph{h-3}{p+1}{r+1}{p+1}\\
&\qquad\hbox{}
+(2+p-r) \bigl( 2\HH_r+4i\XX_0+2p+r-h\bigr)
b\; \Kph{h-3}{p+1}{r-1}{p+1}\,,\\
4(p+1)&\, \sh3{-1} \bigl( b\, \Kph h {p}{r}{p}\bigr)
\;=\; p \,\bigl(2\HH_r-4i \XX_0
-4+h-2p-r\bigr)b\\
&\quad\qquad\hbox{} \cdot
\Kph{h+3}{p-1}{r+1}{p-1}\\
&\qquad\hbox{}
+2p\, \bigl( \XX_1-i \XX_2 \bigr)b\; \Kph{h+3}{p-1}{r-1}{p-1}\,,\\
4(p+1)&\, \sh{-3}{-1} \bigl( b\, \Kph h {p}{r}{p}\bigr)
\;=\; -2p
(\XX_1+i \XX_2) b\, \Kph{h-3}{p-1}{r+1}{p-1}\\
&\qquad\hbox{}
+ p \,\bigl( 2 \HH_r + 4 i \XX_0 -(4+h+2p-r)
\bigr)b\; \Kph{h-3}{p-1}{r-1}{p-1}\,.
\end{align*} \il{shob}{shift operator}
\caption{The action of the shift operators on vectors in
$C^\infty(G)_K.$ See \cite[\S8f]{Math}.}\label{tab-sh}
\end{table}


\def\flnm{rFtm-II-lFtm}


\section{Large Fourier term modules}\label{sect-lFtm}\markright{8. LARGE
FOURIER TERM MODULES} We turn to the submodules of $C^\infty(G)_K$ of
our main interest: the modules $\Ffu_\bt$ and $\Ffu_{\ell,c}$ in which
the Fourier term operators $C^\infty(\Ld_\s\backslash G)_K$ in
\eqref{Frbtdef} and \eqref{Frlcdef} take their values. These modules
should contain the Fourier terms, even if we do not impose any further
conditions on the $\Ld_\s$-invariant functions.

We adapt the Mathematica routines for the shift operators to these
modules, and arrive at explicit descriptions. In the generic abelian
case, $\Ffu_\bt$ with $\bt\in \ZZ[i]\setminus\{0\}$, the upward shift
operators turn out to be injective.

In the non-abelian case the modules $\Ffu_{\ell,c}$ turn out to be a
countably infinite direct sum of submodules $\Ffu_{\ell,c,d}$.
\medskip

\begin{prop}\il{lFtm}{large Fourier term module}\il{Ftml}{Fourier term
module, large} {\rm Large Fourier term modules.}
\begin{enumerate}
\item[i)] Let $\bt\in \CC$. The subspace
$\Ffu_\bt \subset C^\infty(G)_K$ determined by the condition
$ F\bigl( n g \bigr)=\ch_\bt(n)\, F(g)$ on $F\in C^\infty(G)_K$ is a
$(\glie,K)$-submodule of $C^\infty(G)_K$. {\rm We call it a \emph{large
abelian Fourier term module}.}\il{labFtm}{large abelian Fourier term
module.}
\item[ii)] Let $\ell\in \frac\s2\ZZ_{\neq 0}$ and let
$c\in \ZZ\bmod 2\ell$. We define $\Ffu_{\ell,c}$ as the vector space
spanned by functions of the form
\be f\bigl( n\am(t)k \bigr) \= \Th_{\ell,c}(\ph;n) \, f(t)\, \Kph
h{p}{r}{q}(k)\,,\ee
where $\ph \in \Schw(\RR)$ runs over (finite) linear combinations of
normalized Hermite functions $h_{\ell,m}$ with $ \ZZ_{\geq 0}$, where
$f\in C^\infty(0,\infty)$, and where the integers $h,p,r,q$ satisfy
$h \equiv p \equiv r \equiv q \bmod 2$, $|r|\leq p$, $|q|\leq p$.
\il{Ffuellc}{$\Ffu_{\ell,c}$}The space $\Ffu_{\ell,c}$ is a
$(\glie,K)$-submodule of $C^\infty(G)_K$. {\rm We call it a \emph{large
non-abelian Fourier term module}.}\il{nablFtm}{large non-abelian
Fourier term module.}
\end{enumerate}
\end{prop}
\begin{proof}The invariance under the $(\glie,K)$-action in i)
follows from the fact that the actions on the right and on the left
commute.

In~ii) the invariance under the action of $\klie$ and $K$ is clear. The
action of the remaining basis elements $\Z_{ij}$ can be worked out by
interior differentiation, with the approach in \S\ref{sect-exdf}. Since
$A$ normalizes $N$, the action of $\nlie$ on the functions on $NA$
leads to an action of $\nlie $ on theta functions. The relations
\eqref{Thd1} and~\eqref{Thd2} show that the differentiation produces
linear combinations of theta functions with the same parameters $\ell $
and~$c$.
\end{proof}

The Fourier term operators on $C^\infty(\Ld_\s\backslash G)_K$ take
values in these large Fourier term modules. These modules are large,
since we can take $f\in C^\infty(0,\infty)$ arbitrarily. In
\S\ref{sect-Ftm} we will impose the condition that $ZU(\glie)$ acts by
multiplication by a character. With that restriction we will speak of
Fourier term modules.

When we apply the explicit differentiation of the previous paragraph we
use that if a function $b$ on $NA$ is of the form
$b\bigl( n \am(t) \bigr)= u(n)\, f(t)$, then we have, with
$\exp(x\HH_r) \= \am(e^x)$, the relation
\be \HH_r b\bigl(n \am(t) \bigr) = u(n)\, t\, f'(t)\,.\ee
For $\XX\in \nlie$ the action of $\XX$ satisfies
$\XX b \bigl(na \bigr) = \bigl ( (a\XX a^{-1})u \bigr)(n)\, f(a)$. With
\eqref{expnlie} this leads to
\badl{XNA} \XX_j b \bigl( n\am(t) \bigr)&\= t\, (\XX_j u)(n) \, f\bigl(
\am(t)\bigr)\qquad (j=1,2)\,,\\
\XX_0 b \bigl( n\am(t) \bigr)&\= t^2\,
(\XX_0 u)(n) \, f\bigl( \am(t)\bigr)\,.
\eadl
\subsection{Abelian case}
\begin{prop}\label{prop-isogab}All $(\glie,K)$-modules $\Ffu_\bt$ with
$\bt\neq 0$ are isomorphic.
\end{prop}
\begin{proof}Any element $\bt\in \CC^\ast$ can be written as
$\bt=\z^{-3} t$ with $|\z|=1$ and $t>0$. Conjugation by
$x=\am(t)\mm(\z)$ as described in \eqref{Nnorm} transforms $\chi_1$
into $\chi_{\bt}$. So, the left translation
\[( L_x F)(g) = F(x g)\]
gives a bijective linear map $\Ffu_1 \rightarrow\Ffu_\bt.$ Since left
and right translations commute, $L_x$ is an intertwining operator.
\end{proof}

In the abelian case we have $u(n) =\ch_\bt(n)$, and
\bad \XX_1 \bigl( \ch_\bt(n)\, f(t) \bigr) &\= 2\pi i \re(\bt) \,
\ch_\bt(n)\, f(t)\,,& \XX_0 \bigl( \ch_\bt(n)\, f(t) \bigr) &\= 0\,,
\\
\XX_2 \bigl( \ch_\bt(n)\, f(t) \bigr) &\= 2\pi i \im(\bt) \,
\ch_\bt(n)\, f(t)\,,& \HH_r\bigl( \ch_\bt(n)\, f(t) \bigr) &\= t\,
\ch_\bt(n)\, f'(t)\,.
\ead
These relations can be applied to work out the differentiation
relations. In particular we get the description of the shift operators
in Table ~\ref{tab-shab}.
\begin{table}[tp]
\begin{align*}
8(p+1)& \, \sh31 \bigl( \chi_\bt f\, \Kph h {p}{r}{p}\bigr)
\;=\;\chi_\bt \, \Bigl( (2+p+r) \bigl( 2 t f'
+(h+2p-r) f \bigr)
\\
&\qquad\qquad \hbox{} \cdot
\Kph{h+3}{p+1}{r+1}{p+1}\\
&\qquad\hbox{}
- 4\pi i(2+p-r)\bar\bt \, tf\; \Kph{h+3}{p+1}{r-1}{p+1}\Bigr)\,,\\
8(p+1)&\sh{-3}1 \bigl( \chi_\bt \,f\, \Kph h {p}{r}{p}\bigr)
\;=\; \chi_\bt\,\Bigl( 4\pi i\bt(2+p+r)tf \; \Kph{h-3}{p+1}{r+1}{p+1}\\
&\qquad\hbox{}
+(2+p-r) \bigl(2 t f' +(2p+r-h)f\bigr) \;
\Kph{h-3}{p+1}{r-1}{p+1}\Bigr)\,,\\
4(p+1)&\, \sh3{-1} \bigl( \chi_\bt \, f \, \Kph h {p}{r}{p}\bigr)
\;=\;\chi_\bt\Bigl( p \bigl(2tf'
+(-4+h-2p-r)f\bigr)\\
&\qquad\qquad\hbox{} \cdot
\Kph{h+3}{p-1}{r+1}{p-1}\\
&\qquad\hbox{}
+4\pi i \bar \bt p t f\; \Kph{h+3}{p-1}{r-1}{p-1}\Bigr)
\,,\\
4(p+1)&\, \sh{-3}{-1} \bigl( \chi_\bt \, f\, \Kph h {p}{r}{p}\bigr)
\;=\; \chi_\bt\Bigl(-4\pi ip \bt t f\, \Kph{h-3}{p-1}{r+1}{p-1}\\
&\qquad\hbox{}
+ p \bigl( 2 t f' -(4+h+2p-r)f \bigr)\;
\Kph{h-3}{p-1}{r-1}{p-1}\Bigr)\,.
\end{align*}
\il{shopab}{shift operator, abelian case}
\caption{Shift operators in $\Ffu_\bt$, with
$\bt\in \ZZ[i]\setminus\{0\}$.\\
See \cite[\S9b]{Math}.} \label{tab-shab}
\end{table}

\rmrk{Kernel relations} The component functions of a highest weight
function in a given $K$-type $\tau^h_p$ are parametrized by $r$, which
runs over the finitely many values satisfying $|r|\leq p$,
$r\equiv p \bmod 2$. If the context allows it, we will
write\ir{sumr}{\sum_r = \sum_{r\equiv p(2),\; |r|\leq p}}
\be\label{sumr} \sum_r \quad\text{ instead of } \sum_{r\equiv p(2),\;
|r|\leq p}\,.\ee

A generating element of the highest weight space $\Ffu_{\bt;h,p,p}$ has
the form
\be \label{Fcona}F\bigl( n\am(t) k\bigr) \= \ch_\bt(n) \, \sum_r
f_r(t)\, \Kph h{p}{r}{p}(k)\,,\ee
with \il{co}{component}\emph{components}$f_r \in C^\infty(0,\infty)$.
The description of the shift operators shows that if $F$ is in the
kernel of a shift operator, then there are relations between components
$f_r$ and $f_{r+2}$. We call these relations \emph{kernel
relations}.\il{krel}{kernel relations} See Table~\ref{tab-krab}.
\begin{table}[htp]
\begin{align*}
\sh 3 1:\quad& 2 t f_p'+(h+p)\, f_p = 0 \,\\
&(2+p+r)(2 t f_r'+(h+2p-r) f_r = 8\pi i \bar \bt (p-r)\, t f_{r+2}\\
&\qquad \text{for } -p\leq r \leq p-2\,,\\
&\bar \bt f_{-p}=0\,;\\
\sh{-3}1:\quad& \bt f_p = 0\,,\\
&(p-r)(2 t f_{r+2}'+(2p+r+2-h) f_{r+2} = -2\pi i \bt(2+p+r)\, t f_r\\
&\qquad\text{for }-p\leq r \leq p-2\,,\\
&2t f_{-p}' +(p-h)f_{-p} = 0\,;\\
\sh3{-1}:\quad& 2t f_r'+(h-2p-r-4)
f_r = -8\pi i \bar \bt t f_{r+2}\\
&\qquad\text{for }-p\leq r \leq p-2\text{ and }p\geq1\,;\\
\sh{-3}{-1}:\quad& 2t f_{r+2}'-(2+h+2p-r) f_{r+2}= 2\pi i \bt t \, f_r\\
&\qquad \text{for }-p\leq r \leq p-2\text{ and }p\geq1\,.
\end{align*}
\caption[]{Kernel relations on $\Ffu_{\bt;h,p,p}$, with
$\bt\in \ZZ[i]\setminus\{0\}$.\\
The condition $p\geq 1$ for the downward shift operators is in
accordance with the fact that $\sh{\pm 3}{-1}$ vanishes on
one-dimensional $K$-types. See \cite[\S9c]{Math}.} \label{tab-krab}
\end{table}

\begin{samepage}
\begin{prop}\label{prop-uso-ga}Let $\bt\in \CC^\ast$.
\begin{enumerate}
\item[i)] The upward shift operators $\sh 31$ and
$ \sh{-3}1 : \Ffu_{\bt;h,p,p}\rightarrow \Ffu_{\bt;h\pm 3,p+1,p+1}$ are
injective.
\item[ii)] For each $K$-type $\tau^h_p$ the subspaces of
$\Ffu_{\bt;h,p,p}$ on which $\sh 3{-1}$ or $\sh{-3}{-1}$ vanish have
infinite dimension.
\end{enumerate}
\end{prop}\end{samepage}
\begin{proof} In the kernel relations for $\sh31$ in
Table~\ref{tab-krab} we see that if $\sh 3 1 F=0$ for $F$ as
in~\eqref{Fcona}, we have $f_{-p}=0$. Furthermore, since $f_{r+2}$ is
expressed in terms of $f_r$ and its derivative, we conclude that all
$f_r$ vanish, and hence $F=0$. If $\sh{-3}1 F=0$ we proceed similarly,
now starting with $f_p$. The proof of i) clearly breaks down for
$\bt=0$.

For the kernel of $\sh3{-1}$ in~ii) we can pick $f_{-p}$ arbitrarily in
$C^\infty(0,\infty)$. This determines the higher components. For
$\sh{-3}{-1}$ we start with any $f_p$ in $C^\infty(0,\infty)$.
\end{proof}

\subsection{Non-abelian case}\label{sect-lnab}In
$\Ffu_{\ell,c}$ the components are linear combinations of functions on
$NA$ of the form
\[ n\am(t) \mapsto \Th_{\ell,c}\bigl( h_{\ell,m};n)\, f(t) \]
with $f\in C^\infty(0,\infty)$ and $m\in \ZZ_{\geq 0}.$ In a given
module $\Ffu_{\ell,c}$ only $m$ varies, and we abbreviate
\il{thm}{$\th_m=\Th_{\ell,c}(h_{\ell,m})$
}$	\th_m = \Th_{\ell,c}(h_{\ell,m}).$ We obtain from the substitution
rules in Table~\ref{tab-hermdiff}, p~\pageref{tab-hermdiff}, and
\eqref{Thd1}, \eqref{Thd2}.
\begin{align*} R_{NA}(\HH_r)( \th_m f) &= \th_m\, tf'\,,\displaybreak[0]
\\
R_{NA}(\XX_0)( \th_m f) &= \pi i \ell \th_m \, t^2
f\,,\displaybreak[0]\\
R_{NA}\bigl( \sign(\ell) \XX_1+i \XX_2\bigr)( \th_m f)
&=- 2i \sqrt{2\pi|\ell|(m+1)}\, \th_{m+1}\, t f\,,\displaybreak[0]\\
R_{NA}\bigl( \sign(\ell) \XX_1-i \XX_2\bigr)( \th_m f)
&=-2 i \sqrt{2\pi|\ell|m}\, \th_{m}\, t f\,. \end{align*}

The sign of $\ell$ plays a role in these relations, hence also in the
resulting differentiation formulas and in the expressions for the shift
operators in Table~\ref{tab-shnab}. We will often write
\ir{epsdtt}{\e=\sign(\ell)}\il{dtt}{$\dt_x= (x+1)/2$}
\be \label{epsdtt}
\e\=\sign(\ell)\,,
\qquad \dt_x\= \frac{1+x}2\text{ for }x=\pm1\,.\ee
\begin{table}[tp]
\begin{align*}
8(p+1)&\sh 3 1 F \;=\; \th_m\,(2+p+r)\, \bigl( 2 t f'+
(h+2p-r+4\pi\ell t^2)f\bigr)\\
&\qquad\hbox{} \cdot
\Kph {h+3}{p+1}{r+1}{p+1}\\
&\quad\hbox{}
+4i \e \th_{m-\e} (2+p-r)\sqrt{2\pi |\ell|(m+\dt_{-\e})}\, t\, f \,
\Kph{h+3}{p+1}{r-1}{p+1}\,,\\
8(p+1)&\sh{-3}1 F \;=\; -4i\e\th_{m+\e} (2+p+r)
\sqrt{2\pi|\ell|(m+\dt_\e)}\,t\, f\\
&\qquad\hbox{} \cdot
\Kph{h-3}{p+1}{r+1}{p+1}\\
&\quad\hbox{}
+\th_m\,(2+p-r)\bigl(2 t f'-( h-2p-r+4\pi\ell t^2) f \bigr) \\
&\qquad\hbox{} \cdot
\Kph{h-3}{p+1}{r-1}{p+1}\,,\\
\frac{4(p+1)}p & \sh3{-1} F \;=\; \th_m \bigl(2 tf'-(4-h+2p+4-4\pi\ell
t^2)
f \bigr)\\
&\qquad\hbox{} \cdot
\Kph{h+3}{p-1}{r+1}{p-1}\\
&\quad\hbox{}
- 4i \e \th_{m-\e} \sqrt{2\pi|\ell|(m+\dt_{-\e})}\, t\,f \,
\Kph{h+3}{p-1}{r-1}{p-1}\,,\\
\frac{4(p+1)}p & \sh{-3}{-1} F \;=\; 4i\e
\th_{m+\e}\sqrt{2\pi|\ell|(m+\dt_\e)}\,t\, f
\,\Kph{h-3}{p-1}{r+1}{p-1}\\
&\quad\hbox{}
+ \th_m \bigl( 2 t f'-(4+h+2p-r+4\pi\ell t^2)f \bigr)\,
\Kph{h-3}{p-1}{r-1}{p-1}\,.
\end{align*}
\il{shopnab}{shift operator, non-abelian case}
\caption[]{Shift operators in $\Ffu_{\ell,c}$, For
$F=\th_m\, f(t)\, \Kph h{p}{r}{p}$.\\
$\e$ and $\dt_x$ as in \eqref{epsdtt}. See \cite[\S10b]{Math}. }
\label{tab-shnab}
\end{table}

\rmrk{Submodules} In Table~\ref{tab-shnab} we can check that the
quantity \ir{mudef}{d}
\be \label{mudef} d\isdd 3 \sign(\ell)
(2m+1)+h-3r \,\in\, 1+2\ZZ\ee
is preserved by the action of the shift operators. The action of the
elements in $\klie$ and right translation by elements of $K$ preserve
this quantity as well. So we can split $\Ffu_{\ell,c}$ into invariant
submodules
\be \label{Fcl-decomp} \Ffu_{\ell,c} = \bigoplus_{d \equiv 1\bmod 2}
\Ffu_{\ell,c,d}\,,\ee
where \il{Fn}{$\Ffu_{\ell,c,d} =  \Ffu_\n$}$\Ffu_{\ell,c,d}$ consists of
finite linear combinations of functions of the form
\[ n\am(t)k\mapsto \Th_{\ell,c}\bigl( h_{\ell,m};n\bigr)\, f(t) \, \Kph
h{p}{r}{q}(k)\]
with $(6m+3)\sign(\ell) + h-3r=d.$ If we work with fixed $\ell,$ $c$ and
$ d$ we often abbreviate $\Ffu_{\ell,c, d}$ as
$\Ffu_\n.$\il{n}{$\n=(\ell,c, d)$}

\rmrk{Metaplectic action} The splitting of $\Ffu_{\ell,c}$ as a direct
sum can be understood in greater generality (see Weil \cite{We64}, or
Ishikawa \cite[p~ 489, 490]{Ish99}) by an action of the double cover of
$M$ in the module~$\Ffu_{\ell,c}$.

Let us make this more explicit. The one-parameter group of automorphisms
$u(v): n\mapsto \mm(e^{iv}) n \mm(e^{-iv})$, with $v\in \RR$, induces
on the Lie algebra the automorphism determined by
\bad u(v) \XX_1 &\= (\cos 3t) \XX_1 + (\sin 3t) \XX_2\,,\\
u(v) \XX_2
&\= (\cos 3 t )\XX_2 - (\sin3t) \XX_1\,,\quad u(v) \XX_0 \= \XX_0\,,
\ead
and hence
\be\label{u'0} u'(0): \quad \XX_1\mapsto 3\XX_2\,,\quad \XX_2\mapsto - 3
\XX_1\,,
\quad \XX_0 \mapsto 0\,.\ee

On the other hand,
$B=\frac{1}{8\pi i \ell}\partial_\xi^2+2\pi i \ell \xi^2$
defines an operator in the Schwartz space $\Schw(\RR)$. With use of
Table~\ref{tab-hermdiff}, p~\pageref{tab-hermdiff}, we see that
\be Bh_{\ell,m} = \frac i2 (2m+1)\sign(\ell) h_{\ell,m}\,.\ee
(Checked in \cite[\S5g]{Math}.) Hence there is a group homomorphism from
$\RR$ to the unitary operators in $L^2(\RR)$ such that
$e^{v B} h_{\ell,m} = e^{i (m+1/2)\sign(\ell)v}\, h_{\ell,m}$.

With \eqref{dpild} we find for Schwartz functions $\ph$ the relations
\bad B\,d\pi_{2\pi \ell}(\XX_1)\, \ph - d\pi_{2\pi\ell}(\XX_1) B \ph
&\= - d\pi_{2\pi \ell}(\XX_2) \ph\,,\\
B\,d\pi_{2\pi \ell}(\XX_1)\ph - d\pi_{2\pi\ell}(\XX_2) B \ph
&\= d\pi_{2\pi \ell}(\XX_1) \ph\,,\\
B\,d\pi_{2\pi \ell}(\XX_0)\, \ph - d\pi_{2\pi\ell}(\XX_0) B \ph
&\= 0\,.
\ead
Comparison with \eqref{u'0} shows that for $\XX\in \nlie$
\be \partial_v e^{-3vB} d\pi_{2\pi\ell} (\XX) e^{3vB }\bigr|_{v=0} \=
d\pi_{2\pi \ell}\bigl( u'(0) \XX\bigr)\,. \ee
Integrating this, we obtain for $n\in N$
\be e^{-3v B} \pi_{2\pi\ell}(n) e^{3v B} \= \pi_{2\pi \ell} \bigl(
\mm(e^{iv}) n \mm(e^{-iv} ) \bigr)\,.\ee

We note that the right-hand side depends only on
$v\in \RR\bmod 2\pi\ZZ$. However,
$e^{2 \pi  B} h_{\ell,m} \= - h_{\ell,m}$, and $e^{-3(2\pi) B} = -1$.
So $\mm(e^{iv}) \mapsto e^{vN}$ is not defined on $M$, but can be
viewed as a function on the double cover $\tilde M$ of~$M$.\medskip

We define a group homomorphism $\tilde m$ from $\RR/\bmod 4 \pi\ZZ$ to
the operators on $\Ffu_{\ell,c}$ given on basis elements by
\[ \tilde m(v) \Bigl( \Th_{\ell,c}(\ph) \cdot f \cdot \Kph hprq\Bigr)
\= \Th_{\ell,c}\bigl( e^{3vB} \ph) \cdot
f \cdot L\bigl( \mm(e^{iv}) \bigr) \Kph h p r q\,.\]
It turns out that the elements $\Th_{\ell,c}(h_{\ell,m} )\cdot f\cdot
\Kph h p r q$ are eigenvectors of $\tilde m(v)$ with eigenvalue
\[ e^{3i(m+1/2)\sign(\ell) v}\, e^{iv(h-3r)/2} \= e^{\frac i2 v \bigl(
(6m+3)\sign(\ell) + h - 3r\bigr)} \= e^{i v d/2}\,, \]
with $d$ as in~\eqref{mudef}. So the decomposition in~\eqref{Fcl-decomp}
is the decomposition in eigen\-spaces for this action of the double
cover of~$M$. We may call $d$ the \il{mtplprm}{metaplectic
parameter}metaplectic parameter.

\rmrk{$K$-types}The $K$-types $\tau^h_p$ occurring in
$(\glie,K)$-modules can be pictured as points in the $(h/3,p)$-plane
satisfying $\frac h3 \equiv p\bmod 2,$ $p\in \ZZ_{\geq 0}.$ The shift
operators change the $K$-types by $(h/3,p) \mapsto (h/3\pm 1,p\pm 1)$
(occurrences of $\pm$ are not coupled). See Figure~\ref{fig-lattsho}.
\begin{figure}[htp]
\begin{center}\grf{6}{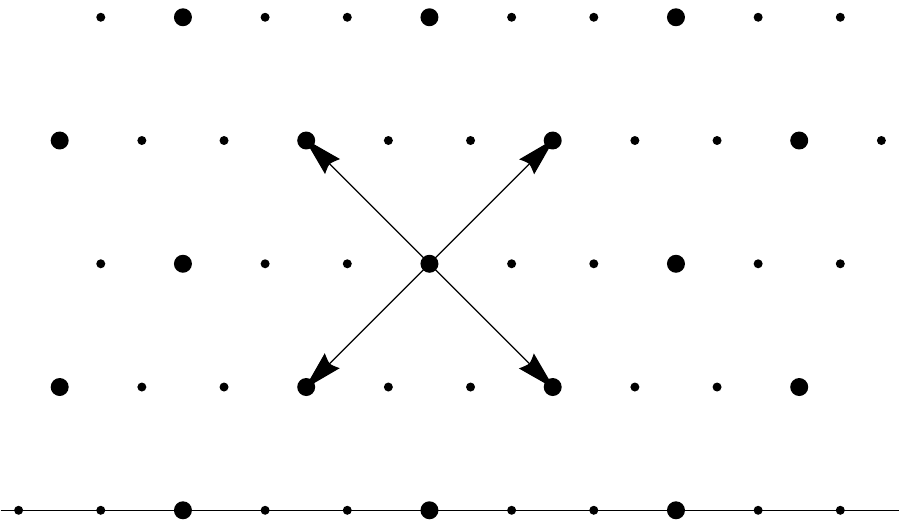}\end{center}
\caption{$K$-types $\tau^h_p$ depicted in the $(h/3,p)$-plane. The
arrows indicate the change of $K$-type given by the four shift
operators. Repeated application of shift operators leaves invariant the
set of thick points, which satisfy $\frac h3\equiv p\bmod 2$. }
\label{fig-lattsho}
\end{figure}

The $K$-type $\tau^h_p$ can occur in the realizations $\tau^h_{p,r}$
with $r\equiv p \bmod 2,$ $|r|\leq p.$ The definition of
$\Ffu_{\ell,c, d}$ gives the condition that
$ d= (6m+3)\sign(\ell) + h-3r.$ Since $m \in \ZZ_{\geq 0}$ this imposes
the requirement that
\begin{align*} 3r \geq 3+h- d&\text{ if }\ell>0\,,\\
3r\leq -3+h- d&\text{ if }\ell<0\,.
\end{align*} This imposes the following condition on the $K$-types:
\badl{mucond} h-3p &\leq d-3&\text{ if }
&\ell>0\,,\\
h+3p&\geq d+3 &\text{ if }&\ell<0\,. \eadl
See Figure~\ref{fig-mucond}.
\begin{figure}[htp]
\begin{center}\hspace*{\fill}\grf{5}{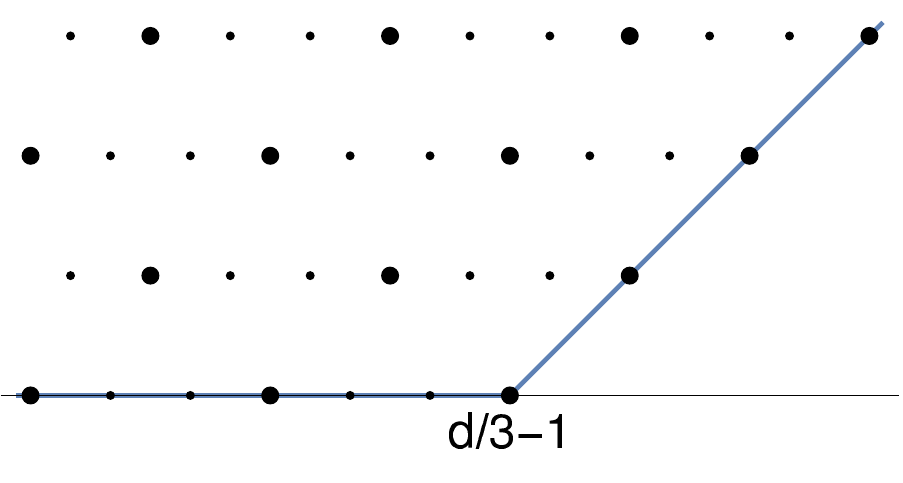}
\hspace{\fill} \grf{5}{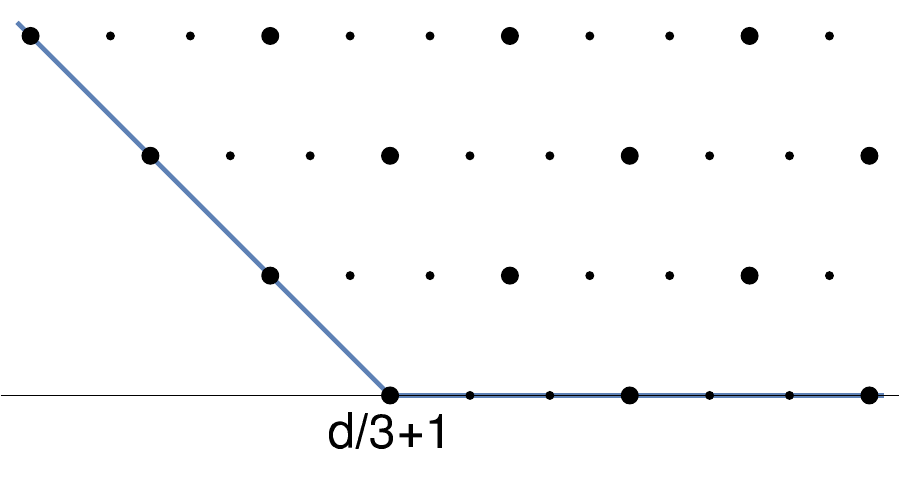}\hspace*{\fill}
\end{center}
\caption{$K$-types allowed in $\Ffu_{\ell,c, d}$ in the $(h/3,p)$-plane;
for $\ell>0$ on the left, and for $\ell<0$ on the right.
}\label{fig-mucond}
\end{figure}
This restriction is special for the modules $\Ffu_{\ell,c, d};$ in the
abelian modules $\Ffu_\bt$ all $K$-types can occur.

For highest weight elements $F\in \Ffu_{\n;h,p,p}$ we get a more
complicated decomposition into components:
\be\label{dcmp-nab}
F \= \sum_r \th_{m(h,r)}\, f_r \, \Kph hprp\,, \ee
where $r\equiv p \bmod 2$, $|r|\leq p$, with the additional condition
\be \label{r0cond} \begin{cases} r\geq r_0(h) & \text{ if }\e=1\,,\\
r \leq r_0(h) &\text{ if } \e=-1\,,
\end{cases}
\ee
and the following quantities, depending implicitly
on~$d$\il{r0def}{$r_0(h)$}\ir{mhrdef}{m(h,r)}
\be\label{mhrdef} r_0(h) \= \frac{h-d}3+\e\,,\qquad m(h,r) \= \frac
\e6(d-h+3r)-\frac12\,.\ee
The $K$-type $\tau^h_p$ does not occur in $\Ffu_\n$ if $r_0(h)>p$ if
$\e=1$, and if $r_0(h) < -p $ if $\e=-1$.

The kernel relations for the shift operators in $\Ffu_\n$ depend on the
quantity $r_0(h)$. We will work them out when we need them.

\rmrk{Fourier term operators}\il{Fto1}{Fourier term operator} The
operator $\Four_{\ell,c} :  C^\infty(\Ld_\s\backslash G)_K 
\rightarrow \Ffu_{\ell,c}$ in \eqref{Frlcdef} can be split up according
to the decomposition~\eqref{Fcl-decomp}:
\be \Four_{\ell,c} F = \sum_{ d\equiv 1\bmod 2} \Four_{\ell,c, d}
F\,.\ee
The sum is finite for each $F\in C^\infty(\Ld_\s\backslash G)_K.$
\ir{Fourmu}{\Four_{\ell,c, d}}
\badl{Fourmu} \Four_{\ell,c, d}& F (nak)
= \sum_{m,h,p,r,q} \Th_{\ell,c}\bigl( h_{\ell,m} ;n\bigr) \, \frac{\Kph
h{p}{r}{q}(k) } {\bigl\| \Kph h{p}{r}{q} \bigr\|^2}\\
&\hbox{} \cdot
\frac \s 2 \int_{n'\in \Ld_\s\backslash N} \int_{k'\in K}
\overline{\Th_{\ell,c}\bigl( h_{\ell,m};n'\bigr)}\, F(n' a k') \,
\overline{\Kph h{p}{r}{q}(k')}\, dk'\, dn'\,,
\eadl
where the sum runs over integers satisfying $m\geq 0,$
$(6m+3)\sign(\ell) + h-3r= d,$ $h\equiv p\equiv r\equiv p \bmod 2,$
$|r|\leq p,$ $|q|\leq p.$

\begin{prop}\label{prop-Ftit}The Fourier term operators
\bad \Four_\bt&:C^\infty(\Ld_\s\backslash G)_K \rightarrow \Ffu_\bt&
&(\bt\in \ZZ[i])\,,\\
\Four_{\ell,c}&: C^\infty(\Ld_\s\backslash G)_K \rightarrow
\Ffu_{\ell,c}&&
(\ell \in \frac\s 2\ZZ_{\neq 0},\; c\bmod 2\ell)\,,\\
\Four_{\ell,c, d}&: C^\infty(\Ld_\s\backslash G)_K \rightarrow
\Ffu_{\ell,c, d}&&( d\equiv 1\bmod 2)\,, \ead
are intertwining operators of $(\glie,K)$-modules.
\end{prop}
\begin{proof}For $\Four_\bt$ the intertwining property was already
noted; in \S\ref{sect-Fe1}.

For $f(nak) = f_{NA}(na) \, \Kph h{p}{r}{q}(k)$, we consider
$\Four_{\ell,c}$ given by
\begin{align*} \Four_{\ell,c} f (nak) &= \sum_{m\geq 0}
\Th_{\ell,c}(h_{\ell,m}; n)
\, \frac \s 2 \int_{n'\in \Ld_\s\backslash N}
\overline{\Th_{\ell,c}(h_{\ell,m};n')}\, f_{NA}(n'a) \,
dn'\\&\qquad\qquad\hbox{} \cdot \Kph h{p}{r}{q}(k)\,.
\end{align*}
It suffices to check that $\Z\bigl( \Four_{\ell,c} f \bigr) = 
\Four_{\ell,c} (\Z f)$ for all $\Z$ in a basis of $\glie_c.$

For basis elements in $\klie_c$ this is directly clear. For other basis
elements we use the discussion in \S\ref{sect-exdf}, which reduces the
question to the action by interior differentiation, between $NA$ and
$K.$ In \eqref{IdL} we give the decomposition
$\Z = \Z_\nlie+\Z_\alie+\Z_\klie$ corresponding to the Iwasawa
decomposition of~$G.$ The action of $M(\Z_\klie)$ is the same for
$\Four_{\ell,c} f$ and $f.$ We have to look at
\badl{RNA} R_{NA}(\XX_0) &= \frac12 t^2
\partial_r&
R_{NA}(\XX_1) &= t\bigl( \partial_x- y
\partial_r)\\
R_{NA}(\XX_2) &= t\bigl( \partial_y + x
\partial r)& R_{NA}(\HH_r)&= t\partial_t
\eadl
in the coordinates $(x,y,r,t) \leftrightarrow \nm(x,y,r)\am(t)$. To see
this we use \eqref{expnlie} and the fact that
$\exp(x\HH_r) \= \am(e^x)$. The action of $\HH_r$ is the same for
$\Four_{\ell,c} f$ and $f.$ For the function $f:n\mapsto f_{NA}(n a)$
with fixed $a\in A$ we have
\[ \Four_{\ell,c} f(na) = \sum_{m\geq 0} \Th_{\ell,c}(h_{\ell,m};n)
\, \Bigl( f , \Th_{\ell,c}(h_{\ell,m})
\Bigr)_{\Ld_\s\backslash N}\]
where
\[ \Bigl( f_1, f_2 \Bigr)_{\Ld_\s\backslash N} = \frac\s2
\int_{\Ld_\s\backslash N} f_1(n) \overline{f_2(n)}\, dn\,.\]
Partial integration gives for $j=0,1,2$
\[ \Bigl( \XX_j f_1, f_2 \Bigr)_{\Ld_\s\backslash N} = - \Bigl( f_1,
\XX_j f_2 \Bigr)_{\Ld_\s\backslash N}\,.\]
This gives the desired formula $ \XX_0 \Four_{\ell,c} f = 
\Four_{\ell,c} \XX_0 f,$ since $d\pi_{2\pi\ell}(\XX_0) h_{\ell,m} =
\pi i \ell h_{\ell,m};$ see \eqref{dpild}.

Let $\e=\sign(\ell).$ For the other basis elements we obtain with
Table~\ref{tab-hermdiff}, p~\pageref{tab-hermdiff}:
\begin{align*}
\Th_{\ell,c}&(h_{\ell,m} )\,\Bigl(
(\e\XX_1\mp i\XX_2) f, \Th_{\ell,c}(h_{\ell,m})
\Bigr)_{\Ld_\s\backslash N}\displaybreak[0]
\\
& =
-\Th_{\ell,c}(h_{\ell,m})\, \Bigl( f,
(\e\XX_1\pm i \XX_2)
\Th_{\ell,c}(h_{\ell,m})
\Bigr)_{\Ld_\s\backslash N}\displaybreak[0]
\\
&= -\Th_{\ell,c}(h_{\ell,m}) \,
\begin{cases}
\Bigl( f, -4i\sqrt{\pi |\ell|(m+1)}\, \Th_{\ell,c}(h_{\ell,m+1})
\Bigr)_{\Ld_\s\backslash N}& \pm=+\,,
\\
\Bigl( f, -4i \sqrt{2\pi |\ell|m} \, \Th_{\ell,c}(h_{\ell,m-1})
\Bigr)_{\Ld_\s\backslash N}&\pm=-\,;
\end{cases}
\displaybreak[0]\\
(\e X_1&\mp i \XX_2) \, \Th_{\ell,c}(h_{\ell,m}) \, \bigl( g,
\Th_{\ell,c}(h_{\ell,m}) \bigr)_{\Ld_\s\backslash N}\displaybreak[0]
\\
&= \bigl( f, \Th_{\ell,c}(h_{\ell,m})\bigr)_{\Ld_\s\backslash N} \cdot
\begin{cases}
-4i \sqrt{2\pi|\ell|m} \,\Th_{\ell,c}(h_{\ell,m-1})
&\pm =+\,,\\
-4i \sqrt{\pi|\ell|(m+1)}\,\Th_{\ell,c}( h_{\ell,m+1})
&\pm=-\,.
\end{cases}
\end{align*}
Taking the sum over $m\in \ZZ_{\geq 0}$ we obtain equality.

In this way we conclude that $\Four_{\ell,c}$ is an intertwining
operator. The subspaces $\Ffu_{\ell,c, d}$ in the decomposition
\eqref{Fcl-decomp} are invariant $(\glie,K)$-modules. Hence the
operators $\Four_{\ell,c, d}$ are intertwining operators.
\end{proof}

\subsection{Normalization of standard lattices}\label{sect-normLd}
In the proof of Proposition~\ref{prop-isogab} we used the left
translation \il{Lh}{$L_h$}$(L_h f)(g)=f(hg)$ with $h\in AM$ to get an
isomorphism of $(\glie,K)$-modules between large abelian Fourier term
modules. Here we consider left translations that preserve
$\Ld_\s$-invariance on the left.
\begin{prop}Let $\Ld_\s$ be a standard lattice. The normalizer
\il{NP}{$\Norm_P(\Ld_\s),\; \Norm_N(\Ld_\s),\, \Norm_M(\Ld_\s)$}
\be \Norm_P(\Ld_\s) \= \bigl\{ p\in NAM \;:\; p \Ld_\s p^{-1}=
\Ld_\s\bigr\}\ee
is the semi-direct product of the groups
\be \Norm_N(\Ld_s) \= \bigl\{ \nm(\bt/\s,\rho\bigr) \in N\;:\; \bt\in
\ZZ[i],\; \rho\in \RR\bigr\}\,,\ee
and
\be \Norm_M(\Ld_\s) \= \bigl\{ \mm(\z) \in M \;:\; \z^{12}=1\bigr\}\ee
\end{prop}
\begin{proof}Suppose that $p=\nm(\bt,\rho)\am(\tau)\mm(\z)$ normalizes
$\Ld_\s$. Then $p \nm(b,r)p^{-1} = \nm(\tau \z^3, r')$ for some
$r'\in \RR$. The projection $N\mapsto N/Z_N$ sends $\Ld_\s$
to~$\ZZ[i]$, and conjugation by $p$ descends to $b\mapsto \tau\z^3 b$
in~$\CC$. This leaves $\ZZ[i]$ invariant only if
$\tau \z^3 \in \ZZ[i]$. So $\tau=1$ and $\z^{12}=1$.

We consider the action by conjugation of $p =\nm(\bt,\rho)$ on basis
elements of $\Ld_\s$:
\bad p\nm(1,0)p^{-1}&\= \nm\bigl(\z^3, -2\im(\bt/\z^3)\bigr)\,,\\
p\nm(i,0)p^{-1}&\= \nm\bigl(i \z^3, 2\re(\bt/\z^3)\bigr)\,,\\
p\nm\bigl(0,2/\s\bigr)p^{-1} &\= \nm\bigl(0,2/\s\bigr) \,.\ead
This gives the requirement that $\s \bt \in \ZZ[i]$. A direct check
shows that these elements normalize~$\Ld_\s$.
\end{proof}

Conjugations by $\nm(0,\rho)$ or by $\mm(\z)$ where $\z$ is a third root
of unity, induce the identity on $\Ld_\s$.

In Table~\ref{tab-ltrFfu} we describe the action of elements in
$\Norm_P(\Ld_\s)$ on basis vectors for the large Fourier term modules.
\begin{table}[ht]
\[ \renewcommand\arraystretch{1.2}\begin{array}{|c|c|c|}\hline
&F= \ch_\bt f \Kph hprq & F=\Th_{\ell,c}(h_{\ell,m}) f \Kph hprq \\
\hline
n= \nm(0,\rho)& L_n F = F & L_n F = e^{2\pi i \rho \ell} F\\
n=\nm(1/\s,0) & L_n F = \ch_\bt(n) F & L_n F = e^{2\pi i c/\s} F\\
n=\nm(i/\s,0) & L_n F = \ch_\bt(n) F & L_n F =
\Th_{\ell,c+2\ell/\s}(h_{\ell,m}) f \Kph h p r q\\
m= \mm(e^{2\pi i/3}) & \multicolumn{2}{|c|}{L_m F = e^{\pi i(h-3r)/3}
F}\\
m=\mm(i) & L_m F = e^{\pi i(h-3r)/4}\,\ch_{-i\bt} f \Kph h p r q& \ast
\\ \hline
\end{array}
\]
\caption[]{Action of elements of $\Norm_P(\Ld_\s)$ on basis vectors in
$\Ffu_\bt$ and $\Ffu_{\ell,c}$. For the case marked with $\ast$
see~\eqref{Lmi}. }\label{tab-ltrFfu}
\end{table}

Left translation by an element of $\Norm_N(\Ld_\s)$ is absorbed in the
character $\ch_\bt$ in the abelian case, and handled with
\eqref{ThNshift} in the non-abelian case. Left translation by $\mm(\z)$
has the effect
\[ \mm(\z) \nm(b,r) \am(t) k \= \nm(\z^3 b,r) \; \am(t)\; \mm(\z)
k\,.\]
We use that $L(\HH_i) \Kph h p r q = i\frac{h-3r}2\, \Kph h p r q$. In
the non-abelian case we need Proposition~\ref{prop-actmi} and arrive at
the following description.
\badl{Lmi} L_{\mm(i)} \Th_{\ell,c}(h_{\ell,m}) f \Kph hprq
&\= \frac{ e^{\pi i(h-3r)/4} \bigl(i\sign(\ell)\bigr)^m
}{\sqrt{2|\ell|}}\\
&\qquad\hbox{} \cdot
\sum_{c'\bmod 2\ell} e^{-\pi i c c'/\ell} \Th_{\ell,c'}(h_{\ell,m})\,
f\, \Kph hprq\,.
\eadl

The left translation gives injective morphisms of $(\glie,K)$-modules.
In most cases left translation by the element in the table gives in the
abelian case isomorphisms $\Ffu_\bt\rightarrow \Ffu_\bt$, except for
$L_{\mm(i)}: \Ffu_\bt \stackrel{\cong}{\rightarrow} \Ffu_{-i\bt}$. In
the non-abelian case the quantity $d = 3\sign(\ell)2m+1) h-3r$ is
preserved. For $\nm(1/\s,\rho)$and $\mm(e^{2\pi i/3})$ we get
isomorphisms $\Ffu_{\ell,c,d}\rightarrow \Ffu_{\ell,c,d}$. The other
cases yield isomorphisms of the direct sum
$\oplus_{c\bmod 2\ell} \Ffu_{\ell,c,d}$ into itself.


\def\flnm{rFtm-II-Ftm}

\section{Central action} \markright{9. CENTRAL ACTION}\label{sect-Ftm}
Automorphic forms are not only $\Gm$-invariant, but also eigenfunctions
of the differential operators corresponding to elements of the center
of the enveloping algebra of $\glie.$ Since the Fourier term operators
are, by Proposition~\ref{prop-Ftit}, intertwining operators of
$(\glie,K)$-modules, the Fourier terms of automorphic forms are
elements of submodules of $\Ffu_\bt$ and $\Ffu_\n$ on which $ZU(\glie)$
acts by a character. We call the resulting submodules Fourier term
modules.

We start the study of these submodules by parametrizing the characters
of $ZU(\glie)$, and next translating the condition that $ZU(\glie)$
acts by a character into a system of coupled linear differential
equations. The coupling makes these `eigenfunction equations' too hard
to allow us to solve them explicitly, except in special cases.

Nevertheless, the explicit availability of these eigenfunction equations
is put to work in the last three subsections. We get information on the
set of $K$-types that are present in Fourier term modules, and some
information on the multiplicity; see Propositions~\ref{prop-sectors}
and~\ref{prop-dim}. We arrive at these results by Proposition
\ref{prop-kdso} and~\ref{prop-kuso}, which give necessary conditions
for shift operators to have a non-trivial kernel on a given
$K$-type.\medskip

Let $\ps$ be a character of $ZU(\glie)$. For $\Nfu=\Nfu_\bt$ of
$\Nfu=\Nfu_\n$ we define $\Ffu_\Nfu^\ps$ as the $(\glie,K)$-submodule
of functions $F\in\Ffu_\Nfu$ that satisfy $ u F= \ps(u) F$ for all
$u\in ZU(\glie)$.\il{Ffups}{$\Ffu_\bt^\ps,\, \Ffu_\n^\ps$} This
submodule is much smaller than $\Ffu_\Nfu$; we call it a
\il{Ftm}{Fourier term module}\emph{Fourier term module}. Since
$ZU(\glie)$ is a polynomial algebra in the Casimir element $C$ and the
element $\Dt_3$ of degree~$3$, the character $\ps$ is determined by its
values $\ps(C), \ps(\Dt_3)\in \CC$.

The left translations by elements normalizing $\Ld_\s$ as discussed in
\S\ref{sect-normLd} are intertwining operators of $(\glie,K)$-modules.
So they preserve the Fourier term modules.
\rmrk{Example} Let us consider the function in $\Ffu_0$ given by
\be \ph\bigl(n\am(t)k\bigr) \= t^{2+\nu} \, \Kph {2j}000(k)\ee
with $\nu \in \CC$, $j\in \ZZ$. Application of the formulas in
Table~\ref{tab-shab}, p~\pageref{tab-shab}, or a computation in
\cite[\S11c]{Math}, gives
\bad \sh 3 1 \ph &\= \bigl( 1+\frac{\nu+j}2\bigr) \,
t^{2+\nu}\Kph{2j+3}{1}{1}{1}\,,\\
\sh{-3}{-1}\sh3 1 \ph &\= \frac18\bigl(\nu^2-(j+2)^2\bigr) \ph\,. \ead
Since $\ph$ has $K$-type $\tau^{2j}_0$, it is a minimal vector in
$\Ffu_0$. With iii) and iv) in Lemma~\ref{lem-mv} we
obtain:\ir{ld23}{\ld_2(j,\nu)\,,\; \ld_3(j,\nu)}\il{ld23}{$\nu$}
\badl{ld23} C \ph &= \ld_2(j,\nu)\,\ph\,,& \qquad
\ld_2(j,\nu)&\;:=\;\nu^2-4+\frac13 j^2\,,\\
\Dt_3 \ph &= \ld_3(j,\nu)\,\ph\,,& \qquad \ld_3(j,\nu)&\;:=\;
(j+3)\bigl(\nu^2 -\tfrac19(j-6)^2\bigr)
\eadl

\rmrk{Parametrization by Weyl group orbits}The functions $\ld_2$ and
$\ld_3$ on $\CC^2$ are invariant under the transformations
\ir{S12}{\Ws1,\;\Ws2\in W}
\badl{S12} \Ws1: (j,\nu) &\mapsto \Bigl(\frac12(3\nu-j),\frac12(j+\nu)
\Bigr)\,,\\
\Ws2: (j,\nu) &\mapsto \Bigl(
-\frac12(3\nu+j), \frac12(\nu-j)\Bigr)\,. \eadl

This is in agreement with a theorem of Harish Chandra for general
reductive Lie groups, stating that characters of the center of the
enveloping algebra correspond bijectively with the Weyl group orbits of
invariant polynomial functions on a Cartan subalgebra. See, eg,
\cite{Wal88}, \S3.2, especially Theorems 3.2.3 and~3.2.4.

In our notations we use the Cartan subalgebra $\alie_c\oplus \mlie_c $,
and identify $\CC^2$ with its dual space, by letting the simple roots
in \eqref{al12} satisfy the correspondence
$\al_1 \leftrightarrow  (-3,1)$, $\al_2 \leftrightarrow (3,1)$. Then
$(j,\nu)\in \CC^2$ corresponds to the linear form that satisfies
\[ \HH_r \mapsto \nu\,,\qquad \HH_i \mapsto ij\,.\]
A linear form on the Lie algebra corresponds to a (possibly
multi-valued)
character of the group. For the linear form corresponding to
$(j,\nu) \in \CC^2$:
\[ \am(e^x) \mm(e^{iy}) \= \exp( x\HH_r+y\HH_i) \mapsto e^{ x \nu + y
(i j)}\= \bigl( e^x\bigr)^\nu\, \bigl( e^{iy}\bigr)^j\,. \]
Since $\mm(e^{2\pi i})=1$, only the
$(j,\nu) \in \ZZ\times\CC \subset \CC^2$ correspond to a character of
the group $AM$. The condition $j\in \ZZ$ would not be necessary if we
were to work with the universal covering group of $\SU(2,1)$. This is a
point where it is important to work with $(\glie,K)$-modules, and not
just $\glie$-modules.

The \il{wg}{Weyl group}Weyl group \il{W}{$W$}$W$ for $\SU(2,1)$
corresponds to the group of linear transformations of $\CC^2$ generated
by $\Ws1$ and $\Ws2$ in~\eqref{S12}. It is isomorphic to the symmetric
group $\mathrm{S}_{3}$. The general theory tells us that the characters
of $ZU(\glie)$ are parametrized by the orbits of $W$ in $\CC^2$. Not
all of these characters can occur in a $(\glie,K)$-module.
\begin{prop}Let $V$ be a $(\glie,K)$-module in which $ZU(\glie)$ acts by
multiplication by the character $\ps$. Then there exist elements
$(j,\nu)\in \ZZ\times \CC$ such that
\be \ps(C) \= \ld_2(j,\nu),\qquad \ps(\Dt_3)\=\ld_3(j,\nu)\,. \ee
If $(j_1,\nu_1)\in \CC^2$ also satisfies this relation, then
$(j_1,\nu_1)$ is in the orbit of $(j,\nu)$ under the Weyl group $W$.
\end{prop}
We call $j\in \ZZ$ and $\nu\in \CC$ \emph{spectral
parameters}.\il{spprm}{spectral parameters}
\begin{proof}We pick a non-zero element in a highest weight subspace of
$V$ of some $K$-type, and apply downward shift operators until we have
reached a minimal vector~$v$. Lemma~\ref{lem-mv} implies $v$ is also an
eigenvector of $\sh {-3}{-1} \sh 3 1$, say with eigenvalue $\th$.
Inserting $\ld_2(j,\nu)$ and $\ld_3(j,\nu)$ as eigenvalues of $C$ and
$\Dt_3$ we get two relations between $j$, $\nu$ and $\th$. Solving
these relations with Mathematica, \cite[\S11d]{Math}, leads to six
solutions for $(j,\nu)$. Among these solutions there are two solutions
of the form
\[ j\= \frac{h-3p}2\,,\quad \nu = \pm\sqrt{\text{a complicated
expression in $p$, $h$ and $\th$}}\,.\]
Since $h\equiv p\bmod 2$ this shows that $(j,\nu)$ with the desired
properties exist.

For the solutions of $\ld_n(j_1,\nu_1)=\ld_n(j,\nu)$ for $n=2,3$ we find
precisely the orbit $W(j,\nu)$, illustrated in Figure~\ref{fig-Wo}.
\end{proof}
\begin{figure}[tp]
\begin{center}
\grf{9}{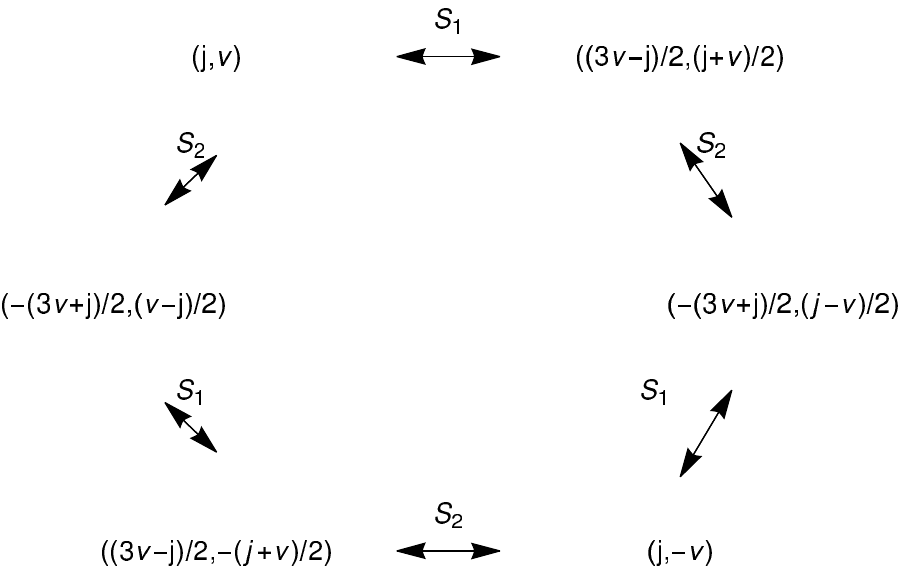}
\end{center}
\caption{Weyl group orbit of a point in $\CC^2.$} \label{fig-Wo}
\end{figure}

In this way we have concluded directly for $\SU(2,1)$ how the characters
of $ZU(\glie)$ that occur in $(\glie,K)$-modules are parametrized. The
result is in accordance with the general theory.

\rmrk{Several types of parametrization}By \il{WO}{$\WO)$}$\WO$ we denote
the collection of the $W$-orbits in $\CC^2$ that intersect
$\ZZ\times \CC$. We use the symbol $\ps$ in two ways. It may denote an
element of $\WO$, and it may denote the corresponding character of
$ZU(\glie)$. If $(j,\nu)\in \CC^2$ we denote the set $W(j,\nu) \in \WO$
by $\psi[j,\nu]$.\il{jnu}{$\psi[j,\nu]$}

We put for
$\ps\in \WO$\ir{wo}{\wo}\il{won}{$\wo(\ps)_\n,\; \wo^1(\ps_\n)$}
\badl{wo} \wo(\ps) &\= \bigl\{ (j,\nu)\in \ps\;:\; j\in \ZZ\bigr\}\,,
\qquad \wo^1(\ps) \= \bigl\{ j\;:\; (j,\nu)\in \wo(\ps)\bigr\}\,,\\
\wo(\ps)_\n&\=\wo(\ps)_{\ell,c,d} \= \bigl\{ (j,\nu) \in
\ps\;:\;\sign(\ell)(2j-d)+3\leq 0\bigr\}\,,\\
\wo^1(\ps)_\n &\= \bigl\{ j\;:\; (j,\nu)\in \wo(\ps)_\n\bigr\}\,. \eadl

In general $\wo\bigl(\psi[j,\nu]\bigr))$ has 2 elements, $(j,\nu)$ and
$(j,-\nu).$ If the number of elements in $\wo(\ps)$ is at most $2$ we
speak of \il{sp}{simple parametrization}simple parametrization. The set
$\wo\bigl(\psi[j,\nu]\bigr)$ has more than 2 elements if
$3\nu \equiv j\bmod 2,$ in which case we speak of \il{mp}{multiple
parametrization} multiple parametrization. We also make a distinction
between \il{ip}{integral parametrization}integral parametrization and
\il{gp}{generic parametrization}generic parametrization, summarized in
the scheme in Table~\ref{tab-parms}. The corresponding subsets of $\WO$
are indicated by the symbols $\WOS$, $\WOG$, $\cdots$ indicated in the
table.\il{WSMGI}{$\WOS$}
\begin{table}[htp]
\[
\renewcommand\arraystretch{1.6}
\begin{array}{|c|c|c|}\hline
3\nu \not\equiv j\bmod 2 & 3\nu \equiv j\bmod 2 & \nu \equiv j \bmod 2
\\
\text{or }j=\nu=0
& \text{and }\nu \not\equiv j\bmod 2& \text{and }(j,\nu) \neq
(0,0)
\\ \hline
\text{simple parametrization}& \multicolumn{2}{|c|} {\text{multiple
parametrization}\;\;\WOM}\\
\bigl| \wo(j,\nu)\bigr| \leq 2 &
\multicolumn{2}{|c|}{\bigl|\wo(j,\nu)\bigr| >2 }
\\ \hline
\multicolumn{2}{|c|}{\text{generic parametrization}\;\;\WOG}&
\text{integral parametrization} \\ \hline
\WOS & \WOGM&\WOI
\\ \hline
\end{array}
\]
\caption[]{Several types of parametrization of characters $\psi[j,\nu]$
of $ZU(\glie)$.\\
For each case we indicate the symbol corresponding to the set of
$\ps\in \WO$ with that type of parametrization. } \label{tab-parms}
\end{table}

The points $(j,\nu)$ giving integral parametrization form a lattice in
$\RR^2$ minus the origin. See Figure~\ref{fig-ip} for an illustration.
\begin{figure}[tp]
\begin{center}\grf{10}{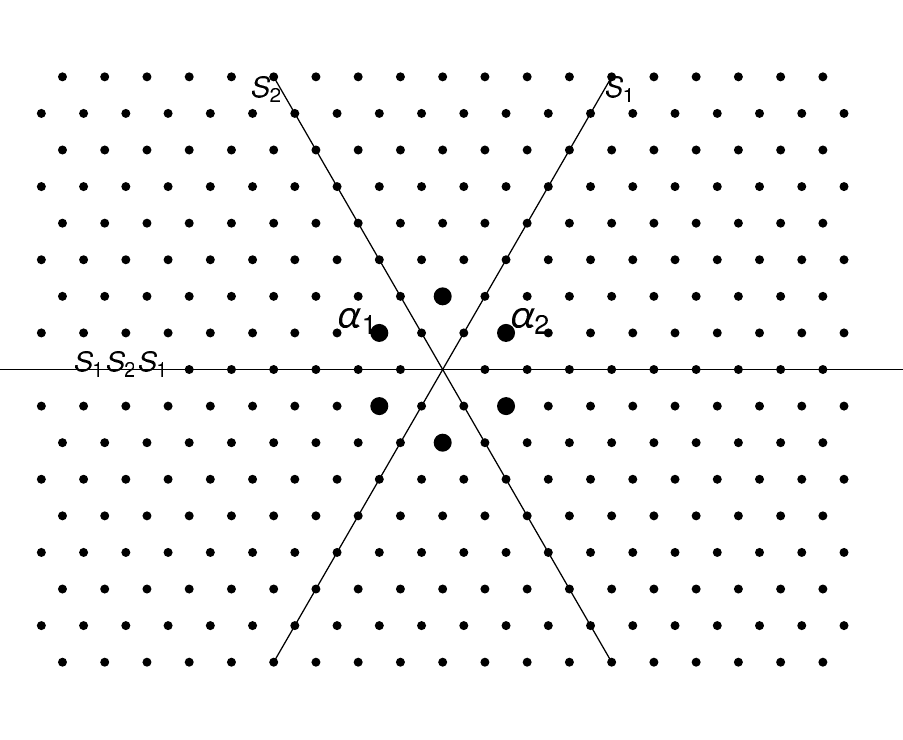}\end{center}
\caption{Points for integral parametrization, depicted in the
$(j,\nu)$-plane ($j$ horizontal, $\nu$ vertical.) We have chosen the
scaling such that the lines fixed by $\Ws 1$ (given by $\nu=j$), fixed
by $\Ws 2$ (given by $\nu=-j$), and fixed by $\Ws 1\Ws 2\Ws 2$ (given
by $\nu=0$) intersect each other in angles of size $\pi/3.$ The thick
points indicate the position of the root system $\mathrm{A}_2.$
}\label{fig-ip}
\end{figure}

\begin{lem}\label{lem-gmp}Let $\ps\in \WOGM$. Then the elements of
$\wo^1(\ps)$ represent all three classes in $\ZZ/3\ZZ.$
\end{lem}
\begin{proof}The element $\psi[j,\nu]$ corresponds to generic multiple
parametrization if $\nu = \frac13(j+2a)$ for some $a\in \ZZ$ and
$a\not\equiv j \bmod 3.$ Then $\wo^1(j,\nu)=\bigl\{ j,a, -a-j\}.$ One
checks that these three elements represent pairwise different elements
of $\ZZ/3\ZZ.$
\end{proof}

\subsection{Eigenfunction equations}\label{sect-eieq} The definition of
the Fourier term modules $\Ffu^\ps_\Nfu$ imposes relations of the form
$CF=\ps(C)F$ and $\Dt_3 F = \ps(\Dt_3)F$. To make these relations more
explicit we note that $\Ffu_\Nfu^\ps$ is the direct sum of the
subspaces $\Ffu_{\Nfu;h,p}^\ps$ of the $K$-types occurring in it, and
that the space $\Ffu^\ps_{h,p}$ is known if we know the highest weight
subspace $\Ffu^\ps_{\Nfu;h,p,p}.$ A highest weight element in
$\Ffu_{\Nfu;h,p,p}^\ps$ has the form
\be\label{Fcomp} F \bigl( n \am(t)k\bigr)
= \sum_r u_r(n) \, f_r(t)\, \Kph h {p}{r}{p} (k)\ee
with basis functions $u_r$ on $N,$ component functions $f_r$ in
$C^\infty(0,\infty),$ and the polynomial basis functions on~$K$
discussed in~\S\ref{sect-K}. The summation parameter runs over
$r\equiv p \bmod 2$, $|r|\leq p$, with the additional condition
\eqref{r0cond} in the non-abelian case. The prescribed action of $C$
and $\Dt_3$ imposes further relations for the components, which we call
the eigenfunction equations.\il{efeq}{eigenfunction equations}

To specify the eigenfunction equations it is convenient to distinguish
the following cases:
\begin{itemize}
\item $N$-trivial Fourier term modules:
$\Nfu=\Nfu_0$\,.\il{Ntefm}{$N$-trivial Fourier term module}
\item Generic abelian Fourier term modules: $\Nfu=\Nfu_\bt,$ $\bt\neq0.$
\il{gabef}{generic abelian Fourier term module} \il{Ftmgab}{Fourier
term module, generic abelian}
\item Non-abelian Fourier term modules: $\Nfu=\Nfu_{\ell,c,d}$ with
$\ell \in \frac \s 2\ZZ_{\neq 0},$ $c\bmod 2\ell,$ $d \in 1+2\ZZ.$
\il{nabef}{non-abelian Fourier term module}\il{Ftmnab}{Fourier term
module, non-abelian}
\end{itemize}
Except for special cases, the eigenfunction equations are complicated
and ask for computer help. We use the differentiation routines in
\S\ref{sect-exdf} and adapt them, in \cite[\S11a]{Math}, to the use
with elements of the enveloping algebra $U(\glie)$. For each of of the
three cases indicated above we develop routines to give the
eigenfunction equations. We choose to write them with
$(j,\nu) \in \ZZ\times \CC$ as the parameters. Hence the equations will
differ if we go to another element of $\wo(\psi[j,\nu])$.

In \cite[\S11efg]{Math} we derive the eigenfunction equations, and keep
in Section 11h routines for later use, to avoid recomputation of the
relation every time that we need them in later sections.
\begin{lem}\label{lem-efeq-t}{\rm $N$-trivial Fourier terms.\ } The
components $f_r$ of the element $F\in \Ffu_0^\ps$ in~\eqref{Fcomp} must
satisfy
\badl{eft} 0&= t^2 f_r'' - 3 t f' + \biggl( \frac{(h-3r)^2}{12}-\nu^2+4
- \frac{j^2}3 \biggr) f_r\,,\\
0&=
(h-3r-2j)(h-3r-3\nu+j)(h-3r+3\nu+j)\, f_r\,. \eadl
\end{lem}
\begin{proof}A Mathematica computation in \cite[\S11e]{Math} shows that
\[ \bigl(C-\ld_2(j,\nu) \bigr) \sum_r f_r \Kph h{p}{r}{q}\]
gives an expression in which the factor of $\Kph h{p}{r}{q}$ depends
only on $f_r$ and its derivatives. We get an uncoupled system of
differential equations for $f_r.$ Up to a factor these are the
differential equations in the first line of~\eqref{eft}.

The equation for $\Dt_3$ also gives an uncoupled system of differential
equations. Subtraction of $\frac12(h-2r+6)$ times the first equation
gives the relation in the second line.
\end{proof}

The three factors in the second equation are permuted by the action of
the Weyl group on $(j,\nu).$ This is like it should be, since the
eigenvalues of $C$ and $\Dt_3$ are invariant under the action of the
Weyl group.

The second equation shows that non-zero solutions are possible only for
certain values of~$r.$
\begin{lem}\label{lem-efrab}{\rm Abelian Fourier terms.\ } The
components $f_r$ of the element $F\in \Ffu_\bt^\ps$ in~\eqref{Fcomp}
must satisfy the following relations, where $r\equiv p\bmod2$,
$|r|\leq p$:
\begin{align*} 0&= t^2 f''_r - 3 t f'_r+\Bigl(
\frac{(h-3r)^2}{12}-\nu^2+4
-\frac{j^2}3-4\pi|\bt|^2 t^2\Bigr)
f_r\\
&\qquad\hbox{}
+ 2\pi i (p-r)\, \bar\bt \,t\, f_{r+2}-2\pi i
(p+r)\,\bt\, t \,f_{r-2}\,, \displaybreak[0]\\
0&= \Bigl(
(h-3r-2j)(h-3r-3\nu+j)(h-3r+3\nu+j)
\\
&\qquad\qquad\hbox{}+ 216 \pi^2 |\bt|^2 r\, t^2 \Bigr)\, f_r
\displaybreak[0]\\
&\qquad\hbox{}
-27\pi i (p-r)\,\bar\bt\,\Bigl( 2t^2 \, f'_{r+2} +(3r-h-2)\,t\,
f_{r+2}\Bigr)
\\
&\qquad\hbox{}
-27\pi i (p+r)\, \bt\, \Bigl( 2t^2 f'_{r-2}
+
(h-3r-2) \, t \, f_{r-2}\Bigr)\,. \end{align*}
\end{lem}
\begin{proof}The computation in \cite[\S11f]{Math} is more complicated
than in the $N$-trivial case, since there are more terms. Moreover,
neighboring components are coupled. We have to rearrange the sum of
individual terms in such a way that all contain the same factor
$\Kph h{p}{r}{p}.$ After that we carry out a simplification by
subtracting a multiple of the first relation from the second one.
\end{proof}

The relation should be valid for all $r$ between $-p$ and $p.$ The terms
with $f_{r\pm 2}$ contain the factor $p\mp r,$ which masks components
that do not exist.

Substitution of $\bt=0$ in Lemma~\ref{lem-efrab} gives \eqref{eft}.
\rmrk{Non-abelian Fourier terms}The computations in \cite[\S11g]{Math}
give eigenfunction equations in non-abelian Fourier term modules that
are complicated, and copying them here seems not to make sense. In
Table~\ref{tab-shnab}, p~\pageref{tab-shnab}, we managed to describe
the shift operators uniformly for $\ell>0$ and $\ell<0$. For the
eigenfunction equations it is simpler to consider separate formulas for
$\e=\sign(\ell)=1$ and $-1$. The final form is in \cite[\S11h]{Math}.

\subsection{One-dimensional K-types}\label{sect-1dKt}
In general, the coupling between the equations for the components is an
obstruction to get explicit solutions. The one-dimensional $K$-types
are an exception, since then there is only one component function.

The functions in $\Ffu_{\Nfu;h,0}^{\psi[j,\nu]}$ have the form
\[ F\bigl( n\am(t)k) = u(n) \, f(t)\, \Kph h000(k)\]
with the function $u\in C^\infty(N)$ determined by $\Nfu,$ and
$f\in C^\infty (0,\infty).$ In all cases the eigenfunction equations
give a second order differential equation for $f$ and a second equation
depending only on $f,$ imposing conditions on $h.$

The computations in \cite[\S12]{Math} show that in all cases this
condition is
\be (h-2j)(h-3\nu+j)(h+3\nu+j) = 0\,. \ee
So $h=2j$ is one possibility. The other factors are relevant only under
multiple parametrization.
(See Table~\ref{tab-parms} on page~\pageref{tab-parms}.) Each element
$j'\in \wo^1(\psi[j,\nu])$ gives $h=2j'$ as a possibility.

\subsubsection{N-trivial Fourier term modules} \label{sect-Ntr1d}From
Lemma~\ref{lem-efeq-t} we get the differential equation
\be t^2 f'' -3 t f' +(4-\nu^2) f=0\,.\ee
It has a two-dimensional solution space spanned by $t\mapsto t^{2+\nu}$
and $t\mapsto t^{2-\nu}$ if $\nu\neq 0,$ and by $t\mapsto t^2$ and
$t\mapsto t^2\log t$ if $\nu=0.$

\subsubsection{Generic abelian Fourier term
modules}\label{sect-Nab1d}For $\Ffu_\bt^{\psi[j,\nu]}$ with $h=2j,$ we
find for $\bt\neq0$ that
\be f(t) = t^2 j_\nu(2\pi|\bt|t)\ee
where $j_\nu$ is a solution of the modified Bessel differential
equation\ir{mBd0}{\text{modified Bessel differential equation}}
\be\label{mBd0} \tau^2 \, j_\nu''(\tau) + \tau \, j_\nu'(\tau)
-(\tau^2+\nu^2) \, j_\nu(\tau)\=0\,.\ee
We get a two-dimensional solution space spanned by modified Bessel
functions, discussed in~\S\ref{sect-mBf}.

\subsubsection{Non-abelian Fourier term modules} \label{sect-Nnab1d}In
\cite[\S12c]{Math} we obtain in the same way for
$\Ffu^{\psi[j,\nu]}_{\ell,c,d}$ a component function of the form
\be f(t) = t\, w_{\k,\nu/2}(2\pi|\ell|t^2)\,,\ee
where $w_{\k,s}$ is a solution of the Whittaker differential
equation\ir{Whd0}{\text{Whittaker differential equation}}
\be\label{Whd0} w_{\k,s}'' = \Bigl( \frac 14
- \frac\k \tau + \frac{s^2-1/4}{\tau^2} \Bigr)\, w_{\k,s}\ee
with parameters $\k=-m-\bigl( j \,\sign(\ell)+1\bigr)/2,$
$m\in \ZZ_{\geq 0},$ and $s=\nu/2.$ So here as well we have a
two-dimensional solution space spanned by known functions. In
\S\ref{sect-Whit} we discuss facts concerning Whittaker functions.

We note that the definition of $\Ffu_{\ell,c,d}$ implies that
\be m = \frac{\sign(\ell)}6(d-2j) - \frac12\,,\qquad \k =
-\frac16 \sign(\ell)
\,(d+j)\,. \ee
In \eqref{mucond} we arrived at a condition that is, for
$\tau^h_p=\tau^{2j}_0$, equivalent to
$\sign(\ell) d \geq 2\sign(\ell)j+3$. It is equivalent to $m \geq 0$.
So $\Ffu_{\n;2j',0}^{\psi[j,\nu]}$ is non-trivial with dimension~$2$ if
and only if $j'$ is in the set $\wo^1\bigl(\psi[j,\nu]\bigr)_\n$.
(See~\eqref{wo}.)

\rmrk{Summary}Let $\bt\in\CC$. The space $\Ffu^\ps_{\bt;2j,0}$ is
non-trivial if and only if $j\in \wo^1(\ps)$. Let $\n=(\ell,c,d)$. The
space $\Ffu^\ps_{\n;2j,0}$ is non-trivial if and only if
$j\in \wo^1(\ps)_\n$. These spaces of $K$-type $\tau^{2j}_0$ have
dimension~$2$.
\rmrk{Remark}For higher-dimensional $K$-types we arrive, for generic
abelian and non-abelian Fourier term modules, at coupled systems of
differential equations for which it is hard to get explicit solutions,
except in special cases.
\rmrk{Aim}We now proceed the investigation of Fourier term modules by
deriving conditions under which the shift operators may have a
non-trivial kernel.

\subsection{Kernels of downward shift operators} We consider Fourier term
modules $\Ffu_\Nfu^\ps$ with $\Nfu=\Nfu_\bt$ or $\Nfu=\Nfu_\n,$
$\n = (\ell,c,d).$ It suffices to consider the kernels of downward
shift operators on $K$-types $\tau^h_p$ with $h\equiv p\bmod 2$ and
$p\in \ZZ_{\geq 0}.$\il{dso1}{downward shift operator}

\begin{prop}\label{prop-kdso}Let $\ps\in \WO$. If the kernel of
$\sh{\pm 3}{-1}:\Ffu^{\ps}_{\Nfu;h,p,p}
\rightarrow \Ffu_{\Nfu;h\pm 3,p-1,p-1}^{\ps}$ is non-trivial then
\be\label{kdso} h\mp 3p = 2j\ee
for some $j\in \wo^1(\ps).$
\end{prop}
This proposition gives only a necessary condition for $\sh{\pm3}{-1}$ to
have a non-trivial kernel on $\Ffu_{\Nfu;h,p,p}^{\ps}.$

\begin{proof}For $p=0$ we established the statements  at the end of
the previous subsection. We proceed under the assumption that
$p\geq 1$. Supporting computations are in \cite[\S13]{Math}.

Elements of $\Ffu^{\ps}_{\Nfu;h,p,p}$ have the form
\[ F \bigl(n\am(t)k\bigr)= \sum_{r:(-p,p)} u_r(n) \, f_r(t)\, \Kph
h{p}{r}{p}(k)\,,\]
with $u_r=\chi_\bt$ if $\Nfu=\Nfu_\bt,$ and $u_r = 	\th_{m(h,r)}$ if
$\Nfu=\Nfu_\n.$

We first consider the $N$-trivial case. Lemma~\ref{lem-efeq-t} shows
that $h-3r=j$ for some $(j,\nu)\in \wo(\ps)$. The kernel relation for
$\sh 3{-1}$ in Table~\ref{tab-krab}, p~\pageref{tab-krab}, shows how
$f_r'$ depends on $f_r$.
(We have $\bt=0$.) Inserting this in the eigenfunction relations shows
that
\[ (h-3p-3\nu+j)(h-3p+3\nu+j)\=0\]
for non-zero $f_r$. So $h=3p + 2j'$ for $j'\in \wo^1(\ps)$. This is what
the lemma states for $\sh 3{-1}$. For $\sh{-3}{-1}$ we proceed
similarly.
\smallskip

In the generic abelian case we also use the kernel relations in
Table~\ref{tab-krab}, but now with $\bt\neq 0$. The components of an
element of the non-zero element of the kernel of $\sh3{-1}$ satisfy a
relation expressing $f_{r+2}$ in terms of $f_r$. So, if the element in
the kernel is non-zero then $f_{-p}$ has to be non-zero. We use the
eigenfunction relations in Lemma~\ref{lem-efrab} for $r=p$ for some
choice of $(j,\nu)$ such that $\ps=[(j,\nu)]$. The occurrences of
$f_{-p-2}$ in the eigenfunction relation are masked by the factor
$p+r=0$. For $f_{2-p}$ we use the relation obtained from the kernel
relations. We obtain two quantities that have to be zero. A suitable
linear combination is
\[ (h-3p-2j)(h-3p+j-3\nu) (h-3p+j+3\nu) f_{-p}\,.\]
So indeed $h=3p+2j'$ for some $j'\in \wo^1(\ps)$ as a necessary
condition. The case of $\sh{-3}{-1}$ goes similarly.\smallskip

In the non-abelian case the description of an element of
$\Ffu_{\n;h,p,p}$ in terms of its components is more complicated. We
use the description and notations in~\eqref{mhrdef}.

We find for $F$ in the kernel of $\sh 3{-1}$ the relation
\be\label{efr1d} f_r = \Bigl( (h-2p-r-2+4\pi\ell t^2)
f_{r-2}+ 2 t f_{r-2}' \Bigr)
\cdot
\begin{cases}
\frac {-i}{4\sqrt{2\pi|\ell|m(h,r)}}&\text{ if }\ell>0\\
\frac i{4\sqrt{2\pi|\ell|(m(h,r)+1)}}&\text{ if }\ell<0
\end{cases}
\ee
under the condition $m(h,r)\geq 1$ if $\ell>0$ and $m(h,r)\geq 0$ if
$\ell<0.$ So non-zero elements of the kernel are determined by the
component $f_{-p}$ if $m(h,-p)\geq 0,$ or by the component $f_{m(r_0)}$
with $r_0>p$ and $m(h,r_0)=0.$ There are many cases to consider, worked
out in \cite[\S13c]{Math}. There are two easy cases: $\e=1$, $r_0=p$,
and $\e=-1$, $r_0=-p$. Then there is only one non-zero component. The
second coordinate of the eigenfunction equations gives the condition
\be\label{hpr} (h-3p-2j)\, (h-3p-3\nu+j)\,(h-3p+3\nu+j)\=0\,.\ee
This implies the necessary condition in the proposition.

The remaining cases are $\e=1$, $r_0<p$, with $r=-p$; $\e=1$
$-p\leq r_0<p$, with $r=r_0$; and $\e=-1$, $r_0>-p$, with $r=-p$. In the
eigenfunction equation occur $f_r$ and $f_{r+2}$. We use \eqref{efr1d}
to replace $f_{r+2}$ and its derivative by expressions concerning
$f_r$. Then we take a suitable linear combination of the two
coordinates of the eigenfunction equations, and observe that it gives
\eqref{hpr} in all cases.

The case of the operator $\sh{-3}{-1}$ requires also the consideration
of many cases, all of which we work out in \cite[\S13c]{Math}.
\end{proof}

The proposition has the consequence that non-trivial kernels of downward
shift operators occur only in $K$-types corresponding to points on at
most three lines in the $(h/3,p)$-plane. Figure~\ref{fig-kerdso}
illustrates this in a case of generic multiple parametrization.
\begin{figure}[tp]
\begin{center}\grf{7.6}{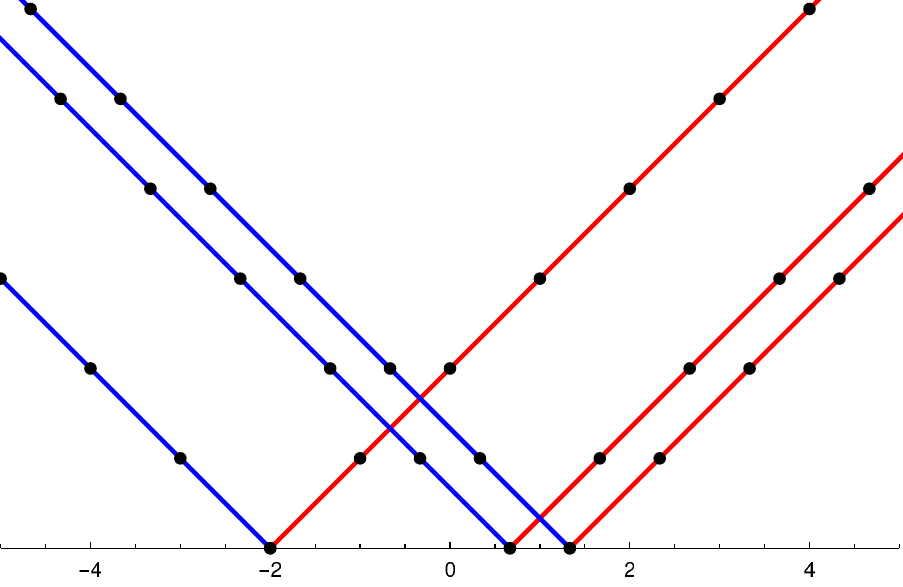}\end{center}
\caption{Points in the $(h/3,p)$-plane corresponding to $K$-types in
$\Ffu_\Nfu^{\psi[j,\nu]}$ with possibly non-trivial kernels of
$\sh3{-1}$
(lines with slope $1$)
and $\sh{-3}{-1}$ (lines with slope $-1$), for $(j,\nu)=(-3,1/3)$
(generic multiple parametrization).} \label{fig-kerdso}
\end{figure}

\begin{defn}\label{def-sect}
To each $j\in \ZZ$ we associate the set $\sect(j)$ of $K$-types
$\tau^h_p$ that satisfy $h\equiv p\bmod 2$ and
$|h-2j| \leq 3p$.\il{sect}{$\sect(j)$}
\end{defn}
The set $\sect(l)$ corresponds in the $(h/3,p)$-plane to the lattice
points satisfying $h\equiv p\bmod 2$ that are on or between the lines
$h=2j-3p$ and $h=2j+3p$ in~\eqref{kdso}.\il{slp}{sector of lattice
points in the $(h/3,p)$-plane} If $\ps\in \WOS$ there is one sector
$\sect(j)$, with $\{j\} = \wo^1(\ps)$.

If $\ps \in \WOGM$ there are three sectors, corresponding to the
elements of $\wo^1(\ps)$. These three sectors have no lattice point in
common, as illustrated in Figure~\ref{fig-sectors}. If $\ps\in \WOI$
the two or three sectors $\sect(j)$ with $j\in \wo^1(\ps)$ do have
lattice points in common.
\begin{figure}[tp]
\begin{center}\grf{9}{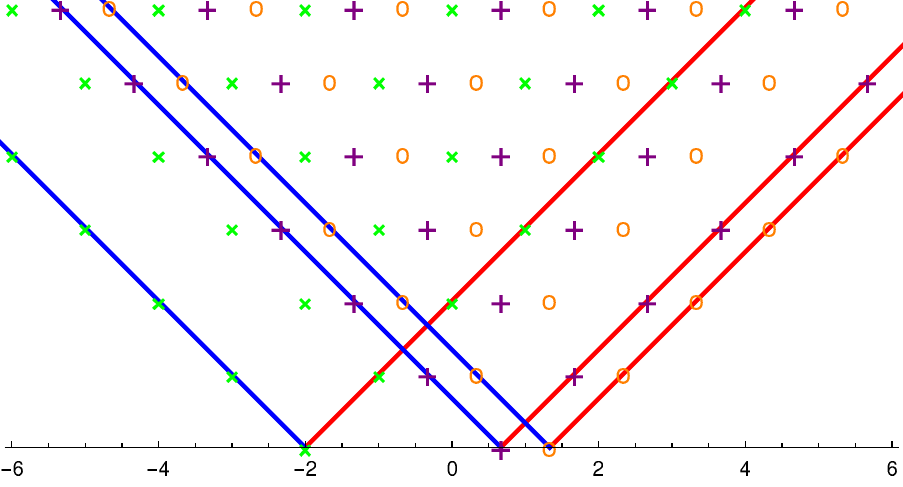}\end{center}
\caption{Points in the $(h/3,p)$-plane corresponding to $K$-types that
can occur in $\Ffu_\Nfu^\ps$ under generic multiple parametrization. We
use the same value $(j,\nu)=(-3,1/3)$ as in Figure~\ref{fig-kerdso}.}
\label{fig-sectors}
\end{figure}

\begin{prop}\label{prop-sectors}Let $\ps\in \WO$.
\begin{enumerate}
\item[i)] The $K$-types occurring in $\Ffu^\ps_\Nfu$ are contained in
\be \begin{cases}
\bigcup_{j\in \wo^1(\ps)} \sect(j)&\text{ if } \Nfu=\Nfu_\bt,\; \bt\in
\CC\,,\\
\bigcup_{j\in \wo^1(\ps)_\n} \sect(j) &\text{ if }\Nfu=\Nfu_\n\,.
\end{cases}
\ee
\item[ii)] Let $\ps\in \WOS$, $j\in \wo^1(\ps)$ if $\Nfu=\Nfu_\bt$ and
$j\in \wo^1(\ps)_\n$ if $\Nfu=\Nfu_\n$. Then for all
$p\in \ZZ_{\geq 0}$
\be\sh3{-1} \Ffu^{\ps}_{\bt;2j+3p,p,p}=\{0\}\,,\qquad \sh{-3}{-1}
\Ffu^{\ps}_{\bt;2j-3p,p,p}=\{0\}\,. \ee
\end{enumerate}
\end{prop}
\begin{proof}We take a point $v\neq 0$ in $\Ffu^\ps_{\Nfu;h,p,p}$ for a
$K$-type $\tau^h_p$ that occurs in $\Ffu^\ps_\Nfu$. If $p\geq 1$ the
point $(h/3,p)$ cannot be on both lines $h-3p=2j$ and $h+3p=2j$ in
Proposition~\ref{prop-kdso} for the same $j\in \wo^1(\ps)$. So there is
a non-zero vector $\sh 3{-1}v$ or $\sh{-3}{-1}v,$ in the $K$-type
$\tau^{h\pm 3}_{p-1}.$ If the point $(h/3,p)$ is in a sector
$\sect(j)$, then the next point is in $\sect(j)$ as well.

The process can be continued until we reach a minimal vector. That can
be a non-zero vector in a one-dimensional $K$-type, with $p=0$. In
\S\ref{sect-1dKt} we saw that this determines at most three $K$-types,
namely $\tau^{2j}_0$ with $j\in \wo^1(\ps)$, in the non-abelian case
$j\in \wo^1(\ps)_\n$.

A minimal vector can also occur in a $K$-type with $p\geq 1$. Then we
have by Proposition~\ref{prop-kdso}
\[ h=3p+2j_1 \= -3p+2j_2\]
with two different elements $j_1,j_2\in \wo^1(\ps)$. In the non-abelian
case at least one of $j_1$ and $j_2$ is in $\wo^1(\ps)_\n$, since
otherwise the intersection point would be in the region ruled out in
Figure~\ref{fig-mucond}.

A point $(h/3,p)$ outside the union of the sectors stays outside the
sectors, since the application of $\sh{\pm3}{-1}$ sends it to a point
$(h/3,p) + (\pm 1,-1)$. A non-zero vector in $\Ffu^\ps_{\Nfu;h,p,p}$
stays non-zero under this process. So we would end in a point on the
$(h/3)$-axis that is not of the form $\frac13 h=\frac23 j$ with
$j\in \wo^1(\ps)$, which yields a contradiction to the results
concerning one-dimensional $K$-types. This ends the proof of~i).

For ii) we consider a non-zero vector
$v\in \Ffu^\ps_{\Nfu;2j\pm 3p,p,p}$ with $\sh{\pm 3}{-1}v\neq 0$ for
$j\in \wo^1(\ps)$ (and $j\in \wo^1(\ps)_\n$ in the non-abelian case).
Then $(h/3\pm 1,p-1)$ is outside the sector $\sect(j)$. Under simple
parametrization, that is a single sector, and we get a contradiction
with~i).
\end{proof}

\subsection{Kernels of upward shift operators}We turn to the possibility
that upward shift operators in Fourier term modules may have a non-zero
kernel. \il{uwso1}{upward shift operator}
\begin{prop}
\label{prop-kuso}\hspace{1cm}
\begin{enumerate}
\item[i)]If $\ps\in \WOG$, then the upward shift operators in
$\Ffu_\Nfu^{\ps}$ are injective.
\item[ii)] Let $\ps\in \WOI$. The operators
\[ \sh{3}1 : \Ffu_{\Nfu;h,p,p}^{\ps} \rightarrow \Ffu^{\ps}_{\Nfu;h+
3,p+1,p+1}\,,
\quad \sh{-3}1 : \Ffu_{\Nfu;h,p,p}^{\ps} \rightarrow \Ffu^{\ps}_{\Nfu;h-
3,p+1,p+1}\]
are injective if $\Nfu=\Nfu_\bt$, $\bt\neq 0$. The operator $\sh{ 3}1$,
respectively $\sh{-3}1$, may have a non-zero kernel for $\Nfu=\Nfu_0$
or $\Nfu_\n$ if there are $j,j'\in \wo^1(\ps)$, $j\neq j'$, such that
\begin{enumerate}
\item[a)] $\tau^h_p\in \sect(j)$,
\item[b)] $h+ 3p+ 6 = 2j'$, respectively $h- 3p- 6 = 2j'$\,,
\item[c)] if $\Nfu=\Nfu_{\ell,c,d}$, then $\ell>0$ (respectively
$\ell<0$) and $j'\not \in \wo^1(\ps)_\n$.
\end{enumerate}
\end{enumerate}
\end{prop}

\begin{proof}Proposition~\ref{prop-uso-ga} implies the statement in the
generic non-abelian case $\Nfu=\Nfu_\bt$, $\bt\neq 0$.

By Proposition~\ref{prop-sectors} we need consider only those $K$-types
$\tau^h_p$ for which there exists $(j,\nu) \in \wo(\ps)$ such that
$|h-2j|\leq 3p.$ We will show that the presence of a non-zero element
$v\in\Ffu^\ps_{\Nfu;h,p,p}$ with $\sh{\pm3}{-1} v\neq 0$ implies that
\be\label{relpm6} (h\pm 6\pm3p-3\nu +j)\,(h\pm 6 \pm 3p+3\nu +j
)\=0\,.\ee

This relation implies that $h\pm 6 \pm 3p$ is equal to $2j'$ for some
element $(j',\nu')$ of the Weyl group orbit of $(j,\nu)$, illustrated
in Figure~\ref{fig-Wo}. In fact, $j'= \frac{3\z\nu+j}2$,
$\z\in \{1,-1\}$. So $h+3p \equiv 2j'\bmod 6$. On the other hand, the
fact that $\tau^h_p$ is in $\sect(j)$ means that $h=2j+3(a-b)$ with
$a,b\in \ZZ_{\geq 0}$ and $a+b=p$. So $2j\equiv h+3p\bmod 6$ as well.
Hence $j\equiv j'\bmod 2$. Under generic parametrization this implies
$j=j'$
(with the use of Lemma~\ref{lem-gmp}). However $|h-2j|\leq 3p$ is not
compatible with $h\pm 3\pm 3p = 2j$. This rules out generic
parametrization, and we are left with $\ps\in \WOI$.
\rmrk{$N$-trivial case}The eigenfunction equations in
Lemma~\ref{lem-efeq-t} show that components of elements of $\Ffu_0^\ps$
satisfy an uncoupled system of differential equations. So we can
consider elements of the simple form
\[ F\bigl( n \am(t) k) = f_r(t)\, \Kph h{p}{r}{p}(k)\]
with $p \equiv r \bmod 2,$ $|r|\leq p,$ and $h=2j+3r.$

We carry out some computations in \cite[\S14a]{Math}. From
$\sh{\pm 3}{-1}F=0$ we see that $f_r$ is a multiple of
$t^{\mp(h+j)/3-p}$. Insertion of this function into the eigenfunction
equations with $(j,\nu)\in \wo(\ps)$ gives relation~\eqref{relpm6}. The
discussion following \eqref{relpm6} concludes the discussion of the
$N$-trivial case.
\rmrk{Non-abelian case}Now we consider
\[ F\bigl( n \am(t)k\bigr) = \sum_r \th_{m(h,r)}(n)\, f_r(t)\, \Kph
h{p}{r}{p}(k)\,,\]
with $r=-p, 2-p,\ldots,p$, with the restriction that only terms with
$m(h,r)\geq 0$ can contribute. We have $m(h,r\pm 2) = m(h,r)\pm \e$
by~\eqref{mudef}, and also $m\bigl(h, \e p\bigr)\geq 0$
(otherwise all components of $F$ vanish). We use the notation
$\e=\sign(\ell)$. We carry out several computations, in
\cite[\S14b]{Math}.

The kernel relations for the upward shift operator $\sh31$ give the
following relations between the components:
\bad f_r& \= \frac1{4(2+p-r) t}\Bigl( i(p+r) \bigl( (2+h+2p-r+4\pi\ell
t^2\bigr)f_{r-2}+ 2 t f'_{r-2} \bigr)\Bigr)
\\
&\qquad\hbox{} \cdot\begin{cases} \frac1{\sqrt{2\pi |\ell|m(r)}}&\text{
if }\e=1\text{ and } m(r)\geq 1\,,\\
\frac{-1}{\sqrt{2\pi|\ell|m(r+1)}}&\text{ if }\e=-1\text{ and } m(r)\geq
0\,,
\end{cases}
\ead
valid for $r \geq 2+\max(r_0,-p)$ if $\e=1$, and for $r\leq \min(r_0,p)$
if $\e=-1$. If $\e=1$ and $r_0< -p$, or if $\e=-1$, the lowest
component $f_{-p}$ determines all other components. A computation shows
that $\sh 3 1 \bigl(\th_m f_{-p} (t) \, \Kph hp{-p}p\bigr)=0$ implies
$f_{-p}=0$, unless $\e=1$ and $m=0$. This implies that the kernel of
$\sh 3 1$ on $\Ffu^\ps_{\n;h,p,p}$ is zero, unless $\e=1$ and
$-p\leq r_0 \leq p$. In this remaining case, the important component is
$f_{r_0}$. If $r_0<p$ we use a kernel relation to express $f_{r_0+2}$
in $f_{r_0}$ and its derivative, and insert this in the eigenfunction
equations for $r=r_0$. A Mathematica computation shows that this gives
for $f_{r_0}\neq 0$ the relation
\be \label{prd3} (h+6-3p-2j) (h+6-3p-3\nu+j)(h+6-3p+3\nu+j)\=0 \,.\ee
The factor $h+6-3p-j$ cannot be zero, since $\tau^h_p \in \sect(j)$. The
other two factors give~\eqref{relpm6}, and we have $h+6+3p=2j'$ for
$j'\in \wo^1(\ps)$, $j'\neq j$. Since $r_0(h)=\frac{h-d}3+1$, by
\eqref{mhrdef}, the requirement $r_0\geq -p$ implies that
$h \geq -3p+d-3$. Hence $2j'=h+6+3p\geq d+3$. The $K$-type
$\tau^{2j'}_0$ satisfies $2j'-3\cdot 0 \geq d+3$, and does not satisfy
the condition in~\eqref{mucond}. So $j'\not\in \wo^1(\ps)_\n$.

If $r_0=p$ the kernel relation leads to a differential equation with
explicit solutions spanned by
$f_{r_0}(t) = t^{-(h+p)/2}\, e^{-2\pi \ell t^2}$. Inserting this
solution in the eigenfunction equations shows that we have an element
of $\Ffu_\n^\ps$ with $\ps$ represented by $j=\frac{h-3p}2$ and
$\nu=\frac{h+p}2+2$. We note that $\tau^h_p \in \sect(j)$ for this
choice of~$(j,\nu)$. It turns out that $(j',\nu') = \Ws1(j,\nu)$
satisfies $j'=h+6+3p$, hence \eqref{prd3} holds in this case as well.
The same argument as for $-p\leq r_0 <p$ shows that
$j'\not\in \wo^1(\ps)_\n$.
\medskip

For the shift operator $\sh{-3}1$ we proceed similarly. The kernel
relations are
\bad f_r &\= \frac1{4(2+p+r) \, t} \Bigl( i(p-r)\bigl(
(2-h+2p+4-4\pi\ell t^2)f_{r+2}+2 t f_{r+2}'\bigr) \Bigr)\\
&\qquad\hbox{} \cdot\begin{cases}
\frac{-1}{\sqrt{2\pi|\ell|(1+m(r))}}&\text{ if }\e=1\text{ and }
m(r)\geq 0\,\\
\frac1{\sqrt{2\pi|\ell| m(r) }}&\text{ if }\e=-1\text{ and }m(r)\geq
1\,.
\end{cases}
\ead
Now the highest non-zero component determines the other components.

Like in the case of $\sh31$ there cannot be a non-trivial kernel if
$\e=1$, or if $\e=-1$ and $r_0> p$. We have to consider the case
$\e=-1$ and $-p\leq r_0 \leq p$.

  In the case of a kernel element with one component we should have
$f_{-p}$ equal to a multiple of
$t\mapsto t^{(h-p)/2} e^{\pi \ell t^2}$. This satisfies the
eigenfunction equations with $j=\frac12(h+3p)$ and
$\nu=1+\frac12(p-h)$, for which $\tau^h_p \in \sect(j)$. It turns out
that $h-6-3p=2j'$ with $(j',\nu')=\Ws2(j,\nu)$.

In the other case we substitute into the eigenfunction equations for
$r=r_0$ the expression for $f_{r_0-2}$ that follows from the kernel
relation. That leads to the relation
\[ (h-6-3p-2j)(h-6-3p+j-3\nu) (h-6-3p+f+3\nu)\=0\,,\]
and then to~\eqref{relpm6}. So in this case as well $h-6-3p=2j'$ for
$j'\in \wo^1(\ps)$. Further we have
\[ 2j' = h-3p-6 \= 3r_0 +d+3 - 3p - 6 \leq 3p+d-3p-3 \=d-3\,,\]
in contradiction to the requirement that $2j' \geq d+3$;
see~\eqref{mucond}.
\end{proof}

\begin{prop}\label{prop-dim}Let $\ps\in \WOG$. The subspaces
$\Ffu_{\Nfu;h,p,p}^{\ps}$ have dimension 2 if $\tau^h_p\in \sect(j)$
for some $j\in \wo^1(\ps)$, and dimension 0 otherwise.
\end{prop}
\begin{proof}Under simple parametrization we have
$\dim\Ffu_{\Nfu;2j,0,0}^\ps=2,$ by the summary at the end of
Subsection~\ref{sect-1dKt}. Repeated application of the upward shift
operators brings us from the $1$-dimensional $K$-type to all $K$-types
corresponding to points in $\sect(j)$ for the sole $j\in \wo^1(\ps)$.
The injectivity in i) in Proposition~\ref{prop-kuso} gives multiplicity
at least $2$ for all $K$-types that occur. From any $K$-type we can go
down to the $K$-type $\tau^{2j}_0$ by application of downward shift
operators on highest weight spaces on which they are injective. So all
multiplicities are equal to~$2.$

Under generic multiple parametrization, the $K$-types correspond to
points in three disjoint sectors, to each of which we can apply the
same reasoning.\end{proof}


\def\flnm{rFtm-II-sFtm}


\section{Special Fourier term modules}\label{sect-spFtm}
\markright{10. SPECIAL FOURIER TERM MODULES}We turn to the structure of
the Fourier term modules $\Ffu_\Nfu^\ps$, under the assumption of
generic parametrization, in which these modules are the direct sum of a
finite number of $(\glie,K)$-modules. The study of Fourier term modules
under integral parametrization will be carried out in the next chapter.

In this section we discuss first the principal series modules
$H^{\xi,\nu}_K$, which are submodules of $\Ffu^\ps_0$, for the
character $\ps=\psi[j_\xi,\nu]$ of $ZU(\glie)$. In the other modules
$\Ffu_\Nfu$ we distinguish submodules by their behavior on~$A$. We
define in this way modules $\Wfu_\Nfu^{\xi,\nu}$ with exponential decay
as $t\uparrow\infty$, and modules $\Mfu_\Nfu^{\xi,\nu}$ with nice
behavior as $t\downarrow 0$. Under generic parametrization the modules
$H^{\xi,\nu}_K$, $\Wfu_\Nfu^{\xi,\nu}$ and $\Mfu^{\xi,\nu}_\Nfu$ are
isomorphic. We discuss a few intrinsically defined intertwining
operators.

\subsection{Principal series and logarithmic modules}\label{sect-prs}For
any choice of $(j,\nu)\in \ZZ\times \CC$ the Fourier term module
$\Ffu_0^{\psi[j,\nu]}$ contains the following
module\ir{HK-def}{H^{\xi,\nu}_K} in the \il{ps}{principal
series}principal series\il{kph}{$\kph h{p}{r}{q}$}
\badl{HK-def} H_K^{\xi,\nu} &= \bigoplus_{h,p,q} \CC\,\kph
h{p}{r}{q}(\nu)\,,\\
\kph h{p}{r}{q}\bigl( \nu;n \am(t) k \bigr)&= t^{2+\nu}\, \Kph
h{p}{r}{q}(k)\,, \eadl
with $\xi=\xi_j$ the character of $M$ corresponding to $j\in \ZZ$. The
sum is over integers satisfying $h\equiv p \equiv q\bmod 2,$
$|q|\leq p$, $|r|\leq p$, and $h=2j+3r$. The element $\kph {2j}000$ is
a solution of the differential equation in~\S\ref{sect-Ntr1d}.

The module $H^{\xi,\nu}_K$ depends on the choice of $(j,\nu)$ in
$\wo(\psi[j,\nu])$. The Fourier term module $\Ffu_0^{\psi[j,\nu]}$
depends only on the Weyl group orbit $\wo(\psi[j,\nu])$, and contains
(in general)
several principal series modules.\smallskip

Specialization of the results in Table~\ref{tab-shab},
p~\pageref{tab-shab}, gives the shift operators.
\badl{shps} \sh{\pm 3}1 \kph h{p}{r}{p}
(\nu)&= \frac{2+p\pm r}{8(1+p)}\, (4\pm h
+2\nu +2p\mp r) \, \kph {h\pm 3}{p+1}{r\pm 1}{p+1}(\nu)\,,\\
\sh{\pm3}{-1} \kph h{p}{r}{p}(\nu)
&= \frac p{4(p+1)} (2\nu \pm h
-2p \mp r) \, \kph{h\pm 3}{p-1}{r\pm 1}{p-1}(\nu)\,. \eadl

Under the assumption of generic parametrization all upward operators are
injective. See Proposition~\ref{prop-kuso}; or alternatively, check it
with \eqref{shps}. So $H_k^{\xi,\nu}$ is a special module
(Definition~\ref{def-scm}). Its type is
$\bigl[\ld_2(j,\nu);2j,0; \infty,\infty\bigr].$
Proposition~\ref{prop-kdso} shows that all downward shift operators
that stay in the sector $\sect(j)$ are also injective. With
Lemma~\ref{lem-alnzir} this implies the following result.
\begin{prop}\label{prop-psscm}Under the condition of generic
parametrization the representations in the principal series are
irreducible special modules.
\end{prop}

We note that the basis vectors $\kph h{p}{r}{q}$ form holomorphic
families depending on $\nu \in \CC.$

\begin{prop}\label{prop-FFu0} Let $\ps \in \WOG$ such that $\nu\neq 0$
for all $(j,\nu)\in \wo(\ps)$. Then
\be \Ffu_0^{\ps} \=\bigoplus_{(j,\nu)\in \wo(\ps)} H_K^{\xi_j,\nu}
\,.\ee
\end{prop}
\begin{proof}This follows directly from the eigenfunction equations in
Lemma~\ref{lem-efeq-t} and from the dimension statements in
Proposition~\ref{prop-dim}. If $\nu=0$ we get solutions $t^2$ and
$t^2\log t$, and then the statement is not right.\end{proof}

\begin{prop}If $\psi[j,\nu]\in \WOG$, then $H^{\xi,\nu}_K$ and
$H^{\xi,-\nu}_K$ are isomorphic.
\end{prop}
\begin{proof}Since $\ld_2(j,-\nu) = \ld_2(j,\nu)$ this follows from
Proposition~\ref{prop-scm} and the type
$\bigl[\ld_2(j,\nu);2j,0; \infty,\infty\bigr]$ of~$H^{\xi,\nu}_K$.
\end{proof}

The isomorphism can be chosen such that\ir{ii0}{\ii_0}
\badl{ii0} \ii_0\= \ii_0(j,\nu): \kph h{p}{r}{q}(\nu)
&\mapsto c(h,p,r,\nu) \, \kph h{p}{r}{q}(-\nu)\,,\\
c(h,p,r,\nu)&= \frac{\Gm( 1 + \frac{p-\nu}2+\frac{h-r}4)
\Gm(1+\frac{p-\nu}2+\frac{r-h}4)} {\Gm( 1 + \frac{p+\nu}2+\frac{h-r}4)
\Gm(1+\frac{p+\nu}2+\frac{r-h}4)} \,. \eadl
This is well defined for generic parametrization. Some checks are in
\cite[\S15b]{Math}. The map $\ii_0(j,\nu)$ is meromorphic in $\nu$; the
singularities occur only for some $(j,\nu)$ corresponding to integral
parametrization.

\rmrk{Remark} Under generic multiple parametrization there are more
principal series modules corresponding to the same Weyl group orbit. If
$j\neq j'$ in the same Weyl group orbit then $H^{\xi,\nu}_K$ and
$H^{\xi',\nu'}_K$ are not isomorphic, since the parameter $h_0$ is
given by $2j$ for principal series representations.

\rmrk{Logarithmic modules}In the case of $\nu=0$ the differential
relations in \eqref{eft} admit solutions with components of the form
$c_1 \,t^2\log t+c_2\, t^2.$\il{lm}{logarithmic module}

To describe these solutions we use the intertwining operator
\be \frac1{2\nu}\bigl(1 - \ii_0(j,\nu)\bigr) : H^{\xi,\nu}_K \rightarrow
\Ffu_0^{\psi[j,\nu]}\,.\ee
Since $c(h,p,r,0)=1$ for $(j,0)$ corresponding to generic
parametrization, this operator is well defined for $\nu=0$. In
particular, the value of $c(h,p,q,0)$ equals~$1.$ Hence\ir{ldphg}{\ldph
h p r q(\nu)}
\be \label{ldphg} \nu \mapsto \frac1{2\nu}\Bigl( 1-c(h,p,r,\nu) \Bigr)
\kph h{p}{r}{p}(\nu) \;=:\; \ldph h p q r(\nu) \ee
extends holomorphically to $\nu=0$ with a term with $t^2 \log t$ in its
components at $\nu=0.$ In this way we obtain an injective intertwining
operator $H^{\xi,\nu} \rightarrow \Ffu^{\psi[j,\nu]}_0$ for $\nu $ in a
neighborhood of~$0$ in~$\CC$. We call the
image\il{Lfu0}{$\Lfu_0^{\xi,\nu}$}$\Lfu_0^{\xi,\nu}\subset \Ffu_0^{\psi[j,\nu]}.$
\begin{prop}\label{prop-iLog}Let $j\in \{0\} \cup\bigl(1+2\ZZ\bigr).$
\begin{enumerate}
\item[i)] $\Lfu_0^{j,\nu}\cong H^{\xi,\nu}_K$ for $\nu\neq0$ in a
neighborhood of $0$.
\item[ii)] $\Ffu_0^{\ps[j,0]} = H^{\xi,0}_K \oplus \Lfu_0^{\xi,0}.$
\end{enumerate}
\end{prop}

\subsection{Submodules characterized by boundary behavior}\label{sect-smbh}
The Fourier term modules $\Ffu_\Nfu^{\ps}$ are generated by highest
weight functions in the $K$-types of the form given in \eqref{Fcomp}:
\[ F\bigl( n \am(t) k \bigr) = \sum_{r:(-p,p)} u_r(n)\, f_r(t) \, \Kph
h{p}{r}{p}(k)\,. \]
Since $\Ffu_\Nfu^\ps$ consists of $K$-finite vectors, each of its
elements is determined by finitely many component functions
$f _r = f_r(h,p).$
\begin{defn}\label{def-grbb}(Boundary behavior)
\begin{enumerate}
\item[i)] A function $f$ on $(0,\infty)$ has \il{regbeh}{$\nu$-regular
behavior at $0$}\emph{$\nu$-regular behavior at $0$} if
\be f(t) = t^{2+\nu} h(t)\,,\ee
where $h$ is the restriction to $(0,\infty)$ of an entire function.
\item[ii)] A function $f$ on $(0,\infty)$ has \il{ed}{exponential decay
at $\infty$}\emph{exponential decay at $\infty$} if there exists $a>0$
such that
\be f(t) \;\ll\; e^{-at}\quad\text{ as }t\rightarrow\infty\,.\ee
\item[iii)] An element $F\in \Ffu_{\Nfu;h,p,p}$, has the property in i),
respectively ii), if all its component functions have this property.
\end{enumerate}
\end{defn}

\rmrk{Examples}\begin{enumerate}
\item[i)] All $\kph h{p}{r}{q}(\nu)$ in the principal series module
$H^{j,\nu}_K$ have $\nu$-regular behavior at~$0.$

\item[ii)] The function
$\mu^{0,0}_\bt(j,\nu) \in \Ffu^{\psi[j,\nu]}_\bt$ with $\bt\neq0$
\ir{mu00a}{\mu^{a,b}_\bt}
\be\label{mu00a}\mu^{0,0}_\bt \bigl(j,\nu;n\am(t)k\bigr)
= \chi_\bt(n)
\, t^2\, I_\nu(2\pi|\bt|t)\, \Kph{2j}000(k)\ee
has $\nu$-regular behavior at $0.$ This follows from the
expansion~\eqref{Inu}. This function is an element of
$\Ffu_{\bt;2j,0,0}^{\psi[j,\nu]};$ see \S\ref{sect-Nab1d}.

\item[iii)] \label{itmu00}Similarly we find with $m_0=\frac{\e}6
(d-2j)-\frac12\in \ZZ_{\geq 0}$ and $\nu\not\in \ZZ_{\leq -1}$ the
function
\be\label{mu00nab}
\mu^{0,0}_{\ell,c,d} \bigl(j,\nu;n\am(t)k\bigr) = \Th_{\ell,c}\bigl(
h_{\ell,m_0};n\bigr)\, t \, M_{\k,\nu/2}(2\pi|\ell|t^2)\,
\Kph{2j}000(k)\,, \ee
where $\k= -m_0
-\frac12 \bigl( \e j+1\bigr).$ See the expansion~\eqref{Mkps} for the
$\nu$-regular behavior at~$0.$

\item[iv)] The function $\om^{0,0}_\bt(j,\nu)$ in
$\Ffu^{\psi[j,\nu]}_\bt$ with $\bt\neq 0,$ given
by\ir{om00ab}{\om^{a,b}_\bt}
\be\label{om00ab}
\om^{0,0}_\bt\bigl(j,\nu;n\am(t)k\bigr)
= \chi_\bt(n) \, t^2\, K_\nu(2\pi|\bt|t)\, \Kph{2j }000(k)\ee
has exponential decay at~$\infty.$ See~\eqref{Kae}.
\item[v)] The function $\om^{0,0}_\n(j,\nu)$ in
$\Ffu_{\ell,c,d}^{\psi[j,\nu]},$ with $m_0$ and $\k$ like in~iii),
given by
\be\label{om00nab}
\om^{0,0}_{\ell,c,d} \bigl(j,\nu;n\am(t)k\bigr)
= \Th_{\ell,c}\bigl( h_{\ell,m_0};n\bigr)\, t \,
W_{\k,\nu/2}(2\pi|\ell|t^2)\, \Kph{2j}000(k)
\ee
has exponential decay at~$\infty.$ See~\eqref{Wae}.
\end{enumerate}

\begin{prop}\label{prop-inv}The properties of $\nu$-regular behavior at
$0$ and exponential decay at $\infty$ are preserved under the action of
$\glie$ and $K.$
\end{prop}
\begin{proof}Since the actions of $K$ and of $\klie$ do not change the
component functions, it suffices to show that the actions of
$\Z_{31}, \Z_{23}, \Z_{32}, \Z_{13} \in \glie_c$ preserve these
properties. On each $K$-type these elements can be related to shift
operators. Tables \ref{tab-shab} and \ref{tab-krab}
pp~\pageref{tab-shab}, \pageref{tab-krab} describe the action of the
shift operators in the module $\Ffu_\Nfu.$ Inspection shows that the
operations on the components are linear combinations of $b\mapsto b,$
$b\mapsto t\, b,$ $b\mapsto t^2\, b,$ and $b\mapsto t\, b'.$ If
$b(t) = t^{2+\nu} h(t)$ with $h$ extending as a holomorphic function on
$\CC,$ these operations change $h$ by $h\mapsto t^c h,$ $c=1,2,3,$ or
by $h\mapsto t\, h'+(2+\nu) h.$ So the shift operators preserve the
property of $\nu$-regular behavior at~$0.$

For the property of exponential decay at $\infty$ we use the convolution
representation theorem of Harish Chandra; Theorem 1 on p.~18 of
\cite{HCh66}. One writes $F\in \Ffu_\Nfu^{\ps}$ in the form
\be F\bigl(nak\bigr) = \int_{ G} F(nak g^{-1})
\al(g)\, dg\ee
with $\al \in C_c^\infty(G).$ So $kg^{-1}$ in the integral runs over a
compact set, and we can write $kg^{-1}=n_1a_1k_1$ where $n_1,$ $a_1,$
and $k_1$ run over compact sets in $N,$ $A,$ and $K,$ respectively.
Then $nakg^{-1} =  (n\, an_1a^{-1}) \, aa_1 \, k_1,$ with
$\am(t) a_1= \am(t_1)$ where $t/b \leq t_1 \leq tb$ for some $b>1$
depending on $\al.$ Right differentiation by an element of $\glie$ can
be carried out on
\be \XX F\bigl( n \am(t) k \bigr ) =\int_G F\bigl(
n\am(t)kg_1^{-1}\bigr) \, \XX\al(g_1)\, dg_1\,.\ee
If $g_1$ varies through a compact set then the Iwasawa components $n_1$,
$a(t_1)$ and $k_1$ in $n\am(t) kg_1 = \allowbreak n_1\am(t_1) k_1$ vary
through compact sets. In particular there is $b>1$ such that
$t/b \leq t_1 \leq tb$. This preserves the estimate of exponential
decay.
\end{proof}
\rmrk{Polynomial growth}The proof shows that differentiation also
preserves the property of polynomial growth:\il{polgr0}{polynomial
growth} $f\bigl( n \am(t) k\bigr) = \oh(t^a)$ as $t\uparrow\infty$ for
some $a>0$.
\medskip

Proposition~\ref{prop-inv} suggests the following definitions:
\begin{defn}\label{defW[]}We put\ir{Wfudef}{\Wfu_\Nfu^{\ps}}
\be\label{Wfudef} \Wfu_\Nfu^{\ps} \= \Bigl\{ F \in \Ffu_\Nfu^{\ps}\;:\;
\text{ $F$ has exponential decay at $\infty$}\Bigr\}\,.\ee
\end{defn}

\rmrk{Notation} It is convenient to use the following subsets of Weyl
group orbits. For $\ps\in \WO$ we put: \ir{WO+}{\wo(\ps)^+\,, \;
\wo(\ps)_\n^+}
\badl{WO+} \wo(\ps)^+ &\= \bigl\{ (j,\nu)\in \wo(\ps) \;:\; \re\nu\geq
0\bigr\}\,,\\
\wo(\ps)_\n^+ &\= \wo(\ps)_\n \cap \wo(\ps)^+\,. \eadl

The restriction of the projection map $\wo(\ps)\rightarrow \wo^1(\ps)$
to $\wo(\ps)^+ \rightarrow  \wo^1(\ps)$ is a bijection.

\begin{defn}\label{defM[]}We define
\il{Mfudef}{$\Mfu_\Nfu^{\ps}$}$\Mfu_\Nfu^{\ps}$ as the
$(\glie,K)$-submodule of linear combinations of functions
$F \in \Ffu_\Nfu^{\ps}$ that have $\nu$-regular behavior at $0$ for
some $(j,\nu) \in \wo(\ps)^+$.
\end{defn}

\rmrk{Remark}Let $\ps\in \WO$. The principal series modules
$H^{\xi,\nu}_K$ with $(j_\xi,\nu) \in \wo(\ps)^+$ are submodules of
$\Mfu^\ps_0$. For other $\Nfu$, we will define, in \eqref{WMwnudef},
$\Mfu_\Nfu^{\xi,\nu}$ and $\Wfu_\Nfu^{\xi,\nu}$ as submodules of
$\Mfu_\Nfu^{\psi[j_\xi,\nu]}$, respectively
$\Wfu_\Nfu^{\psi[j_\xi,\nu]}$, with similar properties.

\begin{lem}\label{lem-WNdisj}Let $\Nfu$ be $\Nfu_\bt$ with $\bt\neq 0,$
or $\Nfu_\n$. If $\ps\in \WOG$ then
$\Mfu_\Nfu^\ps \cap \Wfu_\Nfu^\psi=\{0\}.$
\end{lem}
\begin{proof}The intersection $\Mfu_\Nfu^\ps \cap \Wfu_\Nfu^\ps$ is a
$(\glie,K)$-submodule of $\Ffu_\Nfu^\ps$. Consider a non-zero element
of $\Mfu_{\Nfu;h,p,p}^\ps \cap \Wfu_{\Nfu;h,p,p}^\ps$. Then
$\tau^h_p\in \sect(j)$ for some $j\in \wo^1(\ps)$, by
Proposition~\ref{prop-dim}. Proposition~\ref{prop-kdso} implies that if
$p\geq 1$ at least one of the downward shift operators is injective.
Thus, we get a non-zero element in the intersection of $K$-type
$\tau^{h\pm 3}_{p-1}$. Proceeding in this way we arrive at a non-zero
element in $\Mfu_{\Nfu;2j,0,0}^\ps \cap \Wfu_{\Nfu;2j,0,0}^\ps$. We
know an explicit basis $\om^{0,0}_\Nfu(j,\nu)$, $\mu^{0,0}_\Nfu(j,\nu)$
of the space $\Ffu^\ps_{\Nfu;2j,0,0}$; see
\eqref{mu00a}--\eqref{om00nab}. Of these, only $\om^{0,0}_\Nfu(j,\nu)$
has exponential decay at $\infty$, and it has no $\nu$-regular behavior
at $0$. A non-zero element with both properties does not
exist.\end{proof}

\rmrk{Basis families} Let $\Nfu=\Nfu_\bt$ with $\bt\neq0,$ or
$\Nfu=\Nfu_\n.$ We put for $a,b\in \ZZ_{\geq 0},$
$(a,b)\neq(0,0),$\il{muab}{$\mu^{a,b}_\Nfu$}
\il{omab}{$\om^{a,b}_\Nfu$}
\badl{xab} \mu^{a,b}_\Nfu(j,\nu) &= \bigl(\sh 31\bigr)^a
\bigl(\sh{-3}1\bigr)^b \mu ^{0,0}_\Nfu(j,\nu)\;\in\;
\Mfu_{\Nfu;2j+3(a-b),a+b,a+b}^{\psi[j,\nu]} \,,\\
\om^{a,b}_\Nfu(j,\nu)
&= \bigl(\sh 31\bigr)^a \bigl(\sh{-3}1\bigr)^b \om ^{0,0}_\Nfu(j,\nu)
\;\in\; \Wfu_{\Nfu;2j+3(a-b),a+b,a+b}^{\psi[j,\nu]}\,.
\eadl

\begin{lem}\label{lem-bfs}Let $\Nfu=\Nfu_\bt$, $\bt\neq 0$, or
$\Nfu=\Nfu_\n$. Take $(j,\nu)\in \ZZ\times\CC$, put $\ps=\ps[j,\nu]$,
and assume that $\frac\e6 d - \frac12 - \frac{\e}3j \in \ZZ_{\geq 0}$
if $\Nfu=\Nfu_\n$.
\begin{enumerate}
\item[i)]The functions in \eqref{xab} form meromorphic families
in~$\nu$. The families $\nu\mapsto \om^{a,b}_\Nfu(j,\nu)$ are even and
holomorphic in~$\CC$.

In the abelian case, the families $\nu \mapsto \mu^{a,b}_\bt(j,\nu)$ are
holomorphic on $\CC$ and satisfy
$\mu^{a,b}_\bt(j,-n)=\mu^{a,b}_\bt(j,n)$ for $n\in \ZZ$. In the
non-abelian case, the family $\nu \mapsto \mu^{a,b}_\n(j,\nu)$ may have
singularities at points of $\ZZ_{\leq -1}$. If a singularity occurs at
$-\nu_0\in \ZZ_{\leq -1}$, then it has first order, with a multiple of
$\mu^{a,b}_\n(j,\nu_0)$ as its residue.

\item[ii)] The components of the functions in~\eqref{xab} are linear
combinations of special functions of the following type
\begin{align*} \text{for }\om^{a,b}_\bt(j,\nu): \quad& t\mapsto t^c\,
K_{\nu+k}(2\pi|\bt|t)
&&c\in \ZZ_{\geq 2}\,,\; k\in \ZZ_{\geq 0}\,,
\displaybreak[0]\\
\text{for }\mu^{a,b}_\bt(j,\nu):\quad& t\mapsto t^c\,
I_{\nu+k}(2\pi|\bt|t)
&&c\in \ZZ_{\geq 2}\,,\; k\in \ZZ_{\geq 0}\,,
\displaybreak[0]\\
\text{for }\om^{a,b}_\n(j,\nu):\quad& t\mapsto t^c\,
W_{\k+k,\nu/2}(2\pi|\ell|t^2)&& c\in \ZZ_{\geq1},\; k\in \ZZ_{\geq
0}\,,
\displaybreak[0]\\
\text{for }\mu^{a,b}_\n(j,\nu):\quad& t\mapsto t^c\,
M_{\k+k,\nu/2}(2\pi|\ell|t^2)&& c\in \ZZ_{\geq1},\; k\in \ZZ_{\geq
0}\,,
\end{align*}
with $\k= -m_0(j)
-\frac12 \bigl( \e j+1\bigr)$.

\item[iii)] Let $a,b \in \ZZ_{\geq 0}$. Then
$\om^{a,b}_\Nfu(j,\nu)\in \Wfu^{\ps}_\Nfu$. If
$(j,\nu) \in \wo(\ps)^+$, then
$\mu^{a,b}_\Nfu(j,\nu) \in \Mfu^{\ps}_\Nfu$.

\item[iv)] If $\ps \in \WOG$, then the
spaces\il{WFUximudef}{$\Wfu_\Nfu^{\xi,\nu}$}
\il{MFUximudef}{$\Mfu_\Nfu^{\xi,\nu}$}
\badl{WMwnudef} \Wfu_\Nfu^{\xi,\nu} & \;:=\; \bigoplus_{a,b\geq 0} U
(\klie) \; \om^{a,b}_\Nfu(j,\nu)\\
\Mfu_\Nfu^{\xi,\nu} & \;:=\; \bigoplus_{a,b\geq 0} U(\klie) \;
\mu^{a,b}_\Nfu(j,\nu)&\quad&\text{ provided }\nu\not\in \ZZ_{\leq
-1}\,.
\eadl
are $(\glie,K)$-submodules of $\Ffu^\ps_\Nfu$. In particular,
$\Wfu_\Nfu^{\xi,\nu}\subset \Wfu^\ps_\Nfu$, and if
$(j_\xi,\nu)\in \wo(\ps)^+$, then
$\Mfu^{\xi,\nu}_\Nfu \subset \Mfu^{\psi}_\Nfu$.
\end{enumerate}
\end{lem}

\rmrk{Remark}We postpone the definition of the modules
$\Wfu^{\xi,\nu}_\Nfu$ and $\Mfu^{\xi,\nu}_\Nfu$ under integral
parametrization till Lemma~\ref{lem-strab} and
Definition~\ref{def-VWMa}.

\begin{proof} The statements in i) are valid for $a=b=0$, as can be
checked in Appendix~\ref{app-spf}; see in particular \eqref{Knu},
\eqref{Wkps}, \eqref{Ievint} and~\eqref{Mres}. The properties are
preserved under application of the shift operators. In the proof of
Proposition~\ref{prop-inv} we gave the action of the shift operators on
the component functions. We apply this repeatedly to the special
function in the cases when $(a,b)=(0,0).$ Then we apply the contiguous
relations in \eqref{cr1} and~\eqref{crW1} to see that we stay in the
linear space spanned by the functions indicated in~ii). This gives
also~iii).

Under generic parametrization the upward shift operators are injective.
So the elements $\mu^{a,b}_\Nfu(j,\nu)$ and $\om^{a,b}_\Nfu(j,\nu)$ are
non-zero and linearly independent. Proposition~\ref{prop-dim} (on the
dimensions)
implies that together they span $\Ffu_\Nfu^{\psi[j,\nu]}.$ So the
downward shift operators send both $\om ^{a,b}_\Nfu(j,\nu)$ and
$\mu^{a,b}_\Nfu(j,\nu)$ to a linear combination of
$\om^{a',b'}_\Nfu(j,\nu)$ and $\mu^{a',b'}_\Nfu(j,\nu)$. The downward
shift operators also preserve linear combinations as indicated in~ii).
So they preserve $\Wfu^{\xi,\nu}_\Nfu$ and $\Mfu^{\xi,\nu}_\Nfu$. In
the definition of $\Mfu^{\xi,\nu}_\Nfu$ we imposed the condition
$\nu\not\in \ZZ_{\leq -1}$, thus avoiding the complications that may be
caused by singularities.
\end{proof}

\begin{prop}\label{prop-dsok1}For all $p\in \ZZ_{\geq 0}$
\bad \sh3{-1} \om^{p,0}_\Nfu(j,\nu)&=0&\quad \sh{-3}{-1}
\om^{0,p}_\Nfu(j,\nu)&=0
\quad\text{ for $\nu \in \CC$}\,,\\
\sh3{-1} \mu^{p,0}_\Nfu(j,\nu)&=0&\quad \sh{-3}{-1}
\mu^{0,p}_\Nfu(j,\nu)&=0
\quad\text{ for $\nu \in \CC\setminus \ZZ_{\leq -1}$\,.}
\ead
\end{prop}
\begin{proof}Under simple parametrization this follows from ii) in
Proposition~\ref{prop-sectors}. The families $x^{p,0}_\Nfu(j,\cdot)$
and $x^{0,p}_\Nfu(j,\cdot)$ are holomorphic in their domain, and that
property is preserved under differentiation.
\end{proof}

\begin{proof}[Proofs of Theorems \ref{mnthm-ab-gp} and
\ref{mnthm-nab-gp}]\label{prfAB}
The role of $m_0(j)$ in Theorem~\ref{mnthm-nab-gp} is based on the
discussion in~\S\ref{sect-1dKt} of the functions in the $K$-type
$\tau^{2j}_0$, which have the form
$nak\mapsto \Th_{\ell,c}(h_{\ell,m};n) \, f_t) \Kph{2j}000(k)$, with a
normalized Hermite function $h_{\ell,m}$ with $m\in \ZZ_{\geq 0}$. The
eigenfunction equations show that $K$-type $m=m_0(j)$ in the case of a
one-dimensional $K$-type.

Lemma~\ref{lem-bfs} gives elements of $\Wfu^\ps_\Nfu$ and
$\Mfu^\ps_\Nfu$ in terms of basis functions, both in the generic
abelian case and in the non-abelian case. Since the modified Bessel
functions or Whittaker functions are linearly independent, this shows
that the $K$-types with $|h-2j|\leq 3p$ occur in both modules with
multiplicity at least one. Proposition~\ref{prop-dim} implies that the
multiplicities are exactly one.

Proposition~\ref{prop-kdso} tells that the downward shift operators
vanish on boundaries of the sectors $\sect(j)$. Hence there are no
$K$-types in $\Wfu^{\xi,\nu}_\Nfu$ or $\Mfu^{\xi,\nu}_\Nfu$ that do not
satisfy $|h-2j_\xi|\leq 3p$. Proposition~\ref{prop-kuso} gives the
injectivity of the upward shift operators. So $\Mfu^{\xi,\nu}_\Nfu$ and
$\Wfu^{\xi,\nu}_\Nfu$ are special modules as in
Definition~\ref{def-scm}, with parameter set
$\bigl[ \ld_2(j_\xi,\nu); 2j_\xi,0;\infty,\infty]$. Hence they are
isomorphic to $H^{\xi,\nu}_K$ by Propositions \ref{prop-scm}
and~\ref{prop-psscm}.
\end{proof}

\begin{remark}\label{rmk-WM-gp}Among these modules, $\Ffu^\ps_\Nfu$ and
$\Wfu^\ps_\Nfu$ are defined in an intrinsic way; $\Ffu^\ps_\Nfu$ as the
codomain of the Fourier term operator $\Four_\Nfu$, with the submodule
$\Wfu^\ps_\Nfu$ determined by the condition of exponential decay.

Under generic parametrization, and the additional condition
$\nu \not\in \ZZ_{\leq-1}$ for $\Mfu^{\xi,\nu}_\n$, we define
inside $\Wfu^\ps_\Nfu$ the special modules $\Wfu^{\xi,\nu}_\Nfu$
generated by $\om^{0,0}_\Nfu(j_\xi,\nu)$ with $\nu \in \wo(\ps)^+$. The
modules $\Mfu_\Nfu^{\xi,\nu}$ are intrinsically defined by the
condition of $\nu$-regular behavior at~$0$. The definition of
$\Mfu^\ps_\Nfu$ in~\eqref{defM[]} is much less intrinsic. The
restriction to $\re\nu\geq 0$ is motivated by i) in
Proposition~\ref{prop-dsok1}.
\end{remark}

\subsection{Intertwining operators}\label{sect-inttw}
Under generic parametrization, we have obtained various irreducible
simple modules that are isomorphic by Proposition~\ref{prop-scm},
namely
\badl{II-ex} &H^{\xi,\nu}_K \,,\quad H^{\xi,-\nu}_K\,, \quad
\Mfu_\bt^{\xi,\nu}\,,
\quad \Wfu_\bt^{\xi,\nu}\,,\quad \Mfu_\n^{\xi,\nu}\,,\quad
\Wfu_\n^{\xi,\nu}\,,
\eadl
where we take $\bt\neq 0,$ and $\n=(\ell,c,d)$ such that
$m_0= \frac\e6 \bigl( d-2j\bigr)-\frac12\in  \ZZ_{\geq 0}.$ The
isomorphisms are given by intertwining operators determined up to a
non-zero factor. We may fix them by prescribing them in the $K$-type
$\tau^{2j}_0,$ for instance by letting the basis vectors $x^{0,0}_\Nfu$
correspond to each other. 
In this way we can also build, under generic parametrization, an
injective morphism
\be \Wfu_\Nfu^{\xi,\nu} \rightarrow \Mfu_\Nfu^{\xi,\nu} \oplus
\Mfu_\Nfu^{\xi,-\nu}\ee
based on \eqref{Knu} and~\eqref{Wkps}.

We discuss three types of intertwining operators that are defined in a
more intrinsic way.

\rmrk{Left translations by elements normalizing the lattice}In
\S\ref{sect-normLd} we discussed that left translation by elements
normalizing the lattice $\Ld_\s$ gives intertwining operators of
$(\glie,K)$-modules between large Fourier term modules $\Ffu_\Nfu$.
Hence they preserve the Fourier term modules $\Ffu^\ps_\Nfu$. The
description in terms of basis elements in Table~\ref{tab-ltrFfu},
p~\pageref{tab-ltrFfu}, preserves the functions $f$ on~$A$. Hence they
preserve the families $\om^{0,0}_\Nfu$ and $\mu^{0,0}_\Nfu$, and by the
intertwining property also the derivatives $\om^{a,b}_{\Nfu}$ and
$\mu^{a,b}_\Nfu$. So they preserve the special Fourier term modules as
well.

\rmrk{Evaluation at zero}A basis element of $K$-type $\tau^h_p$ in
$\Mfu^{\xi,\nu}_\Nfu$ has the form
\[ F= \sum_r u_r(n) \, t^{2+\nu} \, h_r(t) \, \Kph h{p}{r}{q}\,.\]
For each $r$, the function $h_r$ is entire, and $u_r$ is a basis element
on~$N$, either a character or a theta function. The summation variable
satisfies $|r|\leq p$, $r\equiv p \bmod 2$, and some further condition
in the non-abelian case. \il{eval0}{evaluation at zero}\emph{Evaluation
at zero} is the operator
$E_0: \Mfu^{\xi,\nu}_\Nfu \rightarrow H^{\xi,\nu}_K$ induced
by\ir{ev0}{E_0}
\be\label{ev0}
E_0 F = \sum_r t^{2+\nu} \, h_r(0) \,\Kph h{p}{r}{q}\,.\ee
So we replace all $u_r$ by $1$, and $h_r(t)$ by its value at~$t=0$.

\begin{prop}Evaluation at zero is an intertwining operator of
$(\glie,K)$-mod\-ules.\end{prop}
\begin{proof}Clearly $E_0$ commutes with the action of $\klie$ and $K.$
So it suffices to check the operation on a basis of a complementary
space of $\klie$ in the Cartan decomposition, or for the shift
operators on highest weight vectors. It is not too hard to do this by
hand, on the basis of Tables \ref{tab-shab}, \ref{tab-shnab},
pp~\pageref{tab-shab}, \pageref{tab-shnab}, and the relations
in~\eqref{shps}. A check is in \cite[\S16]{Math}.
\end{proof}

\rmrk{An inverse operator} In~\cite{GW80} Goodman and Wallach define, in
a much more general context than $\SU(2,1)$, a linear form on the
analytic vectors in principal series representations given by an
infinite sum of differential operators. This induces a family of
intertwining operators $H^{\xi,\nu}_K \rightarrow \Mfu^{\xi,\nu}_\bt$,
which is inverse to evaluation at zero up to a factor that depends
meromorphically on~$\nu$.\medskip

\rmrk{Kunze-Stein operators}An interesting family of intertwining
operators is given by the \il{KSo}{Kunze-Stein operator}Kunze-Stein
operators. See Kunze, Stein \cite{KuSt}, or Schiffmann \cite{Schi}.
These operators turn up in the computation of Fourier coefficients of
Poin\-car\'e series. Here we mention their definition, but do not go
into computations.

The Kunze-Stein operators act on functions $F\in C^\infty(G)_K$ that
satisfy an estimate
\be\label{ccfd} F\bigl( n \am(t) k \bigr)\;\ll\; t^{2+\e}\ee
for some $\e>0,$ uniformly in $n$
(and $k$). For $\bt\in \CC$ and $\eta\in N \wm AMN$ (the big cell in the
Bruhat decomposition) the abelian Kunze-Stein operator is given by
\be\label{aSo}(S_{\!\bt}(\eta)F)(g) = \int_{n'\in N}
\overline{\chi_\bt(n')}\, F(\eta n'g)\, dn'\,,\ee
\il{Scho}{$S_{\!\Nfu}(\eta)$}and for $\n=(\ell,c,d)$ the non-abelian
Kunze-Stein operator is
\badl{nSo} (S_{\!\n}(\eta)F)\bigl( n a k \bigr)
&= \sum_{m,h,p,r,q} \Th_{\ell,c}(h_{\ell,m};n) \, \int_{n'\in N}
\int_{k'\in K} \overline{\Th_{\ell,c}(h_{\ell,m};n')}\\
&\qquad\hbox{} \cdot
F( \eta n' a k')\, \overline{\Kph h{p}{r}{q}(k')}\, dk'\, dn'\,
\frac{\Kph h{p}{r}{q}(k)}{\bigl\|\Kph h{p}{r}{q}\bigr\|^2}\,.
\eadl
The sum is over $m,h,p,r,q \in \ZZ$ satisfying $m\geq 0,$
$h\equiv p\equiv r\equiv q\bmod 2,$ $|r|\leq p,$ $|q|\leq p,$ and
$\sign(\ell)\left( 6m+3\right) + h - 3r=d.$
(These operators are similar to the Fourier term operators in
Proposition~\ref{prop-Ftit}.)
Under the condition \eqref{ccfd} the integrals converge absolutely.
Applied to $F\in \Mfu_\Nfu^{\xi,\nu}$ with $\re \nu>0$ we get
holomorphic families of intertwining operators
\badl{Sch-im} S_{\!\bt}(\eta) &: \Mfu_\Nfu^{\xi,\nu} \rightarrow
\begin{cases}H^{\xi,-\nu}_K\,,& \text{ if }\bt=0\,,\\
\Wfu_\bt^{\xi,\nu}&\text{ if }\bt\neq0\,,\end{cases}\\
S_{\!\n}(\eta)&: \Mfu_\Nfu^{\xi,\nu} \rightarrow \Wfu_\n^{\xi,\nu}\,.
\eadl


\def\flnm{rFtm-chap-III}

\setcounter{tabold}{\arabic{table}}
\setcounter{figold}{\arabic{figure}}


\chapter{Submodule structure}\label{chap-3}\setcounter{section}{10}
\setcounter{table}{\arabic{tabold}}
\setcounter{figure}{\arabic{figold}} \markboth{III. SUBMODULE STRUCTURE}
{III. SUBMODULE STRUCTURE} In the previous chapter we saw that under generic
parametrization (see Table~\ref{tab-parms}, p~\pageref{tab-parms})
the special Fourier term modules, like $\Mfu_\Nfu^{\xi,\nu}$ and
$\Wfu_\Nfu^{\xi,\nu},$ are isomorphic if they determine the same
element $[j_\xi,\nu]\in \WO$. Under integral parametrization this is no
longer the case. Then the special Fourier term modules have non-trivial
submodules, and the submodule lattice is not determined  only
 by an element of $\WO$. The main purpose of this chapter is
to determine the submodule structure of all Fourier term
modules~$\Ffu^\ps_\Nfu$.

By a general theorem all irreducible $(\glie,K)$-modules occur as a
submodule of a principal series module. So the irreducible modules that
we will find in Section~\ref{sect-ps-ism}, discussing the principal
series, represent all isomorphism classes of irreducible
$(\glie,K)$-modules. In comparison the generic abelian Fourier term
modules in Section~\ref{sect-abip} have a simpler submodule structure.
In Section~\ref{sect-nab} we will see that the submodule structure of
the non-abelian Fourier term modules $\Ffu_{\ell,c,d}^\ps$ is
complicated, and depends not only on the spectral parameters, but also
on the parameters $\ell$ and~$d$.


\def\flnm{rFtm-III-im}


\section{Preliminaries}\markright{11.
PRELIMINARIES}\label{sect-iclirr}The investigation of the submodule
structure will be done in separate sections for the $N$-trivial case,
the generic abelian case, and the non-abelian case. Here we carry out
preparations that will be used in all these cases.

The \il{sqth}{subquotient theorem}subquotient theorem of Harish Chandra
states that all isomorphism classes of irreducible $(\glie,K)$-modules
can be realized as subquotients of some representation $H^{\xi,\nu}_K$
in the principal series; see eg \cite[Theorem 3.5.6]{Wal88}. A result
by Casselman-Mili\v ci\'c \cite{CM82} implies that we even get all
irreducible $(\glie,K)$-modules as submodules of some $H^{\xi,\nu}_K$.
In Section \ref{sect-ps-ism} we will give the irreducible submodules of
representations in the principal series. Thus we will have the complete
list of isomorphism classes of irreducible $(\glie,K)$-modules. In
\S\ref{sect-list-irr} below we will give this list.

\subsection{Lattice points}\label{sect-lp}
The modules in \eqref{II-ex} are in the class of irreducible principal
series modules. All other isomorphism classes are represented in some
$H^{\xi,\nu}_K$ with parameters $(j_\xi,\nu)$ corresponding to integral
parametrization. Under integral parametrization we deal with Weyl group
orbits in the lattice\ir{Ldef}{L,\; L^+}
\be\label{Ldef} L \= \bigl\{ (j,\nu)\in \ZZ^2\;:\; j\equiv\nu\bmod
2\bigr\}\,.\ee
See Figure~\ref{fig-ip}, p~\pageref{fig-ip}. The origin $(0,0)$ in this
lattice does not correspond to integral parametrization.

A fundamental set for the action of the Weyl group~$W$ on $L$ is the
closure of the \il{pWch}{positive Weyl chamber} positive Weyl chamber
$ L^+ = \{(j,\nu)\in L\;:\; \nu \geq |j|\bigr\}$. (It is positive for
the choice of $\al_1$ and $\al_2$ in \eqref{al12} as simple positive
roots.) The walls of $L^+$ are the lines $j=\nu\in \ZZ_{\geq 1}$ and
$-j=\nu \in \ZZ_{\geq 1}$.\il{walls}{walls of positive chamber}

We use the convention to denote elements of $L^+$ by $(j_+,\nu_+)$;
hence $j_+\equiv \nu_+\bmod 2$, and $\nu_+\geq |j_+|$. Furthermore, we
denote elements of the adjacent Weyl chambers by
$(j_r,\nu_r) \in \Ws 1L^+$ and $(j_l,\nu_l)\in \Ws2 L^+$.
\il{jprl}{$j_+,\; j_r, \; j_l$}\il{nuprl}{$\nu_+,\; \nu_r, \; \nu_l$}
Hence $j_r\geq 1$, $0\leq \nu_r\leq j_r$, and $j_l \leq -1$,
$0\leq \nu_l \leq -j_l$. On the walls we have
$(j_+,\nu_+) = (j_r,\nu_r) = (j,j)$, $j\in \ZZ_{\geq 1}$, and
$(j_+,\nu_+) =(j_l,\nu_l) = (-j,j)$, $j\in \ZZ_{\geq 1}$.

If $(j_+,\nu_+)$, $(j_r,\nu_r)$ and $(j_l,\nu_l)$ are in the same Weyl
group orbit we have the relation
\badl{jnurels} (j_r,\nu_r)&\= \Ws1(j_+,\nu_+) \= \bigl(
\tfrac{3\nu_+-j_+}2,\tfrac{\nu_++j_+}2\bigr) \,,\\
(j_l,\nu_l)&\= \Ws2(j_+,\nu_+) \= \bigl(
\tfrac{-3\nu_+-j_+}2,\tfrac{\nu_+-j_+}2\bigr) \,.
\eadl
Under these relations we have the following identities:
\be j_r+j_l + j_+ \= 0\,,\qquad \nu_r-\nu_+ +\nu_l\=0\,.\ee

\subsection{Isomorphism classes of irreducible
representations}\label{sect-list-irr}We give in
Subsection~\ref{sect-smps} a list of isomorphism classes of irreducible
$(\glie,K)$-modules that are embedded in a principal series
representation. As discussed in the introduction to this section, this
is the complete list.

All the irreducible $(\glie,K)$-modules turn out to be special modules;
see Definition~\ref{def-scm}.

For the purpose of these notes we classify the isomorphism classes into
four \il{tpirr}{type of irreducible representation}types, $\II$, $\IF$,
$\FI$, and $\FF$, according to the action of the upward shift
operators, which in Table~\ref{tab-sho}, p~\pageref{tab-sho}, are given
by the action of elements of the complexified Lie algebra~$\glie_c$.
The first letter refers to $\sh 3 1$, and the second letter to
$\sh{-3}1$. This letter is $I$ if the shift operator acts injectively
in the module; otherwise it is $F$. Most classes have a $K$-type of
dimension $1$. If the dimension of the minimal $K$-type is larger
than~$1$, we add a subscript $+$. To completely determine the
isomorphism class we add the choice of spectral parameters $(j,\nu)$
such that the irreducible module occurs in the module $H^{\xi_j,\nu}_K$
in the principal series. In many cases this choice is
unique.\il{IIIF}{$\II$, $\IF$}\il{FFFI}{$\FF$, $\FI$}

The following list gives all isomorphism classes of irreducible
$(\glie,K)$-modules. We indicate also which of the classes admit a
unitary structure, to be discussed in \S\ref{sect-un}.
\begin{itemize}
\item \emph{Irreducible principal series}\il{pdsi}{principal series,
irreducible}
\begin{enumerate}
\item[$\II(j,\nu)$] with $j\in \ZZ$, $\re \nu \geq 0$,
$\nu \not\equiv j\bmod 2$, or $(j,\nu)=(0,0)$. The principal series
modules $H^{\xi_j,\nu}_K$ and $H^{\xi_j,-\nu}_K$ are in this
isomorphism class.

A unitary structure occurs in the following cases:
\begin{itemize}\item $\re\nu=0$, \emph{unitary irreducible principal
series}.\il{ps-ui}{principal series, unitary irreducible}
\item $\nu \in \RR$, $0<\nu<2$ if $j=0$, or $0<\nu<1$ if $j$ is odd,
\emph{complementary series}.\il{cs}{complementary series}
\end{itemize}

\end{enumerate}
\item \emph{Discrete series types}
\begin{enumerate}
\item[$II_+(j_+,\nu_+)$] with $j_+\in \ZZ$, $\nu_+\in \ZZ_{\geq 0}$,
$j_+\equiv\nu_+\bmod2$, $\nu_+\geq |j_+|$. \emph{Large discrete series
type.}\il{ldst}{large discrete series type}

\item[$\IF(j_r,\nu_r)$] with $j_r \in \ZZ_{\geq 2}$,
$\nu_r \equiv j_r\bmod 2$, $0\leq \nu_r\leq j_2-2$. \emph{Holomorphic
discrete series type.}\il{hdst}{holomorphic discrete series type}

\item[$\FI(j_l,\nu_l)$] with $j_l\in \ZZ_{\leq -2}$,
$\nu_l\bmod j_l\bmod 2$, $0\leq \nu_l \leq |j_l|-2$.
\emph{Antiholomorphic discrete series type.} \il{adst}{antiholomorphic
discrete series type}
\end{enumerate}
All modules of discrete series type admit a unitary structure.

\item \emph{Langlands representations}\il{Llr}{Langlands representation}
\begin{enumerate}
\item[$\IF_+(j_r,-\nu_r)$] and $\IF(j_r,-j_r)$ with
$j_r\in \ZZ_{\geq 1}$, $\nu_r\equiv j_r\bmod 2$, and
$1\leq \nu_r \leq j_r-2$ for $\IF_+(j_r,\nu_r)$; $\nu_r=j_r$ for
$\IF(-j_r)$.

\item[$\FI_+(j_l,-\nu_l)$] and $\FI(j_l,j_l)$ with
$j_l \in \ZZ_{\leq -1}$, $\nu_l\equiv j_l\bmod 2$, and
$1\leq \nu_l \leq |j_l|-2$ for $\FI_+(j_l,-\nu_l)$; $\nu_l=j_l$ for
$\FI(j_l,j_l)$ .

\item[$\FF(j_+,\nu_+)$] with $j_+ \in \ZZ$, $\nu_+\equiv j_+\bmod 2$,
$\nu_+\geq |j_+|+2$. \emph{Finite-dimensional irreducible modules.}
\end{enumerate}
The Langlands representations that admit a unitary structure are
$\FF(2,0)$, and all classes $\IF(j_r,-1)$, $\FI(j_l,-1)$. We call the
representations with $\nu=-1$ \emph{thin
representations}.\il{cohrepr}{thin representation}
\end{itemize}

\rmrk{Occurrence in principal series modules} These isomorphism classes
are represented by one or more genuine submodules of a principal series
representation and by one or more quotients of principal series
representations. We will see this explicitly in \S\ref{sect-smps}. The
classes $\II(j,\nu)$ are represented by $H^{\xi_j,\nu}_K$, which is a
(trivial)
submodule and a (trivial) quotient of itself.
\rmrk{Discrete series type} \il{ds}{discrete series} Discrete series
representations have spectral parameters $(j,\nu)$ in the interior of a
Weyl chamber in Figure~\ref{fig-ip}, p~\pageref{fig-ip}. They are
characterized by being represented in $L^2(G)$. \il{lds}{limits of
discrete series}Limits of discrete series correspond to $(j,\nu)$ on a
wall between Weyl chambers. These classes are not represented in
$L^2(G)$. We put these classes together into a discrete series type.

We distinguish holomorphic and antiholomorphic discrete series type.
These concepts are interchanged if we change the complex structure of
the symmetric space $G/K$ (or of the space into which $G/K$ is
embedded). We prefer to keep both names, in accordance with
\cite[\S7]{Wal76}.

\rmrk{Langlands representations}All Langlands representations are
non-tempered. They occur as quotient of a unique principal series
module $H^{\xi,\nu}_K$ with $\nu \geq 1$, and are often called
\emph{Langlands quotients}.

\subsection{Submodules determined by shift operators} From Propositions
\ref{prop-kdso} and~\ref{prop-kuso} we obtain the information that
kernels of shift operators can occur only in $K$-types $\tau^h_p$ that
correspond to points in the $(\frac h3,p)$-plane on specific lines.
This allows us to draw conclusions concerning submodules.

\begin{prop}\label{prop-smks}Let $j\in \ZZ$ and let $V$ be a $(\glie,K)$
module such that
\be \label{Vsmks} V \= \bigoplus_{a,b\geq 0} V_{2j +3(a-b),a+b}\,.\ee
\begin{itemize}
\item[i)] For $c\in \ZZ$ and $u$ a linear form on $\RR^2$ let $X_{u,c}$
be the subspace of $V$ given by
\be X_{u,c} = \bigoplus_{(h/3,p)\in \sect(j)\;:\; u(h,p)\geq c} V_{h,p}
\,.\ee
The linear space $X_{u,c}$ is a proper $(\glie,K)$-submodule of~$V$ in
the following cases:
\begin{enumerate}
\item[a)] $u(h,p) = -\frac h3-p$, $c\leq -\frac{2j}3$, and
$\sh 31 V_{h,p,p}=\{0\}$ for all $(h/3,p)\in \sect(j)$ such that
$u(h,p)=c$.
\item[b)] $u(h,p) =\frac h3-p$, $c\leq - \frac{2j}3$, and
$\sh {-3}1 V_{h,p,p}=\{0\}$ for all $(h/3,p)\in \sect(j)$ such that
$u(h,p)=c$.
\item[c)] $u(h,p) = -\frac h3+p$, $c\geq \frac{2j}3$, and
$\sh 3{-1} V_{h,p,p}=\{0\}$ for all $(h/3,p)\in \sect(j)$ such that
$u(h,p)=c$.
\item[d)] $u(h,p)=\frac h3 +p$, $c\geq \frac{2j}3$, and
$\sh {-3}{-1} V_{h,p,p}=\{0\}$ for all $(h/3,p)\in \sect(j)$ such that
$u(h,p)=c$.
\end{enumerate}
\item[ii)]Let $Y\subset V$ be a submodule such that the subspace
$Y_{h,p}$ of $K$-type $\tau^h_p$ is non-zero. If
$\sh{3\al}\bt Y_{h,p,p}\neq 0$ for $\al,\bt\in \{1,-1\}$, then
$Y_{h+3\al,p+\bt}\neq 0$.
\end{itemize}
\end{prop}
\rmrk{Remarks}
(1) In the proposition we do not assume that the module is generated by
a minimal element, or that the Casimir element acts by multiplication
by a scalar. We will apply the proposition to the Fourier term modules
$\Ffu^\ps_\Nfu$, in which the Casimir operator acts as a scalar.

(2) The proposition enables us to reduce the study of submodules for the
action of~$\glie$ to the consideration of shift operators. We will use
this result repeatedly in this chapter. Part i) tells us that lines on
which a shift operator vanishes determine submodules. Part ii) tells
that all submodules are visible in the vanishing of shift operators.
\begin{proof}Part~ii) follows from the fact that
$Y_{h,p} = U(\klie) Y_{h,p,p}$.

In i)c) we have the following situation:
\begin{center}\grf{5}{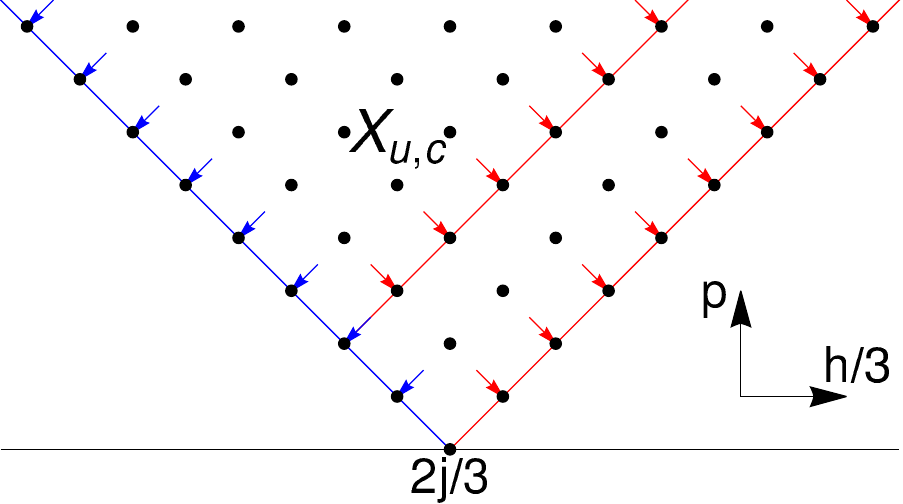}\end{center}
We use here, and later on, the convention that a downward arrow to a
point $(h/3,p)$ indicates that application of the corresponding
downward shift operator ``stops at that point'': the space $V_{h,p,p}$
is in the kernel of that shift operator. In the picture are two lines
of points where the shift operator $\sh 3{-1}$ is zero, and one line of
points where $\sh{-3}{-1}$ is zero.

The action of the Lie algebra sends an element of $V_{h,p}$ to an
element in the sum of spaces $V_{h',p'}$ with
$\frac{|h-h'|}3, |p-p'| \in \{-1,0,1\}$. The critical $K$-types
correspond to points on the line in the interior of the sector. Let
$(h/3,p)$ be such a point, and suppose that for $v\in V_{h,p,q}$ there
exists $u\in U(\glie)$ such that $u v \in V_{h+3,p-1}$. The elements
$\Z_{12}$ and $\Z_{21}$ in $\klie_c$ change the weight in a $K$-type by
one, and are injective except on the lowest or highest weight spaces,
respectively. Using this we can reduce the situation to
$v\in V_{h,p,p}$ and $uv \in V_{h+3,p-1,p-1}$.

We write $u$ as a linear combination of elements in
$U(\klie) \Z_{31}^\al \Z_{23}^\bt \Z_{13}^\gm \Z_{32}^\dt$. A
contribution with $\dt\geq 0$ can be reduced by the assumption that
$\sh 3 {-1} v=0$ to terms with lower values of $\dt$. (We use the
description of the shift operators in Table~\ref{tab-sho},
p~\pageref{tab-sho}.) Repeating this we arrange $\dt=0$. Further, we
note that $\Z_{31}$, $\Z_{23}$ and $\Z_{13}$ preserve the $K$-types
with $u(h,p)\geq c$. (For $\Z_{13}$ we use again Table~\ref{tab-sho}.)
This shows that $u v$ is contained in~$W$. So under condition c) we get
invariance of~$X_{u,c}$.

Under condition~d) we proceed in the same way, with a reversed role of
$\Z_{32}$ and $\Z_{13}$.
\begin{center}\grf{5}{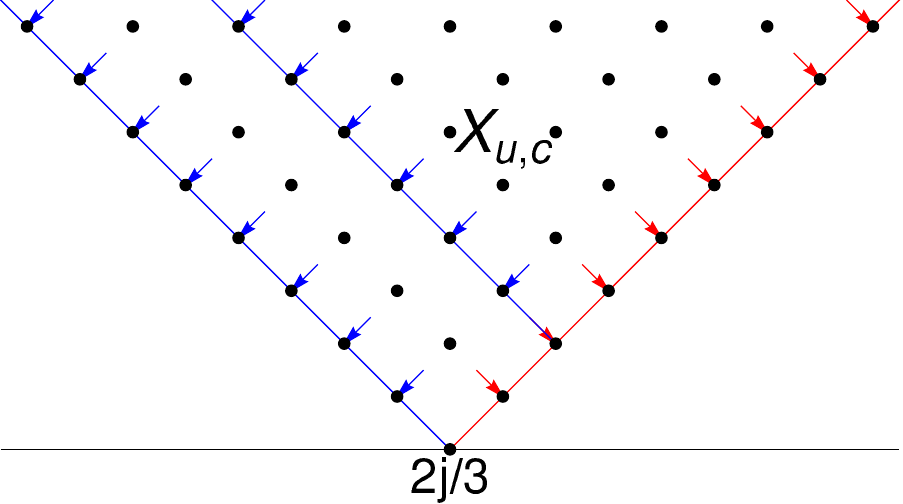}\end{center}

For b) the module $X_{u,c}$ has $K$-types in the region between the
lines with slope 1. For $c=\frac{2j}3$ both lines coincide, and
$X_{u,c}$ is a non-trivial subspace of $V$.
\begin{center}\grf{5}{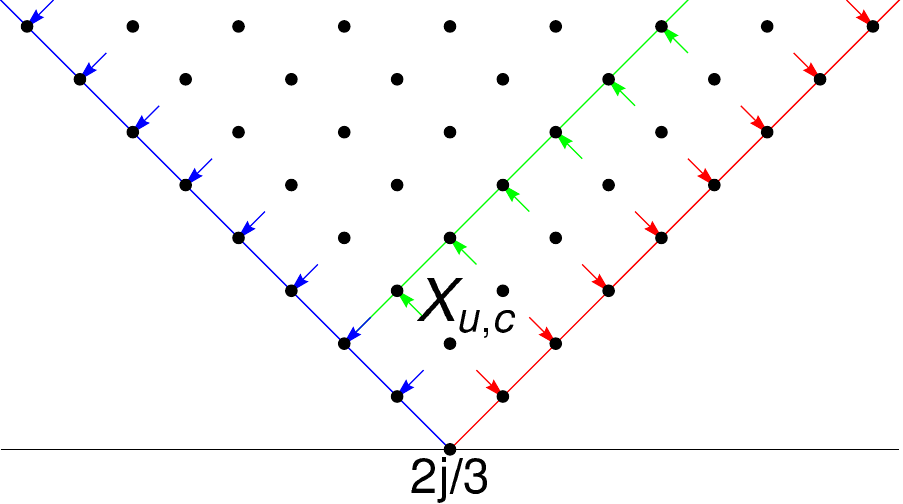}\qquad \grf5{cpropb1}\end{center}
An upward arrow to a point $(h/3,p)$ indicates that the space
$V_{h,p,p}$ is in the kernel of the corresponding upward shift
operator. In the picture on the right the points
$(h/3,p)= (2j/3,0)+a(1,1)$ correspond to spaces $V_{h,p,p}$ on which
both shift operators $\sh3{-1}$ and $\sh{-3}1$ vanish.

Like above we reduce consideration to $v\in V_{h,p,p}$ with $u(h,p)=c$,
for which we know that $\sh{-3}1 v=0$. Decomposing $u$ as above, we
reduce to $\gm=\dt=0$ by the observation that application of
$\sh 3{-1}$, and of $\sh{-3}{-1}$, does not decrease the value of
$u(h,p)$. Since $\sh {-3}1 v=0$ we can remove all terms with
$\bt\geq1$. The proof for assumption b) is completed by the observation
that the action of $\Z_{31}$ does not decrease the value of $u(h,p)$.

\begin{center}\grf{5}{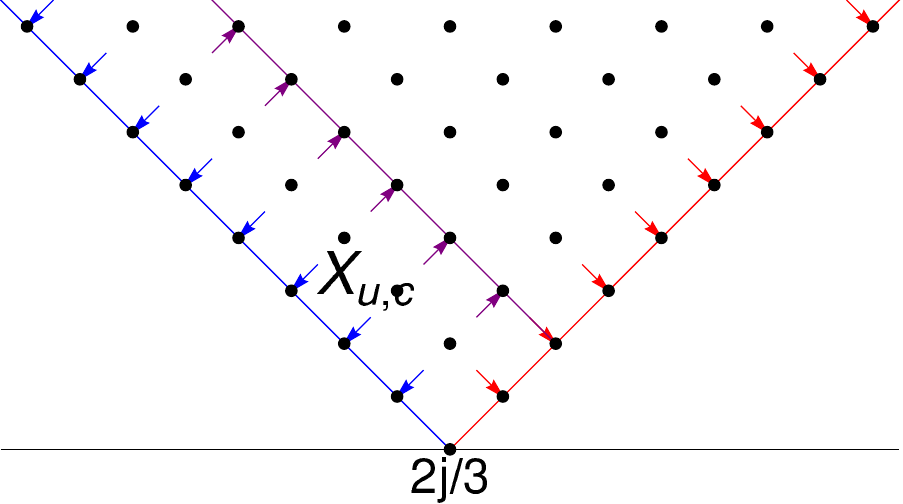}\end{center}
For assumption a) we proceed similarly with a reversed role of $\Z_{31}$
and $\Z_{23}$.
\end{proof}

\rmrk{Conventions in the figures} The conventions explained in the proof
will be used for many figures in the $(h/3,p)$-plane in the sequel.


\def\flnm{rFtm-III-ps}

\section{Principal series and related
modules}\label{sect-ps-ism}\markright{12. PRINCIPAL SERIES AND RELATED
MODULES} Under the assumption of integral parametrization, the variety
of submodule structures of principal series modules is large: in
\S\ref{sect-smps} we will meet twelve different submodule structures.
Most isomorphism classes in the list in~\S\ref{sect-list-irr} occur in
only one principal series module. The type $\II_+(j_+,\nu_+)$ forms the
sole exception.

In most cases of integral parametrization, the Fourier term module
$\Ffu_0^\ps$ is the direct sum of a number of non-isomorphic principal
series modules. If the spectral parameters are on a wall of a Weyl
chamber we need also the logarithmic submodules, discussed
in~\S\ref{sect-log}.

The last subsection, \S\ref{sect-iskdso}, is of a more technical nature.
It is relevant not only for the modules $\Ffu_0^\ps$, but also for
other Fourier term modules.
\subsection{Kernels of shift operators and submodules}
\rmrk{$K$-types}The $K$-types $\tau^h_p$ occurring in $H^{\xi,\nu}_K$
have multiplicity one. These $K$-types correspond to points in the
sector $\sect(j)$ in the $(h/3,p)$-plane indicated in
Figure~\ref{fig-sectps}.
\begin{figure}[htp]
\begin{center}\grf{4.7}{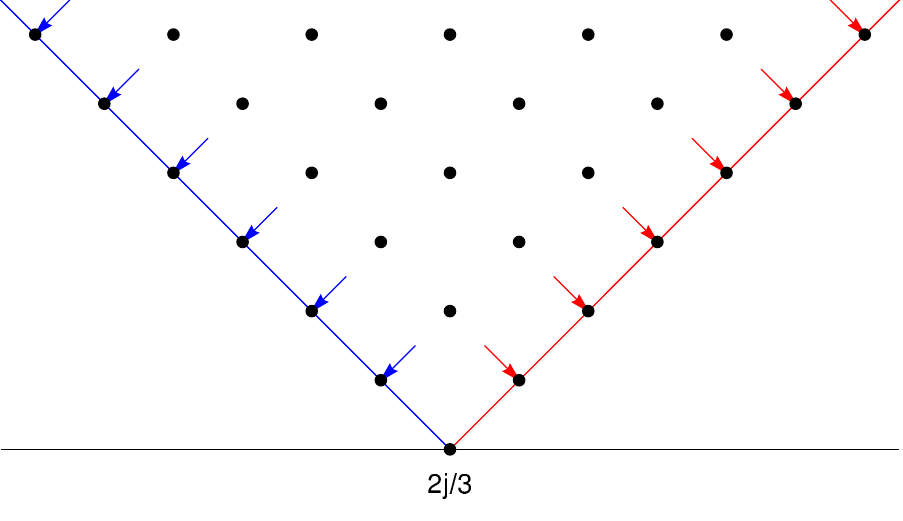}\end{center}
\caption{Points corresponding to the $K$-types occurring in
$H^{\xi,\nu}_K.$ The downward shift operator $\sh3{-1}$, respectively
$\sh{-3}{-1}$, is zero on the highest weight in $K$-types indicated by
line with slope~$1$, respectively~$-1$. }\label{fig-sectps}
\end{figure}

\begin{prop}\label{prop-ksops}Let $(j,\nu)\in L$ and
$(j_j,\nu_j) = \Ws j (j,\nu)$, and let $\tau^h_p\in\sect(j)$. Then
$\kph h {p}{r}{p}(\nu)$, with $h=3r+2j$, is in the kernel of a shift
operator precisely in the cases in Table~\ref{tab-vancond}.
\end{prop}
\begin{table}[htp]
\[
\renewcommand\arraystretch{1.4}
\begin{array}{|c|c|c|}
\hline
&\text{if}&\text{for}
\\ \hline
\sh 31& j_2\geq j+3 \quad (\nu\leq
-j-2) & \frac h3+p= \frac 23 j_2-2\\
\hline
\sh{-3}1& j_1\leq j-3\quad (\nu \leq j-2) & \frac h3-p= \frac23 j_1+2\\
\hline
\sh3{-1}& j_2< j \quad(\nu \geq
-j+2) & \frac h3-p = \frac23 j_2\\
& \text{all }(j,\nu)&\text{or } \frac h3-p = \frac23j
\\ \hline
\sh{-3}{-1}&j_1>j\quad (\nu \geq j+2)
& \frac h3+p=\frac23 j_1\\
& \text{all }(j,\nu)&\text{or } \frac h3+p= \frac23j
\\ \hline
\end{array}
\]
\caption[]{Conditions for vanishing of shift operators on the function
$\kph h{p}{(h-2j)/3}{p}(\nu)$.\\
We use $j_1=\frac12(3\nu-j)$ and $j_2=-\frac12(3\nu+j)$. }
\label{tab-vancond}
\end{table}
\begin{proof}We use that $h=2j+3r$ for
$\kph h{p}{r}{p}(\nu)\in H^{\xi,\nu}_K$.

For all values of $(j,\nu)$ the element $\kph h {p}{\pm p}{p}(\nu)$ is
in the kernel of the downward shift operator $\sh{\pm 3}{-1}$. The
corresponding points in Figure~\ref{fig-sectps} are on the right,
respectively left, boundary of the sector $\sect(j)$. This gives the
second lines of the possibilities for $\frac h3\mp p$ in the last
column.

In \eqref{shps} we see that $\sh{\pm 3}{-1}\kph h {p}{\pm p}{p}(\nu)$
vanishes also if
\[ 0 \= \pm h+2\nu-2p\mp r\= 2\left(\nu \pm 13 j\pm \frac13 h - p
\right)
\,.\]
(We used that $r=\frac13(h-2j)$.)
For $\sh3{-1}$ this gives
\[ \frac13 h-p \= - \nu-j = \frac23\,j_2 \,.\]
This represents a line with slope $1$ in Figure \ref{fig-sectps}
intersecting the horizontal axis in the point
$\bigl( \frac23 j_2,0\bigr)$. If $j_2<j$ this line has points in the
sector minus the right boundary line. So, if $j_2<j$ it gives new
spaces in $H^{\xi,\nu}_{K;h,p,p}$ on which $\sh 3{-1}$ vanishes. This
gives the third box in Table~\ref{tab-vancond}. For $\sh{-3}{-1}$ we
obtain the line $\frac13 h +p= \frac 23 j_1$, which leads to new kernel
elements of $\sh{-3}{-1}$ if $j_1>j$. This gives the bottom box in the
table.

For the vanishing of $\sh{\pm 3}1\kph h{p}{r}{p}(\nu)$ we find in
\eqref{shps} the condition
\[ 4 \pm h + 2\nu + 2p \mp r\=0\,.\]
This yields the vanishing of $\sh31$ for points in the sector on the
line $\frac h3+p=\frac23 j_2-2$ if $j_2\geq j+3$. This gives the first
box in the table. We obtain the second box in a similar way. For
general $(j,\nu)$ the upward shift operators are injective. So if the
line that we find coincides with a boundary line of the sector
$\sect(j)$ we get new information. Hence we need no strict inequalities
in the first two boxes.
\end{proof}
\rmrk{Remark} For two different points in a Weyl orbit in $ \WOI$ the
space $H^{\xi,\nu}_K$ and $H^{\xi',\nu'}_K$ have intersection zero.
Nevertheless the boundary lines of $\sect(j')$ have a significance for
$H^{j,\nu}_K$. On points in $\sect(j) \cap \partial \sect(j')$ at least
one shift operator vanishes. On points in
$\partial\sect(j) \cap \partial \sect(j')$ at least two shift operators
vanish.
\subsection{Submodules of principal series modules}\label{sect-smps}
Table~\ref{tab-vancond} determines half-planes in $\RR^2$ that induce
subsets of $L$ for which there is a line intersecting $\sect(j)$ on
which a shift operator vanishes in $H^{\xi,\nu}_K$. There are non-empty
intersections of two of such half-planes, but no non-empty intersection
of three half-planes.
\begin{figure}[htp]
\begin{center}\grf{5}{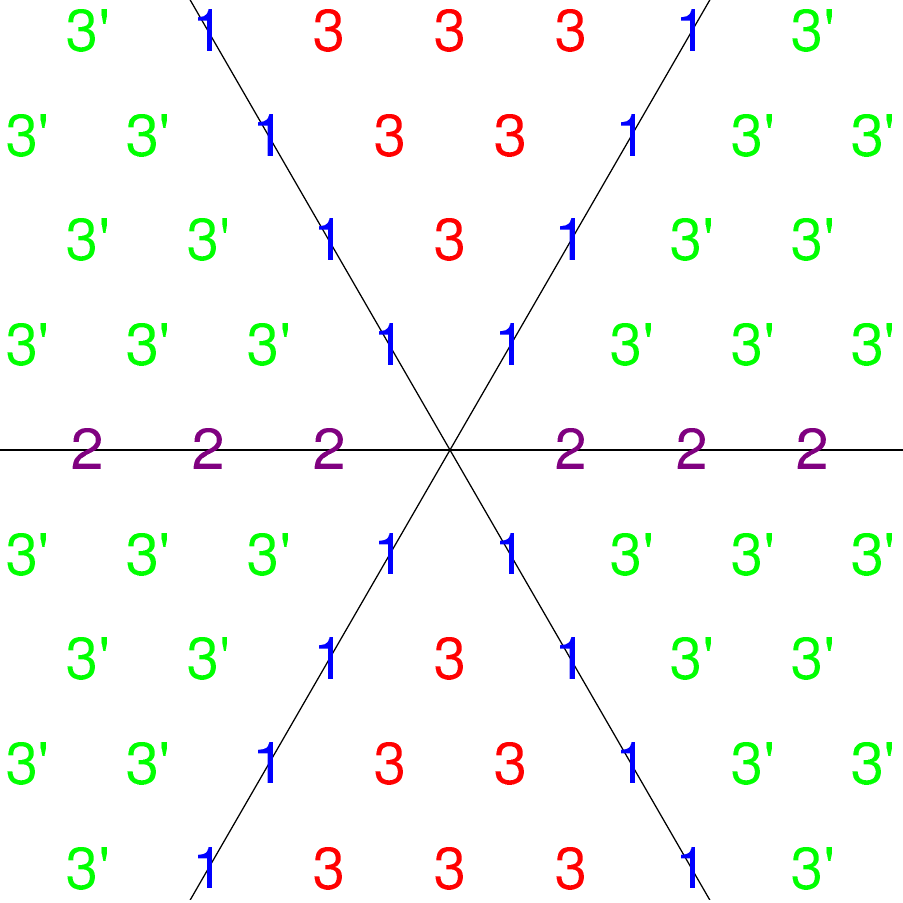}\end{center}
\caption{Points $(j,\nu)$ in the lattice $L$ and the number of genuine
submodules of $H^{\xi,\nu}_K$. }\label{fig-gsm}
\end{figure}
There are several possibilities, indicated in Figure~\ref{fig-gsm}. If a
point is in only one of the half-planes, then $H^{\xi,\nu}_K$ has one
genuine submodule. If $(j,\nu)$ is in the intersection of two
half-planes, there are several possibilities. If cases i)a) and i)d) in
Proposition~\ref{prop-smks} are combined, or cases i)b) and i)d), then
the half-planes are bounded by two parallel lines in the
$(h/3,p)$-plane, leading to two or three genuine submodules, depending
on the relative position, indicated by 2 or $3'$ in
Figure~\ref{fig-gsm}. In the other combinations the half-planes are
bounded by intersecting lines. This results in three genuine
submodules, indicated by 3 in the figure. In this way we obtain twelve
regions in the lattice to investigate.\medskip

We consider these regions separately. For each we give two
illustrations. On the left we sketch the corresponding subset of the
lattice $L$. On the right we give a sketch of the $(h/3,p)$-plane for
one point $(j,\nu)$, using the conventions indicated at the end of
Section~\ref{sect-iclirr}. If we move the point on the left, then the
lines where the shift operators vanish move as well. As long as we stay
in the interior of a Weyl chamber the submodule structure keeps its
general structure. Different structures arise when we cross a wall.

We give parameters for the irreducible submodules and the irreducible
quotients. Furthermore, we give a composition diagram.

We use the notation for the spectral parameters indicated
in~\S\ref{sect-lp}. In particular we assume in each case that
$(j_+,\nu_+)$, $(j_r,\nu_r)$ and $(j_l,\nu_l)$ are related
by~\eqref{jnurels}.

Afterwards, we summarize the occurrence of irreducible modules in
principal series modules in Table~\ref{tab-isot-ps} and
Figure~\ref{fig-isot}.
\medskip

\subsubsection{}\label{sect-lds}Spectral parameters:
$\nu_+\geq |j_+|+2$. See Figure~\ref{fig-ps1}.
\begin{figure}[ht]
\begin{center}
\grf{4.7}{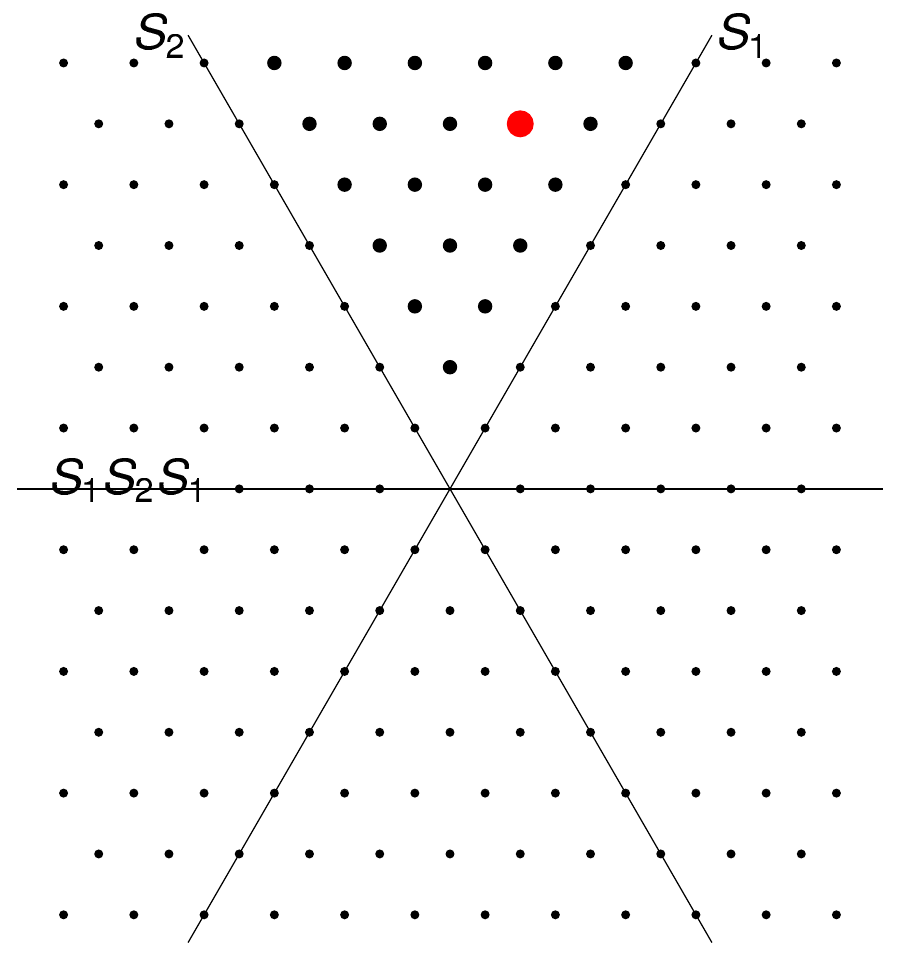} \quad \grf{6.7}{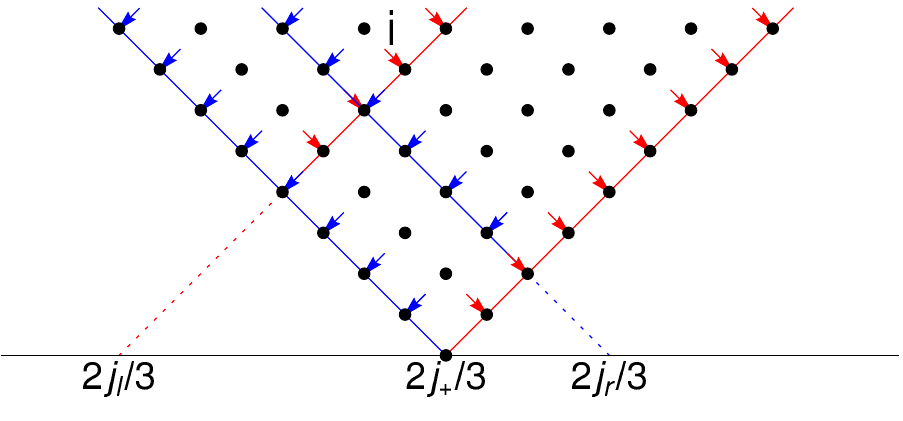}
\end{center}
\caption{The region $\nu_+\geq |j_+|+2$ in \S\ref{sect-lds}. Illustrated
for $(j_+,\nu_+)=(2,6)$. }\label{fig-ps1}
\end{figure}

There are four genuine submodules.
\be \renewcommand\arraystretch{1.4}
\begin{array}{|clll|}\hline
(a)&h_0= j_l+j_r \=-j_+
&p_0 =\frac{j_r-j_l}3 \=\nu_+
&A= B=\infty \\ \hline
(ab)&h_0=j_++j_l\=-j_r
&p_0= \frac{j_+-j_l}3 \=\nu_r
&A=B=\infty\\\hline
(ac)&h_0= j_r+j_+\= -j_l& p_0 = \frac{j_r-j_p}3\=\nu_l & A=B=\infty\\
\hline
(abc) &\multicolumn{3}{l|}{\text{ not a special $(\glie,K)$-module}}
\\ \hline
\end{array}\ee
The letters indicate the sets of $K$-types in the submodule. See
Figure~\ref{fig-ps1a}.
\begin{figure}[ht]
\begin{center}
\grf{6.7}{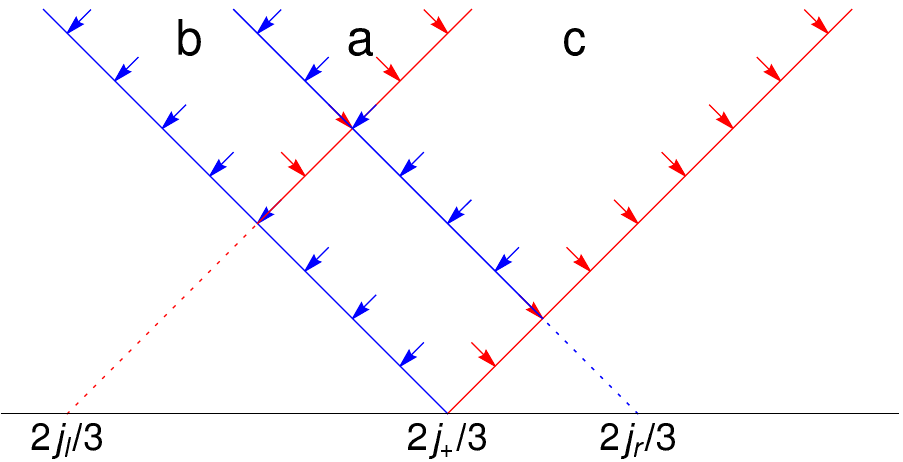}
\end{center}
\caption{Indication of regions corresponding to sets of $K$-types
corresponding to genuine subspaces in
\S\ref{sect-lds}.}\label{fig-ps1a}
\end{figure}

Submodule (a) is irreducible of type $\II_+(j_+,\nu_+)$. The quotient
$H^{\xi_+.\nu_+}_K \bmod (abc)$ is irreducible of type
$\FF(j_+,-\nu_+)$.

We find in addition the following irreducible subquotients
\begin{align*}
&(ab)/\bigm/(a) \;\cong (abc)\bigm/(ac) &\text{type
}&\FI_+(j_l,-\nu_l)\\
& (ac)\bigm/ (a) \;\cong\; (abc)\bigm/(ab) & \text{type
}&\IF_+(j_r,-\nu_r)
\end{align*}
This leads to the following composition diagram.
\badl{cd1} \xymatrix@R.5cm@C1.2cm{
&& (ab) \ar[rd] ^{\IF_+(j_r,-\nu_r)}
\\
\{0\} \ar[r] ^{\II_+(j_+,\nu_+)} & (a) \ar[ru]^{\FI_+(j_l,-\nu_l) }
\ar[rd]_{\IF_+(j_r,-\nu_r)}
&&
(abc) \ar[r] ^{\FF(j_+,-\nu_+)} & H^{\xi_+, \nu_+} _K \\
&& (ac) \ar[ru] _{\FI_+(j_l,-\nu_l) } }
\eadl
\rmrk{Conventions in the decomposition diagrams}\il{cd}{composition
diagram} Each arrow in the diagram denotes an inclusion of submodules
with irreducible quotient. We indicate the isomorphism class of the
quotient above or below the arrow. The left-most arrows correspond to
irreducible submodules, the right-most arrows correspond to irreducible
quotients of the principal series module. The other arrows correspond
to other subquotients.

\subsubsection{}\label{sect-ps2}Spectral parameters:
$j_+=\nu_+\in \ZZ_{\geq 1}$ in \S\ref{sect-ps2}. See
Figure~\ref{fig-ps2}.
\begin{figure}[ht]
\begin{center}
\grf{4.7}{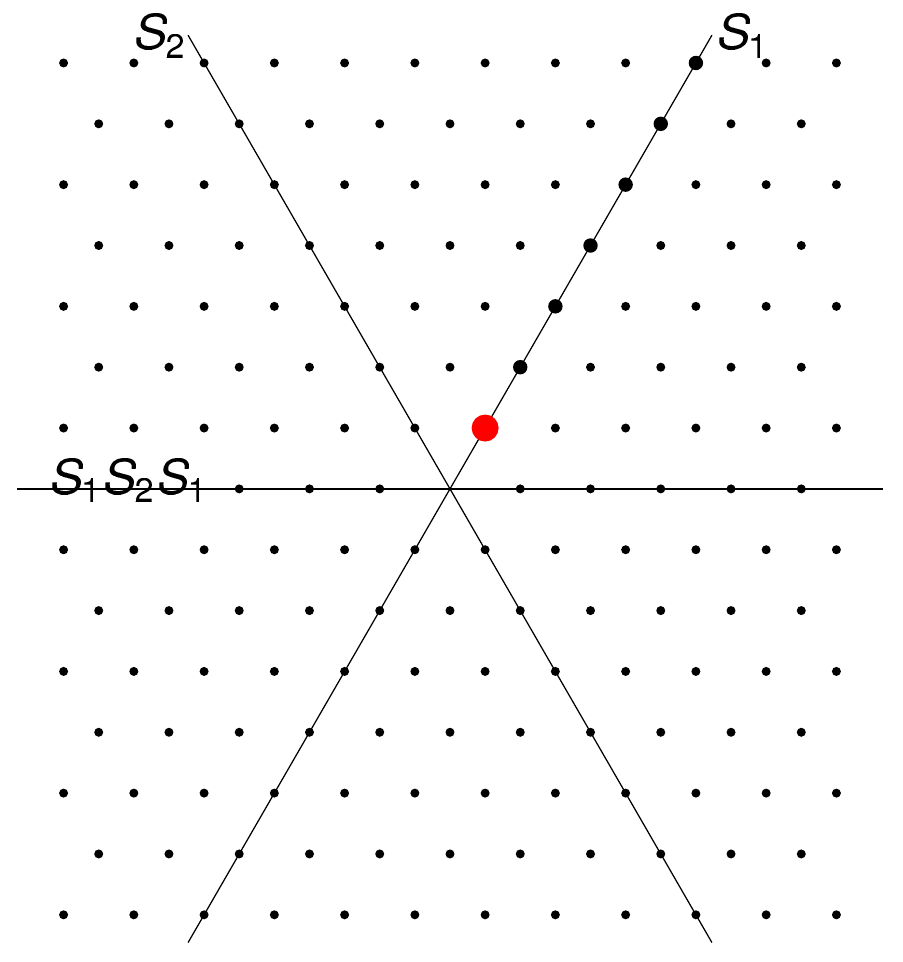} \quad \grf{6.7}{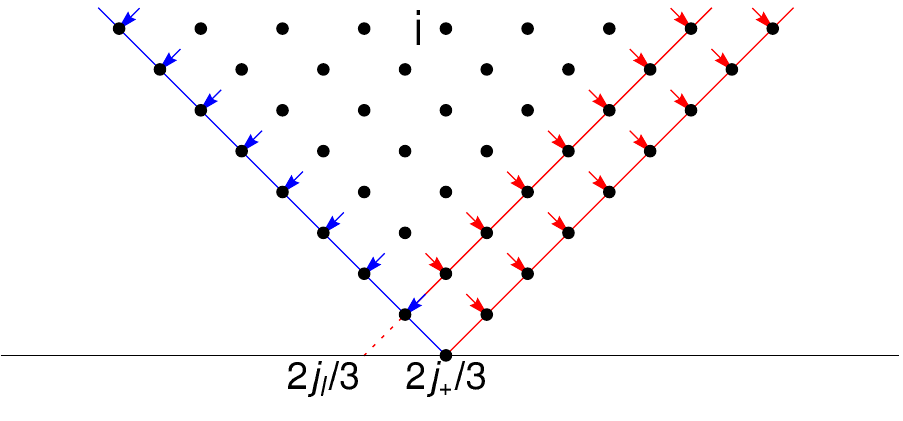}
\end{center}
\caption{The region $\nu_+= j_+ \in \ZZ_{\geq 1}$ in \S\ref{sect-ps2}.
Illustrated for $(j_+,\nu_+)=(1,1)$. }\label{fig-ps2}
\end{figure}

In this case $(j_l,\nu_l)=(j_+,\nu_+)$. The irreducible submodule has
type $\II_+(j_+,j_+)$, with the following parameters.
\be \renewcommand\arraystretch{1.4}
\begin{array}{|lll|}\hline
h_0=j_l+j_+ \=- j_+& p_0= \frac{j_+-j_l}2 = \nu_+ & A=B=\infty
\\ \hline\end{array}
\ee
The unique quotient has type $\IF(j_r,-j_r)$. Composition diagram:
\badl{cd2}\xymatrix@C1.2cm{ \{0\} \ar[r]^{\II_+(j_+,j_+)} & \square
\ar[r]^{\IF(j_r,-j_r)} & H^{\xi_+,j_+}_K }
\eadl

\subsubsection{}\label{sect-ps3} Spectral parameters:
$3\leq \nu_r\leq j_r-2$. See Figure~\ref{fig-ps3}.
\begin{figure}[ht]
\begin{center}
\grf{4.7}{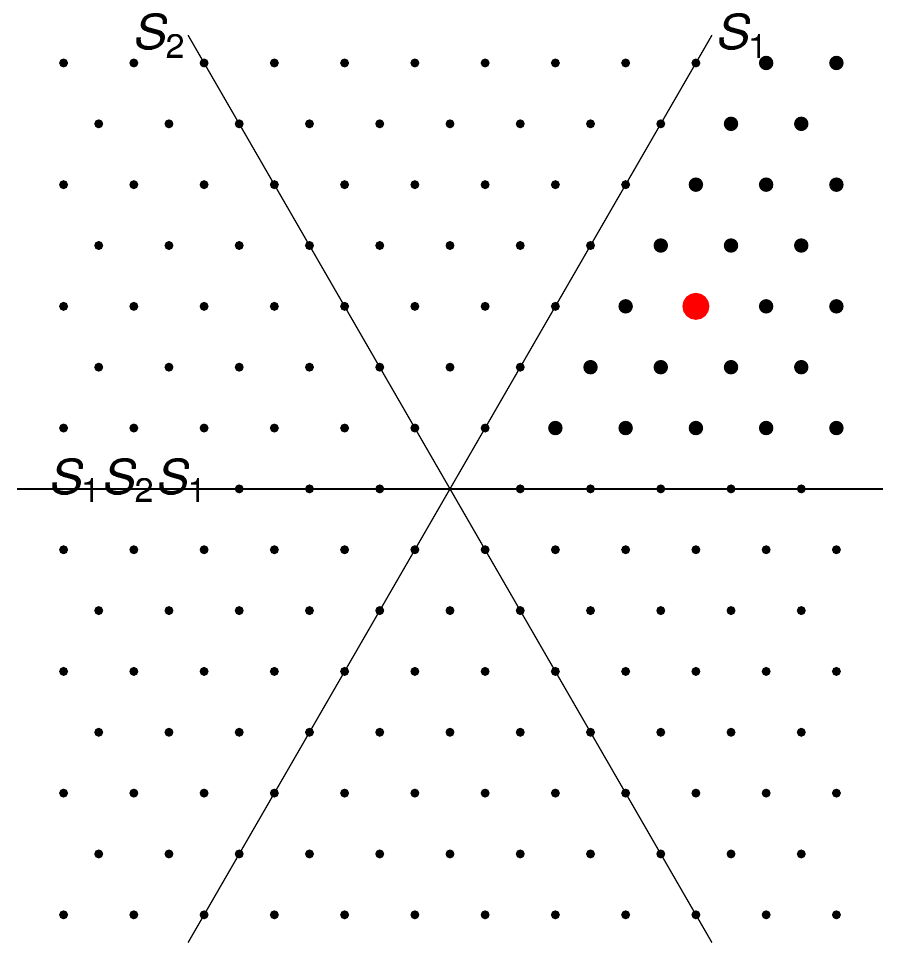} \quad \grf{6.7}{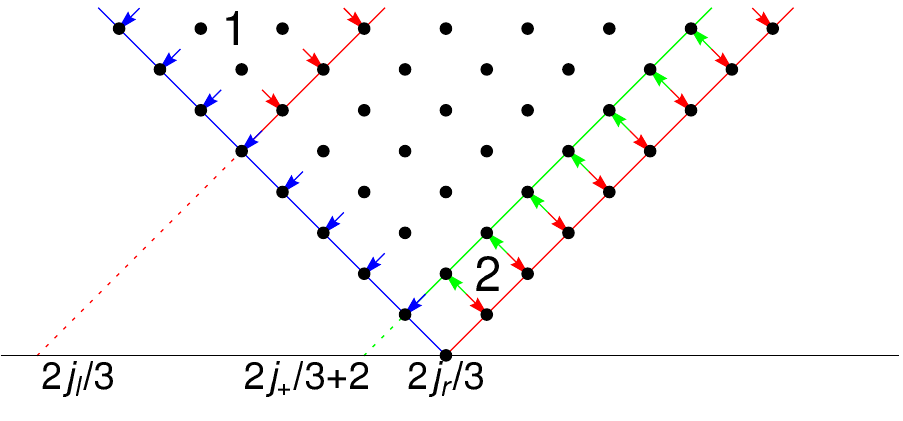}
\end{center}
\caption{The region $1\leq \nu_r \leq j_r-2$ in \S\ref{sect-ps3}.
Illustrated for $(j_r,\nu_r)=(7,3)$. }\label{fig-ps3}
\end{figure}

There are three genuine submodules; two of them are irreducible, the
third one is the direct sum of the irreducible submodules.
\be \renewcommand\arraystretch{1.4}
\begin{array}{|cll|}\hline
1& h_0= j_r+j_l=- j_+& p_0 = \frac{j_r-j_l}3=\nu_+\\
& A=\infty&B=\infty\\ \hline
2& h_0=2j_r & p_0=0\\
&A=\infty& B= \frac{j_r-j_+}3-1=\nu_l-1
\\ \hline\end{array}
\ee
Submodule $1$ has type $\II_+(j_+,\nu_+)$, and submodule $2$ is of type
$\IF(j_r,\nu_r)$. The unique quotient has type $\IF_+(j_r,-\nu_r)$.
Composition diagram:
\badl{cd3}\xymatrix@R.5cm@C1.1cm{
& \square \ar[r]^{\IF_+(j_r,-\nu_r)} & \square \ar[rd]^{\IF(j_r,\nu_r)}
\\
\{0\} \ar[ru] ^{\II_+(j_+,\nu_+)} \ar[rd]_{\IF(j_r,\nu_r)}
&& &H^{\xi_r,\nu_r}_K \\
& \square \ar[r]_{\IF_+(j_r,-\nu_r)} & \square
\ar[ru]_{\II_+(j_+,\nu_+)} }\eadl

A special case occurs if $\nu_r=j_j-2$. Then the $K$-types in the
submodule in the holomorphic discrete series form a boundary line of
the sector $\sect(j_r)$. See Figure~\ref{fig-ps3a}.
\begin{figure}[ht]
\begin{center}
\grf{4.7}{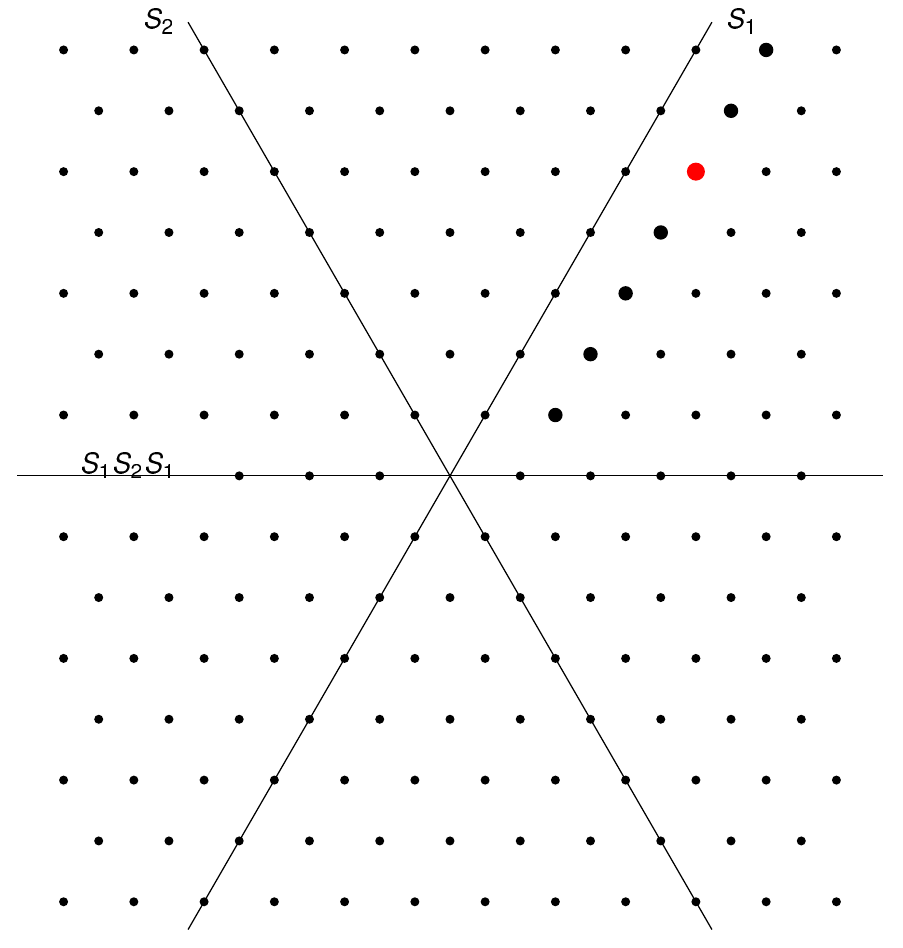} \quad \grf{6.7}{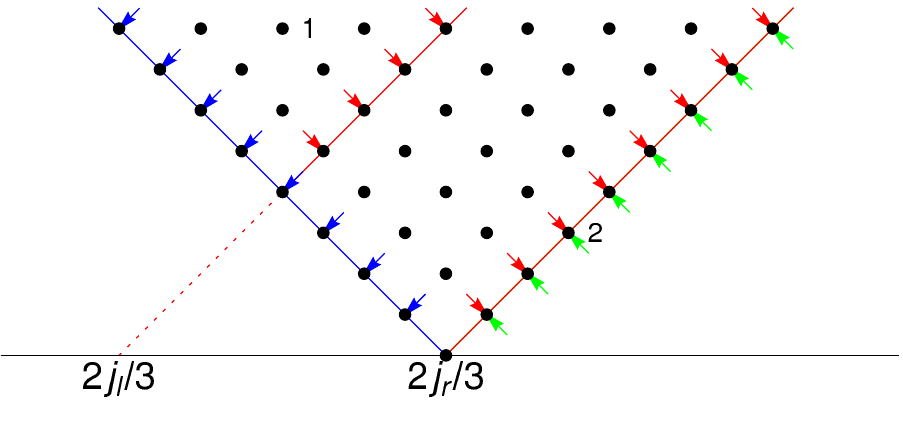}
\end{center}
\caption{The region $3\leq \nu_r= j_r-2$. Illustrated for
$(j_r,\nu_r)=(7,5)$. }\label{fig-ps3a}
\end{figure}

\subsubsection{}\label{sect-ps4}Spectral parameters
$j_r \in 2\ZZ_{\geq 1}$, $\nu_r=0$. See Figure~\ref{fig-ps4}.
\begin{figure}[ht]
\begin{center}
\grf{4.7}{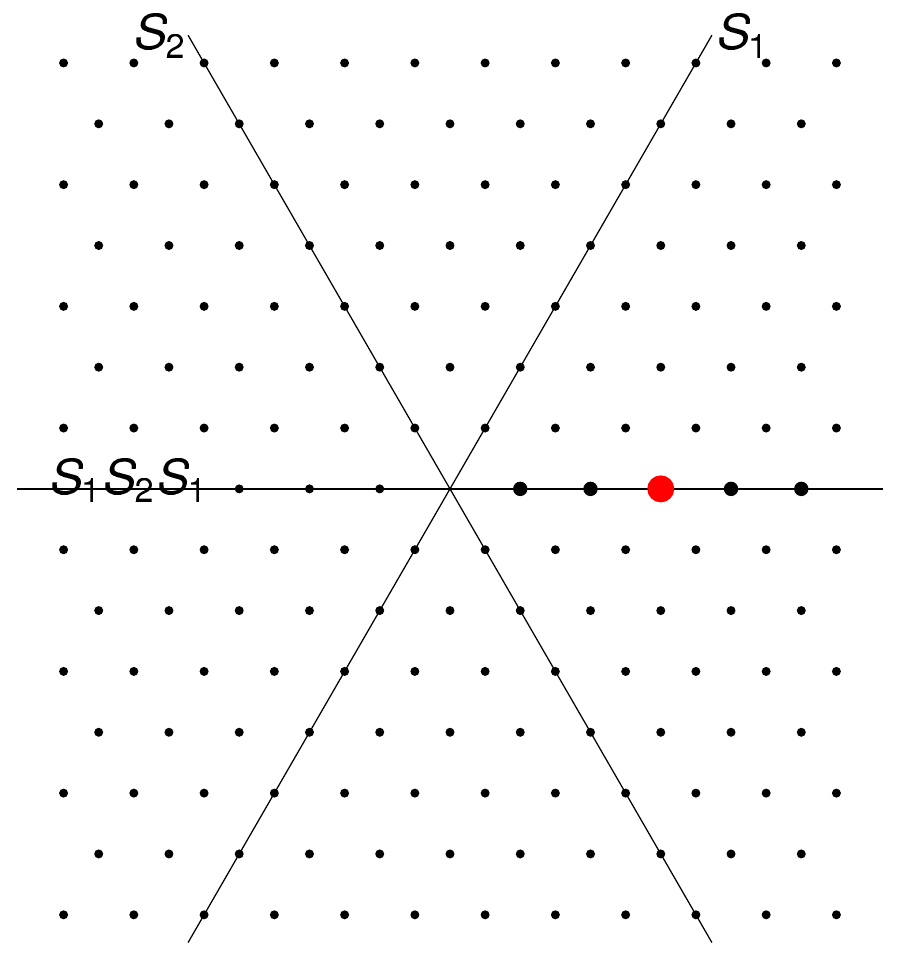} \quad \grf{6.7}{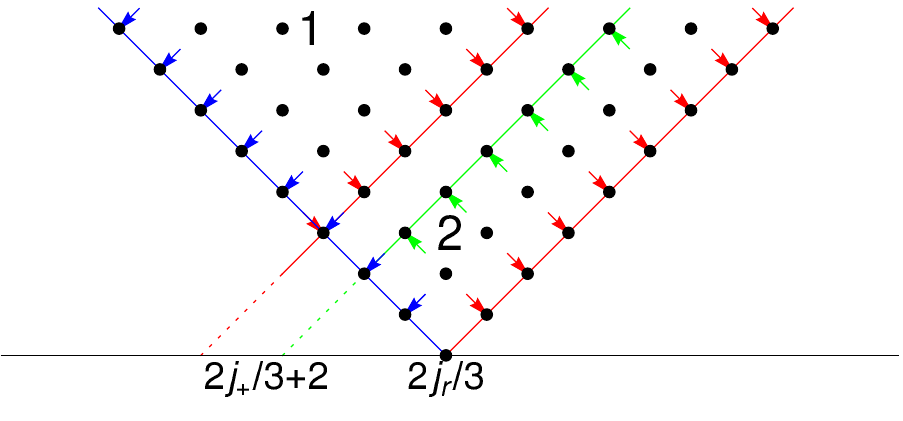}
\end{center}
\caption{The region $j_r\in 2\ZZ_{\geq 2}$, $\nu_r=0$, in
\S\ref{sect-rs5}. Illustrated for $(j_r,\nu_r)=(6,0)$. }\label{fig-ps4}
\end{figure}

In this case $(j_l,\nu_l)=(j_+,\nu_+)$, and in the figure on the right
there are no lattice points in the region between the lines through
$\frac23 j_l$ and through $\frac 23 j_++2$. The principal series module
$H^{\xi_r,0}_K$ is a direct sum of two irreducible submodules, of types
$\II_+(-j_+, j_+)=\II_+(-j_r/2,j_r/2)$ and $\IF(j_r,0)$, respectively.
\be \renewcommand\arraystretch{1.4}
\begin{array}{|cll|}\hline
1& h_0= j_r+j_l=-\ j_+& p_0 = \frac{j_r-j_l}3=\nu_+\\
& A=\infty&B=\infty\\ \hline
2& h_0=2j_r & p_0=0\\
&A=\infty& B= \frac{j_r-j_+}3-1=\frac{j_r}2-1
\\ \hline\end{array}
\ee
Composition diagram:
\badl{cd4} \xymatrix@R.5cm{
&\square \ar[rd]^{\IF(j_r,0)}
\\
\{0\}\ar[ru]^{\II_+(-j_+,j_+)} \ar[rd]^{\IF(j_r,0)}
&& H^{\xi_r,0}_K \\
&\square \ar[ru]_{\II_+(-j_+,j_+)} }\eadl

\subsubsection{}\label{sect-rs5}Spectral parameters $(j_r,-\nu_r)$ with
$1\leq \nu_r\leq j_r-2$. See Figure~\ref{fig-ps5}.
\begin{figure}[ht]
\begin{center}
\grf{4.7}{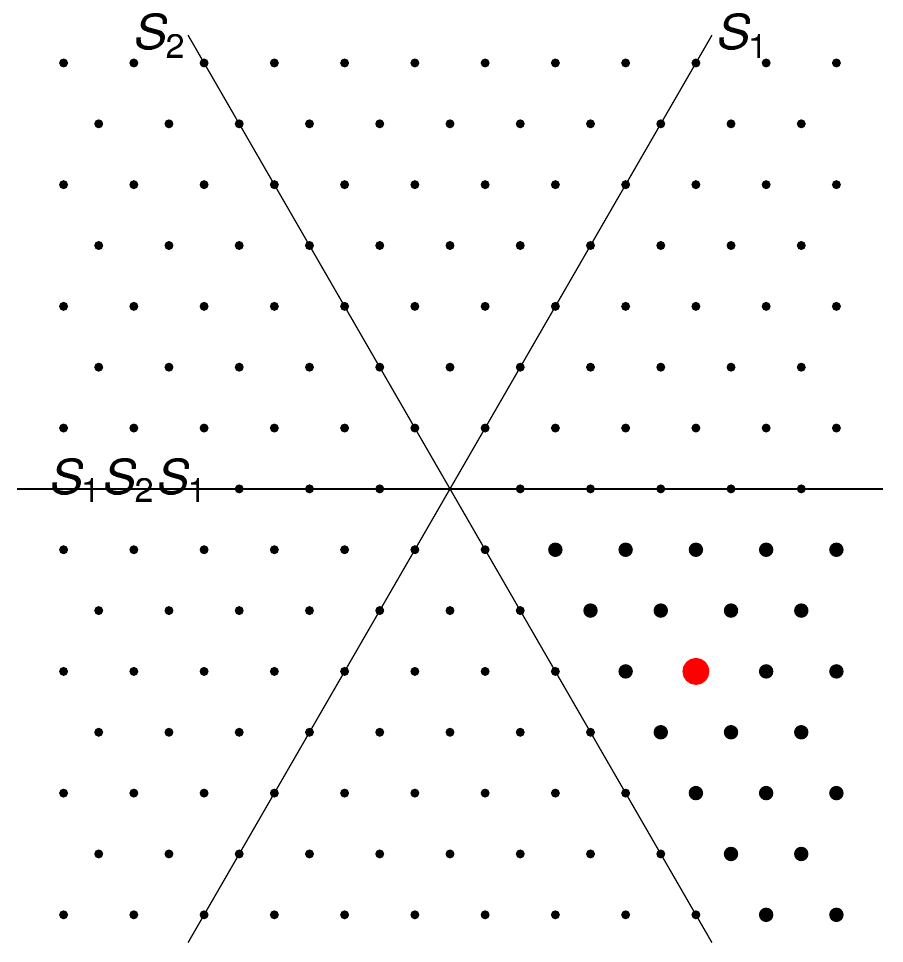} \quad \grf{6.7}{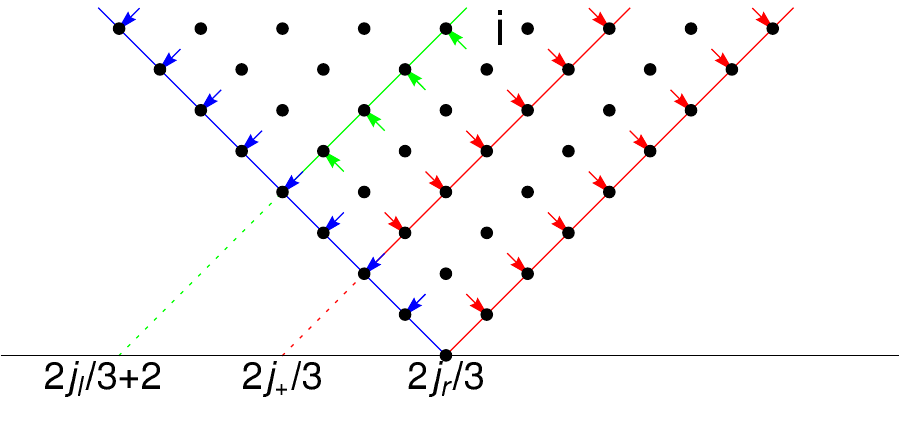}
\end{center}
\caption{The region of $(j_r,-\nu_r)$ for $1\leq \nu_r \leq j_r-2$.
Illustrated for $(j_r,-\nu_r)=(7,-3)$. }\label{fig-ps5}
\end{figure}

There is one irreducible submodule, of type $\IF_+(j_r,-\nu_r)$.
\be \renewcommand\arraystretch{1.4}
\begin{array}{|ll|}\hline
h_0 = j_++j_r =-j_l& p_0 = \frac{j_r-j_+}3 = \nu_l \\
A=\infty& B=\frac{j_+-j_l}3-1 =\nu_r-1
\\ \hline\end{array}
\ee
The two irreducible quotients have types $\II_+(j_+,\nu_+)$ and
$\IF(j_r,\nu_r)$. Composition diagram:\badl{cd5}
\xymatrix@R.5cm@C1.1cm{
&& \square \ar[rd] ^{\IF(j_r,\nu_r)} \\
\{0\} \ar[r]^{\IF_+(j_r,-\nu_r)}
& \square \ar[ru] ^{\II_+(j_+,\nu_+)} \ar[rd]_{\IF(j_r,\nu_r)} &&
H^{\xi_r,-\nu_r}_K\\
&& \square \ar[ru]_{\II_+(j_+,\nu_+)} }\eadl
A special case occurs if $\nu_r=1$. Then $j_l=j_+-3$. See
Figure~\ref{fig-ps5a}.
\begin{figure}[ht]
\begin{center}
\grf{4.7}{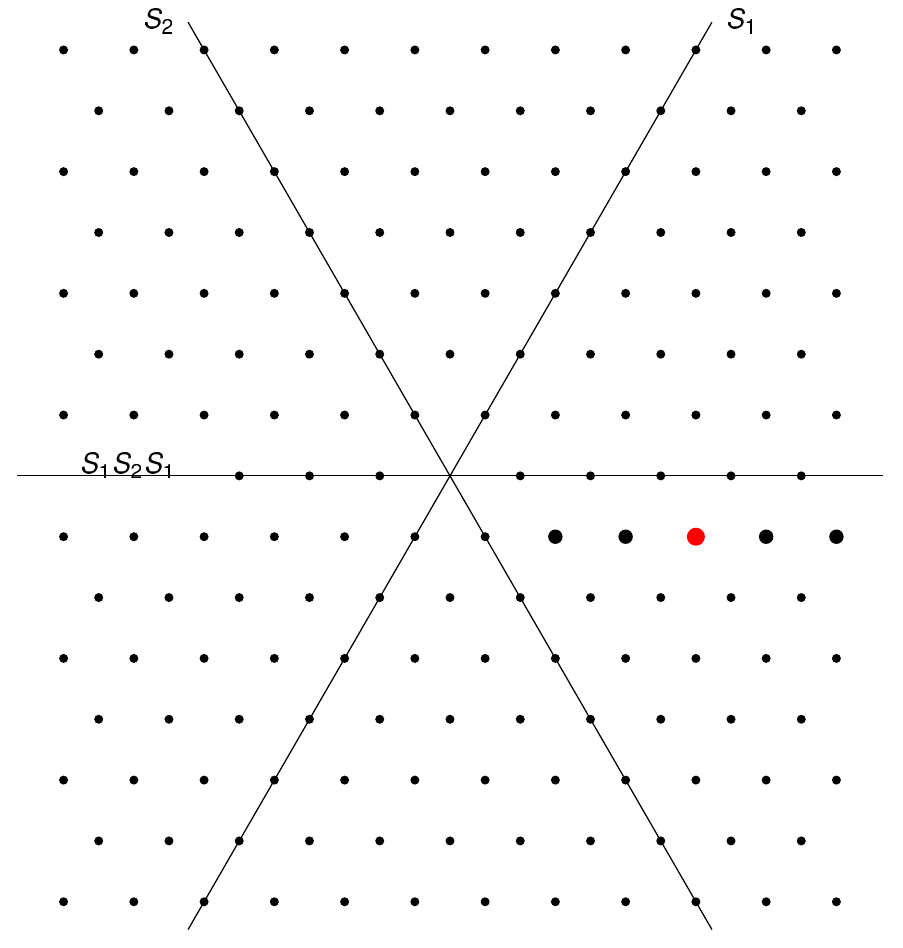} \quad \grf{6.7}{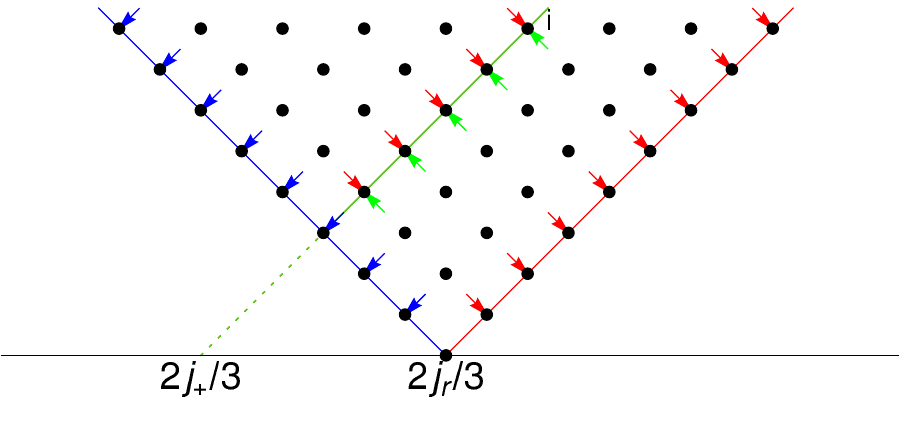}
\end{center}
\caption{The region of $(j_r,-1)$ for $j_r\geq 3$ odd, in
\S\ref{sect-rs6}. Illustrated for $j_r=7$. }\label{fig-ps5a}
\end{figure}

\subsubsection{}\label{sect-rs6}Spectral parameters $(j_r,-j_r)$ with
$j_r\geq 1$. See Figure~\ref{fig-ps6}.
\begin{figure}[ht]
\begin{center}
\grf{4.7}{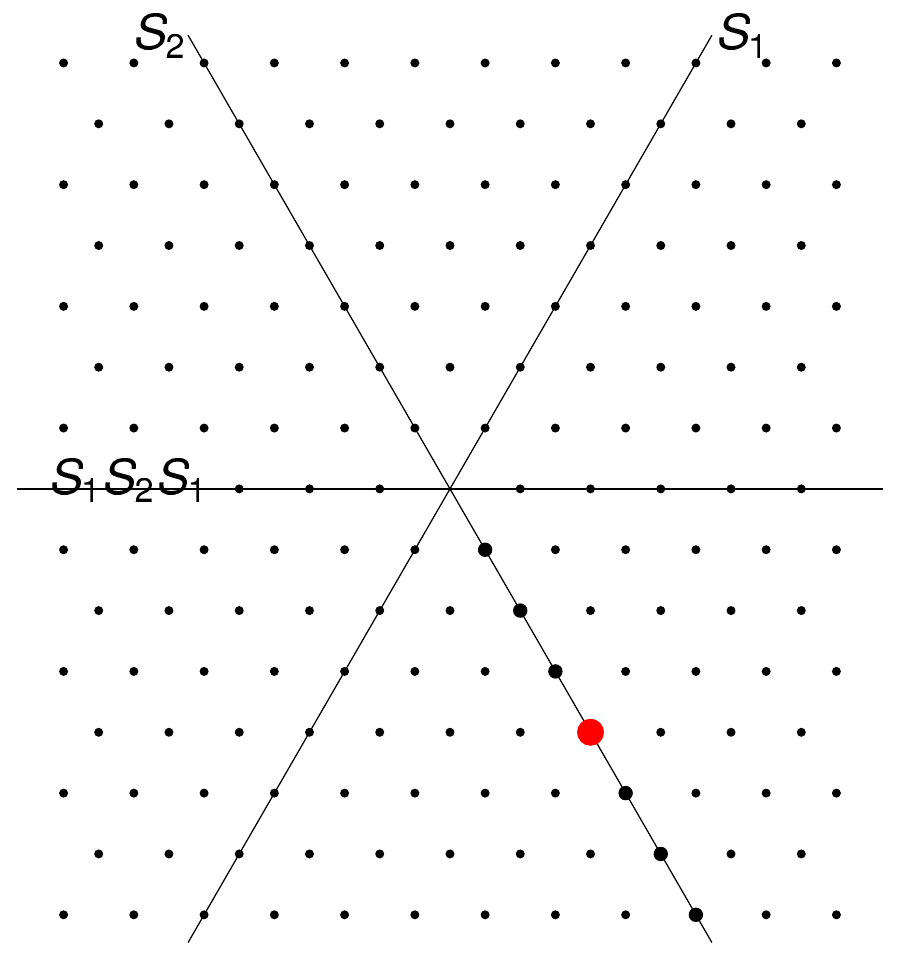} \quad \grf{6.7}{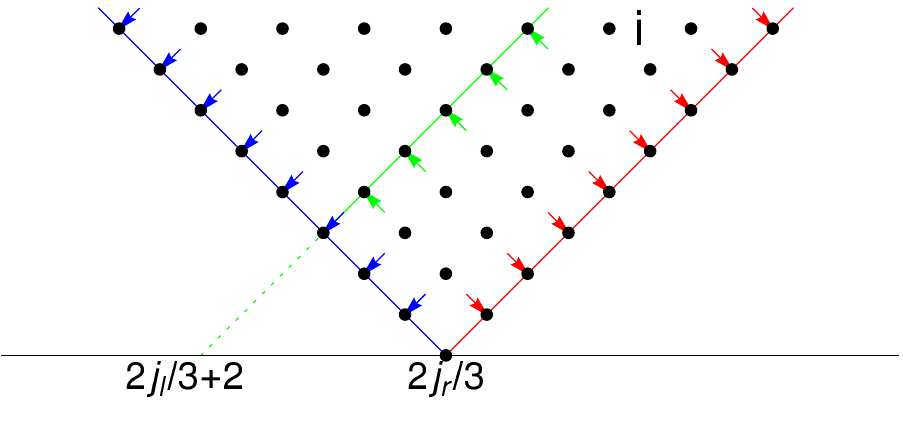}
\end{center}
\caption{The region of $(j_r,-j_r)$ for $j_r\geq 1$, in
\S\ref{sect-rs6}. Illustrated for $j_r=4$. }\label{fig-ps6}
\end{figure}

We have $(j_+,\nu_+) = (j_r,\nu_r)$. There is one irreducible submodule,
which has type $\IF(j_r,-j_r)$.
\be \renewcommand\arraystretch{1.4}
\begin{array}{|ll|}\hline
h_0 =2j_r& p_0 = 0\\
A=\infty& B=\frac{j_r-j_l}3-1 =j_r-1\\
\hline\end{array}
\ee
The irreducible quotient has type $\II_+(j_+,\nu_+)$. Composition
diagram:
\badl{cd6}\xymatrix@C1.2cm{ \{0\}\ar[r]^{\IF(j_r,-j_r)} & \square
\ar[r]^{\II_+(j_+,j_+)} & H^{\xi_r,-j_r}_K }\eadl

We note that in the pictures the modules in the isomorphism class
$\IF(j_r,-j_r)$ look the same as the modules in the class
$\IF(2j_r,0)$. To see that we have different isomorphism classes we
compare the full parameter sets, with $j\in \ZZ_{\geq 0}$.
\be\begin{array}{c|c|ccccc}
j_r &\text{type}&\ld_2 &h_0&p_0&A&B\\
\hline
2j& \IF(2j,0)&\frac 13 (2j)^2-4& 4j&0&\infty&j-1\\
j & \IF(j,-j)& \frac43 j^2-4& 2j &0&\infty& j-1
\end{array}
\ee

\subsubsection{}\label{sect-rs7}Spectral parameters $(j_+,-\nu_+)$ with
$\nu_+\geq |j_+|-2$. See Figure~\ref{fig-ps7}.
\begin{figure}[ht]
\begin{center}
\grf{4.7}{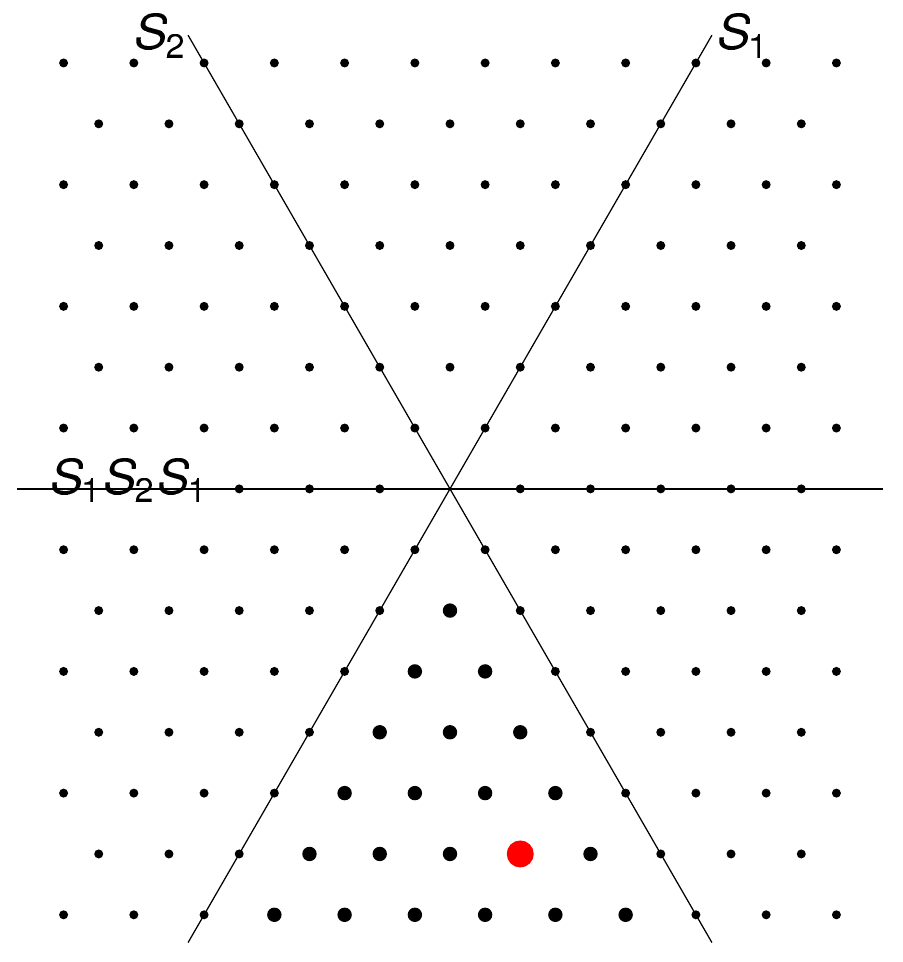} \quad \grf{6.7}{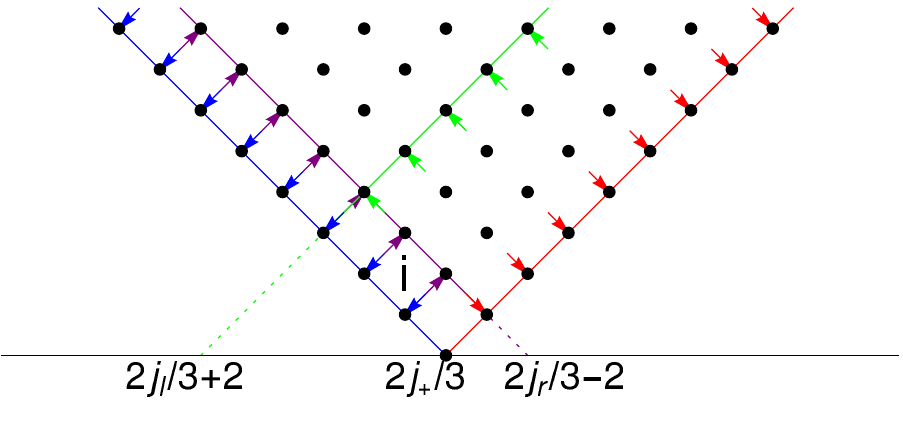}
\end{center}
\caption{The region of $(j_+,-\nu_+)$ for $\nu_+\geq |j_+|+2$, in
\S\ref{sect-rs7}. Illustrated for $j_+=2$, $-\nu_+=-6$.
}\label{fig-ps7}
\end{figure}

There is one irreducible submodule with finite dimension equal
to~$(A+1)(B+1)$.
\be \renewcommand\arraystretch{1.4}
\begin{array}{|ll|}\hline
h_0 = 2j_+ & p_0 = 0 \\
A= \frac{j_r-3-j_+}3 = \nu_l-1& B=\frac{j_+-(j_l+3)}3 \= \nu_r-1 \\
\hline\end{array}
\ee
The irreducible quotient has type $\II_+(j_+,\nu_+)$. Composition
diagram:
\badl{cd7} \xymatrix@R.5cm@C1.2cm{
&&\square \ar[rd]^{\IF_+(j_r,-\nu_r) }
\\
\{0\}\ar[r]^{\FF(j_r,-\nu_r)} & \square \ar[ru]^{\FI_+(j_l,-\nu_l)}
\ar[rd]_{\IF_+(j_r,-\nu_r) }
&& \square \ar[r]^{\II_+(j_+,\nu_+)}
& H^{\xi_+,-\nu_+}_K
\\
&&\square \ar[ru]_{\FI_+(j_l,-\nu_l)} }\eadl

\subsubsection{}\label{sect-rs8}Spectral parameters $(j_l,j_l)$ with
$j_l\leq -1$. See Figure~\ref{fig-ps8}. Compare
Subsection~\ref{sect-rs6}.
\begin{figure}[ht]
\begin{center}
\grf{4.7}{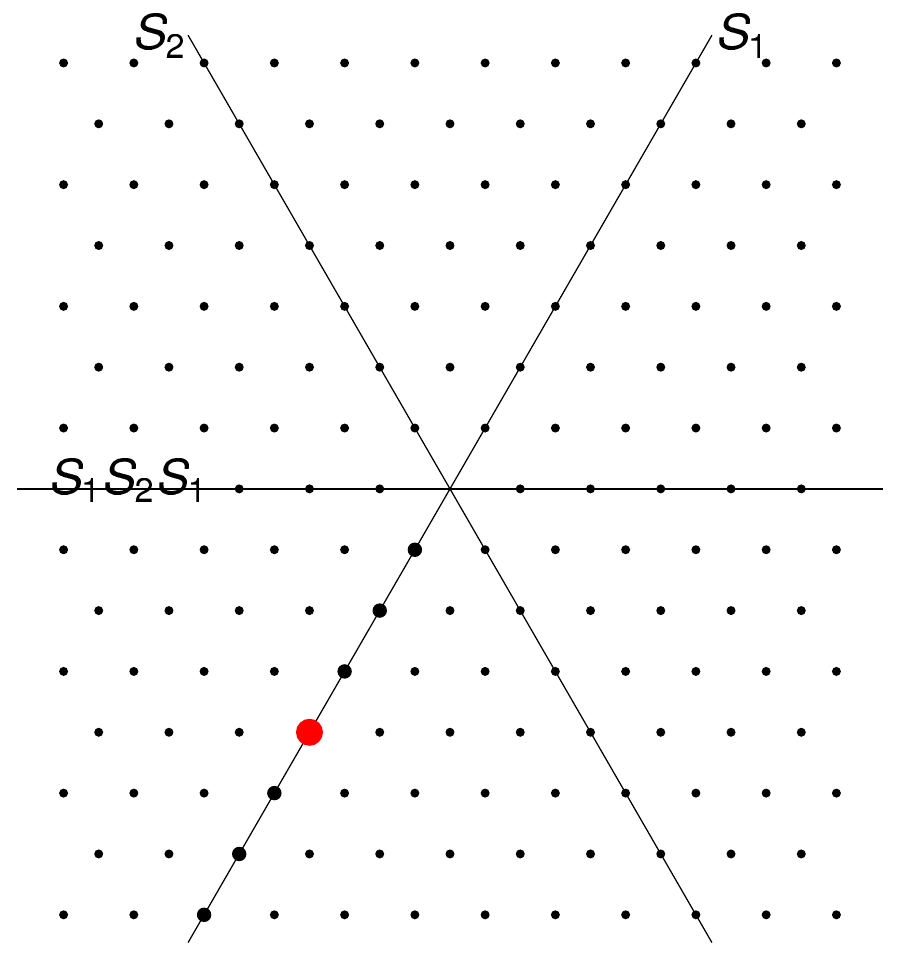} \quad \grf{6.7}{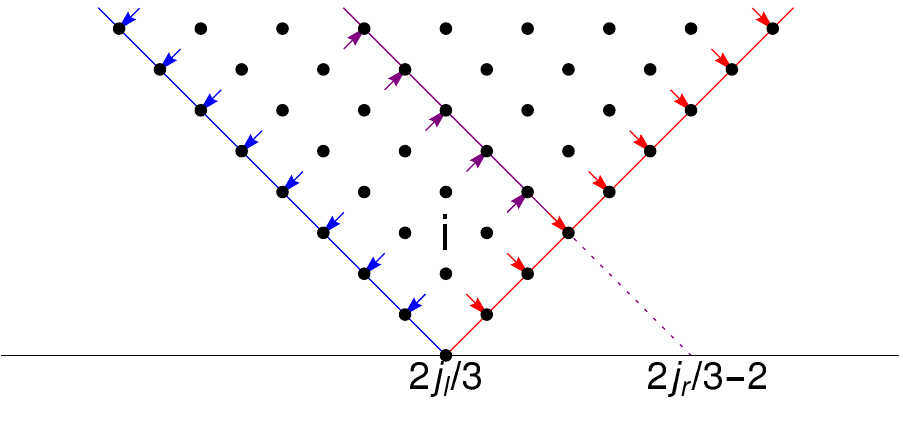}
\end{center}
\caption{The region of $(j_l,j_l)$ for $j_l\leq -1$, in
\S\ref{sect-rs8}. Illustrated for $j_l=-4$. }\label{fig-ps8}
\end{figure}

There is one irreducible submodule of type $\FI(j_l,j_l)$.
\be \renewcommand\arraystretch{1.4}
\begin{array}{|llll|}\hline
h_0= 2j_l & p_0=0 & A= \frac{j_r-3-j_l}3 = |j_l|-1& B=\infty
\\\hline
\end{array}
\ee
The irreducible quotient has type $\II_+(j_+,\nu_+)$. Composition
diagram:
\badl{cd8}\xymatrix@C1.2cm{ \{0\} \ar[r]^{\FI(j_l,j_l)} & \square
\ar[r]^{\II_+(j_+,-j_+)} & H^{\xi_l,j_l}_K }\eadl

\subsubsection{}\label{sect-rs9}Spectral parameters $(j_l,-\nu_l) $ with
$1\leq \nu_r \leq -j_l-2$. See Figure~\ref{fig-ps9}. Compare
\S\ref{sect-rs5}.
\begin{figure}[ht]
\begin{center}
\grf{4.7}{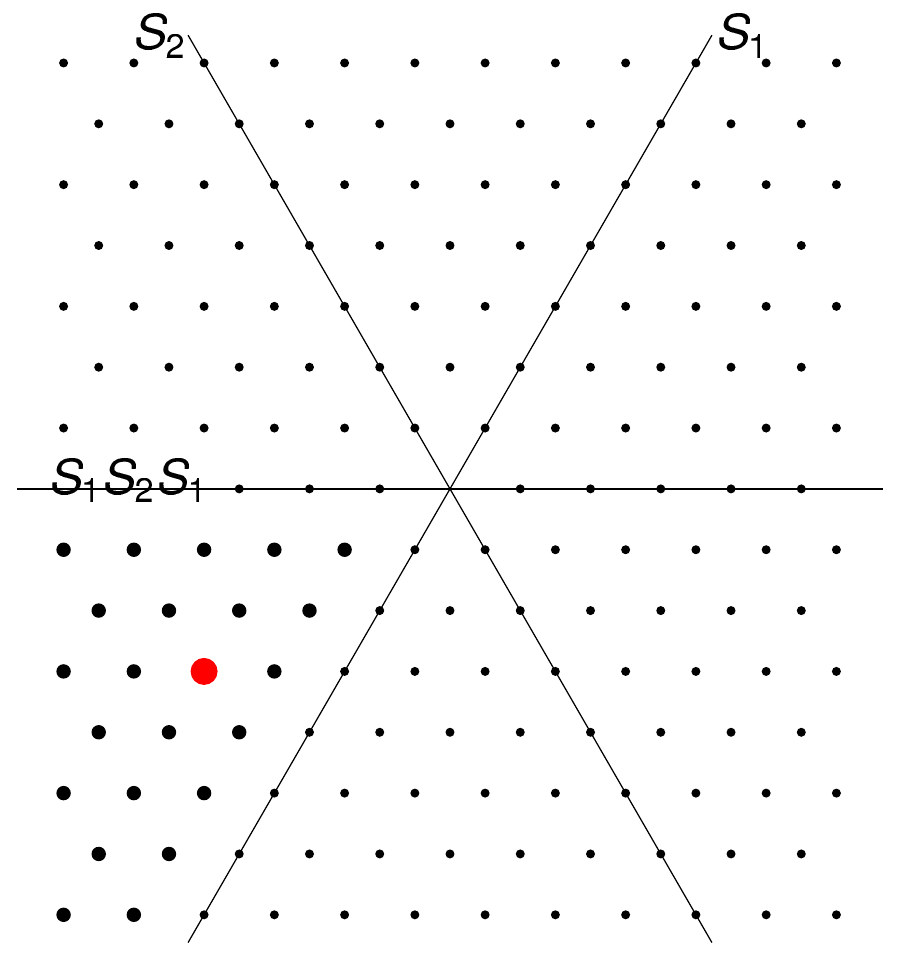} \quad \grf{6.7}{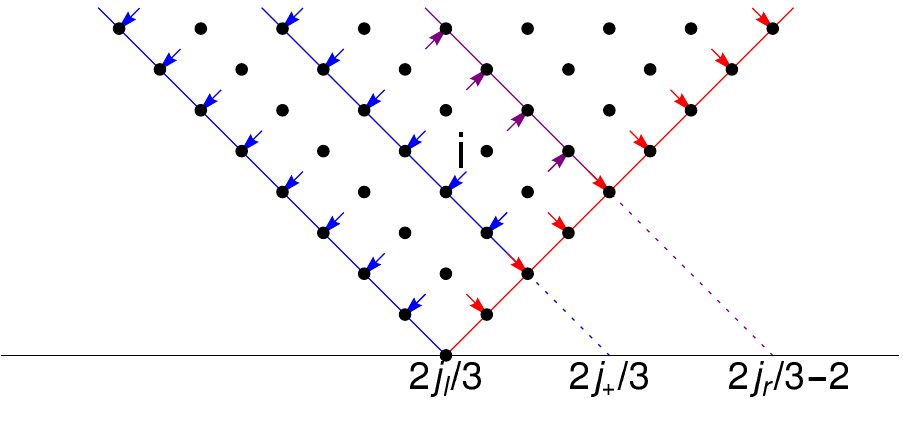}
\end{center}
\caption{The region $j_l+2 \leq -\nu_l \leq -1$, in \S\ref{sect-rs9}.
Illustrated for $j_l=-7$, $\nu_l=3$. }\label{fig-ps9}
\end{figure}

There is one irreducible submodule, of type $\FI_+(j_l,-\nu_l)$.
\be \renewcommand\arraystretch{1.4}
\begin{array}{|ll|}\hline
h_0= j_l+j_+ = -j_r & p_0 = \frac{j_+-j_l}3 = \nu_r \\
A= \frac{j_r-j_+}3-1 = \nu_l-1
&B=\infty\\ \hline
\end{array}
\ee
The irreducible quotients have types $\II_+(j_+,\nu_+)$ and
$\FI(j_l,\nu_l)$. Composition diagram:\badl{cd9}
\xymatrix@R.5cm@C1.1cm{
&&\square \ar[rd]^{\FI(j_l,\nu_l)}
\\
\{0\}\ar[r]^{\FI_+(j_l,-\nu_l)} & \square \ar[ru]^{\II_+(j_+,\nu_+)}
\ar[rd]_{\FI(j_l,\nu_l) }
&& H^{j_l,-\nu_l}_K
\\
&&\square \ar[ru]_{\II_+(j_+,\nu_+)} }\eadl

\subsubsection{}\label{sect-rs10} Spectral parameters
$j_l\in 2\ZZ_{\leq -1}$, $\nu_l=0$. See Figure~\ref{fig-ps10}. Compare
\S\ref{sect-ps4}.
\begin{figure}[ht]
\begin{center}
\grf{4.7}{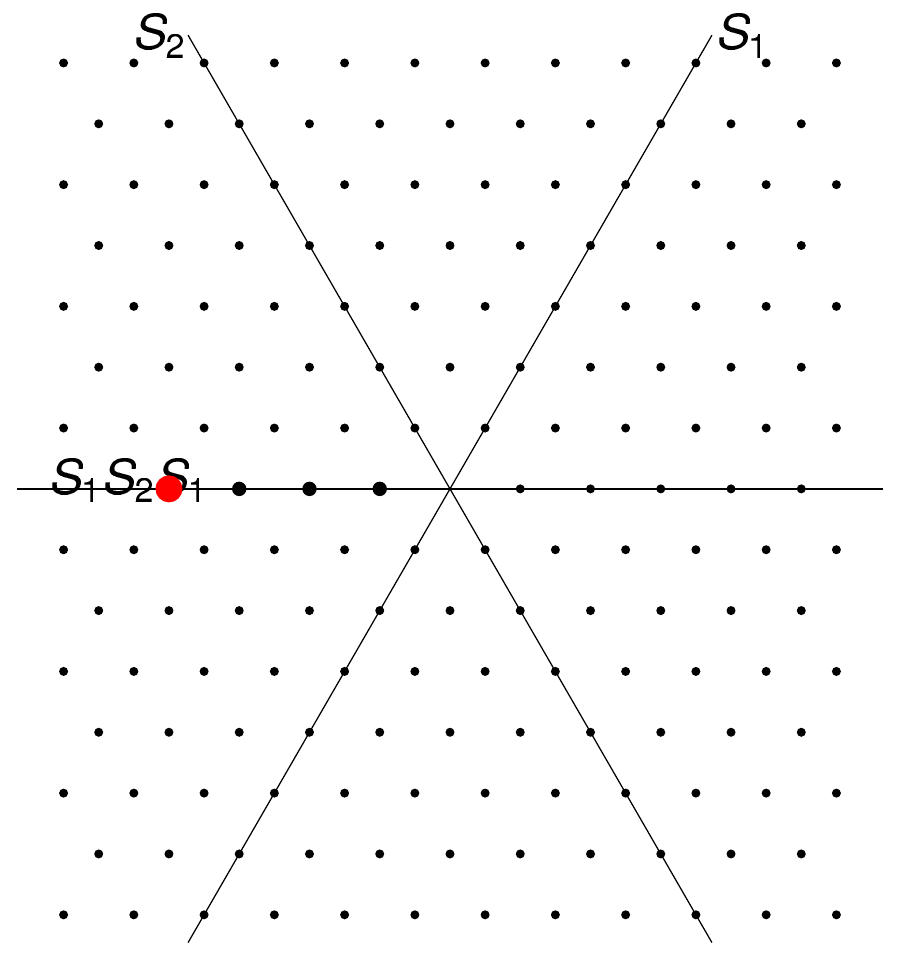} \quad \grf{6.7}{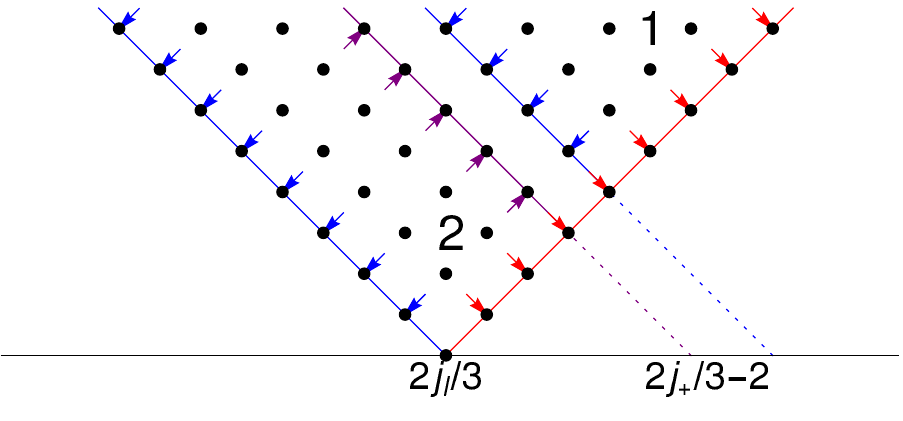}
\end{center}
\caption{The region $j_l\leq -2$ even, $\nu_l=0$, in \S\ref{sect-rs10}.
Illustrated for $j_l=-8$, $\nu_l=0$. }\label{fig-ps10}
\end{figure}

In this case $(j_r,\nu_r) = (j_+,\nu_+)$. The module $H^{\xi_l,0}_K$ is
a direct sum of two irreducible submodules, of types $\II_+(j_+,\nu_+)$
and $\FI(j_l,0)$.
\be \renewcommand\arraystretch{1.4}
\begin{array}{|cll|}\hline
1&h_0=2j_l & p_0=0 \\
&A=\frac{j_+-3-j_l}3=\nu_r-1& B=\infty\\ \hline
2&h=j_r+j_l=- j_+ & p_0=\frac{j_r-j_l}3=\nu_+\\
&A=\infty & B=\infty
\\\hline
\end{array}
\ee
Composition diagram:
\badl{cd10} \xymatrix@R.5cm{
&\square\ar[rd]^{\FI(j_l,0)}\\
\{0\} \ar[ru]^{\II_+(j_+,j_+)} \ar[rd]_{\FI(j_l,0)} && H^{\xi_l,0}_K\\
&\square \ar[ru]_{\II_+(j_+,j_+)}
\\
}\eadl

\subsubsection{}\label{sect-rs11}Spectral parameters
$1\leq \nu_l \leq -j_l-2$. See Figure~\ref{fig-ps11}. Compare
\S\ref{sect-ps3}.
\begin{figure}[ht]
\begin{center}
\grf{4.7}{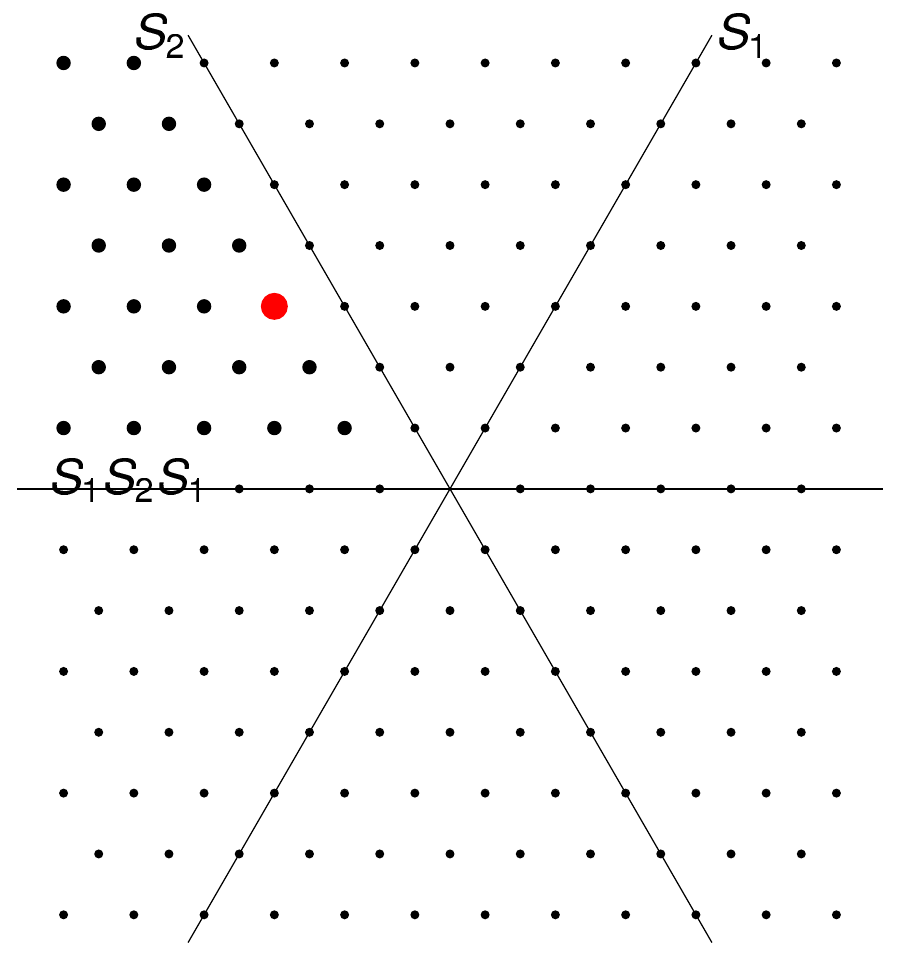} \quad \grf{6.7}{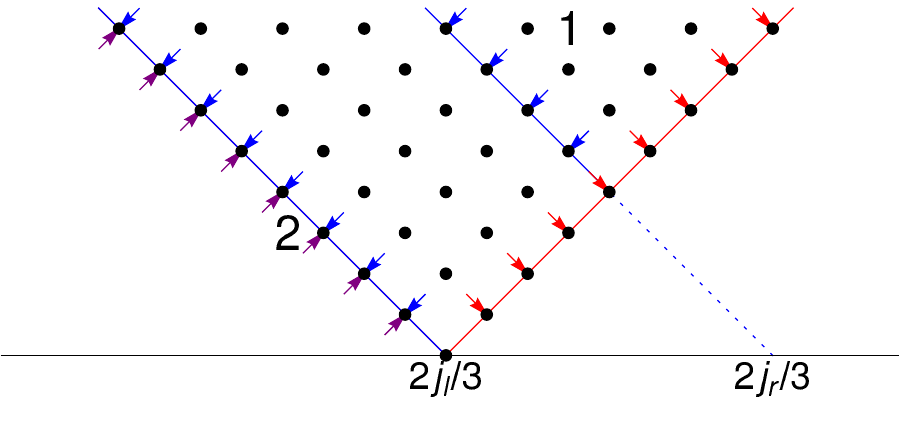}
\end{center}
\caption{The region $1\leq \nu_l \leq -j_l-2$, in \S\ref{sect-rs11}.
Illustrated for $(j_l,\nu_l)=(-5,3)$. }\label{fig-ps11}
\end{figure}

There are two irreducible submodules, of types $\II_+(j_+,\nu_+)$ and
$\FI(j_l,\nu_l)$.
\be \renewcommand\arraystretch{1.4}
\begin{array}{|cll|}\hline
1&h_0=j_r+j_l = - j_+& p_0 \= \frac{j_r-j_l}3=\nu_+\\
&A=\infty&B=\infty\\ \hline
2&h_0=2j_l&p_0=0\\
&A=\frac{j_p-3-j_l}3 = \nu_r-1& B=\infty \\ \hline
\end{array}
\ee
The type of the irreducible quotient is $\FI_+(j_l,-\nu_l)$. Composition
diagram:
\badl{cd11}\xymatrix@R.5cm@C1.1cm{
& \square \ar[r]^{\FI_+(j_l,-\nu_l)}
&\square \ar[rd]^{\FI(j_l,\nu_l)}
\\
\{0\} \ar[ru]^{\II_+(j_+,\nu_+)} \ar[rd]_{\FI(j_l,\nu_l)}
&&& H^{\xi_l,\nu_l}_K
\\
&\square \ar[r]^{\FI_+(j_l,-\nu_l)}
&\square \ar[ru]_{\II_+(j_+,\nu_+)} }\eadl

\subsubsection{}\label{sect-rs12}Spectral parameters $(j_l,\nu_l)$ with
$-j_l = \nu_l\geq 1$. See Figure~\ref{fig-ps12}. Compare
\S\ref{sect-ps2}.
\begin{figure}[ht]
\begin{center}
\grf{4.7}{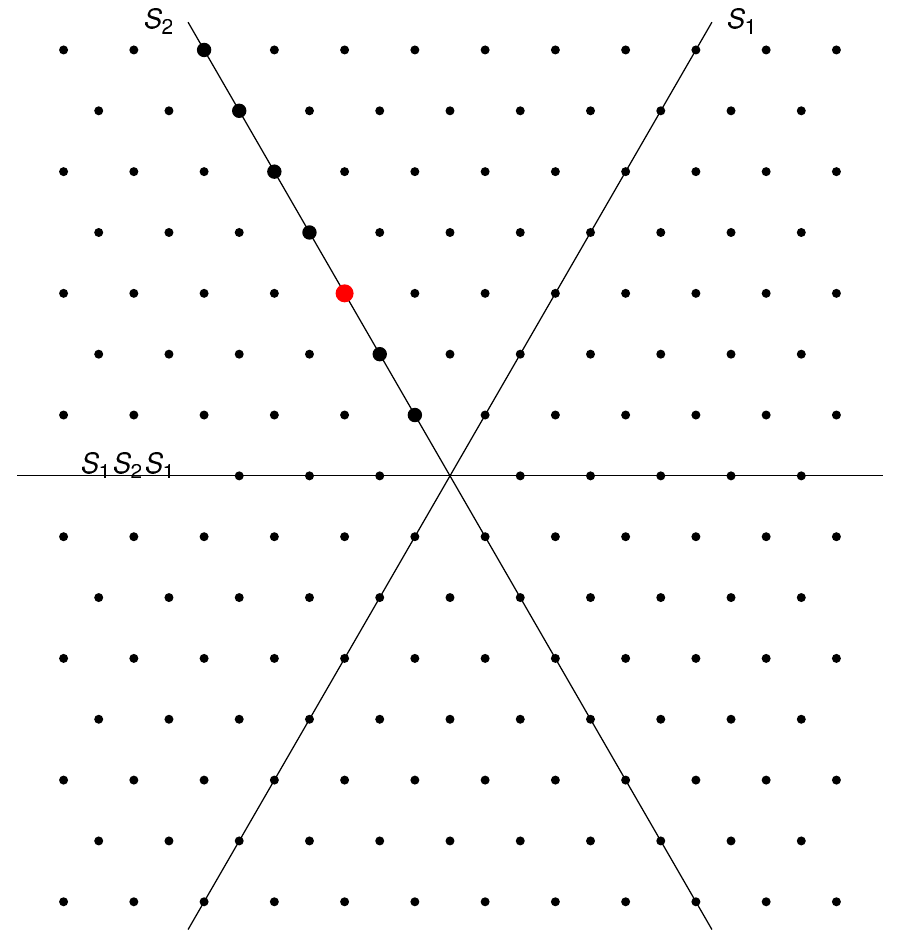} \quad \grf{6.7}{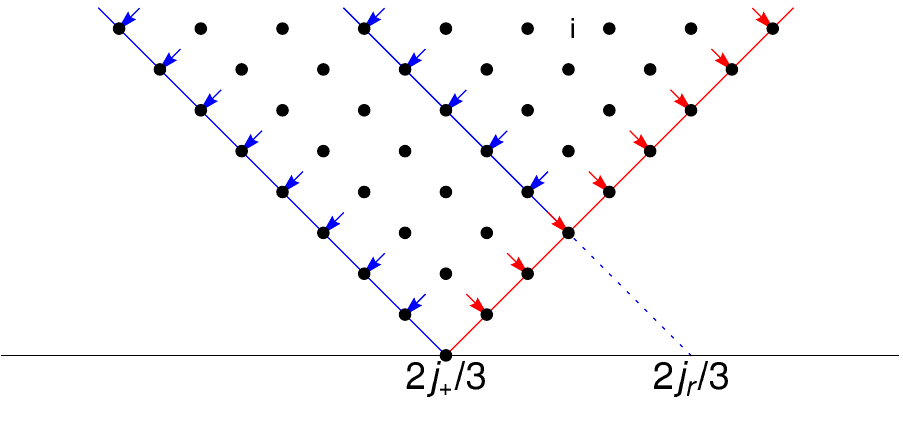}
\end{center}
\caption{The region $-j_+=\nu_+\geq 1$, in \S\ref{sect-rs12}.
Illustrated for $(j_+,\nu_+)=(-3,3)$. }\label{fig-ps12}
\end{figure}

In this case $(j_l,\nu_l) = (j_+,\nu_+)$. There is one irreducible
submodule, of type $\II_+(-\nu_+,\nu_+)  =\II_+(j_+,-j_+)$.
\be \renewcommand\arraystretch{1.4}
\begin{array}{|llll|}\hline
h_0=j_++j_r = - j_+& p_0 = \frac{j_r-j_+}3=\nu_+& A=\infty & B=\infty
\\ \hline
\end{array}
\ee
The irreducible quotient has type $\FI(j_l,-\nu_l) =\FI(j_l,j_l)$.
  Composition diagram:
\badl{cd12}\xymatrix@C1.2cm{ \{0\} \ar[r]^{\II_+(j_+,|j_+|)} & \square
\ar[r]^{\FI(j_l,j_l)}
& H^{\xi_+,|j_+|}_K }\eadl

\rmrk{Observation}We note that in in the cases that if $(j,\nu)$ is in
the interior of a Weyl chamber, which means that $(j,\nu)$ determines a
regular character of $ZU(\glie)$), then $H^{\xi_j,\nu}_K$ has for
$\nu\geq 1$ one or two irreducible submodules and one irreducible
quotient, and for $\nu\leq -1$ one irreducible submodule and one or two
irreducible quotients. If $(j,\nu)$ is on the wall between Weyl
chambers, then $H^{\xi_j,\nu}_K$ is the direct sum of two irreducible
submodules if $\nu=0$, and otherwise has an an irreducible submodule
$V$ such that $H^{\xi_j,\nu}_K/ V$ is the irreducible quotient. See
Collingwood \cite[p~46]{Coll85}, where these observations for a regular
character of $ZU(\glie)$ are placed in a wider context.

\subsubsection{Unique embedding} In Table~\ref{tab-isot-ps} we summarize
the irreducible iso\-mor\-phism types occurring in reducible principal
series modules.
\begin{table}[htp]\renewcommand\arraystretch{1.2}{\small
\[\begin{array}{|c|ccccc|}
\hline
\text{type}&&h_0&p_0&A&B \\ \hline
\II_+(j_+,\nu_+)&\nu_+\geq |j_+|+2&-j_+ & \nu_+ & \infty &\infty \\
&\multicolumn{5}{|l|}{\text{ it occurs in }H^{\xi_+,\nu_+}_K, \text{
Fig.~\ref{fig-ps1}}} \\
&\multicolumn{5}{|l|}{\text{ it occurs in }H^{\xi_r,\nu_r}_K, \text{
Fig.~\ref{fig-ps3}}} \\
&\multicolumn{5}{|l|}{\text{ it occurs in }H^{\xi_l,\nu_l}_K, \text{
Fig.~\ref{fig-ps11}}}\\\hline
\II_+(j_+,j_+) & j_+= \nu_+\in \ZZ_{\geq 1}& -j_+ & \nu_+ &
\infty&\infty \\
&\multicolumn{5}{|l|}{\text{ it occurs in }H^{\xi_+,j_+}_K, \text{
Fig.~\ref{fig-ps2}}} \\
&\multicolumn{5}{|l|}{\text{ it occurs in }H^{\xi_l,-0}_K, \text{
Fig.~\ref{fig-ps10}}}\\ \hline
\II_+(j_+,-j_+) & j_+=-\nu_+\in \ZZ_{\geq 1} & h_0= -j_+ & \nu_+&
\infty&\infty\\
&\multicolumn{5}{|l|}{\text{ it occurs in }H^{\xi_r,0}_K, \; j_r = -2
j_+ \text{ Fig.~\ref{fig-ps4}}} \\
&\multicolumn{5}{|l|}{\text{ it occurs in }H^{\xi_l,-j_l}_K, \; j_r =
-2 j_+ \text{ Fig.~\ref{fig-ps12}}}
\\\hline
\IF(j_r,\nu_r) & 1\leq \nu_r \leq j_r-2& 2j_r & 0 & \infty & \nu_l-1\\
&\multicolumn{5}{|l|}{\text{ it occurs in }H^{\xi_r,\nu_r}_K, \text{
Fig.~\ref{fig-ps3}}} \\ \hline
\IF(j_r,0) & j_r\in 2\ZZ_{\geq 1}& 2j_r & 0 &\infty & \nu_l-1 \\
\mu_2=\frac13j_r^2-4
&\multicolumn{5}{|l|}{\text{ it occurs in }H^{\xi_r,0}_K, \text{
Fig.~\ref{fig-ps4}}} \\ \hline
\IF_+( j_r,-\nu_r) & 1\leq \nu_r\leq j_r-2 & -j_l & \nu_l & \infty &
\nu_r-1\\
&\multicolumn{5}{|l|}{\text{ it occurs in }H^{\xi_r,-\nu_r}_K, \text{
Fig.~\ref{fig-ps5}}} \\ \hline
\IF(j_r,-j_r)& j_r\in \ZZ_{\geq 1} & 2j_r & 0 & \infty & j_r-1\\
\mu_2=\frac 43 j_r^2-4
&\multicolumn{5}{|l|}{\text{ it occurs in }H^{\xi_r,-\nu_r}_K, \text{
Fig.~\ref{fig-ps6}}} \\ \hline
\FI(j_l,\nu_l) & 1\leq \nu_l \leq -j_l-2 & 2j_l & 0 &\nu_r-1& \infty
\\
&\multicolumn{5}{|l|}{\text{ it occurs in }H^{\xi_l,\nu_l}_K, \text{
Fig.~\ref{fig-ps11}}}
\\ \hline
\FI(j_l,0) &j_l = -2 j_r=-j_+\in 2\Z_{\leq -1}& 2j_l& 0& \nu_r-1 &
\infty\\
\mu_2=\frac13j_l^2-4
&\multicolumn{5}{|l|}{\text{ it occurs in }H^{\xi_l,-0}_K, \text{
Fig.~\ref{fig-ps10}}}
\\ \hline
\FI_+(j_l,-\nu_l)&j_l+2 \leq -\nu_l -1& -j_r & \nu_r & \nu_l-1&\infty\\
&\multicolumn{5}{|l|}{\text{ it occurs in }H^{\xi_l,-\nu_l}_K, \text{
Fig.~\ref{fig-ps9}}}
\\ \hline
\FI(j_l,j_l)&j_l \in \ZZ_{\leq -1}& 2j_l & 0 & |j_l|-1&\infty\\
\mu_2=\frac43 j_l^2-4
&\multicolumn{5}{|l|}{\text{ it occurs in }H^{\xi_l,-j_l}_K, \text{
Fig.~\ref{fig-ps8}}}
\\ \hline
\FF(j_+,-\nu_+)& \nu_+ \geq |j_+|+2& 2j_+& 0 & \frac{\nu_+-j_+}2-1 &
\frac{j_++\nu_+}2-1 \\
&\multicolumn{5}{|l|}{\text{ it occurs in }H^{\xi_+,-\nu_+}_K, \text{
Fig.~\ref{fig-ps7}}}
\\ \hline
\end{array}
\]}
\caption{Isomorphism types under integral parametrization, and their
embeddings in principal series representations.} \label{tab-isot-ps}
\end{table}

We see that almost all types of irreducible $(\glie,K)$-modules under
integral para\-metrization occur as a submodule in only one principal
series module $H^{\xi,\nu}_K$. \il{ueb}{unique embedding in principal
series} Only for the large discrete series type there are several
embeddings: three for large discrete series, and two for limits of
large discrete series. See Collingwood, \cite{Coll84}, p~115-119 in
\cite[\S5.3]{Coll85}, for the unique embedding property in a wider
context. See also Figure~\ref{fig-isot}.

For the spectral parameters in the notation $\II_+(j_+,\nu_+)$ we have
chosen to use the dominant Weyl chamber. For all other isomorphism
classes in the list in \S\ref{sect-list-irr} the spectral parameters
$(j,\nu)$ is uniquely determined by the unique principal series module
in which it occurs as a submodule. We note that the Weyl orbit
$(j_l,\pm\nu_l)$, $(j_+,\pm \nu_+)$, $(j_r,\pm \nu_r)$ corresponds to a
unique character of the ring $ZU(\glie)$. The choice of an element of
this orbit determines a character of $AM$, and hence of $NAM$. Induced
up to $G$, this provides a specific principal series module.

\begin{figure}[ht]
\begin{center}\grf{10}{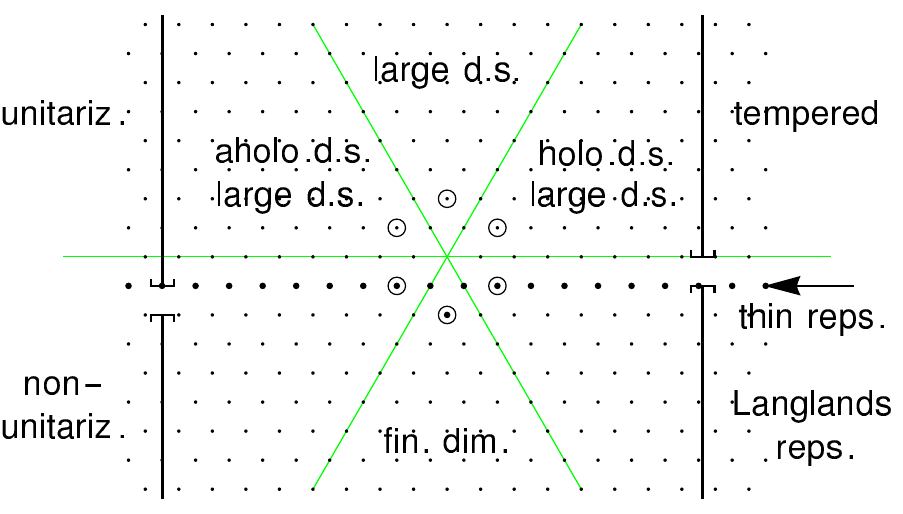}\end{center}
\caption{ Lattice point $(j,\nu) \in L$ and embeddings of isomorphism
types in $H^{\xi_j,\nu}_K$.
\\
The limits of large discrete series are on the walls between the large
discrete series and the holomorphic and antiholomorphic discrete
series. The horizontal walls carry the limits of antiholomorphic and
holomorphic discrete series.\\
The thin representations are the Langlands representations that are
unitarizable (with $\nu=-1$). The point $(j,\nu)=(0,-2)$ corresponds to
the trivial representation. The unitarizable modules occur for
$\nu \geq -1$.\\
The circled dots form the root system. This Weyl orbit corresponds to
the character of $ZU(\glie)$ represented by $\rho=(0,2)$, which is half
the sum of the positive roots.} \label{fig-isot}
\end{figure}
\subsection{Characterization by sets of K-types} Let us take
$j_l<j_+<j_r$. The corresponding spectral parameters $(j_l,\pm \nu_l)$,
$(j_+,\pm \nu_+)$, and $(j_r,\pm\nu_r)$ form one Weyl orbit,
determining a character $\ps$ of $ZU(\glie)$. Each of the six
corresponding principal series modules contains one or two irreducible
$(\glie,K)$-modules. We consider the sets of $K$-types occurring in
each of these irreducible modules. These $K$-types correspond to points
in the union
\[ \sect(j_l)\cup \sect(j_+)\cup\sect(j_r)\,.\]
In Figure~\ref{fig-ps7} we see that the $K$-types $\tau^h_p$ in a
representation in the class $\FF(j_+,-\nu_+)$ satisfy
\[ \frac{2j_+}3 \leq p+\frac h3 < \frac{2j_r}3\,,\qquad \frac{2j_l}3 <
p-\frac h3 \leq \frac{2j_+}3\,.\]
This corresponds to the rectangular region near the base point
$\bigl( \frac{2j_+}3, 0\bigr)$, in Figure~\ref{fig-Ktp}.
\begin{figure}[ht]
\begin{center}\grf{8.5}{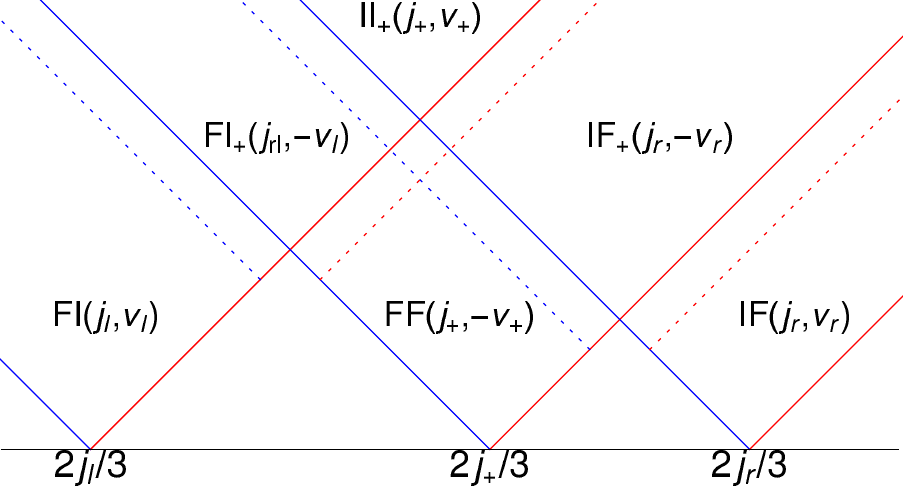}\end{center}
\caption{$K$-types of irreducible representations with character of
$ZU(\glie)$ determined by $(j_+,\nu_+)$ in the interior of the positive
Weyl chamber. }\label{fig-Ktp}
\end{figure}

Going through the other irreducible modules in Figures \ref{fig-ps1},
\ref{fig-ps3}, \ref{fig-ps5}, \ref{fig-ps9}, and~\ref{fig-ps11}, we
find that the $K$-types of the representations are in disjoint regions
in Figure~\ref{fig-Ktp}. The type $\II_+(j_+,\nu_+)$ (large discrete
series) occurs three times, but, of course, always with the same set of
$K$-types.

Figure~\ref{fig-Ktp} illustrates the following result.
\begin{prop}\label{prop-Ktp} Let $V_1$ and $V_2$ be irreducible
$(\glie,K)$-modules in which $ZU(\glie)$ acts by the same character but
which are not isomorphic. Then the set of $K$-types occurring in $V_1$
is disjoint from the set of $K$-types occurring in~ $V_2$.
\end{prop}
\begin{proof}Let us consider the character $\ps$ of $ZU(\glie)$ that is
common to $V_1$ and $V_2$. It is represented by a set of spectral
parameters $(j_0,\nu_0)\in \ZZ \times \left(i\RR \cup \RR\right)$ that
forms one orbit of the Weyl group. Table~\ref{tab-parms},
p~\pageref{tab-parms}, distinguishes several types of Weyl orbits.
Under simple parametrization, the Weyl orbit consists of $(j_0,\nu_0)$
and $(j_0,-\nu_0)$ with $3\nu_0\not\equiv j_0 \bmod 2$, and of the
element $(0,0)$. The $K$-types in the isomorphic $(\glie,K)$-modules
$H^{\xi_{j_0},\nu_0}_K$ and $H^{\xi_{j_0},\nu_0}_K$ correspond to the
points in the sector $\sect(j_0)$, in Definition~\ref{def-sect}.

Under generic multiple parametrization, the Weyl orbit consists of six
elements, occurring in pairs $(j,\nu)$, $(j,-\nu)$. The different pairs
correspond to non-isomorphic modules, and have disjoint sets of
$K$-types. See Figure~\ref{fig-sectors}.

There remains the integral parametrization, with spectral parameters
$(j,\nu_0)\in \ZZ^2$, $j_0\equiv\nu_0\bmod 2$ and $(j_0,\nu_0)\neq
(0,0)$. The Weyl orbits have six or three elements. In constructing
Figure~\ref{fig-Ktp} we have checked that non-isomorphic modules have
disjoint sets of $K$-types for orbits with six elements. We leave to
the reader the analogous checks for the orbits with three elements,
corresponding to $(j_+,\nu_+)$ on one of the boundary lines of the
positive Weyl chamber $L^+$.

By the theorem of Casselman and Mili\v ci\'c quoted in
\S\ref{sect-iclirr}, we conclude that we have taken care of all
isomorphism classes of irreducible $(\glie,K)$-modules.
\end{proof}

\subsection{Logarithmic submodules}\label{sect-log}The module
$\Ffu^{\ps[j,0]}_0$ with an odd value of $j$ (or with $j=0$) contains
the submodule $\Lfu_0^{j,0}$, discussed in Proposition~\ref{prop-iLog}.
We have to look also at $j\in 2\ZZ_{\neq 0}$, which comes under
integral parametrization.\il{ls1}{logarithmic submodule}
\begin{prop}\label{prop-logIP}
For $j\in 2\ZZ_{\neq 0}$ we
put\ir{tLfuip}{\tilde\Lfu_0^{j,0}}\il{ldph}{$\ldph h{p}{r}{q}$}
\badl{tLfuip} \tilde \Lfu_0^{j,0} &\= \bigoplus_{(h/3,p) \in \sect(j),\,
q\equiv p(2),\, |q|\leq p} \Bigl( \CC\, \ldph h{p}{r}{q}(0) + \CC\,\kph
h{p}{r}{q}(0)
\Bigr)\,,\\
\ldph h{p}{r}{q}&(0;n\am(t)k)
\= t^2\log t\; \Kph h{p}{r}{q}(k)\,, \eadl
with $r=\frac13 (h-2j)$. Then $\tilde \Lfu_0^{j,0}$ is a
$(\glie,K)$-submodule of $\Ffu_0^{\ps[j,0]}$ and
\be\label{0q} \tilde \Lfu_0^{j,0} \bigm/ H^{\xi,0}_K \cong
H^{\xi,0}_K\,.\ee
\end{prop}
\begin{proof}We put for $\nu \in \CC$
\be \label{ldph!=0} \ldph h{p}{r}{q}(\nu) = \frac1{2\nu} \bigl( \kph
h{p}{r}{q}(\nu) - \kph h{p}{r}{q}(-\nu) \bigr)\,.\ee
This extends holomorphically to $\nu=0$ with value
$\ldph h{p}{r}{q}(0)$.

For any $\XX\in \glie_c$ the derivative $\XX \kph h{p}{r}{q}(\nu)$ is a
linear combination of elements
$f_j(\nu) \, \kph{h_j}{p_j}{r_j}{q_j}(\nu)$ with holomorphic functions
$f_j$. This implies that $\XX \ldph h{p}{r}{q}(0)$ is a linear
combination of
$f_j(0)\, \ldph h{p}{r}{q}(0) + f_j'(0)\, \kph h{p}{r}{q}(0)$. Applying
this approach to the shift operators, we get with \eqref{shps}:
\begin{align*}
\sh{\pm 3}1 \ldph h{p}{r}{p}(0) &\= \frac{2+p\pm r}{8(p+1)} \Bigl(
(4\pm h + 2p\mp r) \, \ldph{h\pm 3}{(p+1)}{(r\pm 1)}{(p+1)}(0)\\
&\qquad\qquad \hbox{}
+ 2 \kph{h\pm 3}{(p+1)}{(r\pm 1)}{(p+1)}(0)\Bigr)\,,\\
\sh{\pm3}{-1}\ldph h{p}{r}{p}(0) &\= \frac p{4(p+1)} \Bigl( (\pm h-2p
\mp r)
\ldph{h\pm 3}{(p-1)}{r\pm 1}{p-1}(0)\\
&\qquad\qquad\hbox{}
+ 2 \kph {h\pm 3}{(p-1)}{r\pm 1}{p-1}(0) \Bigr)\,. \end{align*}

To determine the kernels of the shift operators, we note first that on
the boundary lines $h=2j\pm 3p$ of the sector $\sect(j)$ we have the
vanishing of $\sh{\pm 3}{-1} \ldph h{p}{\pm p}{p}$ by the properties of
$\Kph h{p}{r}{p}$. We check what happens on the lines indicated in
Propositions \ref{prop-kdso} and~\ref{prop-kuso}. The Weyl orbit of
$(j,0)$ has the two other elements $(-j,j)$ and $(-j,-j)$. On the line
$h= 2j'\pm 3p = -j\pm 3p$ we find
\be \sh{\pm 3}{-1}\ldph h{p}{r}{p} \= \frac p{2(p+1)} \, \kph {h\pm
3}{p-1}{r\pm 1}{p-1}(0)\,.\ee
On the line $h=-j\mp3p\mp6$:
\be \sh{\pm 3}1 \ldph h{p}{r}{p} \= \frac{2+p\pm r}{4(p+1) } \kph{h\pm
3}{p+1}{r\pm 1}{p+1}(0)\,.\ee
This shows that $\tilde \Lfu_0^{j,0}$ is a $(\glie,K)$-submodule of
$\Ffu^{\ps[j,0]}_0$, and that, modulo $H^{\xi,0}_K$, the shift
operators are zero on $\ldph h{p}{r}{p}(0)$ under the same conditions
as for $\kph h{p}{r}{p}(0)$. This gives~\eqref{0q}.
\end{proof}

\subsection{Intersection of kernels for the principal
series}\label{sect-iskdso}
Here we combine the necessary conditions in Proposition \ref{prop-kdso}
to get a condition for both downward shift operators to have a
non-trivial kernel.

First we give a lemma that will be useful in the next sections as well.
\begin{lem}\label{lem-kdso}Let $h\equiv p \bmod 2$, $p\in \ZZ_{\geq 1}$.
Define the integers $j_1<j_2$ by
\be\label{hjp} h-3p\= 2j_1\,, \quad h+3p \= 2 j_2\,.\ee
\begin{enumerate}
\item[i)] The set $ \bigl\{ j_1, j_2,-h\bigr\}$ is equal to the set
$\wo^1(\ps)$ for the unique $\ps\in \WOI$ represented by $(-h,p)$.
\item[ii)] In terms of the notations and relations in~\eqref{jnurels} we
have $\{j_1,j_2,-h\} = \{ j_l,j_+,j_r\}$, and one of the following
situations occurs:
\be\label{jjhp}
\begin{array}{c|cc|rlrl}
\text{\rm case}
& j_1 & j_2 & \multicolumn{2}{c}h& \multicolumn{2}{c}p \\ \hline
1& j_l & j_+& j_l+j_+ &= -j_r& \frac13(j_+-j_l)&= \nu_r
\\
2 & j_l & j_r & j_l + j_ r&= -j_+ & \frac13(j_r-j_l)&=\nu_+
\\
3 & j_+ & j_r & j_+ + j_r&=-j_l & \frac13(j_r-j_+)&= \nu_l
\end{array}\ee
If $j_l=j_+$ case 1 does not occur, and cases 2 and 3 coincide; if
$j_+=j_r$ case 3 does not occur, and cases 1 and 2 coincide.
\item[iii)] Let $\nu_1,\nu_2$ be the unique non-negative integers such
that $(j_n,\nu_n)\in \wo(\ps)^+$ for $n=1,2$. Then
$\nu_1=\frac12|h+p|$, and $\nu_2= \frac12|h-p|$.
\item[iv)]The quantities $\nu_1\pm j_1$ and $\nu_2\pm j_2$ have the
following values:
\be
\begin{array}{|c|cc|cccc|}\hline
&j_1&j_2&\nu_1+j_1&\nu_1-j_1&\nu_2+j_2&\nu_2-j_2 \\ \hline
1&j_l&j_+&-2p&2(\nu_l+p)&2p&2(\nu_+-p)
\\
2&j_l&j_r&2(\nu_l-p)&2p&2p&2(\nu_r-p)
\\
3&j_+&j_r&2(\nu_+-p)&2p&2(\nu_r+p) & -2p
\\ \hline
\end{array}\ee
\end{enumerate}
\end{lem}

\begin{proof}The relations \eqref{hjp} mean that the point
$\bigl( \frac h3,p \bigr)$ is on the right boundary of $\sect(j_1)$ and
on the left boundary of $\sect(j_2)$. That rules out generic multiple
parametrization. See Figure~\ref{fig-sectors}, p~\pageref{fig-sectors}.

Since $h \equiv p\bmod 2$ we have $j_1,j_2\in \ZZ$. Since $p\geq 0$ we
have $(h,p) \neq 0$, and $(-h,p)\in L$. So $\ps=[-h,p]\in \WOI$. This
gives~i). Part ii)
follows from \eqref{hjp} and the order relation $j_l \leq j_+\leq j_r$.

For iii) we carry out a computation (in \cite[\S17b]{Math}) based on
\eqref{jnurels}, which shows that $\nu_1=\pm\frac12(h+p)$ and
$\nu_2= \pm \frac12(h-p)$ with the signs according to the following
scheme:
\be
\begin{array}{|c|cc|cc|}\hline
&j_1&j_2& \nu_1 & \nu_2 \\ \hline
1& j_l&j_+&-1 & -1
\\
2&j_l&j_r&1&-1
\\
3&j_+&j_r&1&1
\\ \hline
\end{array}
\ee

With these signs we can express $j_n$ in $\nu_n$ and $p$. That leads to
the table in~iv).
\end{proof}

The $K$-types $\tau^h_p$ given in the lemma are the intersection points
of the boundaries of sectors $\sect(j)$, as sketched in
Figure~\ref{fig-sectis}.
\begin{figure}[ht]
\begin{center}\grf{8}{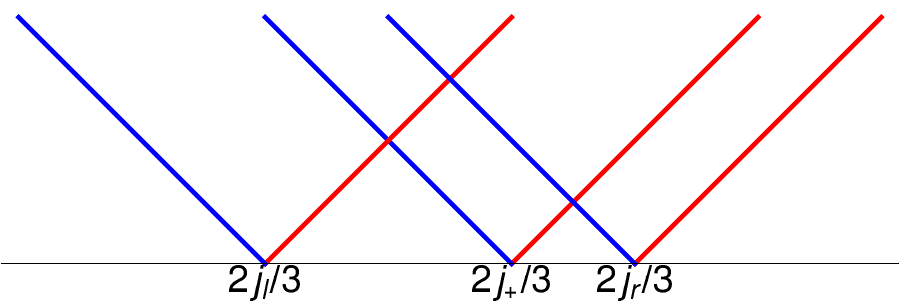}\end{center}
\caption[]{Intersection of sectors $\sect(j)$ corresponding to one Weyl
orbit.\\
If $(j_+,\nu_+)$ is on a wall of the positive chamber $L^+$ two of the
sectors coincide.} \label{fig-sectis}
\end{figure}

\begin{prop}\label{prop-F0do}Let $\tau^h_p$ be a $K$-type with
$p\geq 1$. If the intersection \il{K0hp}{$K_{0;h,p}$}$K_{0;h,p}$ of the
kernels of $\sh3{-1}:\Ffu_{0;h,p,p}\rightarrow \Ffu_{0;h+3,p-1,p-1}$
and $\sh{-3}{-1}:\Ffu_{0;h,p,p}\rightarrow \Ffu_{0;h-3,p-1,p-1}$ is
non-zero, then $K_{0;h,p} \subset \Ffu_{0;h,p,p}^\ps$ for
$\ps= [-h,p]\in \WOI$, by Lemma~\ref{lem-kdso}.

The space $K_{0;h,p}$ is non-zero in the following cases:
\be
\begin{array}{|c|cc|c|}\hline
&h& p\geq 1 & \text{basis of $K_{0;h,p}$} \\ \hline
1&j_l+j_+ & \frac{j_+-j_l}3 &\kph h p {r_l} p ( \nu_l)\,,\; \kph h p
{r_+}p(\nu_+)
\\
2&j_l+j_r & \frac{j_r-j_l}3 &\kph h p {r_l} p ( \nu_l)\,,\; \kph h p
{r_r} p( \nu_r)\,,\; \kph h p {r_+}p(\nu_+)
\\
3&j_++j_r& \frac{j_r-j_+}3&\kph h p {r_r} p( \nu_r)\,,\; \kph h p
{r_+}p(\nu_+)
\\ \hline
\end{array}
\ee
We use $r_l = \frac{h-2j_l}3$, $j_+ = \frac{h-2 j_+}3$, and
$r_r= \frac{h-2j_r}3$. If $j_l=j_+$ then $\kph hp{r_r}p(\nu_r)$ and
$\kph h p {r_+}p(\nu_+)$ form a basis of $K_{0;h,p}$, and if $j_+=j_r$,
a basis is $\kph h p {r_l}p(\nu_l)$, $\kph hp{r_+}p(\nu_+)$.
\end{prop}

\begin{proof} Proposition~\ref{prop-kdso} shows that the lemma can be
applied, thus reducing the proof to a computation of the downward shift
operators on $\kph h p r p (\nu)$ where $r$ is determined by $h=2j+3r$
for all $(j,\nu) $ in the Weyl orbit of $(j_+,\nu_+)$.

We use Table~\ref{tab-vancond}, p~\pageref{tab-vancond}, which uses the
quantities $\frac h3\pm p$, which we compute for the three cases:
\be\renewcommand\arraystretch{1.2}
\begin{array}{|c|cc|}\hline
& \frac h3+p & \frac h3-p \\ \hline
1&\frac23 j_+ & \frac23 j_l \\
2& \frac23 j_r & \frac23 j_l \\
3& \frac 23 j_r & \frac23 j_+ \\ \hline
\end{array}\ee

We first consider elements of $H^{\xi_l,\pm \nu_l}_K$. We have
$\sh 3{-1}\kph h p {r_l} p (\nu)=0$ for all $\nu \in \CC$. Furthermore
we have $\Ws 1 (j_l,\nu_l) = (j_+,\nu_+)$ and
$\Ws 1 (j_l,-\nu_l)= (j_r,\nu_r)$. So we have $j_1 > j_l$ in both
cases. This gives $\kph h p {r_l} p ( \nu_l) \in K_{0;h,p}$ in case 1,
and $\kph h p {r_l} p ( -\nu_l) \in K_{0;h,p}$ in case~2.
(Computations of the Weyl group action in~\cite[\S17a]{Math}.)

We handle elements of $H^{\xi_r,\nu_r}_K$ in an analogous way. We need
to consider $\sh 3 {-1} \kph h p {r_r} p(\pm \nu_r)$, and have
$\Ws 1 (j_r,\nu_r) = (j_+,\nu_+)$ and
$\Ws 1 (j_r,-\nu_r) = (j_l,\nu_l)$. This gives elements in $K_{0;h,p}$
in cases 2 and~3.

Finally we consider elements of $H^{\xi_+,\nu_+}_K$. We have
\begin{align*}
\Ws2 (j_+,\nu_+) &\= (j_l,\nu_l)\,,&
\quad \Ws1(j_+,\nu_+) &\= (j_r,\nu_r)\,,\\
\Ws2(j_+,-\nu_+)
&\=(j_r,-\nu_r)\,,& \Ws1(j_+,-\nu_+) &\=
(j_l,-\nu_l)\,.\end{align*}

With Table~\ref{tab-vancond}, p~\pageref{tab-vancond}, this gives the
possibilities in Table~\ref{tab-j123}.
\begin{table}[ht]
\[\renewcommand\arraystretch{1.4}
\begin{array}{|cc|cc|cc|}\hline
& \pm& 3h-p & &3h+p & \\ \hline
1& + & j_l = j_2 < j & \sh 3 {-1}\ph_+ =0 & j_+ = j & \sh{-3}{-1}\ph_+=0
\\
& - & j_l = j_1 < j &j_2=j_r>j&j_+ = j & \sh {-3}{-1} \ph_-=0 \\ \hline
2&+& j_l= j_2<j & \sh 3{-1} \ph_+=0 &j_r=j_1 > j & \sh{-3}{-1}\ph_+=0
\\
&-& j_l= j_1 < j &j_2=j_r>j & j_r=j_2>j & j_1=j_l<j\\\hline
3&+&j_+ = j & \sh 3{-1} \ph_+=0 & j_r=j_1>j &\sh{-3}{-1}\ph_+=0
\\
&-&j_+ = j & \sh3{-1}\ph_-=0 &j_r=j_2>j&j_1=j_l<j
\\ \hline
\end{array}
\]
\caption{Combinations for $j=j_+$ and
$\ph_\pm = \kph h p {r_+}p(\pm \nu_+)$. }\label{tab-j123}
\end{table}
If $j_+$ is equal to $j_l$ or $j_r$, we have to deal with only one
intersection point of boundaries of sectors; one case does not exist,
and the two other cases coincide. The same reasoning leads to two
elements in the intersection.

Non-zero elements of different principal series modules are linearly
independent. We have to see whether logarithmic modules
(Proposition~\ref{prop-logIP}) may be in the intersection of kernels. A
computation in \cite[\S17c]{Math} shows that logarithmic solutions do
not give elements of $K_{0;h,p}$.
\end{proof}


\def\flnm{rFtm-III-ab}


\section{Generic abelian Fourier term
modules}\label{sect-abip}\markright{13. GENERIC ABELIAN FOURIER TERM
MODULES} We turn to the submodule structure of \il{gabFtm}{generic
abelian Fourier term modules}generic abelian Fourier term modules
$\Ffu^{\ps}_\bt$ with $\bt\neq 0$. Under integral para\-metrization, we
still have to define the submodules $\Mfu^{\xi,\nu}_\bt$ and
$\Wfu_\bt^{\xi,\nu}$.

These modules differ from those in the $N$-trivial case in two aspects:
(1) In these modules there are only irreducible submodules of type
$\II_+$. (2) The modules and $H^{\xi',\nu'}_K$ and $H^{\xi,\nu}_K$ in
the principal series have zero intersection if $(j,\nu)$ and
$(j',\nu')$ are different elements in the same Weyl group orbit. We
will see that $\Wfu^{\xi,\nu}_\bt \cap \Wfu^{\xi',\nu'}_\bt$ and
$\Mfu^{\xi,\nu}_\bt \cap \Mfu^{\xi',\nu'}_\bt$ are non-zero modules.

\rmrk{Preliminaries} There are some facts that hold for both the abelian
and the non-abelian case.

In Lemma~\ref{lem-bfs} we obtained the families
$\nu\mapsto \om_\Nfu^{a,b}(j,\nu)$ and
$\nu\mapsto \mu^{a,b}_\Nfu(j,\nu)$ that are holomorphic in $\nu\in \CC$
(for $\om$), or in $\CC\setminus \ZZ_{\leq -1}$ (for $\mu^{a,b}_\n$). In
the cases $a=0$ or $b=0$, Proposition~\ref{prop-dsok1} states a
vanishing result for the downward shift operators, which stays valid in
the case of integral parametrization.

These families may be compared to the basis families
$\nu\mapsto\kph h{p}{r}{p}(\nu)$ for the principal series. An important
difference is that $\kph h{p}{r}{p}(\nu)$ is explicitly known, whereas
the families $\om^{a,b}_\Nfu$ and $\mu^{a,b}_\Nfu$ have a much more
complicated description. We have to look for situations in which we can
obtain a relatively simple description. One of such situations occurs
when the intersection of the kernels of the downward shift operators is
non-zero. Proposition~\ref{prop-kdso} and Lemma~\ref{lem-kdso} show
that this happens for a given $K$-type only for one character of
$ZU(\glie)$.

\rmrk{Families with a fixed K-type}In \eqref{xab} we defined the
families $\om^{a,b}_\bt$ and $\mu^{a,b}_\bt$ by repeated application of
the upward transfer operators to $\om^{0,0}_\bt$ and $\mu^{0,0}_\bt$.
Proposition~\ref{prop-uso-ga} implies that these families are non-zero
for all $\nu$.

\begin{lem}\label{lem-ext-ab} In the decomposition
$F=\ch_\bt \sum_r f_r\, \Kph h{p}{r}{p}$ the following holds.
\begin{align*} \om_\bt^{p,0}(j,\nu) &\text{ has lowest component }
f_{-p}(t)
\ddis t^{p+2}\, K_\nu(2\pi|\bt|t)\,,\displaybreak[0]
\\
\mu_\bt^{p,0}(j,\nu) &\text{ has lowest component } f_{-p}(t) \ddis
t^{p+2}\, I_\nu(2\pi|\bt|t)\,,\displaybreak[0]\\
\om_\bt^{0,p}(j,\nu)
&\text{ has highest component } f_{p}(t) \ddis t^{p+2}\,
K_\nu(2\pi|\bt|t)\,,\displaybreak[0]\\
\mu_\bt^{0,p}(j,\nu) &\text{ has highest component } f_{p}(t) \ddis
t^{p+2}\, I_\nu(2\pi|\bt|t)\,.
\end{align*}
\end{lem}
We use \il{.=}{$\dis$}$\dis$ to indicate equality up to a non-zero
factor.

\rmrk{Determining components} Proposition~\ref{prop-dsok1} implies that
$\sh 3{-1}x^{p,0}_\bt$ and $\sh{-3}{-1}x^{0,p}_\bt$ are identically
zero for $x=\om$ and $x=\mu$. The kernel relations in
Table~\ref{tab-krab}, p~\pageref{tab-krab}, imply that the components
in the lemma \il{dc}{determining component}determine all other
components.

\begin{proof}[Proof of Lemma~\ref{lem-ext-ab}]The families
$x^{a,b}_\Nfu$ are defined inductively. For $a=b=0$ relations
\eqref{mu00a} and \eqref{om00ab} give the component
$f_0(t) = t^2 \, j_\nu(2\pi |\bt|t)$, with $j_\nu=I_\nu$ or $K_\nu$.
The description of the shift operators in Table~\ref{tab-shab},
p~\pageref{tab-shab}, implies that the relevant upward shift operator
acts on the lowest, respectively highest, component as multiplication
by a non-zero multiple of~$t$.
\end{proof}

\rmrk{Intersection of kernels}
\begin{prop}\label{prop-nsefa}Let $\bt\neq 0$, and let $\tau^h_p$ with
$p\geq 1$ be a $K$-type that occurs in $\Ffu_\bt$. We denote by
\il{Khp}{$K_{0;h,p},\; K_{\bt;h,p},\; K_{\n;h,p}$}$K_{\bt;h,p}$ the
intersection of the kernels of
$\sh 3{-1} : \Ffu_{\bt;h,p,p} \rightarrow \Ffu_{\bt;h+3,p-1,p-1}$ and
$\sh {-3}{-1} : \Ffu_{\bt;h,p,p} \rightarrow \Ffu_{\bt;h-3,p-1,p-1}$.
\begin{enumerate}
\item[i)] The dimension of $K_{\bt;h,p}$ is equal to $2$. A basis is
\ir{expl-kdso}{\kk^I_{\bt;h,p},\; \kk^K_{\bt;h,p}}
\badl{expl-kdso} \kk^I_{\bt;h,p}&\=\sum_{ r\equiv p(2),\, |r|\leq p}
\chi_\bt\;
(i\bt/|\bt|)^{(r+p)/2} \, t^{2+p} \, I_{|h-r|/2}(2\pi|\bt|t)
\, \Kph h{p}{r}{p}\,,\\
\kk^K_{\bt;h,p}&\=\sum_{r\equiv p(2),\, |r|\leq p} \chi_\bt\;
(-i\bt/|\bt|)^{(r+p/2} \, t^{2+p}\, K_{|h-r|/2}(2\pi|\bt|t)
\, \Kph h{p}{r}{p}\,. \eadl

\item[ii)]
\begin{enumerate}
\item[a)] $K_{\bt;h,p}$ is a subspace of the module $\Ffu^\ps_\bt$ where
$\ps=\ps[-h,p]\in \WOI$.
\item[b)] Put $j_1=\frac12(h-3p)$, $j_2=\frac12(h+3p)$. There are unique
$\nu_1,\nu_2\in \ZZ_{\geq 0}$ such that $(j_n,\nu_n)\in \wo(\ps)^+$ for
$n=1,2$. With these values we have:
\badl{kkbtmuom} \kk^I_{\bt;h,p} &\ddis \mu^{p,0}_\bt(j_1,\nu_1)
\ddis \mu^{0,p}_\bt(j_2,\nu_2)\,,\\
\kk^K_{\bt;h,p} &\ddis \om^{p,0}_\bt(j_1,\nu_1) \ddis
\om^{0,p}_\bt(j_2,\nu_2)\,.
\eadl
\end{enumerate}
\end{enumerate}
\end{prop}
\rmrk{Remarks} The $K$-types $\tau^h_p$ discussed here are the sole
higher-dimensional $K$-types for which we obtained reasonably simple
descriptions of elements of $\Ffu_{\bt;h,p,p}$.
Corollary~\ref{cor-jl+r-ab} describes one more instance with explicit
expressions similar to those in~\eqref{kkbtmuom}.

Such $K$-types correspond to points $(h/3,p)$ on the intersection of two
sectors, here denoted $\sect(j_1)$ and $\sect(j_2)$.

\begin{proof}We consider an element
$F = \ch_\bt \sum_r f_r\, \Kph h{p}{r}{p}$ in
$K_{\bt;h,p} \subset \Ffu_{\bt;h,p,p}$. Its components $f_r$ satisfy
the kernel relations for $\sh3{-1}$ and $\sh{-3}{-1}$ in
Table~\ref{tab-krab}, p~\pageref{tab-krab}. These two relations lead to
a second order differential equation for~$f_r$. A computation in
\cite[\S18a]{Math} shows that this differential equation is related to
the modified Bessel differential equation \eqref{mBd} with
$\nu=(h-r)/2$, and $f_r(t) = t^{2+p} \, j(2\pi|\bt|t)$, where we can
take $j$ equal to $I_{(h-r)/2}$ or $K_{(h-r)/2}$. The function $K_\nu$
is even in the parameter $\nu$. The same holds for $I_\nu$ for integral
values of~$\nu$.

To determine the relation between the coefficients for various values of
$r$ in these linear combinations, we use again the kernel relations.
With use of the contiguous relations in \eqref{cr1} we get in
\cite[\S18b]{Math} a two-dimensional solution space for the
coefficients, with basis as indicated in~\eqref{expl-kdso}.

In this way we have two explicit linearly independent elements of
$K_{\bt;h,p}$. In \cite[\S18c]{Math} we see that both functions are
eigenfunction of $ZU(\glie)$ with character $\ps=\ps[-j,p]$.

To finish the proof of i) and ii)a), we still have to show that the
dimension of $K_{\bt;h,p}$ is two.
\smallskip

The lowest component of $\mu_\bt^{p,0}(j_1,\nu_1)$ is equal to
$t^{p+2}\, I_{\nu_1}(2\pi|\bt|t)$ by Lemma~\ref{lem-ext-ab}. The lowest
component of $\kk^I_{\bt;h,p}$ is $t^{2+p} \, I_{|h+p|/2}(2\pi|\bt|t)$.
By iii) in Lemma~\ref{lem-kdso}, these two functions are proportional.
The other relations in~\eqref{kkbtmuom} follow analogously.

For the dimension assertions we use the following induction result.
\begin{lem}\label{lem-dimdso}Let $h$, $p$, $\ps$, $(j_1,\nu_1)$ and
$(j_2,\nu_2)$ as in the proposition.

If $\dim \Ffu^\ps_{\bt;h,p,p}>2$ then at least one of the following
statements holds:
\[ \dim \Ffu^\ps_{\bt;h+3,p-1,p-1}>2\,,\qquad \dim
\Ffu^\ps_{\bt;h-3,p-1,p-1}>2\,.\]
\end{lem}
\begin{proof}
We know two linearly independent elements
\be\label{b-}\om_\bt^{0,p-1}(j_2,\nu_2)\,,\; \mu_\bt^{0,p-1}(j_2,\nu_2)
\in \Ffu^\ps_{\bt;h+3,p-1,p-1}\,,\ee
and two linearly independent elements
\be\label{b+}
\om_\bt^{p-1,0}(j_1,\nu_1)\,,\; \mu_\bt^{p-1,0}(j_1,\nu_1)
\in \Ffu^\ps_{\bt;h-3,p-1,p-1}\,.\ee
To prove the lemma we assume that $F\in \Ffu^\ps_{\bt;h,p,p}$ is not in
$K_{\bt;h,p}$, and want to show that at least one of
\be F_- \= \sh3{-1}F\in \Ffu^\ps_{\bt;h+3,p-1,p-1}
\quad\text{ and }\quad F_+ \= \sh{-3}{-1}F \in
\Ffu^\ps_{\bt;h-3,p-1,p-1}\ee
is not a linear combination of the two elements in \eqref{b+} or
\eqref{b-}.

Suppose that $F_-$ is a linear combination of basis functions
in~\eqref{b+}. These functions are in the kernel of $\sh{-3}{-1}$, and
hence determined by their highest component. It suffices to show that
the component $b$ of $F_-$ of order $p-1$ vanishes.

On the other hand, since $b$ is the highest component of $\sh{3}{-1}F$,
it can be expressed in the highest two components of~$F$.
\be b\= \frac p{4(p+1)} \Bigl( 4\pi i \bar\bt t f_p + 2 t f_{p-2}'
+(h-3p-2) f_{p-2} \Bigr)\,.\ee
This enables us to express $f_{p-2}'$ in $f_{p-2}$, $f_p$ and the
function $b$. We substitute this in the eigenfunction equations for the
highest component of $F$. We have to take $(j,\nu)$ in the
eigenfunctions equations equal to $(j_2,\nu_2)$, and use iii) in
Lemma~\ref{lem-kdso}. The second eigenfunction relation takes the form
\[ -216\pi i \bt (p+1) t b\=0\,.\]
Hence $b=0$.

In a similar way we obtain that $F_+=0$ from the assumption that it is a
linear combination of the functions in~\eqref{b-}. Both computations
are in \cite[\S18d]{Math}.
\end{proof}

To finish the proof of proposition~\ref{prop-nsefa} we show more than is
needed for the present proposition, namely that for $\ps=\ps[-h,p]$
\be \label{dbts}\dim \Ffu_{\bt;h',p',p'}^\ps \= 2 \text{ or } 0\,,\qquad
\text{ for all $K$-types $\tau^{h'}_{p'}$}\,. \ee

For $p'=0$ we know \eqref{dbts} from \S\ref{sect-1dKt}: the dimension is
$2$ for $h=\frac 23 j$ with $j\in \wo^1(\ps)$ and 0 otherwise. The
injectivity of the upward shift operators
(Proposition~\ref{prop-uso-ga}) shows that the dimension is at least $2$
whenever $\Ffu_{\bt;h',p',p'}\neq \{0\}$. When we apply a downward
shift operator, the dimension cannot decrease if that shift operator is
injective. So from a given $K$-type $\tau^{h'}_{p'}$ we can go down to
a $K$-type $\tau^{h'\pm 3}_{p'-1}$ without decreasing the dimension as
long as one of the downward shift operators is injective on
$\Ffu_{\bt;h',p',p'}^\ps$. The only possibility of a change in
dimension occurs if both downward shift operators have a non-trivial
kernel. Lemma~\ref{lem-dimdso} shows that a dimension larger than 2
cannot drop to 2. Thus we get~\eqref{dbts}.
\end{proof}

\subsection{Structure results}
We still have to define the modules
$\Wfu^{\xi,\nu}_\bt$ and $\Mfu^{\xi,\nu}_\bt$ under integral
parametrization, which implies that $\nu \in \ZZ$. We restrict
ourselves to the case $\nu \in \ZZ_{\geq 0}$.

\begin{lem}\label{lem-strab} Let $\ps\in \WOI$. 
For each $(j,\nu) \in \wo(\ps)^+$ we
put\il{Mfudefia}{$\Mfu_\bt^{\xi,\nu}$}\il{Wfudefia}{$\Wfu_\bt^{\xi,\nu}$}
\badl{MWbd} \Mfu_\bt^{\xi,\nu} &\= \sum_{a,b\geq 0} U(\klie)\;
\mu^{a,b}_\bt(j,\nu)
\subset \Mfu_\bt^\ps\,,\\
\Wfu_\bt^{\xi,\nu} &\= \sum_{a,b\geq 0}U(\klie)\; \om^{a,b}_\bt(j,\nu)
\subset \Wfu_\bt^\ps\,.
\eadl
\begin{enumerate}
\item[i)] These spaces are $(\glie,K)$-submodules of $\Ffu^\ps_\bt$, and
$\Mfu^{\xi,\nu}_\bt\cap\Wfu^{\xi,\nu}_\bt= \{0\}$.

The $K$-types $\tau^h_p$ in $\Wfu^{\xi,\nu}_\bt$ and
$\Mfu^{\xi,\nu}_\bt$ have multiplicity one if
$(h/3,p) \in \sect(j_\xi)$, and do not occur in these modules
otherwise.

\item[ii)] Suppose that $(h/3,p) \in \sect(j) \cap \sect(j')$ for
$(j,\nu), \; (j',\nu')  \in \wo (\ps)^+$. Then
$\Wfu^{\xi_j,\nu}_{\bt;h,p,p} = \Wfu^{\xi_{j'},\nu'}_{\bt;h,p,p}$ and
$\Mfu^{\xi_j,\nu}_{\bt;h,p,p} = \Mfu^{\xi_{j'},\nu'}_{\bt;h,p,p}$
\item[iii)] We have
\be\label{dsd-ab} \Ffu^\ps_\bt = \Bigl( \sum_{(j,\nu)\in \wo(\ps)^+}
\Wfu^{\xi,\nu}_\bt\Bigr) \oplus \Bigl( \sum_{(j,\nu)\in \wo(\ps)^+}
\Mfu^{\xi,\nu}_\bt\Bigr)\,.\ee
\item[iv)]
\be \Mfu^\ps_\bt \= \sum_{(j,\nu)\in \wo(\ps)^+} \Mfu^{\xi,\nu}_\bt\,,
\quad \Wfu^\ps_\bt \= \sum_{(j,\nu)\in \wo(\ps)^+} \Wfu^{\xi,\nu}_\bt
\,.\ee
\end{enumerate}
\end{lem}
\begin{proof}The space $\Wfu^{\xi,\nu}_\bt$ is invariant under $\klie$.
To see that it is invariant under $\glie$, it suffices to consider the
shift operators on a highest weight vector in a $K$-type. We use the
definition in Table~\ref{tab-sho}, p~\pageref{tab-sho} and apply ii) in
Lemma~\ref{lem-mv} to see that $\om^{a,b}_\bt(j_\xi,\nu)$ is sent to a
linear combination of elements $u \,\om^{a',b'}_\bt(j_\xi,\nu)$ with
$u\in U(\klie)$. By Proposition~\ref{prop-dsok1} the point
corresponding to the $K$-type cannot leave the sector $\sect(j_\xi)$.
Hence $\Wfu^{\xi,\nu}_\bt$ is a $(\glie,K)$-module. The same reasoning
works for $\Mfu^{\xi,\nu}_\bt$.

Now consider an element $f\in \Wfu^\bt_\bt\cap \Mfu^{\xi,\nu}_\bt $ of a
given $K$-type $\tau^h_p$. Using a downward path in the $(h/3,p)$-plane
given by injective downward shift operators we ultimately arrive at a
minimal vector $v\in \Wfu^\ps_\bt$, on which both downward shift
operators vanish. In that situation we know that minimal vector
explicitly, from \S\ref{sect-1dKt}
($p=0$) or \eqref{expl-kdso} ($p\geq 1$). Since the I-Bessel functions
and the $K$-Bessel functions with the same parameter are linearly
independent, this minimal vector vanishes. We used a path of injective
downward shift operators, and conclude that $f=0$.

The upward shift operators are injective by Proposition~\ref{prop-kuso}.
So all $K$-types corresponding to points of $\sect(j_\xi)$ occur in
$\Wfu^{\xi,\nu}_\bt$ and in $\Mfu^{\xi,\nu}_\bt$ with multiplicity at
least~$1$. By \eqref{dbts} the multiplicities are exactly one. This
gives~i).\smallskip

Consider $j_l, j_r\in \wo^1(\ps)$ (in the conventions of
\eqref{jnurels}). Proposition~\eqref{prop-nsefa} ii)b) implies that
$\Wfu^{\xi_l,\nu_l}_{\bt;h,p,p} = \Wfu^{\xi_r,\nu_r}_{\bt;h,p,p}$ for
$(h/3,p)$ equal to the lowest point of the intersection
$\sect(j_l) \cap \sect(j_r)$. From any other $K$-type corresponding to
a point of $\sect(j_l) \cap \sect(j_r)$ we can go down to the $K$-type
by a path of injective downward shift operators. This gives the
analogous equality for the $K$-types in the intersection.
\begin{center}\grf{5}{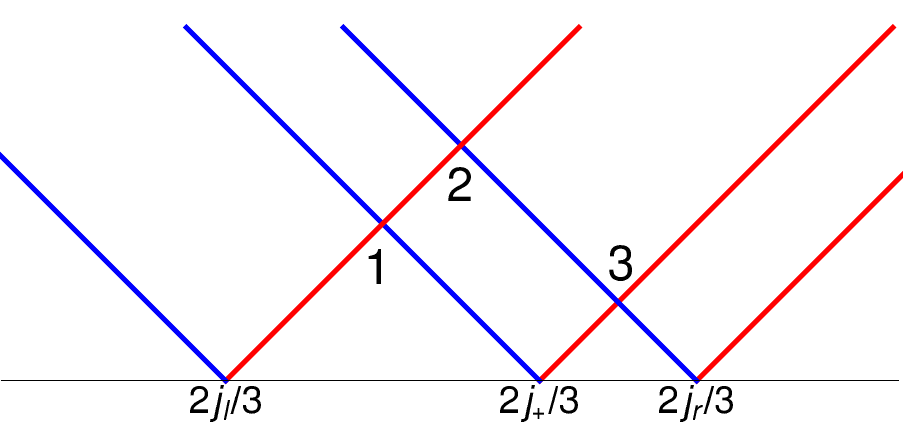}\end{center}
This corresponds to the triangular region above 2 in the picture.

For $j_l < j_+< j_r$ we are in the rectangular regions above 1 and 3. In
the same way the equalities
$\Wfu^{\xi_l,\nu_l}_{\bt;h,p,p} = \Wfu^{\xi_+,\nu_+}_{\bt;h,p,p}$ for
$(h/3,p) \in \bigl( \sect(j_l) \cap \sect(j_+) \bigr) \setminus \sect(j_r)$,
and $\Wfu^{\xi_+,\nu_+}_{\bt;h,p,p} = \Wfu^{\xi_r,\nu_r}_{\bt;h,p,p}$
for
$(h/3,p) \in \bigl( \sect(j_+) \cap \sect(j_r) \bigr) \setminus \sect(j_l)$.
The upward shift operators are injective by
Proposition~\ref{prop-kuso}. So relations
$\om^{a,b}_\bt(j_1,\nu_1) \dis \om^{a,b}_\bt(j_2,\nu_2)$ are preserved
if we increase $a$ \ and/or $b$. This shows that in all $K$-types the
spaces $\Wfu_{\bt;h,p,p}^{\xi,\nu}$ are the same for all $(j_\xi,\nu)$
such that $(h/3,p) \in \sect(j_\xi)$. The same reasoning goes through
for~$\Mfu$. This gives ii), and implies~iii).

We turn to the submodules $\Wfu^\ps_\bt$ and $\Mfu^\ps_\bt$ of
$\Ffu^\ps_\bt$ in Definitions \ref{defW[]} and~\ref{defM[]}. From the
inclusions
\[\sum_{(j,\nu)\in \wo(\ps)^+} \Wfu^{\xi,\nu}_\bt\subset \Wfu^\ps_\bt\,,
\qquad \sum_{(j,\nu)\in \wo(\ps)^+} \Mfu^{\xi,\nu}_\bt\subset
\Mfu^\ps_\bt\,, \]
we obtain iv) by comparing multiplicities of $K$-types
in~\eqref{dsd-ab}.
\end{proof}

\begin{proof}[Proof of Theorem~\ref{mnthm-ab-ip}]\label{prfC}
Lemma~\ref{lem-strab} gives most of the statements of
Theorem~\ref{mnthm-ab-ip}. We still have to prove part~iv) of the
theorem. That implies the reducibility in~i).

The intersection $\cap_{(j,\nu) \in \wo(\ps)^+} \Wfu^{\xi_j,\nu}_\bt$ is
of course an invariant submodule. From the maximal weight in the
minimal $K$-type in the intersection we can reach all $K$-types in the
intersection by injective upward shift operators, and we can go back by
injective downward shift operators. Since the highest weight vectors in
a subspace of a given $K$-type generate the whole subspace, this
suffices for irreducibility.
\end{proof}

\rmrk{Illustration}Figure~\ref{fig-strab} gives an illustration of the
submodule structure.
\begin{figure}[t]
\begin{center}\grf{7}{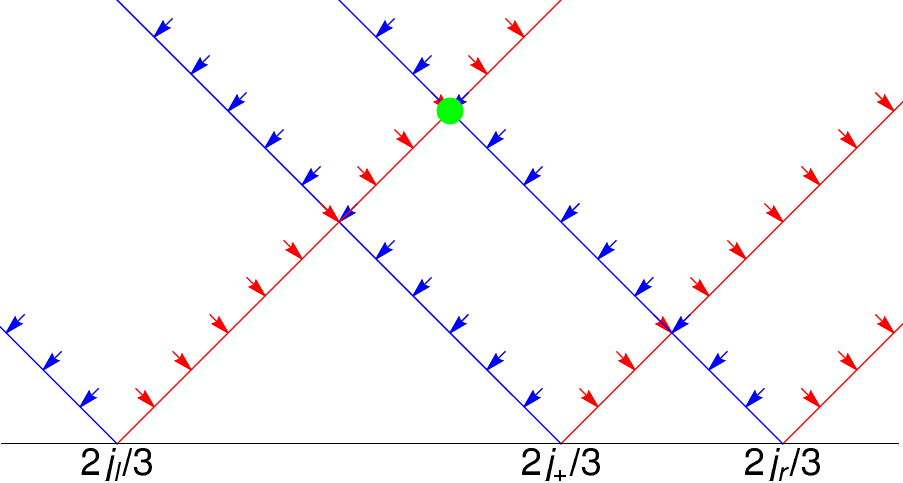}\\
\grf{5}{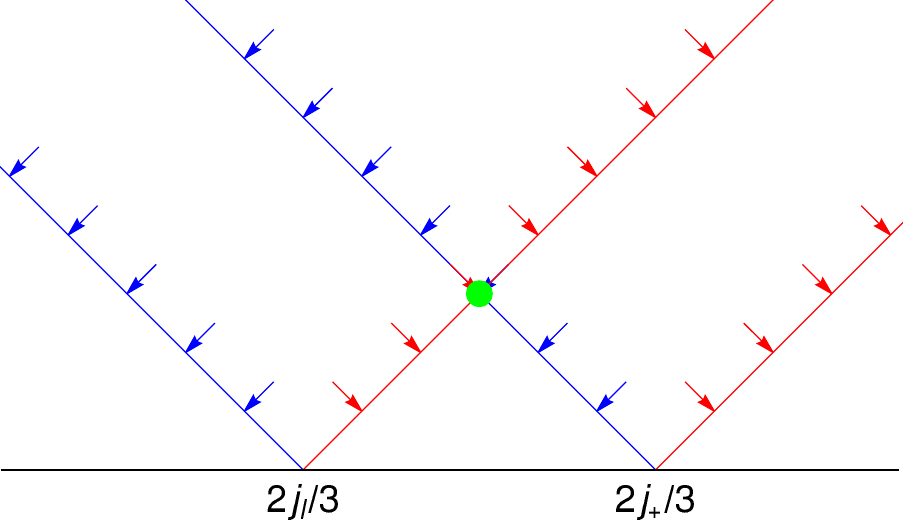}\qquad \grf{5}{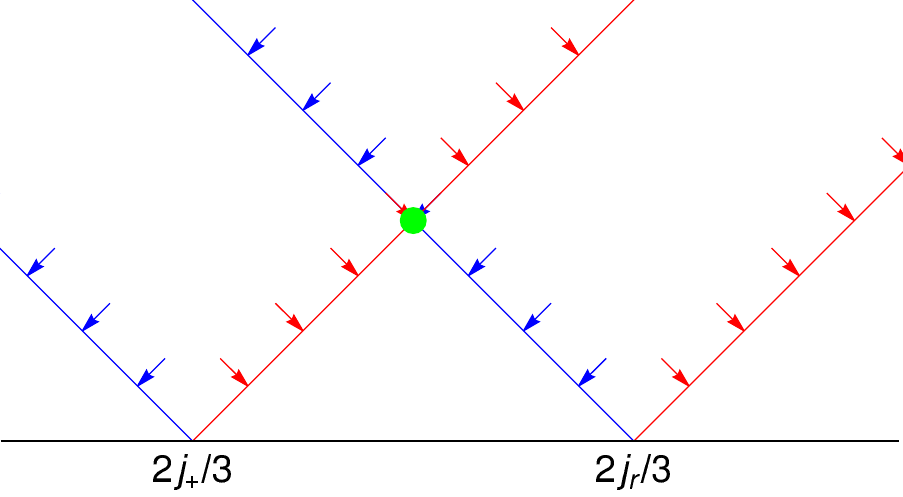}
\end{center}
\caption[]{Sketches of the submodule structure in the generic abelian
cases, applying to $\Wfu^\ps_\bt$ and to $\Mfu^\ps_\bt$. On top the
case $(j_l,j_+,j_r) = (-15,3,12)$, at the bottom the cases $(-6,3,3)$
(left) and $(-4,-4,8)$ (right).\\
We use the conventions in \eqref{jnurels}. The dots indicate the minimal
$K$-type of the irreducible submodule. } \label{fig-strab}
\end{figure}

\begin{remark}\label{rmk-WM-ab}The status of the submodules of the
Fourier term module $\Ffu_\bt^\ps$ under integral parametrization is
the same as in Remark~\ref{rmk-WM-gp}. The difference concerns the
intersections of the special modules in~\eqref{MWbd}. The modules
$\Wfu_\bt^{\xi,\nu}$ coincide in all $K$-types that they have in
common; the same holds for the modules $\Mfu_\bt^{\xi,\nu}$. On the
order hand, all principal series modules have zero intersection.
\end{remark}

With \eqref{MWbd} we get the following addition to ii)b) in
Proposition~\ref{prop-nsefa}.
\begin{cor}\label{cor-jl+r-ab}
Consider $(j_1,\nu_1)=(j_l,\nu_l)$, $(j_+,\nu_+)$,
$(j_r,\nu_r)=(j_2,\nu_2)$ in one Weyl group orbit, with $j_l<j_+<j_r$,
in the notation of~\eqref{jnurels}. The minimal $K$-type $\tau^h_p$ in
the intersection $\sect(j_1) \cap \sect(j_+)\cap \sect(j_2)$ has
descriptions $h=2j_1+3p=2j_2-3p=2j_++3(a-b)$, $p=a+b$ with
$a,b\in \ZZ_{\geq 0}$.

In this situation we have the following addition to~\eqref{kkbtmuom}:
\be \mu^{a,b}_\bt(j_+,\nu_+) \ddis \kk^I_{\bt;h,p}\,,\qquad
\om^{a,b}_\bt(j_+,\nu_+) \ddis \kk^K_{\bt;h,p}\,.\ee
\end{cor}


\def\flnm{rFtm-III-nab}


\section{Non-abelian Fourier term
modules}\label{sect-nab}\markright{14. NON-ABELIAN FOURIER TERM
MODULES} Under integral parametrization, the non-abelian case is more
complicated than the generic abelian case. The modules $\Mfu_\n^\ps$
and $\Wfu_\n^\ps$ have in some cases a non-zero intersection. For this
reason we also use modules based on the unusual Whittaker functions
$V_{\k,s}$ in~\eqref{Vkps}. Furthermore, the families $\om^{a,b}_\n$
and $\mu^{a,b}_\n$ may have zeros. As a consequence, we need more
complicated families to describe submodules.

The aim of this section is to determine the structure of non-abelian
Fourier term modules, and to prove Theorem~\ref{mnthm-nab-ip}. We also
give a detailed description of the structure of the special Fourier
term modules.

Proposition \ref{prop-ik} is of independent interest. It discusses a
situation for which we have a reasonably simple description of elements
of $\Ffu_{\n;h,p,p}$ with $p>1$.

\subsection{Notations and conventions}For the non-abelian Fourier term
modules we need several notations. We used some of them in an earlier
section, to be recalled here.\medskip

The type of the realization of the Stone-von Neumann representation is
indicated by $\n=(\ell,c,d)$ with $\ell \in \frac12\ZZ_{\neq 0}$,
$c\bmod 2\ell$, $d\equiv 1\bmod 2$. We abbreviate
$\sign(\ell)=\varepsilon$.\il{eps}{$\varepsilon$} The $K$-types that
occur in $\Ffu^\psi_\n$ have to satisfy the condition \eqref{mucond},
which can be written as
\be\label{mucond1} 3p -3\leq \varepsilon\bigl( h-d\bigr)\,.\ee

The decomposition of $F\in \Ffu_{\n;h,p,p}$ into component functions has
the form
\be \label{decomp-n}F\bigl(n\am(t)k\bigr) = \sum_r
\vartheta_{m(h,r)}(n)\, f_r(t)\, \Kph h{p}{r}{p}(k)\,,\ee
where \il{th-m}{$\vartheta_m$ abbreviates
$\Theta_{\ell,c}(h_{\ell,m})$}$\vartheta_m$ is an abbreviation of
$\Theta_{\ell,c}(h_{\ell,m})$ with the convention that $	\vartheta_m=0$
if $m\in \ZZ_{<0}$. This puts a further restriction on the summation
variable $r$, $|r|\leq p$, $r\equiv p\bmod 2$. The quantity
\il{mhrn}{$m(h,r)$}$m(h,r)$ is determined by the relation
\be \varepsilon \bigl( 6m(h,r)+3\bigr) + h-3r = 0\,.\ee

The quantity \il{r0n}{$r_0(h)$}$r_0(h)$ is the solution of
$m\bigl( h,r_0(h)\bigr)=0$. Since $m(h,r)$ is increasing in $r$ if
$\varepsilon=\sign(\ell)=1$, and decreasing in $r$ if $\varepsilon<0$,
the sum in~\eqref{decomp-n} effectively runs over $r$ satisfying
$ \max(r_0(h),-p)\leq r\leq  p$ if $\varepsilon=1$, and
$-p\leq r \leq \min(r_0(h),p )$ if $\varepsilon=-1$. We use the
standing assumption that the components $f_r$ are $0$ if $r$ does not
satisfy these relations. All this is subsumed in the notation
$\sum_r$.\il{sumrnot}{$\sum_r$}

\il{minc}{minimal component}\il{maxc}{maximal component}We use the
following terminology
\be\label{maxc} \begin{array}{|c|cc|}\hline
& \varepsilon=1 & \varepsilon=-1 \\
\hline
\text{ minimal component of }F:& f_{\max(r_0(h),-p)} & f_{-p}
\\
\text{maximal component of }F:& f_p & f_{\min(r_0(h),p)}
\\ \hline
\end{array}
\ee

A basis of $\Ffu^\psi_{\n;2j,0,0}$ has been determined in
\S\ref{sect-1dKt}, in terms of Whittaker functions with parameters
\il{kap0}{$\k_0(j)$}$\k=\k_0(j)$ and $s$. We use also the conventions
and relations in~\eqref{jnurels}. Table~\ref{tab-parms-na} lists the
notations that we use for the non-abelian Fourier term modules.
\begin{table}[htp]
\[\renewcommand\arraystretch{1.4}
\begin{array}{|rll|}\hline
\n&\;=\; (\ell,c,d) &d\in 1+2\ZZ \\
\varepsilon&\;=\; \sign(\ell)
&\\\hline
\vartheta_m &\;=\; \Theta_{\ell,c}(h_{\ell,m})&\vartheta_m=0\text{ if
}m< 0\\ \hline
m(h,r) &\;=\; \frac \varepsilon2\bigl( r - r_0(h)
\bigr)\in \ZZ &\\
&\;= \;m_0(j)+\frac \varepsilon6\bigl(3r+2j-h\bigr) & \\
\hline
m_0&\;=\;m_0(j)= \frac \varepsilon 6(d-2j)-\frac12 \in \ZZ_{\geq 0}& \\
\hline
r_0(h)&\;=\;\frac{h-d}3+\varepsilon&
\\
&\;=\;\frac13(h-2j)-2\varepsilon m_0(j)
&\\
\hline
\k_0= \k_0(j) &\;=\; -\frac\varepsilon 6
(d+j)
\in \frac12\ZZ &\\
&\;=\; -m_0(j)
- \frac12(\varepsilon j+1)&
\\ \hline
(j_+,\nu_+) &\in L^+\,,&\\
(j_l,\nu_l) &\;=\; \Ws2 (j_+,\nu_+)\,,&
(j_r,\nu_r) \;=\; \Ws 1
(j_+,\nu_+)\\\hline
\wo(\psi)_\n&\;=\; \bigl\{ (j,\nu)\in \wo(\psi)\;:\; m_0(j)\geq 0\bigr\}
& \\
\wo(\psi)_\n^+&\;=\; \bigl\{ (j,\nu)\in \wo(\psi)\;:\; m_0(j)\geq 0\,,&
\re \nu\geq 0 \bigr\}
\\ \hline
\end{array}\]
\caption{Overview of notations for the non-abelian
cases}\label{tab-parms-na}
\il{m0}{$m_0(j)$}
\end{table}

\subsection{Families with fixed K-types}The families
$\omega^{a,b}_\n(j,\nu)$ and $\mu^{a,b}_\n(j,\nu)$ have been defined in
\eqref{mu00nab} and \eqref{om00nab} for $a=b=0$, and then recursively
in~\eqref{xab} for all $a,b\in \ZZ_{\geq 0}$. Since the $W$-Whittaker
function $W_{\k,\nu/2}$ is holomorphic and even in $\nu$ the families
$\omega^{a,b}_\n$ are holomorphic and even in $\nu\in \CC$. First order
singularities of $\mu_\n^{a,b}$ may occur at $\nu \in \ZZ_{\leq 1}$; so
$\mu^{a,b}_\n$ is a meromorphic family on $\CC$. See \eqref{Mkps} and
\eqref{Wkps}.

The relation \eqref{specWV} implies that $\omega^{a,b}_\n$ and
$\mu^{a,b}_\n$ may not provide us with linearly independent elements in
highest weight spaces for all combinations of $w$ and $\nu$. We employ
the
(unusual)
Whittaker function $V_{\k,s} $ in \eqref{Vkps} to define\ir{upsab}{
\upsilon_{\ell,c,\mu}^{a,b}}
\badl{upsab} \upsilon^{0,0}_{\ell,c,\mu}\bigl(j,\nu ;n\am(t)k\bigr) &=
\Theta_{\ell,c}\bigl( h_{\ell,m_0(j)};n)
\, t\, V_{\k_0(j)
,\nu/2}(2\pi|\ell|t^2)\, \Kph{2j}000(k)\,, \\
\upsilon^{a,b}_\n
(j,\nu) &= \bigl( \sh{-3}1\bigr)^b \bigl( \sh31)^a
\upsilon^{0,0}_\n(j,\nu)\,. \eadl
These families are holomorphic and even in~$\nu$.

The definition in \eqref{Vkps} implies that $\upsilon^{0,0}_\n(j,\nu)$
is a meromorphic linear combination of $\omega^{0,0}_\n(j,\nu)$ and
$\mu^{0,0}_\n(j,\nu)$ that is holomorphic for $\nu\in \CC\setminus\ZZ$.
The families $\upsilon^{a,b}_\n$ inherit this property.

Applying Proposition~\ref{prop-dsok1} for $\nu \in \CC \setminus \ZZ$,
and extending the result by holomorphy, we obtain that
\be \label{ect-ups}
\sh3{-1}\upsilon^{p,0}_\n(j,\nu)=0\,,\quad
\sh{-3}{-1}\upsilon^{0,p}_\n(j,\nu)=0\quad\text{ for all }\nu\in
\CC\,.\ee

\begin{lem}\label{lem-dcpt}If $F\in \Ffu_{\n;h,p,p}$ satisfies
$\sh 3{-1} F=0$, then $F$ is determined by its minimal component. If
$\sh{-3}{-1}F=0$, then $F$ is determined by its maximal component.
\end{lem}
\begin{table}[htp]
\[\renewcommand\arraystretch{1.4}
\begin{array}{|rl|}\hline
&\sh 3{-1}:\\
f_{r+2}&= \Bigl(2t f'_r +
(h-2p-r-4+4\pi\ell t^2) f_r\Bigr)/ 4i t \sqrt{ 2\pi|\ell|(1+m(h,r))} \\
&\text{for }\max\bigl(-p,r_0(h)\bigr)
\leq r \leq p-2\text{ and }\varepsilon=1\,,\\
\hline
f_{r+2}&=
-\Bigl(2t f'_r + (h-2p-r-4+4\pi\ell t^2)
f_r\Bigr)/ 4i t \sqrt{ 2\pi|\ell|m(h,r)} \\
&\text{for }-p \leq r \leq \min\bigl(r_0(h),p\bigr)-2\text{ and }
\varepsilon=-1\\ \hline
0&=2t f'_{r_0(h)}+
(h-2p-r_0(h)-4+4\pi\ell t^2) f_{r_0(h)}\\
&\text{for } -p \leq r_0(h) \leq p\text{ and }\varepsilon=-1\\
\hline\hline
&\sh{-3}{-1}:\\
f_{r-2}&= \Bigl( (h+2p-r+4+4\pi\ell t^2)
f_r - 2 t f_r'\Bigr)/4it \sqrt{2\pi|\ell| m(h,r) }\\
&\text{for } \max(r_0(h),-p) +2\leq r \leq p\text{ and }\varepsilon=1
\\
\hline
f_{r-2} &= -\Bigl( (h+2p-r+4+4\pi\ell t^2) f_r - 2 t f_r'\Bigr)/4it
\sqrt{2\pi|\ell|(1+ m(h,r) )}\\
&\text{for }2-p\leq r \leq\min(r_0(h),p)\text{ and }\varepsilon=-1
\\ \hline
0&=(h+2p-r_0(h)+4+4\pi\ell t^2)
f_{r_0(h)} - 2 t f_{r_0(h)}'\\
&\text{for }-p \leq r_0(h) \leq p\text{ and }\varepsilon=1 \\ \hline
\end{array}\]
\caption{Kernel relations for downward shift operators in $\Ffu_\n$
applied to $F= \sum_r \vartheta_{m(h,r)} \, f_r \, \Kph h p r p$.
Computations in \cite[\S19]{Math}.} \label{tab-krnab}
\end{table}
\begin{proof}Table~\ref{tab-shnab}, p~\pageref{tab-shnab}, gives an
explicit description of $\sh {\pm3}{-1} $ in terms of the component
function of an element of $\Ffu_{\n;h,p,p}$. This leads to the
\il{krn}{kernel relations}kernel relations in Table~\ref{tab-krnab}.
These relations imply the statement in the lemma, and, moreover, impose
in some cases a differential equation for this determining component.
\end{proof}
~

The families $x^{a,b}_\n(j,\nu)$ with $x\in \{\omega,\upsilon,\mu\}$ may
have zeros as a function of~$\nu$. At a zero, the derivative with
respect to $\nu$ is an element of $\Ffu^\psi_\n$. Lemma~\ref{lem-dcpt}
shows that we can investigate these zeros by consideration of the
maximal or minimal component.

\begin{prop}\label{prop-extr.na}Let $x\in \{\omega,\upsilon,\mu\} $ and
let $j\in \ZZ$ such that the $K$-type $\tau^{2j}_0$ occurs in
$\Ffu_\n$.
\begin{enumerate}
\item[i)] For each $p\in \ZZ_{\geq 0}$ there are uniquely
determined\il{tom}{$\tilde\omega_\n^{p,0},\;\tilde\omega_\n^{0,p}$}
\il{tmu}{$\tilde\mu_\n^{p,0},\;\tilde\mu_\n^{0,p}$}
\il{tups}{$\tilde\upsilon_\n^{p,0},\;\tilde\upsilon_\n^{0,p}$} families
$\tilde x^{p,0}_\n(j,\nu)$ and $\tilde x^{0,p}_\n(j,\nu)$ such that
\begin{enumerate}
\item[a)] these families are holomorphic in $\nu$ on
$\CC\setminus \ZZ_{\leq -1}$, and are non-zero for each
$\nu \in \CC\setminus\ZZ_{\leq -1}$;
\item[b)]
$x^{p,0}_\n(j,\nu) = \varphi_+^p(j,\nu)\, \tilde x^{p,0}_\n(j,\nu)$ and
$x^{0,p}_\n(j,\nu) = \varphi_-^p(j,\nu)\, \tilde x^{0,p}_\n(j,\nu)$
with holomorphic functions $\nu\mapsto \varphi_\pm^p(j,\nu)$
on~$\CC\setminus \ZZ_{\leq -1}$;
\item[c)] $\tilde x^{p,0}_\n(j,\nu)$ is determined by its minimal
component, and $\tilde x^{0,p}_\n(j,\mu)$ is determined by its maximal
component, given explicitly in Table~\ref{tab-mmc}.
\begin{table}[tp]
\[
\renewcommand\arraystretch{1.4}
\begin{array}{|c|c|c|}\hline
& \bigl(p\leq m_0(j), \;\varepsilon=1\bigr)\text{ or }\varepsilon=-1
& p\geq m_0(j),\; \varepsilon=1 \\
\tilde x_\n^{p,0} & \text{min.\ comp.\ }f_{-p} & \text{min.\ comp.\
}f_{r_0(2j+3p)}
\\\hline
\tilde\omega_\n^{p,0}& t^{p+1}\, W_{\k_0(j),\nu/2}(2\pi|\ell|t^2) &
t^{m_0(j)+1} \, W_{\k,\nu/2}(2\pi|\ell|t^2)
\\
\tilde\upsilon_\n^{p,0}& t^{p+1}\, V_{\k_0(j),\nu/2}(2\pi|\ell|t^2)
& t^{m_0(j)+1} \, V_{\k,\nu/2}(2\pi|\ell|t^2) \\
\tilde\mu_\n^{p,0}& t^{p+1}\, M_{\k_0(j),\nu/2}(2\pi|\ell|t^2) &
t^{m_0(j)+1} \, M_{\k,\nu/2}(2\pi|\ell|t^2)
\\
\
& \k_0(j)= -m_0(j) -\frac{\varepsilon j+1}2
& \k=-p -\frac{j+1}2 \\ \hline
&\varepsilon=1 \text{ or } \bigl( p\leq m_0(j),\; \varepsilon=-1\bigr)
& p\geq m_0(j),\, \varepsilon=-1
\\
\tilde x^{0,p}_\n&\text{max.\ comp.\ } f_p & \text{max.\ comp.\
}f_{r_0(2j-3p)}
\\ \hline
\tilde\omega_\n^{0,p}& t^{p+1}\, W_{\k_0(j),\nu/2}(2\pi|\ell|t^2) &
t^{m_0(j)+1} \, W_{\k,\nu/2}(2\pi|\ell|t^2)
\\
\tilde\upsilon_\n^{0,p}& t^{p+1}\, V_{\k_0(j),\nu/2}(2\pi|\ell|t^2)
& t^{m_0(j)+1} \, V_{\k,\nu/2}(2\pi|\ell|t^2) \\
\tilde\mu_\n^{0,p}& t^{p+1}\, M_{\k_0(j),\nu/2}(2\pi|\ell|t^2) &
t^{m_0(j)+1} \, M_{\k,\nu/2}(2\pi|\ell|t^2)
\\
& \k_0(j)= -m_0(j) -\frac{\varepsilon j+1}2
& \k=-p +\frac{j-1}2 \\ \hline
\end{array}
\]
\caption[]{Determining components of the families $\tilde x^{p,0}_\n$
and $\tilde x^{0,p}_\n$ for
$x\in \{\omega,\upsilon,\mu\}$.}\label{tab-mmc}
\end{table}
\end{enumerate}

\item[ii)] \begin{enumerate}
\item[a)]$\sh 3{-1}\tilde x_\n^{p,0}(j,\nu)=0$ and
$\sh{-3}{-1}\tilde x^{0,p}_\n(j,\nu)=0$ for all $\nu \in \CC$, or for
all $\nu \in \CC\setminus \ZZ_{\leq -1}$ if $x=\mu$.

\item[b)]For some of the other shift operators we know the behavior on
these families explicitly: $\sh31\tilde x^{p,0}_\n$ and
$\sh{-3}1 \tilde x^{0,p}_\n$ in Table~\ref{tab-x-up}, and
$\sh{-3}{-1}\tilde x^{p,0}_\n$ and $\sh3{-1}\tilde x^{0,p}$ in
Table~\ref{tab-x-dn}.
\end{enumerate}

\item[iii)]The functions $\tilde \omega^{p,0}\bigl(j,\nu;n\am(t)k\bigr)$
and $\tilde \omega^{0,p}\bigl(j,\nu;n\am(t)k\bigr)$ have exponential
decay as $t$ tends to~$\infty$.
\end{enumerate}
\end{prop}

\begin{table}[tp]
\[ \renewcommand\arraystretch{1.4}
\begin{array}{|c|cc|}\hline
& \sh{3}{1}\tilde x^{p,0}_\n/\tilde x^{p+1,0}_\n&\\ \hline
\varepsilon=1,\;0\leq p< m_0(j) & i
\sqrt{2\pi|\ell|}\,\sqrt{m_0(j)-p}&\\
\hline
\varepsilon=1, \; p\geq m_0(j) & \bigl( p+1+\frac{j+\nu}2\bigr)\,
\bigl(p+1+\frac{j-\nu}2\bigr)
& \tilde \omega^{p,0}_\n\\
&-1 & \tilde \upsilon^{p,0}_\n \\
& p+1+\frac{j+\nu}2
& \tilde \mu^{p,0}_\n \\ \hline
\varepsilon=-1,\; p\geq 0 & -i \sqrt{2\pi|\ell|}\,
\sqrt{m_0(j)+1+p}&\\\hline
\hline& \sh{-3}1 \tilde x^{0,p}_\n/\tilde x^{0,p+1}_\n& \\
\hline
\varepsilon=1& -i \sqrt{2\pi|\ell|}\,\sqrt{1+m_0(j)+p}
&\\\hline
\varepsilon=-1,\;0\leq p< m_0(j) & i\sqrt{2\pi|\ell|}\,
\sqrt{m_0(j)-p}&\\
\hline
\varepsilon=-1,\; p\geq m_0(j)&\bigl( p+1-\frac{j+\nu}2\bigr)\, \bigl(
p+1
+\frac{\nu-j}2\bigr)
& \tilde \omega^{0,p}_\n\\
& -1 & \tilde \upsilon^{0,p}_\n\\
& p+1+\frac{\nu-j}2
& \tilde \mu^{0,p}_\n
\\ \hline
\end{array}
\]
\caption[]{Upward shift operator on families $\tilde x^{p,0}_\n$ and
$\tilde x^{0,p}_\n$.\\
The factors in the table are equal to the quotients
$\varphi_+^p(j,\nu)/\varphi_-^p(j,\nu)$ of the holomorphic functions in
i)b) in Proposition~\ref{prop-extr.na}. }\label{tab-x-up}
\end{table}

\begin{table}[tp]
\[ \renewcommand\arraystretch{1.8}
\begin{array}{|c|cc|}\hline
& \sh{-3}{-1}\tilde x^{p,0}_\n/\tilde x^{p-1,0}_\n & \\ \hline
\varepsilon=1,\; 1\leq p \leq m_0(j) & \frac{ ip \, \bigl(
p+\frac{j-\nu}2\bigr)\, \bigl( p+\frac{j+\nu}2\bigr) } {(p+1)\,
\sqrt{2\pi|\ell|(m_0(j)-p+1)}}& \\
\hline
\varepsilon=1,\; p>m_0(j)& \frac{-p}{p+1}&\tilde\omega^{p,0}_\n \\
& \frac p{p+1}\bigl( p + \frac{j+\nu}2 \bigr)\, \bigl( p+\frac{j-\nu}2
\bigr)
& \tilde\upsilon^{p,0}_\n\\
& -\frac p{p+1} \bigl( p+\frac{j-\nu}2 \bigr) &\tilde \mu^{p,0}_\n
\\ \hline
\varepsilon=-1& \frac{ -ip \bigl( p+ \frac{j-\nu}2\bigr)\, \bigl( p
+ \frac{j+\nu}2\bigr)} {(p+1)\, \sqrt{2\pi |\ell|(m_0(j)+p)}}& \\
\hline\hline
& \sh{3}{-1}\tilde x^{0,p}_\n/\tilde x^{0,p-1}_\n & \\ \hline
\varepsilon=1& \frac{-i p \,\bigl(p-\frac{j+\nu}2\bigr)\, \bigl(
p-\frac{j-\nu}2\bigr)} { (p+1)\, \sqrt{2\pi|\ell|\,(m_0(j)+p)}}
&\\
\hline
\varepsilon=-1,\; 1\leq p \leq m_0 & \frac{ ip \,\bigl( p -
\frac{j+\nu}2\bigr)\, \bigl( p - \frac{j-\nu}2\bigr)}
{(p+1)\,\sqrt{2\pi|\ell|\,(m_0(j)-p+1)}}&
\\ \hline
\varepsilon=-1,\; p>m_0(j) & -\frac p{p+1}&\tilde \omega^{0,p}_\n\\
& \frac p{p+1}\, \bigl( p - \frac{j+\nu}2\bigr)\, \bigl( p+
\frac{\nu-j}2\bigr)
& \tilde\upsilon^{0,p}_\n\\
& -\frac p{p+1}\,\bigl(p-\frac{j+\nu}2\bigr)& \tilde \mu^{0,p}_\n \\
\hline
\end{array}
\]
\caption[]{ Downward shift operators on families $\tilde x^{p,0}_\n$ and
$\tilde x^{0,p}_\n$ for $p\in \ZZ_{\geq 1}$.\\
We note that $\sh3{-1}\tilde x^{p,0}_\n$ and
$\sh{-3}{-1}\tilde x^{0,p}_\n$ are identically zero. }\label{tab-x-dn}
\end{table}

\rmrk{Remarks} 1. We recall the use of $m_0(j)$ and $r_0(h)$ in this
proposition and the accompanying tables. In the general situation of
Table~\ref{tab-parms-na} the quantities $m_0=m_0(j)$ and $r_0=r_0(h)$
have the following significance:
\begin{itemize}
\item $m_0(j)\geq 0$ is equivalent to the occurrence of the $K$-type
$\tau^{2j}_0$ in $\Ffu_\n$.
\item $r_0(h)$ determines which components $f_r$ of an element of
$\Ffu_{\n;h,p,p}$ can be non-zero, namely
$\max\bigl(r_0(h),-p\bigr)\leq r\leq p$ if $\varepsilon=\sign(\ell)=1$,
and $-p\leq r \leq \min\bigl(r_0(h),p\big)$ if $\varepsilon=-1$.
\end{itemize}

In Proposition~\ref{prop-extr.na} we have that $h=2j+3p$ for
$\tilde x^{p,0}_\n(j,\nu)$ and $h=2j-3p$ for
$\tilde x^{0,p}_\n(j,\nu)$. This interpretation leads to the scheme in
Table~\ref{tab-r0m0}.
\begin{table}
\[
\begin{array}{|c|cc|}\hline
\varepsilon& \tilde x^{p,0}_\n & \tilde x^{0,p}_\n \\ \hline
1 & \begin{array}{rl}m_0\geq 0
&\Leftrightarrow r_0 \leq p \\
m_0 \geq p & \Leftrightarrow r_0\leq -p \text{ (ac)}
\end{array}
& m_0\geq 0 \Rightarrow r_0\geq p \text{
(ac)} \\ \hline
-1 & m_0 \geq 0 \Rightarrow r_0 \geq p \text{ (ac)} &
\begin{array}{rl}
m_0\geq 0 & \Leftrightarrow r_0\geq -p\\
m_0\geq p& \Leftrightarrow r_0\geq p \text{ (ac)}
\end{array}
\\ \hline
\end{array}
\]
\caption{Relation between $m_0=m_0(j)$ and $r_0=r_0(h)$. By (ac) we
indicate that the value of $r_0$ imposes no restriction on the
components with $|r|\leq p$.} \label{tab-r0m0}
\end{table}\smallskip

\noindent
2. \ Proposition \ref{prop-extr.na} may be compared to
Lemma~\ref{lem-ext-ab} in the generic abelian cases. Since in the
abelian case the upward shift operators are injective, there was in
Section~\ref{sect-abip} no need to introduce families $\tilde x^{p,0}$
and $\tilde x^{0,p}$ by dividing out zeros.

\begin{proof}[Proof of Proposition~\ref{prop-extr.na}]In the case when
$p=0$ we take $\tilde x^{0,0}_\n(j,\nu)$ equal to $x^{0,0}_\n(j,\nu)$,
and the assertions hold by \S\ref{sect-Nnab1d}. We proceed by
induction. Most steps can be carried out by hand with the description
of the upward shift operators in Table \ref{tab-shnab},
p~\pageref{tab-shnab}. We prefer to carry out all steps with
Mathematica. See \cite[\S20]{Math}.

In many steps the determining component of $x^{p,0}_\n$, respectively
$x^{0,p}_\n$, is multiplied by a simple non-zero factor:
\be\label{ind1}
\renewcommand\arraystretch{1.4}
\begin{array}{|cc|cc|}\hline
& x^{p,0}_\n\mapsto x^{p+1,0}_\n \\
\hline
\varepsilon=1,\;p< m_0 & i \sqrt{2\pi|\ell|}\,\sqrt{m_0-p}\\
\varepsilon=-1& -i \sqrt{2\pi|\ell}\, \sqrt{m_0+1+p}\\ \hline
& x^{0,p}_\n\mapsto x^{0,p+1}_\n\\ \hline
\varepsilon=1& - i\sqrt{2\pi|\ell|}\,\sqrt{1+m_0+p} \\
\varepsilon=-1,\; p< m_0 & i\sqrt{2\pi|\ell|}\, \sqrt{m_0-p}\\
\hline
\end{array}
\ee
This gives in many cases part i) of the proposition for the action of
$\sh 31 $ on $\tilde x^{p,0}_\n$, and for the action of $\sh{-3}1$ on
$\tilde x^{0,p}_\n$. Since
\[ \sh 3 1 \tilde x^{p,0}_\n = \bigr(\varphi_+^p\bigr)^{-1} \sh 31
x^{p,0}_\n = \bigr(\varphi_+^p\bigr)^{-1} x^{p+1,0} _\n=
\varphi_+^{p+1} \bigl(\varphi_+^p\bigr)^{-1} \tilde x^{p+1}_\n \,,\]
and similarly for $\sh{-3}1 \tilde x^{0,p}_\n$, the entries in
Table~\ref{tab-x-up} give the quotients of successive values of
$\varphi_\pm^p$. Hence the factors $\varphi_\pm^p$ are essentially
known. See \cite[\S20a]{Math}.

The remaining cases, with $p\geq m_0(w)$, are more complicated. The
lowest component $x^{p,0}_\n$ has order $r_0(h)$ and for $x^{p+1,0}_\n$
the order increases to $r_0(h+3)=r_0(h)+1$. This has the consequence
that we need also the component of $x^{p,0}_\n$ of order $r_0(h)+2$.
This can be determined by the kernel relation for $\sh3{-1}$ in
Table~\ref{tab-krnab}. The lowest component of $\sh 3 1 x^{p,0}_\n$ can
now be expressed in the lowest component of $F$. It is necessary to
write it in the form $t^{ m_0+1} w_{\k,\nu/2}(2\pi |\ell|t^2)$ for
$w_{\k,\nu}$ equal to one of the three Whittaker functions
$W_{\k,\nu/2}$, $V_{\k,\nu/2}$ and $M_{\k,\nu/2}$. The computations in
this case and in the case of $x^{0,p}_\n$ are in
\cite[\S20b]{Math}.\smallskip

Proposition \ref{prop-dsok1} and equation~\eqref{ect-ups} give the
vanishing of the holomorphic families $\sh 3{-1}x^{p,0}_\n$ and
$\sh{-3}{-1} x^{0,p}_\n$ on $\CC\setminus \ZZ_{\leq -1}$. Since
$\tilde x^{p,0}_\n = x^{p,0}_\n \bigm/ \varphi_+^p$ is holomorphic on
$\CC\setminus \ZZ_{\leq 1}$, the vanishing of
$\sh 3{-1}\tilde x^{p,0}_\n$ follows. We proceed similarly for
$\tilde x^{0,p}_\n$.

The action of the downward shift operators in Table~\ref{tab-x-dn} is
obtained by computations similar to those for the upward shift
operator. We have to use that there is one downward shift operator for
which the image is zero, by Proposition \ref{prop-dsok1}, and
equation~\eqref{ect-ups} for $x=\upsilon$, to get a relation between
the components. See \cite[\S20cd]{Math}.

The minimal component of $\tilde \omega^{p,0}_\n(j,\nu)$ and the maximal
component of $\tilde \omega^{0,p}_\n(j,\nu)$ are of the form
$t^c \, W_{\k,s}$ with $c\in \ZZ$, $\k\in \frac12\ZZ$ and $\nu\in \CC$.
The recursive relations between the components used in the proof of
Lemma~\ref{lem-dcpt}, together with \eqref{crWd}, imply that the other
components are linear combinations of functions of the same form, and
hence have exponential decay at~$\infty$ according to~\eqref{Wae}.
\end{proof}

Relation~\eqref{MinWV} implies that each $\tilde \mu_\n^{p,0}(j,\nu)$ is
a linear combination of $\tilde \omega_\n^{p,0}(j,\nu)$ and
$\tilde \upsilon^{0,p}_\n(j,\nu)$, and analogously for
$\tilde \mu^{0,p}_\n(j,\nu)$. In general both coefficients in the
linear combination are non-zero. If $\nu\equiv j\bmod 2$ and
$\nu\geq 0$, then the following special relations occur.
\begin{lem}\label{lem-mu-omups}Let $\n=(\ell,c,d)$, put
$\varepsilon=\sign(\ell)$, and let $j\equiv p \equiv \nu\pmod 2$,
$p,\nu\in \ZZ_{\geq 0}$. Suppose that $m_0(j) \geq 0$.
\begin{enumerate}
\item[i)]
$\tilde \mu_\n^{p,0}(j,\nu)\dis \tilde \upsilon_\n^{p,0}(j,\nu)$ in the
following cases:
\begin{itemize}
\item $\varepsilon=1$ and $-j+\nu \leq 2\max\bigl( p,m_0(j)\bigr)$,
\item $\varepsilon=-1$ and $j+\nu \leq 2m_0(j)$.
\end{itemize}
\item[ii)]
$\tilde \mu_\n^{0,p}(j,\nu) \dis \tilde \upsilon_\n^{0,p}(j,\nu)$ in
the following cases:
\begin{itemize}
\item $\varepsilon=1$ and $-j+\nu \leq 2m_0(j)$,
\item $\varepsilon=-1$ and $j+\nu \leq 2\max\bigl( p, m_0(j) \bigr)$.
\end{itemize}
\item[iii)]
$\tilde \mu^{p,0}_\n(j,\nu) \dis \tilde \omega^{p,0}_\n(j,\nu)$ in the
following cases:
\begin{itemize}
\item $\varepsilon=1$ and $j+\nu \leq -2-2\max\bigl( p,m_0(j)\bigr)$,
\item $\varepsilon=-1$ and $-j+\nu \leq -2-2m_0(j)$.
\end{itemize}
\item[iv)]
$\tilde \mu^{0,p}_\n(j,\nu) \dis \tilde \omega^{0,p}_\n(j,\nu)$ in the
following cases:
\begin{itemize}
\item $\varepsilon=1$ and $j+\nu \leq -2-2m_0(j)$,
\item $\varepsilon=-1$ and $-j+\nu \leq -2-2\max\bigl( p,m_0(j)\bigr)$.
\end{itemize}
\end{enumerate}
\end{lem}
By $\dis$ we denote equality up to a non-zero factor.

We note that the conditions in i) and ii)
and the corresponding conditions in iii)
and~iv) exclude each other.
\begin{proof}With \eqref{MWV} we have $M_{\k,s} \dis V_{\k,s}$ if and
only if $\frac12+\k+s\in \ZZ_{\leq 0}$. We apply this to the
determining components of $\tilde \mu_\n^{p,0}(j,\nu)$ and
$\tilde \upsilon_\n^{p,0}(j,\nu)$ as given in Table~\ref{tab-mmc},
p~\pageref{tab-mmc}, to get $j\equiv \nu \bmod 2$ and
\begin{align*} j + \nu &\leq 2m_0(j)& \text{ if }&\varepsilon=-1\,,\\
-j+\nu &\leq 2m_0(j)& \text{ if }
&\varepsilon=1 \text{ and } p\leq m_0(j)\,,\\
-j+\nu &\leq 2p&\text{ if }
&\varepsilon=1\text{ and }p\geq m_0(j)\,.
\end{align*}
This can be reformulated as the conditions in~i). For ii)
we proceed analogously.

Equation \eqref{MWV} gives also $M_{\k,s} \dis W_{\k,s}$ if and only if
$\frac12-\k+s\in \ZZ_{\leq 0}$. We use again Table~\ref{tab-mmc} to get
for $\tilde \mu^{p,0}_\n(j,\nu) \dis \tilde \omega^{p,0}_\n(j,\nu)$ the
conditions
\begin{align*}
-j+\nu &\leq -2\bigl( m_0(j)+1\bigr)
&\text{ if }&\varepsilon=-1\,,\\
j+\nu &\leq -2\bigl(m_0(j)+1\bigr)& \text{ if }& \varepsilon=1\text{ and
} p\leq m_0(j)\,,\\
j+\nu &\leq
-2(p+1) &\text{ if }& \varepsilon=1\text{ and } p\geq m_0(j)\,.
\end{align*}
This gives iii), and, analogously,~iv).
\end{proof}

\subsection{Intersection of kernels}The intersection of kernels of
downward shift operators was considered for the abelian cases in
Propositions \ref{prop-F0do} and~\ref{prop-nsefa}. Here again, we use
the notations and conventions of Lemma~\ref{lem-kdso}.

\begin{prop}\label{prop-ik}Let $\tau^h_p$ be a $K$-type occurring in
$\Ffu_\n$. Denote by \il{Khpn}{$K_{\n;h,p}$}$K_{\n;h,p}$ the
intersection of the kernels of
$\sh 3{-1}:\Ffu_{\n;h,p,p} \rightarrow \Ffu_{\n;h+3,p-1,p-1}$ and of
$\sh{-3}{-1}:\Ffu_{\n;h,p,p} \rightarrow \Ffu_{\n;h-3,p-1,p-1}$.

We define\ir{kkVW}{\kk^W_{\n;h,p},\; \kk^V_{\n;h,p}}
\badl{kkVW} \kk^W_{\n;h,p} &=\sum_r \vartheta_{m(h,r)}\,
t^{p+1}\,c^W(r)\, W_{\k(r),s(r)}(2\pi|\ell|t^2)\, \Kph h{p}{r}{p}\,,\\
\kk^V_{\n;h,p}
&=\sum_r \vartheta_{m(h,r)}\, t^{p+1}\, c^V(r)\,
V_{\k(r),s(r)}(2\pi|\ell|t^2)\, \Kph h{p}{r}{p}\,,
\eadl
where $r\equiv p\bmod 2$, $|r|\leq p$, $m(h,r) \geq 0$, and where
$\th_m$ is an abbreviation for $\Th_{\ell,c}(h_{\ell,m})$. We use
$m(h,r)$ as indicated in Table~\ref{tab-parms-na},
p~\pageref{tab-parms-na}, and
\be\label{kapr} \k(r) = -m(h,r)-\varepsilon \, s(h,r)
-\frac12, \qquad s(r) = s(h,r)= \frac{h-r}4\,.\ee
The coefficients are\ir{cWVexpl}{c^W(r)\,, c^V(r)}
\be \label{cWVexpl}
c^W(r) \= i^{m(h,r)}\sqrt{m(h,r)!}\,,\qquad c^V(r)
\= \frac{(-1)^{m(h,r)}}{\sqrt{m(h,r)!}}\,,\ee
for $r\equiv p\bmod 2$ such that $m(h,r)\geq 0$.

\begin{enumerate}
\item[i)] \begin{enumerate}
\item[a)] If $|r_0(h)|>p$, then $K_{\n;h,p}$ has dimension $2$, and is
spanned by $\kk^W_{\n;h,p}$ and $\kk^V_{\n;h,p}$.
\item[b)] If $|r_0(h)|\leq p$, then $\kk^V_{\n;h,p}$ spans $K_{\n;h,p}$.
\end{enumerate}

\item[ii)] The subspace $K_{\n;h,p}$ of the large Fourier term module
$\Ffu_\n$ is contained in the Fourier term module $\Ffu_\n^\psi$, where
$\psi=[-h,p]$. With the notation of Lemma~\ref{lem-kdso}, we have the
following equalities up to a non-zero factor:
\begin{enumerate}
\item[a)] If $m_0(j_1)\geq 0$ and $m_0(j_2) \geq 0$, then
\bad \kk^W_{\n;h,p} & \ddis \tilde \omega^{p,0}_\n(j_1,\nu_1) \ddis
\tilde \omega^{0,p}_\n(j_2,\nu_2)\,,\\
\kk^V_{\n;h,p} &\ddis \tilde \upsilon^{p,0}_\n(j_1,\nu_1)
\ddis \tilde \upsilon^{0,p}_\n(j_2,\nu_2)\,,\\
\ead
(If $r_0(h)=-\e p$, then $\kk^W_{\n;h,p}\not\in K_{\n;h,p}$.)
\item[b)] If $\varepsilon=1$ and $0\leq m_0(j_1)<p$ (and hence
$m_0(j_2)<0$), then
\be \kk^V_{\n;h,p}\ddis \tilde\upsilon^{p,0}_\n(j_1,\nu_1) \,.\ee
\item[c)] If $\varepsilon=-1$ and $0\leq m_0(j_2)<p$ (and hence
$m_0(j_1))<0$), then
\be \kk^V_{\n;h,p}= \tilde\upsilon^{0,p}_\n(j_2,\nu_2)\,.\ee
\end{enumerate}
\end{enumerate}
\end{prop}

\noindent\rmrk{Remarks}(1) The Whittaker functions $W_{\k,s}$ and
$V_{\k,s}$ are well defined and linearly independent for all values of
the parameters. In Proposition~\ref{prop-kkM} we will define
$\kk^M_{\n;h,p}$ based on $M_{\k,s}$ in a similar way.

(2) The condition $|r_0(h)|>p$ for dimension $2$ in~i) is stricter than
the conditions on $m_0(j_1)$ and $m_0(j_2)$ in~ii)a). This part is
valid if one of the $m_0$'s is equal to $p$ and the other equal to~$0$.
In that case $\kk^W_{\n;h,p}$ is not an element of~$K_{\n;h,p}$.

(3) This proposition is analogous to Proposition~\ref{prop-nsefa} in the
abelian case. In Proposition~\ref{prop-idkk} we will discuss a result
analogous to Corollary~\ref{cor-jl+r-ab}.

\begin{proof}Suppose that
\[ F= \sum_ r \vartheta_{m(h,r)} \,f_r \, \Kph hprp\]
is an element of $K_{\n;h,p}$. Then it has to satisfy the kernel
relations in Table~\ref{tab-krnab}, p~\pageref{tab-krnab}.

The computations are mostly done with Mathematica. In \cite[\S21a]{Math}
we check that the components are indeed given by Whittaker functions.

We first consider the case that only one component can be non-zero. That
happens if $m(h,\varepsilon p)=0$, with component $f_{\varepsilon p}$.
Then the kernel relations impose a linear differential equation for
$f_{\varepsilon p}$ which implies that it is of the form
\be\label{1compn} f_{\varepsilon p}(t) \ddis t^{p+1}\, V_{\k(\varepsilon
p), s(h,\varepsilon p)}(2\pi |\ell|t^2)
\,.\ee
Hence $F \ddis \kk^V_{\n;h,p}$ spans $K_{\n;h,p}$ in the case of one
component. Thus we have i)b) if $m(h,\e p)=0$.

In all other cases there are more components to consider. We take $r$
such that $f_r$ and $f_{r+2}$ can occur in the sum. We combine the
kernel relations to get a second order differential equation that
implies that $f_r (t) = t^{p+1}\, g_r(2\pi|\ell|t^2)$ where $g$ is a
solution of the Whittaker differential equation \eqref{Wh-deq} with
parameters $\k(r)$ and $s(h,r)$. So each component is in a well-defined
two-dimensional subspace of $C^\infty(0,\infty)$.

The kernel relations involve differentiations and multiplications by
powers of $t$. The first six contiguous relations in \eqref{crWd} and
\eqref{crW1} imply that if one component has the form
$t^{p+1}\, W_{\k,s}(2\pi|\ell|t^2)$, then all other components can be
expressed in $W$-Whittaker functions. Analogously for $V$-Whittaker
functions.

So we look for expressions of the form given in \eqref{kkVW}, and try to
determine how the coefficients are related. For that purpose we need
also contiguous relations in which the parameters are shifted by
$\frac12$. The complicated computations are in \cite[21b]{Math} and
lead to recursive relations for the coefficients, for which
\eqref{cWVexpl} gives solutions.

If $|r_0(h)|\leq p$ there is one more kernel relation. The case
$r_0(h)=\varepsilon p$ has already been discussed. If $\varepsilon=1$
the find that the $W$-Whittaker function does not satisfy the relation.
So in this case we are left with $\kk^V_{\n;h,p}$.

Let $\varepsilon=-1$ and $-p< r_0(h) \leq p$. In this case as well, only
$\kk^V_{\n;h,p}$ satisfies the kernel relations. This completes the
proof of~i).
\smallskip

For~ii) we consider the identifications up to a non-zero factor. The
results then imply that $K_{\n;h,p} \subset \Ffu_\n^\psi$.

It suffices to compare the determining components in Table~\ref{tab-mmc}
with the corresponding component of the functions in~\eqref{kkVW}. We
have to check the resulting relations for Whittaker functions, which we
consider in Lemma~\ref{lem-VWrels} below. Part i) of the lemma deals
with the determining components for ii)a). We work with $W_{\k,s}$ and
$V_{\k,s}$, which are even in~$s$.

Part~ii) of the lemma gives the proof of ii)b) and~ii)c). In the proof
of ii)b) and ii)c) we use that $\nu_1=\pm (2p+j_1)$ and
$\nu_2=\pm(2p+j_2)$. This is in agreement with the kernel conditions in
Table~\ref{tab-x-dn}, and seem to make the proof circular. However, it
is also a consequence of $h=2j_1+2p=2j_2-3p$ together with the
relations in Lemma~\ref{lem-kdso}.
\end{proof}

\begin{lem}\label{lem-VWrels}We use the notations of
Proposition~\ref{prop-ik}.
\begin{enumerate}
\item[i)] The condition $m_0(j_1)\geq 0$, $m_0(j_2)\geq 0$ is equivalent
to $\e r_0(h) \geq p$.

These equivalent conditions imply equality of the parameters in the
Whittaker differential equation.
\bad \k(-p) &\= -m_0(j_1) - \frac{\e j_1+1}2\,,& \qquad s(-p) &\= \pm
\nu_1/2\,,\\
\k(p) &\= - m_0(j_2) - \frac{\e j_2+1}2\,,& \qquad s(p)
   &\= \pm \nu_2/2\,.
\ead
\item[ii)] The condition $m_0(j_2)<0 \leq m_0(j_1)$ is equivalent to
$\e=1$ and $-p < r_0(h) \leq p$. Under these equivalent conditions
\begin{align*} t^{m_0(j_1)+1 }& V_{-p-(j_1+1)/2,(h+p)/4}(2\pi|\ell|t^2)
\ddis t^{p+1}V_{\k(r_0(h)),s(r_0(h))}(2\pi|\ell|t^2)
\\
&\ddis t^{2+j_1+2p+m_0(j_1) } e^{\pi\ell t^2}\,.
\end{align*}
\item[iii)]The condition $m_0(j_1)<0 \leq m_0(j_2)$ is equivalent to
$\e=-1$ and $-p \leq r_0(h)<p$. Under these equivalent conditions
\begin{align*}
t^{m_0(j_2)+1} &V_{-p+(j_2-1)/2,(h-p)/4}(2\pi|\ell|t^2)
\ddis t^{p+1} V_{\k(r_0(h)),s(r_0(h))}(2\pi|\ell|t^2)\\
&\ddis t^{2-j_2+2p+m_0(j_2) } e^{\pi|\ell|t^2}\,.
\end{align*}
\end{enumerate}
\end{lem}
\begin{proof}The equivalences follow from $ m_0(j_2)- m_0(j_1)= - \e p$,
and $r_0(h) = p-2\e m_0(j_1)= -p-2\e m_0(j_2)$. The equality of the two
sets of parameters $\k$ and $s$ in i) can be checked by a computation,
using Tables \ref{tab-parms-na} and~\ref{tab-x-dn}, and
Lemma~\ref{lem-kdso}.

The parameters of the Whittaker functions in~ii) are not equal. Working
them out, we arrive at values of the parameters to which we can apply
the specialization in~\eqref{specWV}. This leads for both functions to
an explicit expression in $t$, with different non-zero factors. The
computations are carried out in \cite[21c]{Math}.
\end{proof}

\subsection{Dimension}
\begin{lem}\label{lem-dimhp}In the notations of
Proposition~\ref{prop-ik}: If $\dim \Ffu^\psi_{\n;h,p,p}>2$ then at
least one of the following statements holds:
\[ \dim \Ffu^\psi_{\n;h+3,p-1,p-1}>2\,,\qquad \dim
\Ffu^\psi_{\n;h-3,p-1,p-1}>2\,.\]
\end{lem}
\begin{proof}
We know that that $\dim\Ffu_{\n;h,p,p}^\psi$ is at least two, by the
presence of $b_\omega=\tilde\upsilon_\n^{p,0}(j_1,\nu_1)$ and
$b_\upsilon=\tilde \omega^{p,0}_\n(j_1,\nu_1)$ if $m_0(j_1)\geq 0$, and
the presence of $b_\omega =\tilde\upsilon_\n^{0,p}(j_2,\nu_2)$ and
$b_\upsilon =\tilde \omega^{0,p}_\n(j_2,\nu_2)$ if $m_0(j_2)\geq 0$. If
both $m_0(j_1)\geq 0$ and $m_0(j_2)\geq 0$, the two choices of
$b_\omega$ and $b_\upsilon$ are proportional.

Similarly, $\dim\Ffu_{\n;h+3,p-1,p-1}^\psi\geq 2$ if $m_0(j_2)\geq 0$,
and $\dim \Ffu^\psi_{\n;h-3,p-1,p-1}\geq 2$ if $m_0(j_1)\geq 0$. Since
the $K$-type $\tau^h_p$ occurs, we cannot have both $m_0(j_1)<0$ and
$m_0(j_2)<0$. We choose
\bad b_\omega^+ &= \tilde\omega^{p-1,0}_\n(j_1,\nu_1)\,,& b_\upsilon^+
&= \tilde\upsilon_\n^{p-1,0}(j_1,\nu_1)
&&\bigl(m_0(j_1)\geq 0\bigr)\,,\\
b_\omega^- &= \tilde\omega^{0,p-1}_\n(j_2,\nu_2)\,,& b_\upsilon^- &=
\tilde\upsilon_\n^{0,p-1}(j_2,\nu_2)
&&\bigl(m_0(j_2)\geq 0\big)\,. \ead
If $m_0(j_1)<0$, then $\Ffu^\psi_{\n;h-3,p-1,p-1}=\{0\}$, and we take
$b_\omega^+=b_\upsilon^+=0$. Similarly, we take
$b_\omega^-=b_\upsilon^-=0$ if $m_0(j_2)=0$.

To prove the lemma we suppose that $\dim \Ffu^\psi_{\n;h,p,p}>2$, and
that $\Ffu^\psi_{\n;h-3,p-1,p-1}$ and $\Ffu^\psi_{\n;h+3,p-1,p-1}$ are
either zero or spanned by $b_\omega^\pm$ and $b_\upsilon^\pm$. We take
$F\in \Ffu^\psi_{\n;h,p,p}$ that is not a linear combination of
$b_\omega$ and $b_\upsilon$, and put
$F^\pm =\sh{\mp 3}{-1} F \in \Ffu^\psi_{\n;h\mp3,p-1,p-1}$. See
Figure~\ref{fig-Fpm}.
\begin{figure}[ht]
\begin{center} \grf{5}{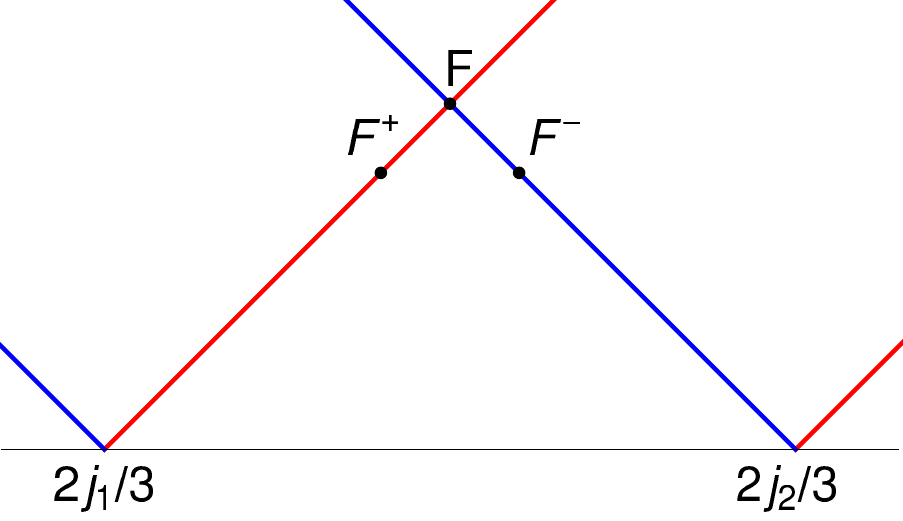}\end{center}
\caption{$F^\pm =\sh{\mp 3}{-1} F$.} \label{fig-Fpm}
\end{figure}
We can write these derivatives as
\be F^\pm = c_\omega^\pm \, b^\pm_\omega + c_\upsilon^\pm\,
b^\pm_\upsilon\,. \ee
That should lead to a contradiction.\smallskip

If both $F^+$ and $F^-$ are zero, then $F\in K_{h,p}$, hence $F$ is a
linear combination of $b_\omega$ and $b_\upsilon$. So we assume that at
least one of the derivatives is non-zero. We have many cases, requiring
essentially two different approaches. \medskip

\rmrk{Case $0< m_0(j_1)<p$ and $\varepsilon=1$}Then $m_0(j_2)<0$, and
$F^-=0$. Computations in \cite[\S22a]{Math}.

The function $F$ has components of order $r$ between $r_0(h)$ and $p$.
Since $\sh3{-1}F=0$ the function $F$ is determined by its lowest
component $f_{r_0}$. The lowest component of $F^+$ has order $r_0+1$
and a computation shows that it is equal to a non-zero multiple of
\[ 2 t f_{r_0}'-(4+h+2p-r_0+4\pi \ell t^2)
f_{r_0}\,.\]

In this case $b_\omega$ is not an element of
$K_{h,p} = \CC \, b_\upsilon$, and
$\sh{-3}{-1} b_\omega \dis b^+_\omega$. See Proposition \ref{prop-ik}.
So we can subtract a multiple of $b_\omega$ from $F$ to arrange that
$F\dis b^+_\upsilon$. With Table~\ref{tab-mmc} we get
\be\label{cpr} 2 t f_{r_0}'-(4+h+2p-r_0+4\pi \ell t^2) f_{r_0} \ddis
t^{0+1}\, V_{-p+1-(j_1+1)/2,\nu'/2}(2\pi|\ell|t^2)\,. \ee
This gives expressions for the derivatives of $f_{r_0}$ in terms of
$f_{r_0}$ and $V$-Whittaker functions.

In the eigenfunction equations for $r=r_0$ also terms with $f_{r_0+2}$
are present. Since $\sh 3 {-1}F=0$ we can express $f_{r_0+2}$ in terms
of $f_{r_0}$ by the kernel relation for $\sh 3{-1}$. Substitution of
all these expressions into the eigenfunction equations gives two linear
combination of $V$-Whittaker functions. The asymptotic behavior in
\eqref{Vae} shows that the implicit factor in \eqref{cpr} has to be
zero. Hence $f_{r_0}$ and $F$ have to vanish.

\rmrk{Case $m_0(j_1)=0$ and $\varepsilon=1$}So $m_0(j_2)<0$ and $F^-=0$.
Computations in \cite[\S22b]{Math}.

The function $F$ has in this case only one component, $f_{r_0}=f_p$. So
there is no need to determine $f_{r_0+2}$. The further computation is
similar to the previous one.

\rmrk{Cases $0 \leq m_0(j_2)<p$ and $\varepsilon=-1$}Then $m_0(j_1)<0$
and $F^+=0$. We proceed as in the previous cases. See
\cite[\S22cd]{Math}.

\rmrk{Cases $m_0(j_1)\geq 0$, $m_0(j_2)\geq 0$}Now both $F^+$ and $F^-$
may be nonzero. We can write
\be\label{Fomups} F^\pm = c^\pm_\omega \, b^\pm_\omega + c^\pm_\upsilon
\, b^\pm_\upsilon\,.\ee
The components of $F$ can have order $-p\leq r\leq p$, and the
components of $F^{\pm}$ can have order between $1-p$ and $p-1$. The
determining components of $b^\pm_\omega $ and $b^\pm_\upsilon$ have
order $\pm (1-p)$.

We compute the highest component $f^+_{p-1}$ of $F^+$, and the lowest
component $f^-_{1-p}$ of $F^-$. The function $f^\pm_{\mp(p-1)}$ is a
linear combination of $f_{\mp p}$, $f_{\mp(p-2)}$ and $f_{\mp(p-2)}'$.
Since we know the determining components of the functions in the right
hand side of \eqref{Fomups}, we can solve for $f_{\mp(p-2)}'$.
Substitution into the eigenfunction equations for $r=-p$ and for $r=p$
leads to a relation involving Whittaker functions only, which shows
that all four coefficients in \eqref{Fomups} have to vanish.

We carry out the actual computations for $\varepsilon=1$ and for
$\varepsilon=-1$ separately, in \cite[\S22ef]{Math}.

\rmrk{Conclusion} In all cases we conclude that the presence of
$F\in \Ffu^\psi_{\n;h,p,p}$ linearly independent of $b_\omega$ and
$b_\upsilon$ would lead to derivatives $F^+$ and $F^-$, of which at
least one is linearly independent of $b^\pm_\omega$ and
$b^\pm_\upsilon$.
\end{proof}

\begin{prop}\label{prop-dim2-nab}For $\psi\in \WOI$ the dimension of
$ \Ffu_{\n;h,p,p}^\psi $ is equal to 2 for all $K$-types $\tau^h_p$
that satisfy $\bigl| h-2j\bigr|\leq 3p$ for some
$j\in \wo^1(\psi)^+_\n$. All other $K$-types do not occur in
$\Ffu^\psi_\n$.
\end{prop}
\begin{proof}
By \S\ref{sect-Nnab1d} we have $\dim \Ffu_{\n;h,p,p}^\psi = 2$ for
$(h/3,p) = (2j/3,0)$ for all $j\in \wo^1(\psi)_n$.
Proposition~\ref{prop-extr.na} provides us with
$\tilde\omega^{p,0}_n(j,\nu)$ and $\tilde\upsilon^{p,0}_\n(j,\nu)$ in
$\Ffu^\psi_{\n;2j +3p,p,p}$ for each $p\geq 1$ and each
$j\in \wo^1(\psi)_\n$. These elements are linearly independent.
Similarly, the spaces $\Ffu^\psi_{\n;2j-3p,p,p}$ have dimension at
least two for $j\in \wo^1(\psi)_\n$. At least one of the upward shift
operators is injective, by ii) in Proposition~\ref{prop-kuso}. Hence
all spaces $\Ffu^\psi_{\n;h,p,p}$ indicated in the lemma have dimension
at least two.

For any point $(h/3,p)\in \sect(j)$ there is a path to the base point
$\bigl( 2j/3,0\bigr)$, corresponding to downward shift operators. Along
this path the dimension of $\Ffu^\psi_{\n;h,p,p}$ cannot decrease by
Lemma~\ref{lem-dimhp}. So all $K$-types mentioned in the lemma have
multiplicity exactly equal to~$2$.

Starting from points outside the sectors with base points $(2j'/3,0)$,
we obtain a path to a point $(h/3,0)$ on the horizontal axis, for which
$\dim \Ffu_{\n;h,0,0}^\psi=0$. This concludes the proof.
\end{proof}

\subsection{Special submodules}Inside $\Ffu_\n^\psi$ we have the special
submodules $\Wfu_\n^\psi$ and $\Mfu_\n^\psi$, defined by their behavior
as $t\uparrow\infty$ and $t\downarrow 0$ on $n\am(t)k$. See Definitions
\ref{defW[]} and~\ref{defM[]}. We define for $\psi\in \WOI$ subspaces
$\Wfu_\n^{\xi,\nu}$, $\Vfu_\n^{\xi,\nu}$ and $\Mfu^{\xi,\nu}$ of
$\Ffu_\n^\ps$, for $j\in \wo(\ps)_\n^+$. We show that $\Wfu^{\xi,\nu}$
and $\Vfu^{\xi,\nu}$ are $(\glie,K)$-modules, and that
$\Wfu^{\xi,\nu}\subset \Wfu^\ps_\n$ and
$\Mfu^{\xi,\nu}\subset \Mfu^\ps_\n$. We define
$\Vfu^\ps_\n \supset \Vfu^{\xi,\nu}_\n$, with properties similar to
$\Wfu_\n^\ps$.

\begin{defn}\label{def-VWMa}Let $\psi\in \WOI$. We define for
$(j,\nu)\in \wo(\psi)^+_\n$ the following $K$-modules
\il{Wfuxinu}{$\Wfu^{\xi,\nu}_\n$} \il{Mfuxinu}{$\Mfu^{\xi,\nu}_\n$}
\il{Vfuxinu}{$\Vfu^{\xi,\nu}_\n$}
\begin{align*} \Wfu^{\xi,\nu}_\n &= \sum_{p\geq 0} \sum_{a,b\geq 0}
\Bigl( U(\klie)\,
(\sh 31)^a (\sh{-3}1)^b \tilde\omega^{p,0}_\n(j,\nu)
+ U(\klie)\, (\sh 31)^a (\sh{-3}1)^b
\tilde\omega^{0,p}_\n(j,\nu)\Bigr)\,,
\displaybreak[0]
\\
\Vfu^{\xi,\nu}_\n &= \sum_{p\geq 0} \sum_{a,b\geq 0} \Bigl( U(\klie)\,
(\sh 31)^a (\sh{-3}1)^b \tilde\upsilon^{p,0}_\n(j,\nu)
+ U(\klie)\, (\sh 31)^a (\sh{-3}1)^b
\tilde\upsilon^{0,p}_\n(j,\nu)\Bigr)\,,
\displaybreak[0]
\\
\Mfu^{\xi,\nu}_\n &= \sum_{p\geq 0} \sum_{a,b\geq 0} \Bigl( U(\klie)\,
(\sh 31)^a (\sh{-3}1)^b \tilde\mu^{p,0}_\n(j,\nu)
+ U(\klie)\, (\sh 31)^a (\sh{-3}1)^b \tilde\mu^{0,p}_\n(j,\nu)\Bigr)\,.
\end{align*}
\end{defn}

\rmrk{Remarks}i) \ We use this complicated description, since in the
non-abelian case the upward shift operators are not always injective.
So $\omega^{a+p,b}_\n(j,\nu)$ may be zero in situations where
$(\sh31)^a (\sh{-3}1)^b \tilde \omega^{p,0}_\n(j,\nu)$ is non-zero.

ii) \ In the definition we speak of $K$-modules. We still have to show
that these spaces are invariant under the action of~$\glie$.

iii) \ The spaces $\Wfu_\n^{\xi,\nu}$ are contained in
$\Wfu_\n^{\ps[\xi,\nu]}$, and similarly for $\Mfu_\n^{\xi,\nu}$. Lemma
\ref{lem-mu-omups} implies that the modules $\Wfu^\psi_\n$ and
$\Mfu^\ps_\n$ may have non-zero intersection. This makes it useful to
use also the families $\tilde \upsilon_\n^{p,0}$ and
$\tilde \upsilon^{0,p}_\n$. We have not yet defined $\Vfu_\n^\psi$.

\begin{lem}\label{lem-Mxinumod}
The $(\klie,K)$-modules in Definition~\ref{def-VWMa} are
$(\glie,K)$-modules.
\end{lem}
\begin{proof}We use $x^{p,0}$ and $x^{0,p}$ as a general notation. The
shift operators on these elements satisfy
$\sh31x^{p,0} \in \CC\, x^{p+1,0}$,
$\sh{-3}{-1}x^{p,0}\in \CC \,x^{p-1,0}$, $\sh3{-1}x^{p,0}=0$, and
analogously for $x^{0,p}$. See Propositions \ref{prop-shdef},
\ref{prop-dsok1}, and \eqref{ect-ups}. Each element of $U(\glie)$ can
be written as a linear combination of
$u \Z_{31}^a\Z_{23}^b \Z_{32}^c \Z_{13}^d$, with $u\in U(\klie)$, and
$a,b,c,d\in \ZZ_{\geq 0}$. Applied to $x^{p,0} $ this can be rewritten,
with the expressions in Table~\ref{tab-sho}. p~\pageref{tab-sho}, as a
linear combination of elements of the form
$u' (\sh 3 1)^{a'} (\sh {-}1)^{b'}  (\sh{-3}{-1})^{d'} x^{p,0}$, which
is in the space under consideration. The image of $x^{0,p}$ is handled
analogously.
\end{proof}

The $(\glie,K)$-modules $\Wfu^\ps_\n$ and $\Mfu^\ps_\n$ in Definitions
\ref{defW[]} and~\ref{defM[]}.
\begin{lem}Let $\ps\in \WOI$.
\be \sum_{(j,\nu)\in \wo(\ps)^+_\n} \Wfu^{\xi_j,\nu}_\n \subset
\Wfu^\ps_\n\,,\qquad \sum_{(j,\nu)\in \wo(\ps)^+_\n}
\Mfu^{\xi_j,\nu}_\n \subset \Mfu^\ps_\n\,. \ee
\end{lem}
\begin{proof}Let $x$ stand for $\om$ or $\mu$, and $\X$ for $\Wfu$ or
$\Mfu$. Let $(j,\nu_0) \in \wo(\ps)_\n^+$.

The elements $x^{0,0}_\n(j,\nu)$ are in $\X^{\ps[j,\nu)]}_\n$ for all
$\nu$ in a neighborhood of $\nu_0$. This property is preserved by the
upward shift operators. However, they might give the result zero. In
Proposition~\ref{prop-extr.na} we form the elements $\tilde x^{p,0}$
and $\tilde x^{0,p}$ by dividing out such zeros. They are determined by
their minimal or maximal component, which has exponential decay for
$x=\om$ and $a$-regular behavior at $0$ for $x=\mu$, with
$a\geq 1+\min\bigl(p,m_0(j)\bigr)+\nu_0\geq 0$. The other components
are determined by the kernel relations in Table~\ref{tab-krnab},
p~\pageref{tab-krnab}. This clearly preserves exponential decay.

Application of the kernel relations in the case $x=\mu$ is problematic.
At each transition $r\mapsto r\pm 2$ we may loose a factor~$t$. Let us
look at the family $\nu\mapsto \mu^{p,0}(j,\nu)$, obtained by
differentiation; see~\eqref{xab}. It is $C^\infty$ in $g\in G$ and
$\nu$ in a neighborhood of $\nu_0$ and holomorphic in~$\nu$, and has
$\nu$-regular behavior at~$0$, according to Proposition~\ref{prop-inv}.
Its components have the form $t\mapsto t^\nu \, h_r(t,\nu)$ where $h_r$
extends holomorphically to $\CC$ times a neighborhood on~$\nu_0$. There
may be common zeros to be divided out in the recursion leading to
$\tilde \mu^{p,0}_\n(j,\nu)$. So we have
$h_r(t,\nu) = k_r(t,\nu)\, (\nu-\nu_0)^a$, with $a\in \ZZ_{\geq 0}$
depending on $p$, but not on~$r$. Then $\tilde \mu^{p,0}_\n(j,\nu_0)$
is a multiple of the function with components
$t\mapsto t^{\nu_0}\, k_r(t,\nu_0)$, and has $\nu_0$-regular behavior
at $0$. The same approach can be followed for
$\tilde \mu^{0,p}_\n(j,\nu_0)$.

Since $\Mfu^\ps_\n$ and $\Wfu^\ps_\n$ are $(\glie,K)$-modules, the lemma
follows.
\end{proof}

We define analogously the following
$(\glie,K)$-module.\ir{Vps}{\Vfu^\ps_\n}
\be \label{Vps}\Vfu^\ps_\n \= \sum_{(j,\nu)\in \wo(\ps)^+_\n}
\Vfu^{\xi_j,\nu}_\n\,.\ee

\begin{lem}\label{lem-VW-nab}Let $\ps\in \WOI$.
\begin{enumerate}
\item[i)]
$\Wfu_\n^\psi = \sum_{(j,\nu)\in \wo(\psi)_\n^+ }\Wfu_\n^{\xi_j,\nu}$.

\item[ii)] Let $(j,\nu), (j',\nu') \in \wo(\ps)_\n^+$. If the $K$-types
$\tau^h_p$ corresponds to a point
$(h/3,p) \in \sect(j) \cap \sect(j')$, then
$ \Wfu^{\xi,\nu}_{\n;h,p,p} = \Wfu^{\xi_{j'},\nu'}_{\n;h,p,p}$.
\end{enumerate}
Fix $(j,\nu)\in \wo(\ps)_\n^+$.
\begin{enumerate}

\item[iii)] The $K$-types $\tau^h_p$ occurring in $\Wfu^{\xi,\nu}_\n$
and in $\Vfu^{\xi,\nu}_\n$ correspond to the points
$(h/3,p) \in \sect(j)$, and have multiplicity one.

\item[iv)] We define for $a,b\in \ZZ_{\geq 0}$ the families
$\tilde\ups^{a,b}_\n$ and $\tilde \om^{a,b}$
by\ir{tuod}{\tilde\ups^{a,b}_\n} \il{tod}{$\tilde\om^{a,b}_\n$}
\be \label{tuod}
\tilde x^{a,b}_\n(j,\nu) \=
\begin{cases}
\bigl( \sh{-3}1\bigr)^b \tilde x^{a,0}(j,\nu)& \text{ if } \ell>0\,,\\
\bigl( \sh 3 1\bigr)^a \tilde x^{0,b}(j,\nu)&\text{ if }\ell<0\,,
\end{cases}
\ee
for $x=\ups$ or~$\om$.

Put $h=2j+3(a-b)$, $p=a+b$, The space $\Vfu^{\xi,\nu}_{\n;h,p,p}$ is
spanned by $\tilde \ups^{a,b}_\n(j_\xi,\nu)$, and
$\Wfu^{\xi,\nu}_{n;h,p,p}$ is spanned by
$\tilde\om^{a,b}_\n(j_\xi,\nu)$.

\item[v)] $\Wfu^{\xi,\nu}_\n \cap \Vfu_\n^{\xi,\nu}=\{0\}$.
\end{enumerate}
\end{lem}
\begin{proof}The kernel relations in Table~\ref{tab-krnab},
p~\pageref{tab-krnab}, and the contiguous relations in \eqref{crW1}
imply that the components of $\tilde \om^{p,0}_\n$ and $\tilde
\om^{0,p}_\n$ are linear combinations of functions $t^c
W_{\k+m,\nu/2}(2\pi|\ell|t^2)$ with $c,m\in \ZZ$. Application of the
upward and downward shift operators stay within the space of functions
in $\Ffu^\ps_\n$ with components of this form. The asymptotic behavior
\eqref{Wae} implies that these functions have exponential decay
at~$\infty$. Hence $ \Wfu_\n^{\xi,\nu}\subset \Wfu^\psi_\n$.

For the modules $\Vfu_\n^{\xi,\nu}$ we have similar descriptions, now
with $V$-Whittaker functions instead of $W$-Whittaker functions.
Lemma~\ref{lem-VI} implies that non-zero elements of $U(\glie)
\Vfu_\n^{j,\nu}$ cannot have exponential decay at~$\infty$. This
gives~v).

We get in particular $\dim \Wfu_{\n;h,p,p}^{\xi_j,\nu} \geq 1$ and $\dim
\Vfu_{\n;h,p,p}^{\xi_j,\nu} \geq 1$ for each $(h/3,p) \in \sect(j)$.
With Proposition~\ref{prop-dim2-nab} we conclude that these dimensions
are exactly equal to~$1$. This gives~iii). In this way we get also
$ \dim \Wfu^\ps_{\n;h,p,p}=1$ and $\dim \Vfu^\ps_{\n;h,p,p}=1$ for
$(h/3,p) \in \sect(j)$. Hence
$\Wfu^{\xi_j,\nu}_{\n;h,p,p}= \Wfu^\ps_{\n;h,p,p}$, and similarly for
$\Vfu$. This implies ii) and~i).

In Proposition~\ref{prop-extr.na} we constructed the non-zero families
$\tilde x^{a,0}_\n(j,\nu)$ and $\tilde x^{0,b}_\n(j,\nu)$ by dividing
out zeros of $x^{a,0}_\n(j,\nu)$ and $x^{0,b}_\n(j,\nu)$. The
injectivity of the shift operators $\sh{\mp 3}1$ for $\pm \ell>0$
(Proposition~\ref{prop-kuso}) leads to non-zero families $\tilde
x^{a,b}_\n(j,\nu)$, which span the corresponding highest weight spaces
by the multiplicity one of the $K$-types.
\end{proof}

This result gives immediately the following identifications, completing
ii)
in Proposition~\ref{prop-ik}.
\begin{prop}\label{prop-idkk}Let $\tau^h_p$, $K_{\n;h,p}$,
$\kk^V_{\n;h,p}$ and $\kk^W_{h,}$ be as in Proposition~\ref{prop-ik}.
\begin{enumerate}
\item [i)] $\kk^V_{\n;h,p}$ spans the space $\Vfu^\ps_{\n;h,p,p}\subset
K_{\n;h,p}$.
\item[ii)] If $r_0(h)>p$, then $\kk^W_{\n;h,p}$ spans the space
$\Wfu^\ps_{\n;h,p,p} \subset K_{\n;h,p}$.
\end{enumerate}
\end{prop}

\rmrk{Remark}This result is analogous to Corollary~\ref{cor-jl+r-ab} in
the generic abelian case. Here we cannot formulate the result in terms
of the families $\ups^{a,b}_\n$ and $\om^{a,b}_\n$, since these may be
zero at the relevant parameter values.

\rmrk{Vanishing of shift operators}We turn to the determination of the
lines in the $(h/3,p)$-plane corresponding to $K$-types in
$\Wfu^{\xi,\nu}_\n$ and $\Vfu^{\xi,\nu}_\n$ on which one of the shift
operators vanish.

\begin{lem}\label{lem-usho-upsom}Let $\psi\in \WOI$, and let $(j,\nu)
\in \wo(\psi)_\n^+$.
\begin{enumerate}
\item[i)] The upward shift operators in $\Vfu^{\xi,\nu}_\n$ are
injective.
\item[ii)] Let $\bigl( \frac h3,p\bigr)\in \sect(j)$. The upward shift
operator
\[ \sh{\pm 3}1 : \Wfu_{\n;h,p,p}^{\xi,\nu} \rightarrow
\Wfu^{\xi,\nu}_{\n;h\pm 3,p+1,p+1}\]
is zero if and only if $\pm \ell>0$ and there is $j'\in \wo^1(\psi)
\setminus \wo^1(\psi)_\n$ such that $m_0(j')<0$, and $h-2j'\pm 3p \pm
6=0$.
\end{enumerate}
\end{lem}
\rmrk{Remark} The number of such elements $j'$ is at most equal to~$2$.
We have in the notations of~\eqref{jnurels}:
\be \begin{array}{cc|c}
\e & j & \text{set of }j'\\ \hline
1 & j= l_l < j_+ & \{ j_+,j_r\}\\
& j=j_+ \neq j_r & \{j_r\}\\ \hline
-1 & j=j_r>j_+ & \{ j_l,j_+\}\\
& j_+ \neq j_l & \{j_l\}
\end{array}
\ee

\begin{proof}
Proposition~\ref{prop-kuso} shows that a non-zero kernel of
$\sh{\pm3}{-1}$ can occur if $\varepsilon=\sign(\ell)=\pm 1$ for
$K$-types on a line $h-2j'\pm 3p\pm 6=0$ with $j'\in \wo^1(\psi)$ such
that $m_0(j') <0$, which means $j'\in \wo^1(\psi)\setminus
\wo^1(\psi)_\n$. In the pictures these points are on the dashed lines
inside the sector.
\begin{center}
\grf{5}{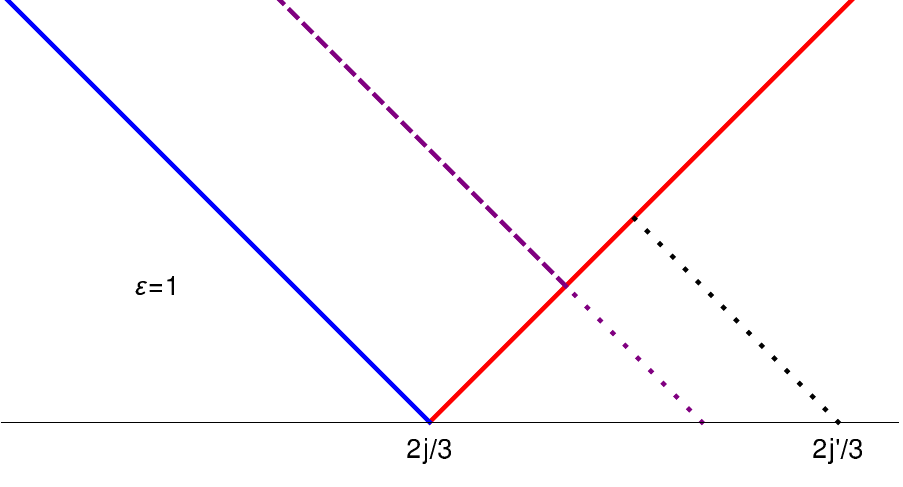}\qquad \grf{5}{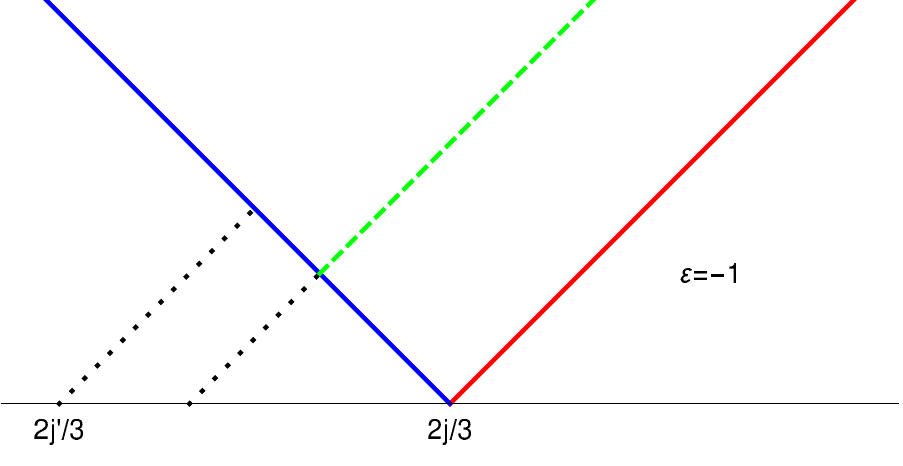}
\end{center}

Let $\varepsilon=1$, hence $\sh{-3}1$ is injective. Consider a point
$(h/3,p)$ on the line $h=2j'-6-3p$, and functions $y_j \in
\Ffu^\psi_{\n;h-3j,p+j,p+j}$ related by $\sh {-3}1 y_j = y_{j+1}$.
Suppose that $\sh31 y_0 = 0$.
\begin{center}\grf{5}{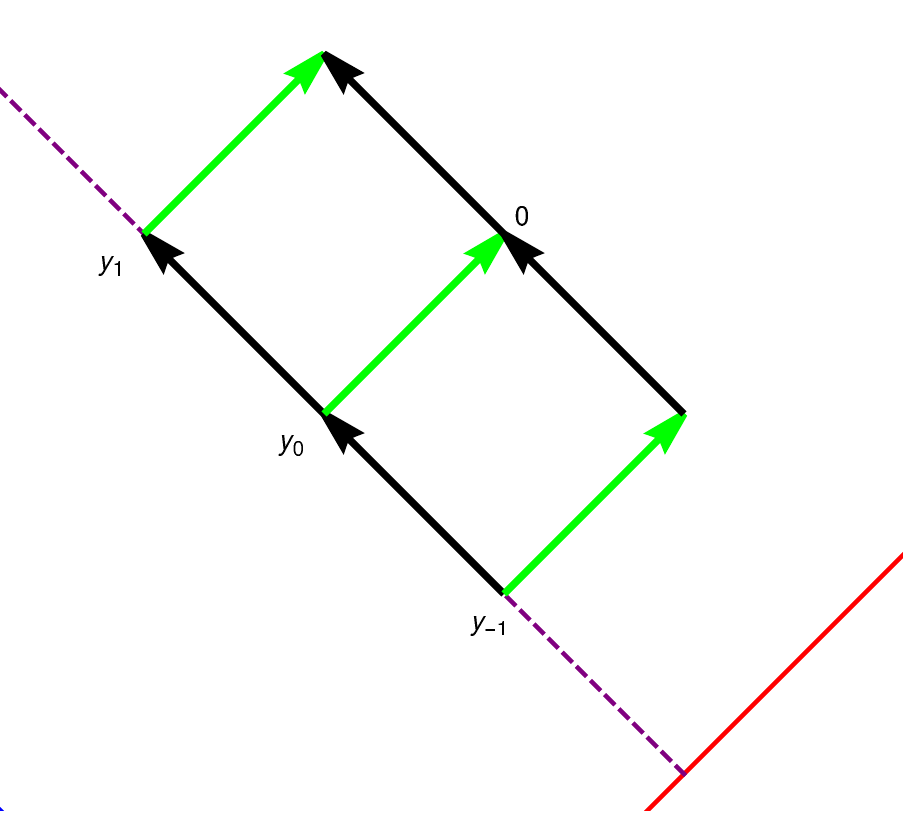}\end{center}

The upward shift operators commute
(Proposition~\ref{prop-shdef}). So we have $\sh31 y_1=0$, and also
$\sh31 y_{-1}=0$ if $(h/3,p)$ is in the sector $\sect(j)$, and $y_{-1}$
corresponds to a point on the dashed line. This shows that the kernels
of $\sh 3 1$ on $K$-types corresponding to points on the dotted line
are related by the injective map $\sh {-3}1$.

This brings us to the intersection point $\bigl( \frac {h_0}3,p_0
\bigr)$ of the line $h=2j'-6-3p$ with the right boundary line $h=2j+3p$
of the sector $\sect(j)$. Table~\ref{tab-x-up}, p~\pageref{tab-x-up},
shows that $\sh3{-1} \tilde\upsilon_\n^{p_0,0}(j,\nu)$ is non-zero for
all $\nu$, and that $\sh3{-1} \tilde\omega_\n^{p_0,0}(j,\nu)=0$ for
$\nu = \pm \bigl(2p_0+2+j\bigr)$. So we have i) in the case
$\varepsilon=1$.

For ii) we consider the following picture:
\begin{center}\grf{6}{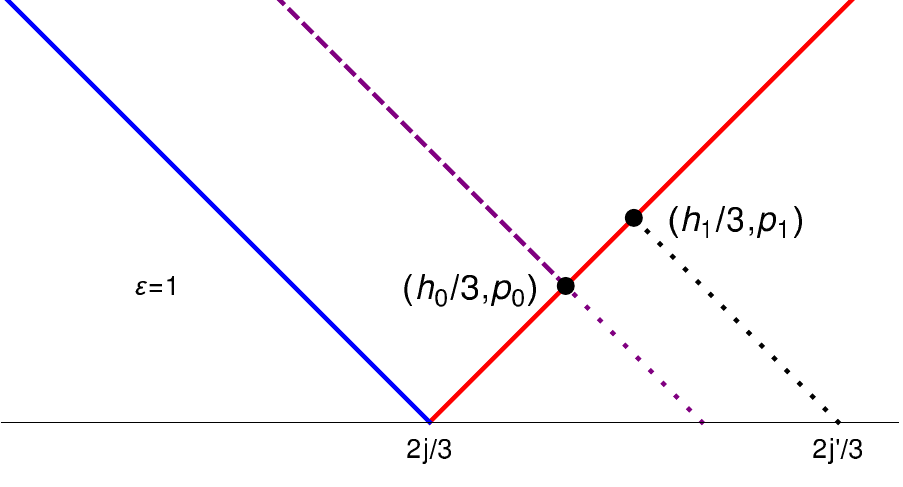}
\end{center}
One step up from $\bigl( \frac{h_0}3,p_0\bigr)$ is the point $\bigl(
\frac{h_1}3,p_1\bigr)$, which is the lowest point in $\sect(j) \cap
\sect(j')$. We can apply Lemma~\ref{lem-kdso} to this situation, with
$j$ in the role of $j_1$ and $j'$ in the role of~$j_2$. This gives
$\nu = \frac{|h_1+p_1|}2$.
\[ \frac{h_1+p_1}2 = \frac12\bigl( 2j + 4 p_1) = j + 2p_0+2\,.\]
So indeed $\sh31$ vanishes on $\Wfu_{\n;h_0,p_0,p_0}^{\xi,\nu}$, and
hence on all $K$-types corresponding to points in $\sect(j)$ on the
line $h=2j'-6-3p$. This gives ii) in the case $\varepsilon=1$.

The case of $\sh{-3}1$ goes analogously, now using $\e=-1$.
\end{proof}

\begin{lem}\label{lem-kdso-na}Let $\psi \in \WOI$, and let $\tau^h_p$ be
a $K$-type occurring in $\Ffu^\psi_\n$.
\begin{enumerate}
\item[i)] The downward shift operator $\sh{\pm 3}{-1}$ is zero on
$\Vfu_{\n;h,p,p}^\psi$ if $h=2j\pm 3p$ for some $j\in \wo^1(\psi)$, and
injective otherwise.
\item[ii)] The downward shift operator $\sh{\pm 3}{-1}$ is zero on
$\Wfu^\psi_{\n;h,p,p}$ if $\pm \ell<0$ and $h=2j'\pm 3p$ for some
$j'\in \wo^1(\psi)_\n$, and injective otherwise.
\end{enumerate}
\end{lem}
\begin{proof}The spaces $\Wfu^\psi_{\n;h,p,p}$ and
$\Vfu^\psi_{\n;h,p,p}$ have dimension at most $1$, so injectivity and
vanishing are the only possibilities. By ii)a) in
Proposition~\ref{prop-extr.na} the operator $\sh{\pm 3}{-1}$ vanishes
on $K$-types corresponding to points on the line $h=2j\pm 3p$ if
$(j,\nu)\in \wo(\psi)^+_\n$.

Proposition~\ref{prop-kdso} shows that we also have to consider points
in $\sect(j)$ on lines $h\mp 3p = 2j'$ for $j' \in \wo^2(\psi)\setminus
\wo^1(\psi)_\n$. Let $\bigl( \frac h3,p\bigr)$ be the intersection
point of the lines $h\pm 3p = 2j$ and $h\mp 3p = 2 j'$.
\begin{center}
\grf{5}{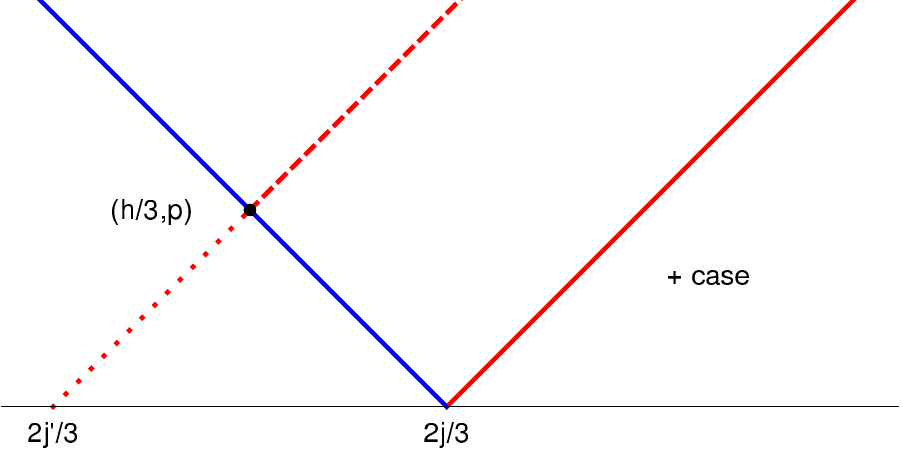}\quad\grf{5}{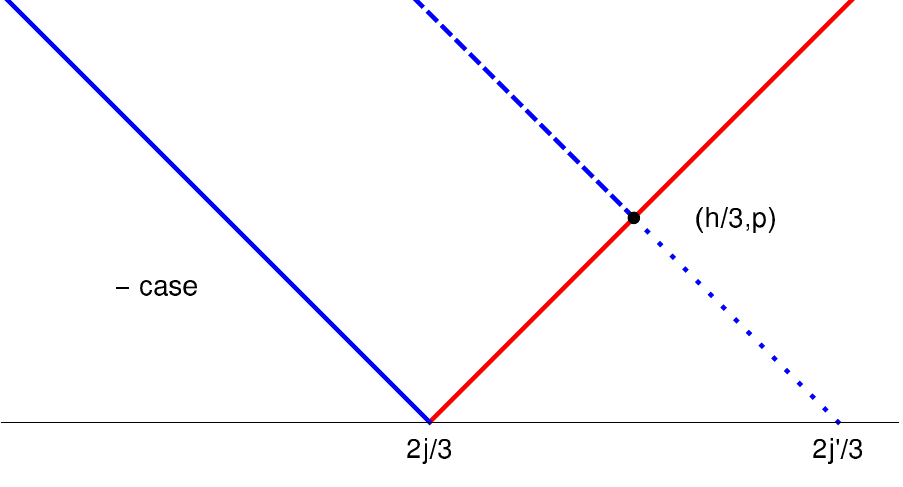}
\end{center}

Table~\ref{tab-parms-na}, p~\pageref{tab-parms-na}, shows that $m_0(j) -
m_0(j') = - \frac13(j-j') \sign(\ell)$. Since $m_0(j')<0\leq m_0(j)$ we
need $\mp \ell>0$. Furthermore $\pm(j-j') = 3p$, and $m_0(j)
= m_0(j') +p < p$. We use this in the application of
Table~\ref{tab-x-dn}, p~\pageref{tab-x-dn}, with
$(j,\nu) \in \wo(\psi)_\n^+$. Lemma~\ref{lem-kdso} shows that $\nu =
\frac12 \bigl|h\mp p \bigr|= \bigl| j \mp 2p\bigr|$. We obtain
\bad
&\text{case +}&\sh3 {-1} \tilde \omega_\n^{0,p}(j,\nu) &\;\neq \; 0\,,&
\quad \sh 3 {-1}\tilde \upsilon_\n^{p,0}(j,\nu)&=0\,,\\
&\text{case --}& \sh{-3} {-1} \tilde\omega_\n^{0,p}(j,\nu) &\;\neq\;
0\,,& \sh{-3}{-1} \tilde\upsilon_n^{0,p}(j,\nu)&=0\,.
\ead

We can extend these properties of points upward along the line $h\mp
3p=2j'$ as long as the shift operator $\sh{\mp 3}{-1}$ is injective on
the corresponding $K$-types in $\Ffu^\psi_\n$, in analogy with the
approach in the proof of Lemma~\ref{lem-usho-upsom}. This injectivity
does not hold if we meet a line $h=2j''\mp 3p$ with another
$j''\in \wo^2(\psi)$.
\begin{center}\grf5{dshe2a}\end{center}
Since $j''$ and $j'$ are on different sides of $j$ we have $m_0(j'')\geq
m_0(j) \geq 0$, and $j''\in \wo^1(\psi)_\n$. There we can start the
reasoning again. This gives i) and~ii).
\end{proof}

\subsection{Modules with regular behavior at 0}\label{sect-Mfu}
The discussion in the previous subsections looks rather satisfactory,
except for the fact that $\Vfu_\n^\ps$ is not naturally defined. It
depends on the choice of the unusual $V$-Whittaker function $V_{\k,s}$
in \eqref{Vkps}, depending on a choice of a branch of the continuation
of $W_{\k,s}$. The function $M_{\k,s}$ is much more natural, it leads
to functions with $\nu$-regular behavior at~$0$; see
Definition~\ref{def-grbb}. It has the disadvantage of being
proportional to $W_{\k,s}$ or to $V_{\k,s}$ for some combinations of
the parameter values that are relevant in the non-abelian case.

In this subsection we establish results for $\Mfu^{\ps[j,\nu]}_\n$ under
integral parametrization. We restrict ourselves to values $\nu \in
\ZZ_{\geq 0}$. Propositions \ref{prop-kkM} and~\ref{prop-spk} consider
for $\Mfu^\ps_\n$ questions studied in Proposition~\ref{prop-ik},
concerning subspaces in $\Mfu^\ps_{\n;h,p,p}$ with $p\geq 1$ on which
both downward shift operators vanish. We extend in
Lemma~\ref{lem-Mfud1} the multiplicity one result in iii) in
Lemma~\ref{lem-VW-nab}. Proposition~\ref{prop-M-VW} determines the
$K$-types for which $\Mfu^\ps_{\n;h,p}$ coincides with
$\Vfu^\ps_{\n;h,p,p} $ or with $\Mfu^\ps_{\n;h,p,p}$.

\rmrk{Intersection of kernels of downward shift operators}
Proposition~\ref{prop-ik} describes the intersection of the kernel
$K_{\n;h,p}$ of both downward shift operators on a more-dimensional
$K$-type $ \tau^h_p$ in $\Ffu_\n^\ps$. It gives two types of
information: It tells how kernel elements arise as values of families
$x^{p,0}_\n$ and $x^{0,p}$ for $x=\tilde\ups $ or $\tilde \om$, and
moreover it gives an explicit expression of basis elements of
$K_{\n;h,p}$.

Here we consider the intersection of $K_{\n;h,p}$ with $\Mfu^\ps_\n$,
and determine which of the families $\tilde\mu^{p,0}_\n$ and
$\tilde\mu^{0,p}_\n$ have values in this intersection. In some cases an
explicit description of a basis of $K_{\n;h,p}\cap \Mfu_\n^\ps$ is
possible. We recall that $\Mfu^\ps_\n$ is the submodule of
$\Ffu^\ps_\n$ generated by $\nu$-regular behavior at~0 for some $\nu$
with $\re\nu\geq 0$. See Definitions \ref{def-grbb} and~\ref{defM[]}.

The point $(h/3,p)$ corresponds to the $K$-type~$\tau^h_p$ in the
intersection of the boundaries of the sectors $\sect(j_1)$ and
$\sect(j_2)$ for element $j_1,j_2 \in \wo(\ps)^+$, determined by
$h=2j_1+3p=2j_2-3p$ with $p\geq 1$. The description of $K_{\n;h,p} \cap
\Mfu^\ps_\n$ depends on the choices of the combination $(j_1,j_2)$. In
general there are three choices, depicted in Table~\ref{fig-comb},
corresponding to points $(j_1,\nu_1)$ and $(j_2,\nu_2)$ in the interior
of Weyl chambers in~Figure~\ref{fig-ip}, p~\pageref{fig-ip}. If
$(j_1,\nu_1)$ and $(j_2,\nu_2)$ are on walls of Weyl chambers, the
three combinations reduce to one, which we arbitrarily take under
combination~2.\il{comb}{combinations 1, 2, 3}
\begin{figure}[ht]
\begin{center}\grf{5}{ints123}\end{center}
\caption{Position of $(h/3,p)$ in the three combinations of $(j_1,j_2)$
in Proposition~\ref{prop-kkM}:\\
Combination 1: $j_1=j_l<j_2=j_+<j_r$, combination 2: $j_1 =j_l<
j_2=j_r$, combination 3: $j_l<j_1=j_+<j_2=j_r$. We use the conventions
in~\eqref{jnurels}. }\label{fig-comb}
\end{figure}

The values of $h=j_1+j_2$ and $p=\frac13(j_2-j_1)$ depend on the
combination. We use a subscript 1, 2, or 3 when needed. With
Lemma~\ref{lem-kdso} we can check that $h<-p$ for combination~1,
$|h|\leq p$ for combination~2, and $h>p$ for combination~3.

\begin{prop}\label{prop-kkM}We use the notations of
Proposition~\ref{prop-ik}. We
put\ir{kkMdef}{\kk^M_{\n;h,p}}\il{coM}{$c^M(r)$}
\badl{kkMdef} \kk^M_{\n;h,p} &\= \sum_r \th_{m(h,r)} c^M(r)\, t^{p+1} \,
M_{\k(r),|s(r)|}(2\pi |\ell|t^2) \, \Kph hprp\,,\\
\c^M(r) &\= - e^{\pi i(m(h,r)-\k(r))}\,\frac{\Gf \bigl(
\frac12+|s(r)|-\k(r)\bigr)}{\sqrt{m(h,r)!}\, \bigl( 2 |s(r)|
\bigr)!}\,,
\eadl
where the sum runs over $r\equiv p \bmod 2$, $|r|\leq p$, $m(h,r)\geq
0$.

The function $\kk^M_{\n;h,p}$ is an element of $K_{\n;h,p}$ that spans
the intersection $K_{\n;h,p} \cap \Mfu^\ps_\n$ if one of the following
conditions is satisfied:
\begin{enumerate}
\item[a)] $m_0(j_1)\geq 0$ and $m_0(j_2)\geq 0$, and one of both is
larger than~$p$.

Under these conditions,
$\kk^M_{\n;h,p}\dis \tilde\mu^{p,0}_\n(j_1,\nu_2)
\dis \tilde \mu^{0,p}_\n(j_2,\nu_2)$. If, in addition, $\e \bigl(
h-r_0(h) \bigr) \geq 0$, then $\kk^M_{\n;h,p} = \kk^V_{\n;h,p}$.
\item[b)] $\e=1$, $0\leq m_0(j_1) \leq p$, and $r_0(h) \leq h$.

Under these conditions, $\kk^M_{\n;h,p} = \kk^V_{\n;h,p}$, and
$\kk^M_{\n;h,p}\dis \tilde \mu^{p,0}_\n(j_1,\nu_1)$.
\item[c)] $\e=-1$, $0\leq m_0(j_2) \leq p$, and $r_0(h) \geq h$.

Under these conditions, $\kk^M_{\n;h,p}= \kk^V_{\n;h,p}$, and
$\kk^M_{\n;h,p} \dis\tilde\mu^{0,p}_\n(j_2,\nu_2)$.
\end{enumerate}
\end{prop}

We prepare the proof by two lemmas. The first lemma gives elements that
we know to be in $K_{\n;h,p}\cap \Mfu^\ps_\n$ under some conditions.
These conditions can be formulated in terms of $m_0(j_1)$ and
$m_0(j_2)$. In particular, $m_0(j_1)\geq 0$ if and only if
$\Mfu^{\xi_1,\nu_1}_\n$ is non-zero, and similarly for $m_0(j_2)$. It
can also be formulated in terms of $r_0=r_0(h) = p-2\e m_0(j_1)=-p-2\e
m_0(j_2)$.

\begin{lem}\label{elm-kkM-mu}We use the notations indicated above.
\begin{enumerate}
\item[i)] The element $\tilde\mu^{p,0}_\n(j_1,\nu_1) $ is in $
K_{\n;h,p} \cap \Mfu^{\xi_1,\nu_1}_\n$ under the following conditions.
\[
\begin{array}{|c|ccc|}\hline
\text{comb.}& \multicolumn{3}{|c|}{\text{conditions}}
\\\hline
1&\e=1& m_0(j_1)\geq p& r_0\leq -p \\
& \e=-1& m_0(j_1)\geq 0 &r_0 \geq -p\\
\hline
2, \; 3& \e=\pm 1& m_0(j_1) \geq 0 & \e r_0 \leq p \\\hline
\end{array}
\]
\item[ii)] The element $\tilde\mu^{0,p}_\n(j_2,\nu_2) \in K_{\n;h,p}
\cap \Mfu^{\xi_2,\nu_2}_\n$ under the following conditions.
\[
\begin{array}{|c|ccc|}\hline
\text{comb.}& \multicolumn{3}{|c|}{\text{conditions}}
\\\hline
1, \; 2&\e=\pm 1& m_0(j_2) \geq 0 & \e r_0 \leq p \\ \hline
3 & \e=1 & m_0(j_2) \geq 0 & r_0 \leq p \\
&\e=-1 & m_0(j_2) \geq p & r_0 \geq p \\ \hline
\end{array}
\]
\end{enumerate}
\end{lem}

\begin{proof}
The element $\tilde\mu_\n^{p,0}(j_1,\nu_1) \in
\Mfu^{\xi_1,\nu_1}_{\n;h,p,p}$ is in the kernel of the shift operator
$\sh 3{-1}$, and
$\tilde\mu_\n^{0,p}(j_2,\nu_2) \in \Mfu^{\xi_2,\nu_2}_{\n;h,p,p}$ is in
the kernel of the shift operator $\sh {-3}{-1}$. We use
Table~\ref{tab-x-dn} on p~\pageref{tab-x-dn} to determine the behavior
of the other downward shift operator on these functions.
\badl{ker-cond} \tilde\mu_\n^{p,0}(j_1,\nu_1)\in K_{\n;h,p}
&\Leftrightarrow 2p+j_1 \=
\begin{cases} \nu_1 & \text{ if }\e=1\text{ and }p> m_0(j_1)\,,\\
\pm \nu_1 & \text{ otherwise}\,;
\end{cases}\\
\tilde \mu_\n^{0,p}(j_2,\nu_2) \in K_{n;h,p}
&\leftrightarrow 2p-j_2 \=
\begin{cases}\nu_2 & \text{ if } \e=-1\text{ and } p>m_0(j_2)\,,\\
\pm \nu_2 & \text{ otherwise}\,.
\end{cases}
\eadl
Computations based on Lemma~\ref{lem-kdso} lead to the following table.
\be\label{cmbt}\begin{array}{|c|cc|cc|cc|}\hline
&j_1&j_2& 2p+j_1-\nu_1 & 2p+j_1+\nu_1
& 2p-j_2-\nu_2 & 2p-j_2+\nu_2
\\ \hline
1&j_l&j_+&-2\nu_l &0&0&2\nu_+
\\
2&j_l&j_r&0&\nu_+-j_+&0&\nu_++j_+
\\
3&j_+&j_r&0&2\nu_+&-2\nu_r&0
\\ \hline
\end{array}
\ee

For combinations 2 and 3 this implies that $\tilde
\mu^{p,0}_\n(j_1,\nu_1) \in K_{\n;h,p}$ if $m_0(j_1) \geq 0$. Under
combination~1 we have $\nu_l\geq 1$, and to have $\tilde
\mu^{p,0}_\n(j_1,\nu_1) \in K_{\n;h,p}$ we need an additional
condition: $\e=1$, or $m_0(j_1) \geq p$. This gives~i). For
$\tilde\mu^{0,p}(j_2,\nu_2)$ we proceed similarly, now using that
$\nu_r \geq 1$ under combination~3. This gives~ii).
\end{proof}

\begin{lem}\label{lem-kkM}In the notation introduced above, we define
$\kk^M_{\n;h,p}$ as in Proposition~\ref{prop-kkM}.
\begin{enumerate}
\item[i)] $\kk^M_{\n;h,p}$ spans the space
$K_{\n;h,p} \cap \Mfu^\ps_\n$.
\item[ii)] If $\e(h-r_0) \geq 0$, then $\kk^M_{\n;h,p}=\kk^V_{\n;h,p}$.
\end{enumerate}
\end{lem}
\begin{proof}
Suppose that $g \in K_{\n;h,p} \cap \Mfu^\ps_\n$. Then $g = \al
\kk^V_{\n;h,p} + \bt \kk^W_{\n;h,p}$, with $\bt=0$ if $|r_0(h)|\leq p$,
by i) in Proposition~\ref{prop-ik}. The component of order $r$ of $g$
has the form
\[ t^{p+1} \Bigl( \al C^V(r) \, V_{\k(r),s(r)} + \bt c^W(r)\,
W_{\k(r),s(r)} \Bigr)\,,\]
with the notations in \eqref{kapr} and~\eqref{cWVexpl}. Here $r$ runs
over $r\equiv p\bmod 2$, $|r|\leq p$ such that $m(h,r) \geq 0$. Each
component has to have $\nu$-regular behavior at~$0$ for $\re\nu\geq 0$,
and hence should be a multiple of $t^{p+1}\, M_{\k(r),|s(r)|}$. The
functions $W_{\k,s}$ and $V_{\k,s}$ are even in~$s$, so going over to
$|s(r)|\in \frac12\ZZ_{\geq 0}$ is the sensible thing to do.

By \eqref{MinWV} we have
\bad M_{\k(r),|s(r)|} &\= A(r)\, V_{\k(r),|s(r)|} + B(r) \,
W_{\k(r),|s(r)|} \,,\\
A(r) &\= \frac{ -e^{\pi i\k(r)}\, \Gf\bigl(1+2|s(r)|\bigr)} {\Gf\bigl(
\frac12+|s(r) |- \k(r)\bigr)}\,,\\
B(r) &\= \frac{-ie^{\pi i \bigl(\k(r) - |s(r)|\bigr)}\,
\Gf\bigl(1+2|s(r)|\bigr)} {\Gf\bigl( \frac12+|s(r)|+\k(r) \bigr)}\,.
\ead
Since $\frac12+|s(r)|-\k(r) \= 1 + |s(r)|+\e s(r) \geq 1$ for all $r$ in
the sum, we have $A(r)\neq 0$. The factor $B(r)$ may be zero for the
values that we use here.

For the assumed $g\in K_{\n;h,p}\cap \Nfu^\ps_\n$ we get coefficients
$c^M(r)$ such that for all relevant values of $r$
\be \al c^V(r) \= c^M(r) A(r)\,,\qquad \bt c^W(r) \= c^M(r) B(r)\,.\ee
The first relation implies that if $\al$ were zero, then all
coefficients $c(r)$ would vanish. So if the supposed $g$ exists as a
non-zero function, then we normalize it so that $\al=1$, and put
\be c^M(r) \= c^V(r)/A(r)\,.\ee
This leads to the relation
\be \bt \= \frac{c^V(r)}{c^W(r)} \, \frac{B(r)}{A(r)}\,, \ee
valid for all $r$ occurring in the sum. The factor $B(r)$ is the sole
factor that may vanish. So if $B(r)=0$ for one relevant value of~$r$,
then $B(r')=0$ for all $r'$ occurring in the sum, and $\bt=0$. In that
case the hypothetical function $g$ is a multiple of $\kk^V_{\n;h,p}$,
which we know explicitly.\smallskip

The factor $B(r)$ vanishes if and only if
\[\frac12+|s(r)|+\k(r) \= |s(r)|-\e s(r) - m(h,r)
\= |s(r)|- \e s(r)
-\frac\e 2(r-r_0)\in \ZZ_{\leq 0}\,. \]
For all $r$ with $m(h,r)\geq 0$ we have
\begin{align*} B(r) \= 0 &\Leftrightarrow \begin{cases} - \frac
\e2(r-r_0) \leq 0
&\text{ if } \e s(r) \geq 0\,,\\
\frac\e 2(r_0-h)\leq 0
&\text{ if } \e s(r) \leq 0\,;
\end{cases}\\
&\Leftrightarrow
\begin{cases} r\geq r_0\text{ and } r_0 \leq h &\text{ if } \e=1\,,\\
r\leq r_0 \text{ and } r_0 \geq h &\text{ if }\e=-1\,.
\end{cases}
\end{align*}
The condition $\e(r-r_0)\geq 0$ is just the condition $m(h,r) \geq 0$.

If $\e=1$ we have $\max(-p,r_0) \leq r \leq p$ for all $r$ relevant for
the sum. That rules out combination~3, and gives the condition $r_0
\leq h$ for the other combinations. For $\e=-1$ combination 1 cannot
occur, and the other cases go similarly. Hence we find $g \dis
\kk^V_{\n;h,p}$ under the following conditions.
\be
\begin{array}{|c|c|c|}\hline
& \e=1& \e=-1\\ \hline
1& r_0\leq h & \\
2& r_0\leq h & r_0\geq h \\
3 && r_0 \geq h \\
\hline
\end{array}\ee\medskip

The other possibility is that $\bt\neq 0$, and $B(r) \neq 0 $ for all
$r$ occurring in the sum. Then we should have $r_0>h$ if $\e=1$, and
$r_0< h$ if $\e=-1$. A computations gives
\[ \bt \= ie^{\pi i (m(h,r)-2|s(r)|)/2} \frac{ \Gf(|s(r)|+\e
s(r)+m(h,r))}{m(h,r)!\, \Gf(|s(r)|-\e s(r)-m(h,r))}\,.\]
Using that $\e(r_0-h)>-$ and $\e(r-r_0)\geq 0$ implies that $|s(r)|=-\e
s(r)$ we obtain
\be \bt \= \frac {i e^{-\pi i |h-r_0|/4}}{\bigl(
|h-r_0|/2-1\bigr)!}\,,\ee
which does not depend on~$r$.

Whether $\bt=0$ or not, we conclude that the hypothetical element $g \in
K_{\n;h,p} \cap \Mfu^\ps_\n$ is a multiple of $ \kk^V_{\n;h,p} + \bt
\kk^W_{\n;h,p}$, and has the expansion \eqref{kkMdef} with coefficients
$c^M(r) = c^V(r)/A(r)$. So the element $\kk^M_{\n;h,p} $ indeed spans
the space $ K_{\n;h,p} \cap \Mfu^\ps_\n$
\end{proof}

\begin{proof}[Proof of Proposition~\ref{prop-kkM}] By
Lemma~\ref{lem-kkM} we know that $\kk^M_{\n;h,p}$ spans the
intersection of $K_{\n;h,p}$ with $\Mfu^\ps_\n$ whenever it is defined.
If $|r_0|>p$ it is a linear combination of $\kk^V_{\n;h,p}$ and
$\kk^W_{\n;h,p}$. The identification of $\tilde\mu^{p,0}_\n(j_1,\nu_1)$
and $\tilde\mu^{0,p}_\n(j_2,\nu_2)$ then are automatic. They can be
confirmed by comparison of a determining component with help of ~i)
in Lemma~\ref{lem-VWrels}. The identification with $\kk^V_{\n;h,p}$
under an additional condition follows from ii) in Lemma~\ref{lem-kkM}.

By i) in Proposition~\ref{prop-ik} we should have $\kk^M_{\n;h,p} \dis
\kk^V_{\n;h,p}$ if $|r_0|\leq p$. Hence we need the condition
$\e(h-r_0)\geq 0$. Comparison of the determining component gives the
proportionality with $\tilde\mu^{p,0}_\n(j_1,\nu_1)$, respectively
$\tilde\mu^{0,p}_\n(j_2,\nu_2)$.
\end{proof}

\begin{lem}\label{lem-Mfud1}Let $\ps\in \WOI$. For each $K$-type
$\tau^h_p$
\be \dim \Mfu^\ps_{\n;h,p,p} \= \begin{cases}
1&\text{ if } (h/3,p) \in \sect(j) \text{ for some } j \in
\wo(\ps)^+_\n\,,\\
0&\text{ otherwise}\,.
\end{cases}
\ee

In particular, if the $K$-type $\tau^h_p$ occurs in both
$\Mfu^{\xi,\nu}_\n$ and $\Mfu^{\xi',\nu'}_\n$ for $(j_\xi,\nu)$,
$(j_{\xi'},\nu')\allowbreak \in \wo(\ps)_\n^+$, then
$\Mfu^{\xi,\nu}_{\n'h,p} = \Mfu^{\xi',\nu'}_{\n;h,p}$.
\end{lem}
\begin{proof}
The spaces $\Mfu^{\xi,\nu}_{\n;h,p,p}$ with $(j_\xi,\nu) \in
\wo(\ps)_\n^+$ are non-trivial, see Definition~\ref{def-VWMa}.

The dimension of $\Mfu^\ps_{\n;h,p,p}$ does not change if we go to a
lower $K$-type by application of an injective downward shift operator.
A path given by successive applications of injective downward shift
operators can stop at a $K$-type for which both downward shift
operators have a non-trivial kernel. This may occur at a
one-dimensional $K$-type $\tau^{2j}_0$. Then the dimension of
$\Mfu^{\xi,\nu}_{\n;h,p,p}$ is one if $(j,\nu)\in \wo(\ps)^+_\n$, and
zero otherwise. The path may also stop at a $K$-type studied in
Proposition~\ref{prop-kkM}. That proposition implies that $\dim
\Mfu^\ps_{\n;h,p,p} \geq 1$. Moreover, $\dim\Mfu^\ps_{\n;h,p,p} \leq 2$
by Proposition~\ref{prop-dim2-nab}. Any element of $
\Mfu^\ps_{\n;h,p,p}$ should have components with $\nu$-regular behavior
at $0$. A determining component is a solution of a Whittaker
differential equation, which has only a one-dimensional space of
solutions with $\nu$-regular behavior at~$0$.
\end{proof}

With Lemma~\ref{lem-Mfud1} we know that, analogously to \eqref{tuod}, we
have \ir{tumu}{\tilde\mu^{a,b}_\n}
\be\label{tumu}
\tilde\mu^{a,b}_\n(j,\nu) \= \begin{cases}
\bigl( \sh{-3}1\bigr)^b\tilde \mu^{a,0}(j,\nu) &\text{ if }\ell>0\,,\\
\bigl( \sh 3 1 \bigr)^a \tilde\mu^{0,b}(j,\nu) &\text{ if } \ell<0\,,
\end{cases}
\ee
for $a,b\in \ZZ_{\geq 0}$ spanning $\Mfu^{\xi,\nu}_{\n;h,p,p}$ for
$h=2j+3(a-b)$, $p=a+b$.

\rmrk{Identifications}With Propositions \ref{prop-ik} and~\ref{prop-kkM}
we have explicit descriptions of elements $\tilde
x^{p,0}_\n(j_1,\nu_1)$ and $\tilde x^{0,p}_\n(j_2,\nu_2)$, with
$x=\ups$, $\om$, or $\mu$, if they happen to be in the kernel of both
downward shift operators. In the case that we call combination~2
($j_1=j_l$, $j_2=j_r$) there may be a third element to be considered,
like we did in Corollary~\ref{cor-jl+r-ab} in the abelian cases. There
we could use the notation $x^{a,b}_\bt(j_+,\nu_+)$. In the non-abelian
case, the upward shift operators are not always injective. We need to
use the construction in Proposition~\ref{prop-extr.na}.

\begin{prop}\label{prop-spk}Let $\ps\in \WOI$. We use the notations of
Proposition~\ref{prop-ik}, and the further notations introduced at the
start of this subsection (p~\pageref{fig-comb}). We consider
combination~2, with $j_1=j_l< j_+<j_2=j_r$, and take
$h=h_2=2j_1+3p_2=2j_2-3p_2$, $p=p_2=\frac13(j_r-j_l)$.

Assume that $m_0(j_+)\geq 0$.
\begin{itemize}
\item If the one-dimensional space $\Wfu^{\xi_+,\nu_+}_{\n;h,p,p}$ is
contained in $K_{\n;h,p,p}$, then it is spanned by $\kk^W_{\n;h,p}$.
\item If the one-dimensional space $\Vfu^{\xi_+,\nu_+}_{\n;h,p,p}$ is
contained in $K_{\n;h,p,p}$, then it is spanned by $\kk^V_{\n;h,p}$.
\item If the one-dimensional space $\Mfu^{\xi_+,\nu_+}_{\n;h,p,p}$ is
contained in $K_{\n;h,p,p}$, then it is spanned by $\kk^M_{\n;h,p}$.
\end{itemize}
\end{prop}
\begin{proof}
We let $\X$ denote any of $\Wfu$, $\Vfu$, and $\Mfu$. We know that $\dim
\X^{\xi_+,\nu_+}_{\n;h,p,p} = 1$ by Lemma~\ref{lem-Mfud1}. If $\e=1$,
then $m_0(j_l) = m_0(j_+)+p_1$, and $\X^{\xi_l,\nu_l}_{\n;h,p,p} =
\X^{\xi_+,\nu_+}_{\n;h,p,p}$ by iv)
in Lemma~\ref{lem-VW-nab} and Lemma~\ref{lem-Mfud1}. Suppose that
$\Wfu^{\xi_+,\nu_+}_{\n;h,p,p}$ is contained in $K_{\n;h,p}$, then the
same holds for $\Wfu^{\xi_l,\nu_l}_{\n;h,p,p}$, and by
Proposition~\eqref{prop-kkM} we conclude that $\kk^X_{\n;h,p}$ spans
$\X^{\xi_l,\nu_l}_{\n;h,p,p} = \X^{\xi_l,\nu_l}_{\n'h,p,p}$ (with the
obvious notation $\kk^X_{\n;h,p}$. For $\e=-1$ we proceed similarly,
now by the identification $\X^{\xi_+,\nu_+}_{\n;h,p,p}=
\X^{\xi_r,\nu_r}_{\n;h,p,p}$.
\end{proof}

\begin{remark}The proof of Proposition~\ref{prop-spk} is based on a
rather unspecified identification of $\X^{\xi_l,\nu_l}_{\n;h,p,p} =
\X^{\xi_l,\nu_l}_{\n'h,p,p}$. It is not too hard to specify an element
of $\X^{\xi_+,\nu_+}_{\n;h,p,p}$ (in the notations used in the proof).
We discuss this for the case $\e=1$. We start with
$ x^{0,0}(j_+,\nu_+) \in \X^{0,0}_{\n;2j_+,0,0}$. The element
$(\sh 3 1)^{p_3} x^{0,0}_\n(j_+,\nu_+)$ may be zero.
Proposition~\ref{prop-extr.na} gives a non-zero element $\tilde
x^{p_3,0}_\n(j_+,\xi_+) \in \X^{\xi_+,\nu_+}_{\n;2j_++3 p_3,p_3,p_3}$,
by working with the family $x^{0,0}_\n(j_+,\nu)$, and dividing out
common zeros in $\nu$ whenever possible. For $\e=1$ the upward shift
operator $\sh{-3}1$ is injective, by Proposition~\ref{prop-kuso}. This
produces a non-zero element $(\sh{-3}1)^{p_1} \tilde
x^{p_3,0}_\n(j_+,\nu_+)$ spanning $\X^{\xi_+,\nu_+}_{\n;h,p,p}$.
\end{remark}

The Whittaker function $M_{\k,s}$ may be a multiple of the basis
solutions $W_{\k,s}$ and $V_{\k,s}$ of the Whittaker differential
equation; see~\eqref{MWV}. In the case of integral parametrization this
brings the need to get an overview of the $K$-types for which
$\Mfu^\ps_{\n;h,p,p}$ might be equal to $\Wfu^\ps_{\n;h,p,p}$ or to
$\Vfu^\ps_{\n;h,p,p}$.
\begin{prop}\label{prop-M-VW}Let $\ps\in \WOI$, with $(j_l,\nu_l),
(j_+,\nu_+), (j_r,\nu_r) \in \wo(\ps)^+$ according to the conventions
in~\ref{jnurels}.

Let $\tau^h_\n$ be a $K$-type occurring in $\Mfu^\ps_\n$.
\begin{enumerate}
\item[i)] The space $\Mfu^\ps_{\n;h,p} $ is equal to the space
$\Vfu^\ps_{\n;h,p}$ in the following cases:
\begin{enumerate}
\item[a)] $m_0(j_l)\geq 0$ and $m_0(j_r)\geq 0$.
\item[b)] $\e=1$, $m_0(j_r) <0$, and $(h/3,p) \in \sect(j_r)$.
\item[c)] $\e=-1$, $m_0(j_l) < 0$, and $(h/3,p)\in \sect(j_l)$.
\end{enumerate}
\item[ii)] The space $\Mfu^\ps_{\n;h,p}$ is equal to the space
$\Wfu^\ps_{\n;h,p}$ in the following cases:
\begin{enumerate}
\item[a)] $m_0(j_r) \leq m_0(j_+)<0$ and
$(h/3,p) \in \sect(j_l) \setminus \sect(j_+)$.
\item[b)] $m_0(j_l) \leq m_0(j_+)<0$ and $(h/3,p)  \in 
\sect(j_r) \setminus \sect(j_+)$.
\end{enumerate}
\end{enumerate}
\end{prop}
\rmrk{Remarks}The $K$-types occurring in $\Mfu^\ps_\n$ correspond to the
points in the union of the sectors $\sect(j)$ with $j\in
\{j_l,j_+,j_r\}$ for which $m_0(j) \geq 0$. So in case i)a) we have
$\Mfu^\ps_\n = \Vfu^\ps_n$. In the cases not mentioned under i) and ii)
the space $\Mfu^\ps_{\n;h,p}$ is not equal to one of $\Vfu^\ps_{\n;h,p}$
or $\Wfu_{\n;h,p}^\ps$. In the pictures in
Subsection~\ref{sect-str-nab}, pp \pageref{fig-j3}--\pageref{fig-mj1a},
all possibilities are illustrated.

\begin{proof}
In the proof we will use many times that $j\mapsto m_0(j)$ is a strictly
decreasing function if $\e=1$ and a strictly increasing function if
$\e=1$. If we view the formula for $m_0$ as describing a function
on~$\RR$, then the derivative is $-\frac\e 3$.

The basis is Lemma~\ref{lem-mu-omups}. Part i) gives information for
$(h/3,p)$ on a boundary line of the sector $\sect(j)$. It suggests that
the quantity $Q(j,\nu) = 2m_0(j) + \e j-\nu$ is crucial. Let $\e=1$. If
$Q(j,\nu)\geq 0$, then i)a) in the lemma shows that $\Mfu_{\n;h,p}
= \Vfu_{\n;h,p}$ for all points $(h/3,p)$ on the right boundary of the
sector $\sect(j)$. For $\e=1$ the shift operator $\sh{-3}1$ is
injective by Proposition~\ref{prop-kuso}. Hence
$\Mfu_{\n;h,p} = \Vfu_{\n;h,p}$ for all $K$-types corresponding to
points $(h/3,p) \in \sect(j)$.

Still assuming that $\e=1$, let us suppose that $m_0(j) \geq 0$ for all
$j\in \{j_l,j_+,j_r\}$. For $(j_r,\nu_r)$ we know that $\nu_r \leq
j_r$. Hence $Q(j_r,\nu_r) \geq 2\cdot 0 + j_r-\nu_r\geq 0$.
Furthermore, the relations for $m_0$ and the relations in
Lemma~\ref{lem-kdso} imply that $Q(j_l,\nu_l) = Q(j_+,\nu_+) = 2
m_0(j_r)\geq 0$. A check for $Q(j_+,\nu_+)$ goes as follow:
\begin{align*}
2m_0(j_+) & + j_+-\nu_+ \= 2\bigl( m_0(j_r) - \frac13(j_+-j_r) \bigr) +
j_+ -\frac13(j_r-j_l)
\\
&\= 2m_0(j_r) + \frac13 j_++ \frac 13 j_r+ \frac13 j_l \= 2m_0(j_r)\,.
\end{align*}
(See \cite[23a]{Math} for further checks.)
Thus, we get $\Mfu^\ps_{\n;h,p} = \Vfu^\ps_{\n;h,p}$ for all $K$-types
corresponding to points in $\sect(j_r) \cup \sect(j_+) \cup
\sect(j_l)$. This gives i)a)
in the case of $\e=1$. The case of $\e=-1$ goes analogously, working
with the left boundary of a sector, and using the injectivity of
$\sh 3 1$. We have $m_0(j_l)=2m_0(j_l) - j_l -\nu_l\geq 0$, and check
that $Q(j_+,\nu_+) = Q(j_r,\nu_r) = 2m_0(j_l)$.\medskip

We turn to the case that at least one of the $m_0(j)$ is negative. For
$\e=1$, this means that $m_0(j_r)<0$, and $m_0(j_l) \geq 0$; otherwise
$\Mfu^\ps_\n=\{0\}$.

In this case, we have $Q(j_l,\nu_l) = Q(j_+,\nu_+)<0$, and we need to
take into account the role of $p$ in Lemma~\ref{lem-mu-omups}. For
points $(h/3,p)=(2j_l/3+a,a)$ on the right boundary of $\sect(j_l)$ we
have $\Mfu_{\n;h,p,p}^{\xi_l,\nu_l} = \Vfu_{\n;h,p,p}^{\xi_l,\nu_l}$ if
and only if $\nu_l-j_l \leq 2a$. The lowest of these points occurs for
$a_0=\frac{\nu_l-j_l}2$. We note with Lemma~\ref{lem-kdso} for this
lowest value that
\[ 2j_l+6 a_0 \= 2j_l+3\nu_l - 3 j_l \= -j_l+(j_r-j_+) = 2j_r\,.\]
Since $2j_l+3a_0=2j_r-3a_0$, the point $(h/3,p)$ is at the intersection
of the right boundary of $\sect(j_l) $ and the left boundary of the
sector $\sect(j_r)$. Taking into account that for $\e=1$ the shift
operator $\sh{-3}1$ is injective
(Proposition~\ref{prop-kuso}), we conclude that all points $(h/3,p)$ in
$\sect(j_l) \cap \sect(j_r)$ satisfy
$\Mfu_{\n;h,p}^{\xi_l,\nu_l} = \Vfu_{\n;h,p}^{\xi_l,\nu_l}$.

If $m_0(j_+)\geq 0$ we have also to apply the same reasoning (and an
analogous computation) to get $\Mfu_{\n;h,p}^{\xi_l,\nu_l} =
\Vfu_{\n;h,p}^{\xi_l,\nu_l}$ for $K$-types corresponding to points in
$\sect(j_+) \cap \sect(j_r)$. This gives i)b). For $\e=-1$ we proceed
analogously to get~i)c). (Computations in \cite[\S23b]{Math}.)
\medskip

We turn to ii) in the proposition, for $\e=1$. For the base point of the
sector $\sect(j_l)$ Lemma~\ref{lem-mu-omups} gives the relation
$\frac12(j_l+\nu_l )\leq -1 -m_0(j_l)$. Since $- \frac12(j_l+\nu_l) =
\frac13(j_+-j_l) =p_1\geq 0$, we get the relation $m_0(j_l) \leq
-1+p_1$, which implies $m_0(j_+) \leq -1$. Hence we can restrict our
attention to $\sect(j_l)$. For the points $(2j_l/3+a,a)$ on the right
boundary line of $\sect(j_l)$, Lemma~\ref{lem-mu-omups} gives the
condition $b+1 \leq - \frac{l_j+\nu_l}2 \= p_1$. So we get all points
that are not in the sector $\sect(j_+)$. By the injectivity of
$\sh{-3}1$ we conclude that $\Mfu^{\xi_l,\nu_l}_{\n;h,p} =
\Wfu^{\xi_l,\nu_l}_{\n;h,p}$ for all $K$-types corresponding to points
of $\sect(j_l) \setminus \sect(j_+)$. This gives ii)b). For $\e=-1$ we
obtain ii)a) in an analogous way.
(Computations in \cite[\S23c]{Math}.)
\end{proof}

\begin{remark}In Lemmas \ref{lem-usho-upsom} and~\ref{lem-kdso} we
determine the lines of $K$-types in $\Vfu^\ps_\n$ and $\Wfu^\ps_\n$ on
which shift operators vanish. We do not need to repeat that work for
$\Mfu^\ps_\n$. For $K$-types such that $\Mfu^\ps_{\n;h,p,p} =
\Vfu^\ps_{\n;h,p,p}$ we can use the results for $\Vfu^\ps_\n$, and
similarly for $K$-types where $\Mfu^\ps_\n$ and $\Wfu^\ps_\n$ agree. On
the other $K$-types a shift operator vanishes on $\Mfu^\ps_{\n;h,p,p}$
if and only if it vanishes on both $\Wfu^\ps_{\n;h,p,p}$ and
$\Vfu^\ps_{\n;h,p,p}$.
\end{remark}

\subsection{Structure results}\label{sect-str-nab}
In the non-abelian case the submodule structure of $\Wfu^\ps_\n$,
$\Vfu^\ps_\n$ and $\Mfu^\ps_\n$ depends strongly on the question for
which $j\in \wo^1(\ps)^+$ the condition $m_0(j)\geq 0$ is satisfied.
That leads to many combinations that we will consider in detail. First
we prove the last main theorem stated in the introduction.

\begin{proof}[Proof of Theorem~\ref{mnthm-nab-ip}]\label{prfD}
The description of the $K$-types occurring in the special Fourier term
modules is in iii) in Lemma~\ref{lem-VW-nab} (for $\Wfu$ and $\Vfu$)
and in Lemma~\ref{lem-Mfud1} (for $\Mfu$). Part~i) of the theorem also
states that these modules are reducible, and often non-isomorphic. This
becomes clear in the detailed discussion of the many cases later on in
this subsection.

For ii) we use ii) in Lemma~\ref{lem-VW-nab} and Lemma~\ref{lem-Mfud1}.
Part iii) follows from Proposition~\ref{prop-dim2-nab} and v) in
Lemma~\ref{lem-VW-nab}.

The statements in iv) and~v) are a reformulation of
Proposition~\ref{prop-M-VW}.
\end{proof}

\begin{remark}\label{rmk-inab}\emph{Irreducible submodules} The study of
the various possibilities of the subset $\wo(\ps)^+_\n$ of
$\wo(\ps)^+$ together with Lemmas \ref{lem-usho-upsom}
and~\ref{lem-kdso-na} lead to the list of irreducible submodules in
$\Wfu^\ps_\n$ in Table~\ref{tab-isoWna}. It turns out that
$\Vfu^\ps_\n$ has always a module of large discrete series type as its
unique irreducible submodule.
\end{remark}

\begin{table}[ht]{\small
\[ \begin{array}{|cccc|cl|c|}\hline
& \multicolumn{3}{c}{ m_0(j)} &\text{type} && \text{Fig.}\\ \hline
\ell&j_l & j_+ &j_r&\text{conditions}&\\ \hline
\neq 0& \geq 0&\geq 0 &\geq 0& \II_+(j_+,\nu_+) & \nu_+ \geq |j_+| &
\ref{fig-j3}, \ref{fig-j3a}\\
>0&&&&& m_0(j_+) \geq \frac12(\nu_+-j_+)
&\\
<0&&&&& m_0(j_+) \geq \frac12(\nu_++j_+)&\\ \hline
>0& \geq0&\geq0&<0& \FI_+(j_l,-\nu_l)
& 1\leq \nu_l \leq -j_l-2 &\ref{fig-pj2VW}
\\
>0&&&&& -\frac12(j_l+\nu_l) \leq m_0(j_l) < \frac12(\nu_l-j_l)&\\\hline
>0 & \geq 0 & \geq 0 & < 0 & \FI(j_l,j_l) &j_l=-\nu_l\leq -1
&\ref{fig-pj2b}\\
>0&&&&& 0 \leq m_0(j_l) < -j_l&\\\hline
<0 & < 0 & \geq 0 & \geq 0 &\IF_+(j_r,-\nu_r)& 1 \leq \nu_r \leq j_r-2
& \text{\ref{fig-mj2}}\\
<0&&&&&\frac12 (j_r-\nu_r) \leq m_0(j_r) < \frac12(j_r+\nu_r)&\\ \hline
<0 & <0 & \geq 0 & \geq 0 & \IF(j_r,-j_r) & j_r= \nu_r \geq 1 &
\text{\ref{fig-mj2b}}\\
<0&&&&& 0 \leq m_0(j_r) < j_r &\\ \hline
>0 & \geq 0 & <0 & <0 & \FI(j_l,\nu_l) & 1 \leq \nu_l \leq -j_l-2&
\text{\ref{fig-pj1VW}}\\
>0 &&&&& 0 \leq m_0(j_l)
<-\frac12(j_l+\nu_l)&\\ \hline
>0 & \geq 0 & <0 & <0 & \FI(j_l,0) & j_l \in 2 \ZZ_{\leq -1}&
\text{\ref{fig-pj1a}}\\
>0 &&&&& 0 \leq m_0(j_l) <-\frac12j_l&\\
\hline
<0 & <0 & <0 & \geq 0 & \IF(j_r,\nu_r) & 1 \leq \nu_r \leq j_r-2&
\text{\ref{fig-mj1}}\\
<0 &&&&& 0 \leq m_0(j_r) < \frac12(j_r-\nu_r)& \\ \hline
<0 & <0 & <0 & \geq 0 & \IF(j_r,0) & j_r\in 2\ZZ_{\geq 1} &
\text{\ref{fig-mj1a}}\\
<0 &&&&& 0 \leq m_0(j_r) < \frac12j_r&
\\\hline
\end{array}\]}
\caption[]{Isomorphism types of the sole irreducible submodule of
$\Wfu_\n^\ps$. The main spectral parameters $j_+$ and $\nu_+$ are in
$\ZZ$, and satisfy $\nu_+ \equiv j_+\bmod 2$, $\nu_+ \geq |j_+|$. They
determine $(j_l,\nu_l)$ and $(j_r,\nu_r)$ according
to~\eqref{jnurels}.\\
See \cite[\S24]{Math} for some computations. }\label{tab-isoWna}
\end{table}

Like in~\S\ref{sect-smps}, we consider the various cases in some detail.
There the cases were described by the spectral parameters $(j,\nu)$.
Here the set of $j\in \wo^1(\psi)$ for which $m_0(j) \geq 0$ determines
the cases, and we use the notations in \eqref{jnurels}. Since we
consider the modules $\Wfu_\n^\ps$, $\Vfu^\ps_\n$ and $\Mfu_\n^\ps$ for
a given $\ps\in \WOI$, we can use $(j_+,\nu_+)$ as the main spectral
parameters. By the letter `i' we indicate the irreducible submodule(s).

We depict the vanishing of shift operators with conventions similar to
those in \S\ref{sect-smps}. In most cases we get different pictures for
$\Vfu^\ps_\n$, $\Wfu^\ps_\n$ and $\Mfu^\ps_\n$. We indicate the
irreducible submodule by the letter `i'. For the pictures of
$\Mfu^\ps_\n$ we use the letters 'V' and 'W' to indicate the sets of
$K$-types in which $\Mfu^\ps_{\n;h,p}$ is equal to $\Vfu^\ps_{\n;h,p}$,
respectively $\Wfu^\ps_{\n;h,p}$.

\medskip\par
\subsubsection{}\emph{All $m_0(j)$ non-negative. }\label{sect-j3}
Let $m_0(j_l), m_0(j_+), m_0(j_r)\geq 0$.

\rmrk{Three different values $j_r<j_+<j_l$} According to
Lemma~\ref{lem-usho-upsom} the upward shift operators in $\Vfu^\psi_\n$
and in $\Wfu_\n^\psi$ are injective. Lemma~\ref{lem-kdso-na} implies
that the downward shift operators in $\Vfu^\psi_\n$ and $\Wfu_\n^\psi$
are zero on the boundaries of the three sectors, and injective
elsewhere. This leads to the configuration in Figure~\ref{fig-j3}.
\begin{figure}[htp]
\begin{center}\grf8{j3}\end{center}
\caption{Structure of $\Wfu^\psi_\n$, $\Vfu^\psi_\n=\Mfu^\psi_\n$ if
$\nu_+\geq|j_+|+2$, and all $m_0(j)\geq 0$. See \S\ref{sect-j3}.
}\label{fig-j3}
\end{figure}
Proposition~\ref{prop-M-VW} ii)a) implies that
$\Mfu_\n^\psi=\Vfu^\psi_\n$. So the figure describes this case as well.
This configuration is identical to the submodule structure in all
abelian cases with $j_l<j_+<j_r$. We saw the same submodule structure
in all generic abelian cases; see Figure~\ref{fig-strab},
p~\pageref{fig-strab}.

Each of the modules $\Wfu^\psi_\n$, $\Vfu^\psi_\n$, and $\Mfu^\psi_\n$
has an irreducible submodule of type $\II_+(j_+,\nu_+)$ with minimal
$K$-type satisfying $h_0=j_l+j_r=-j_+ $ and
$p_0=\frac13(j_r-j_l)=\nu_+$; we use \eqref{jjhp}.

\rmrk{Two coinciding $j$'s} This conclusion is also valid in the cases
that $j_+=\nu_+$ or $j_+=-\nu_+$. In those cases there are only two
sectors of $K$-types, and we get a configuration sketched in
Figure~\ref{fig-j3a}.
\begin{figure}[htp]
\begin{center}\grf8{j3a}\end{center}
\caption{Structure of $\Wfu^\psi_\n$, $\Vfu^\psi_\n$, and $\Mfu^\psi_\n$
if $\nu_+= |j_+|$, and all $m_0(j)\geq 0$. See \S\ref{sect-j3}.
}\label{fig-j3a}
\end{figure}

\medskip\par
\subsubsection{}\emph{Two $m_0(j)$ non-negative, $\ell>0$.
}\label{sect-pj2}
For $\ell>0$ we have $m_0(j_l) \geq m_0(j_+) \geq m_0(j_r)$. Now we
consider the situation that $m_0(j_r)<0$, and $m_0(j_l)\geq
m_0(j_+)\geq 0$.

\rmrk{No coinciding values of $j_l$, $j_+$ and $j_r$} Let $m_0(j_r) < 0
\leq m_0(j_+) < m_0(j_l)$. For $\Vfu_\n^\psi$ the position of the lines
where the shift operators are not injective does not depend on the
$m_0(j)$; the only difference is that the sector $\sect(j_r)$ does not
contribute $K$-types. This gives the configuration on the left in
Figure~\ref{fig-pj2VW}.
\begin{figure}[htp]
\begin{center}\grf{5}{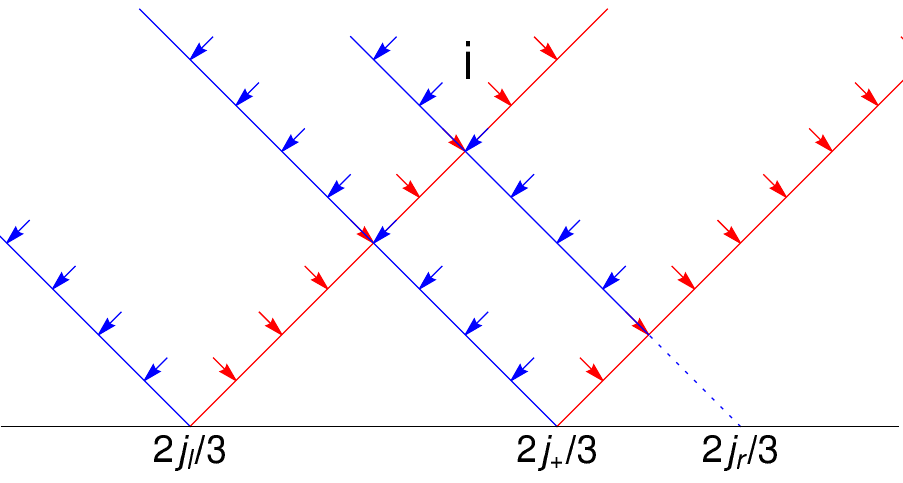}\quad \grf{5}{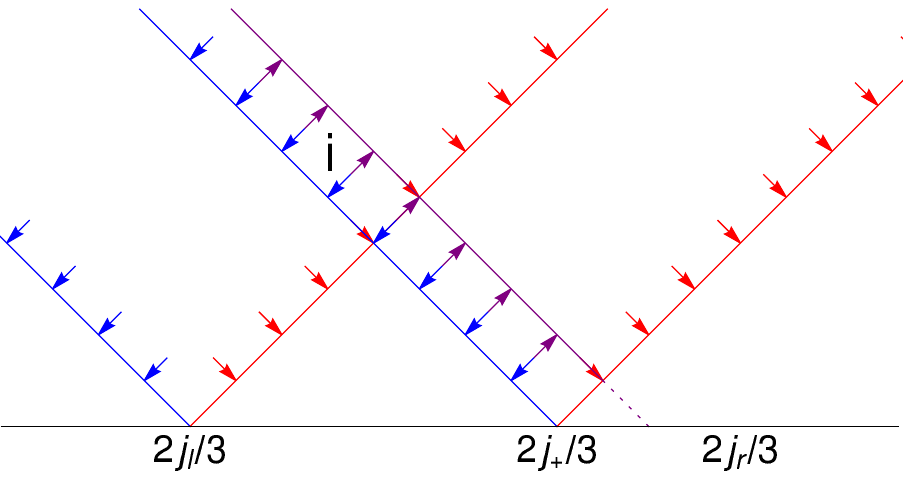}\\
$\Vfu_\n^\psi$
\hspace{4.5cm} $\Wfu_\n^\psi$
\end{center}
\caption{Structure of $\Vfu_\n^\psi$ and $\Wfu_\n^\psi$ for $\ell>0$,
$\nu_+\geq |j_+|+2$, and $m_0(j_r)<0$, $m_0(j_l) > m_0(j_+) \geq 0$.
See \S\ref{sect-pj2}.} \label{fig-pj2VW}
\end{figure}

In the module $\Wfu_\n^\psi$ the downward shift operators vanish on the
$K$-types corresponding to points on the boundary lines of the sectors
$\sect(j_l)$ and $\sect(j_+)$. See ii)
in Lemma~\ref{lem-kdso-na}. Lemma~\ref{lem-usho-upsom} ii) gives
vanishing of $\sh 3 1$ on the line $h-2j_l+3p+6=0$. This gives the
picture on the right in Figure~\ref{fig-pj2VW}.

The module $\Vfu_\n^\ps$ has an irreducible submodule of type
$\II_+(j_+,\nu_+)$, with lowest $K$-type satisfying $h_0=-j_+$ and
$p_0 = \nu_+$. The module $\Wfu_\n^\psi$ has also one irreducible
submodule, of type $\FI_+$. In the picture we read off that
$h_0=j_l+j_+=-j_r$, $p_0=\frac13(j_+-j_l)=\nu_r$,
$A=\frac13\bigl( j_r-3-j_+)= \nu_l-1$, $B=\infty$, with~\eqref{jjhp}.
In Table~\ref{tab-isot-ps}, p~\pageref{tab-isot-ps}, we see that the
complete isomorphism type of the irreducible submodule of $\Wfu_\n^\ps$
is $\FI_+(j_l,-\nu_l)$.

Proposition~\ref{prop-M-VW} ii)b) states that $\Mfu^\psi_{\n;h,p,p}
= \Vfu_{n;h,p,p}^\psi$ for the $K$-types that correspond to $(h/3,p)
\in \sect(j_r)$. Comparison of the lines in Figure~\ref{fig-pj2VW} on
which the downward shift operators vanish in $\Vfu^\psi_\n$ and in
$\Wfu^\psi_\n$ gives the vanishing of the downward shift operators in
$\Mfu^\psi_\n$ as indicated in Figure~\ref{fig-pj2M}. The upward shift
operator does not vanish in $\Mfu^\psi_\n$, since the relevant
$K$-types of $\Mfu^\psi_\n$ have a component in both $\Vfu^\psi_\n$ and
$\Wfu^\psi_\n$.
\begin{figure}[htp]
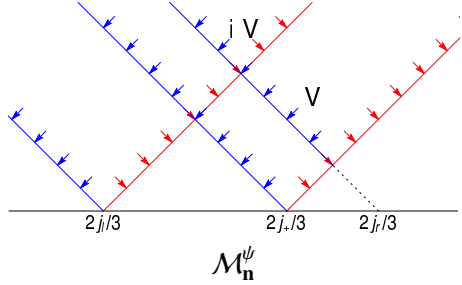

\begin{center}\grf6{pj2M}\\
$\Mfu_\n^\psi$\end{center}
\caption{Structure of $\Mfu^\psi_\n$ for $\ell>0$, $\nu_+\geq |j_+|+2$
and $m_0(j_r) < 0 \leq m_0(j_+)\leq m_0(j_l)$. See \S\ref{sect-pj2}.
}\label{fig-pj2M}
\end{figure}
The module $\Mfu^\psi_\n$ has one irreducible submodule, coinciding with
the irreducible submodule of $\Vfu_\n^\psi$ of type $\II_+(j_+,\nu_+)$.

\rmrk{Coinciding values $j_l=j_+$} The conclusions go through if
$j_l=j_+$, which happens for $\nu_+ = - j_+ \geq 1$. Then we obtain the
configuration in Figure~\ref{fig-pj2b}.
\begin{figure}[htp]
\begin{center}
\grf5{pj2bV}\quad\grf5{pj2bW}\\
$\Vfu_\n^\psi$
\hspace{4.5cm} $\Wfu_\n^\psi$\\
\grf5{pj2bM}\\
$\Mfu^\psi_\n$
\end{center}
\caption{Structure of $\Vfu_\n^\psi$, $\Wfu_\n^\psi$, and $\Mfu^\psi_\n$
for $\ell>0$, $\nu_+=-j_+\geq 1$, and $m_0(j_r)< 0 \leq m_0(j_+) =
m_0(j_l)$. See \S\ref{sect-pj2}. }\label{fig-pj2b}
\end{figure}
In this special case, $\nu_+=-j_+=\nu_l=-j_l\in \ZZ_{\geq 1}$, and the
irreducible submodules have the isomorphism types $\II_+(j_+,-j_+) $,
$h_0=-j_r=2j_+$, $p_0=j_+$; and $\FI(j_l,j_l) $ with $j_l\leq -1$. For
the latter type we note that $h_0=2j_+=2j_l$, $p_0=0$,
$A=\frac13(j_r+j_+)-1 = \nu_l-1 = |j_l|-1$, and consult
Table~\ref{tab-isot-ps}, p~\pageref{tab-isot-ps}.

\medskip\par
\subsubsection{}\emph{Two {$m_0(j)$ non-negative, $\ell<0$}.
}\label{sect-mj2}
For $\ell<0$ the function $m_0$ is strictly increasing, and $m_0(j_l) <
0 \leq m_0(j_+) \leq m_0(j_r)$. By a reasoning similar
to~\S\ref{sect-mj2} we find the configurations in Figures \ref{fig-mj2}
(for $j_+<j_r$)
and~\ref{fig-mj2b} (for $j_+=j_r$).
\begin{figure}[htp]
\begin{center}\grf{5}{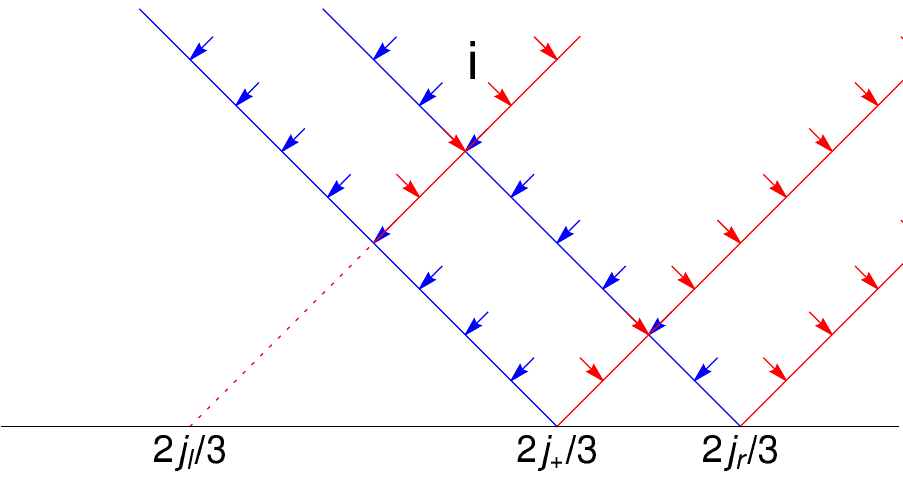}\quad \grf{5}{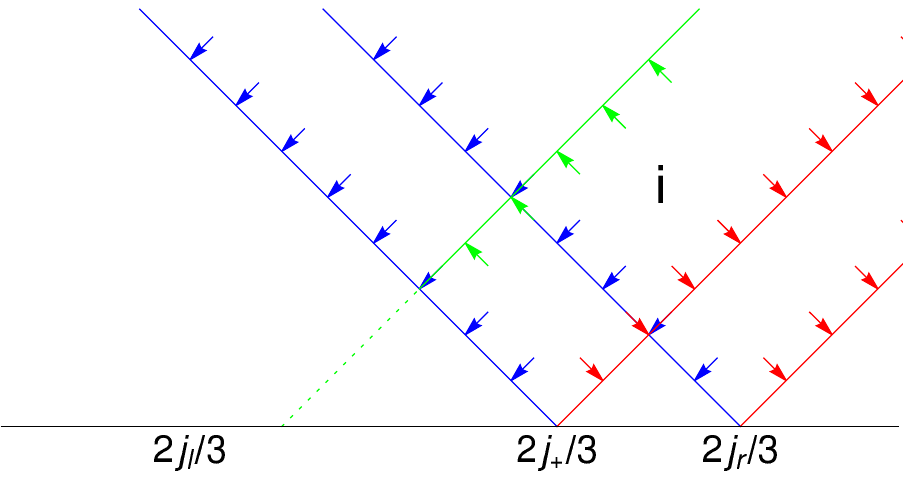}\\
$\Vfu_\n^\psi$
\hspace{4.5cm} $\Wfu_\n^\psi$\\
\grf5{mj2M}\\
$\Mfu^\psi_\n$
\end{center}
\caption{Structure of $\Vfu_\n^\psi$, $\Wfu_\n^\psi$, and $\Mfu^\psi_\n$
for $\ell<0$, $\nu_+\geq |j_+|+2$, and $m_0(j_l)<0 \leq m_0(j_+) <
m_0(j_r)$. See \S\ref{sect-mj2}.} \label{fig-mj2}
\end{figure}
\begin{figure}[htp]
\begin{center}\grf{5}{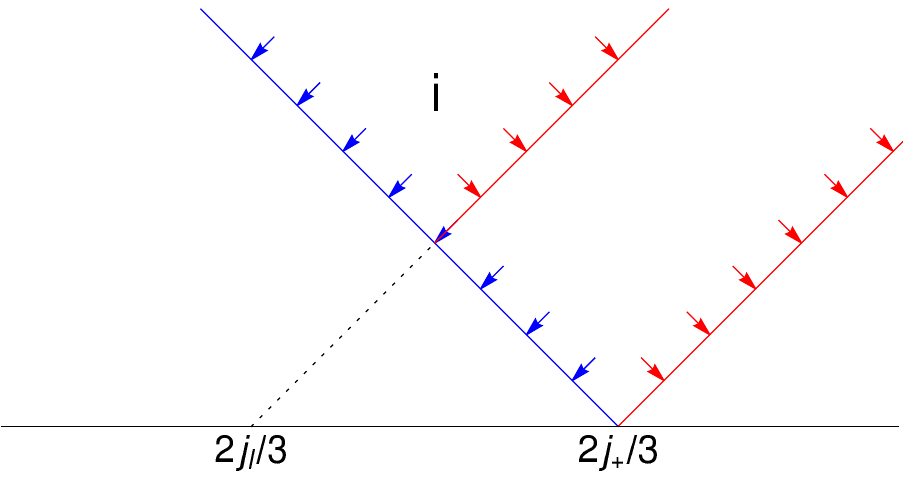}\quad \grf{5}{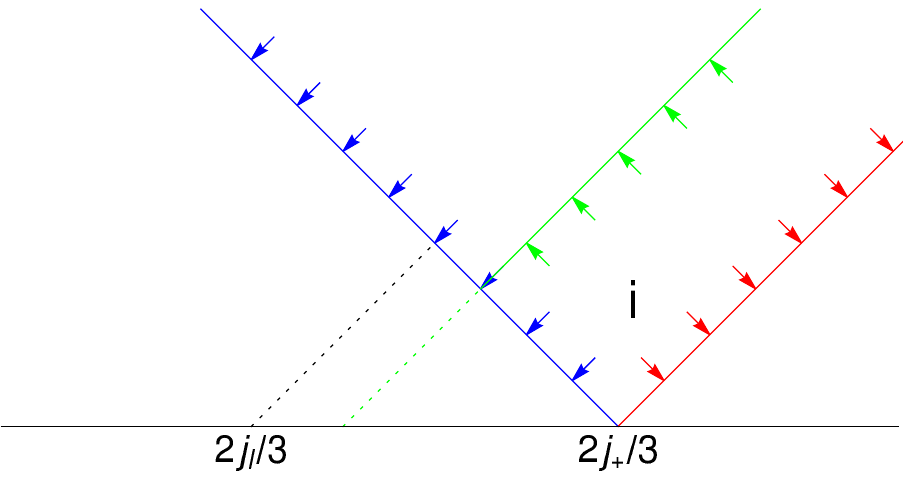}\\
$\Vfu_\n^\psi$
\hspace{4.5cm} $\Wfu_\n^\psi$\\
\grf5{mj2bM}\\
$\Mfu^\psi_\n$
\end{center}
\caption{Structure of $\Vfu_\n^\psi$, $\Wfu_\n^\psi$, and $\Mfu^\psi_\n$
for $\ell<0$, $\nu_+=|j_+|$, and $m_0(j_l)<0 \leq m_0(j_+) = m_0(j_r)$.
See \S\ref{sect-mj2}.} \label{fig-mj2b}
\end{figure}

\rmrk{Three different $j$-values} For $\nu_+ \geq |j_+|+2$, we obtain an
irreducible module of type $\II_+(j_+,\nu_+)$ in the intersections of
$\Vfu^\psi_\n$ and $\Mfu^\psi_\n$, with $h_0=j_++j_r=-j_l $ and
$p_0=\nu_+ $, and in $\Wfu^\ps_\n$ an irreducible module of type
$\IF_+(j_r,-\nu_r)$ with $h_0= j_r+j_+=-j_l$,
$p_0 =\frac13(j_r-j_+)= \nu_l$, $A=\infty$, and $B=
\frac13(j_+-j_l) -1= \nu_r-1$.

\rmrk{Coinciding values $j_+=j_r$} In the situation of
Figure~\ref{fig-mj2b} we have $j_r=j_+=\nu_r-\nu_+$ with $j_r\in
\ZZ_{\geq 1}$. Then $\Vfu^\ps_\n$ has again an irreducible submodule of
type $\II_+(j_+,j_+)$. The irreducible submodule of $\Wfu_\n^\ps$ has
parameters $h_0=2j_+=2j_r$, $p_0=0$, $A=\infty$, and
$B=\frac13(j_+-j_l)-1 = \nu_r-1$, and has type $\IF(j_r,-j_r)$, with
$j_r\geq 1$.

\medskip\par
\subsubsection{}\emph{One $m_0(j)$ non-negative, $\ell>0$.
}\label{sect-pj1}
We consider $m_0(j_r)\leq m_0(j_+) <0 \leq m_0(j_l)$ for $\ell>0$.

\rmrk{Three different $j$-values} Lemma~\ref{lem-usho-upsom} gives lines
$h-2j_++3p+6=0$ and $h-2j_r+3p+6=0$ with $K$-types on which $\sh 3 1$
vanishes in $\Wfu_\n^\ps$, and Lemma~\ref{lem-kdso-na} gives the
vanishing of the downward shift operators on all boundary lines of
sectors for $\Vfu^\ps_\n$, and on the boundary lines of $\sect(j_l)$
for $\Wfu_\n^\ps$. For $j_l\neq j_+$, this leads to the configuration
in Figure~\ref{fig-pj1VW}.
\begin{figure}[htp]
\begin{center}
\grf{5}{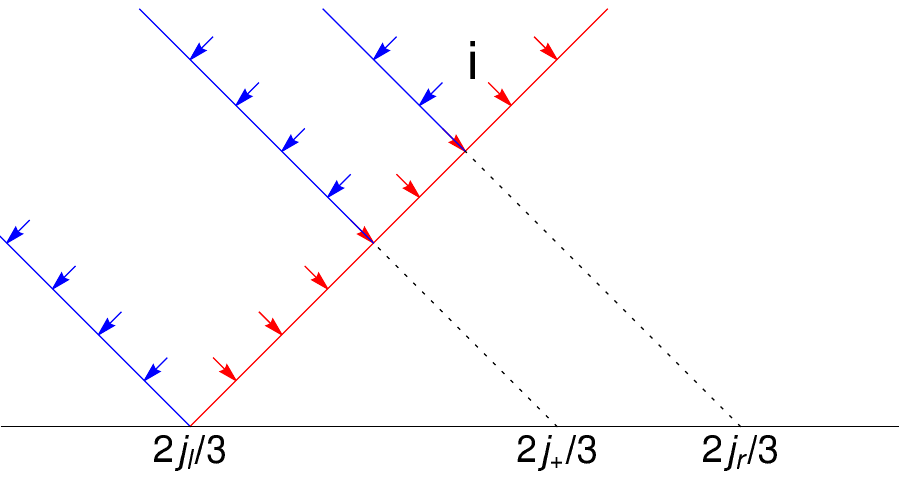} \quad \grf5{pj1W}\\
$\Vfu_\n^\ps$
\hspace{4.5cm} $\Wfu^\ps_\n$
\end{center}
\caption{Structure of $\Vfu^\ps_\n$ and $\Wfu^\ps_\n$ for $\ell>0$, $
m_0(j_l) \geq 0 > m_0(j_+) > m_0(j_r)$. See~\S\ref{sect-pj1}. }
\label{fig-pj1VW}
\end{figure}

The module $\Vfu_\n^\ps$ has an irreducible submodule of type
$\II_+(j_+,\nu_+)$, $h_0=-j_+$, $p_0=\nu_+$. The module $\Wfu_\n^\ps$
has an irreducible submodule of type $\FI(j_l,\nu_l)$, with
$h_0=-2j_l$, $p_0=0$, and $A=\nu_r-1$, with $1\leq \nu_l \leq -j_l-2$.

\begin{figure}[htp]
\begin{center}
\grf{5}{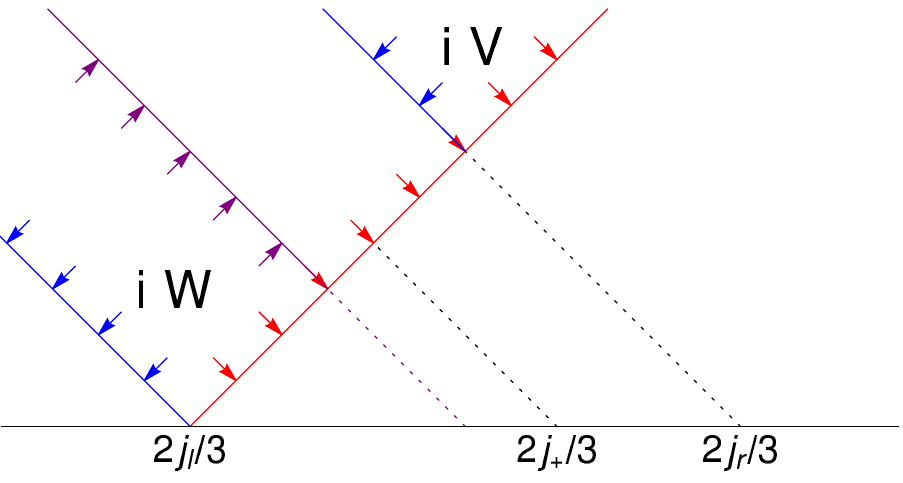} \\
$\Mfu^\ps_\n$
\end{center}
\caption{Structure of $\Mfu^\ps_\n$ for $\ell>0$, $ m_0(j_l) \geq 0 >
m_0(j_+) > m_0(j_r)$. See~\S\ref{sect-pj1}. } \label{fig-pj1M}
\end{figure}
Proposition~\ref{prop-M-VW} ii)a) tells us that the irreducible
submodule of $\Wfu_\n^\ps$ is contained in $\Mfu_\n^\ps$, and by i)b)
in that proposition, the irreducible submodule of $\Vfu^\ps_\n$ is also
contained in~$\Mfu^\ps_\n$. This leads to the sketch of $\Mfu^\ps_\n$
in Figure~\ref{fig-pj1M}.

\rmrk{Coinciding values $j_+=j_r$} If $j_+=j_r$ we get the configuration
described in Figure~\ref{fig-pj1a}.
\begin{figure}[htp]
\begin{center}\grf{5}{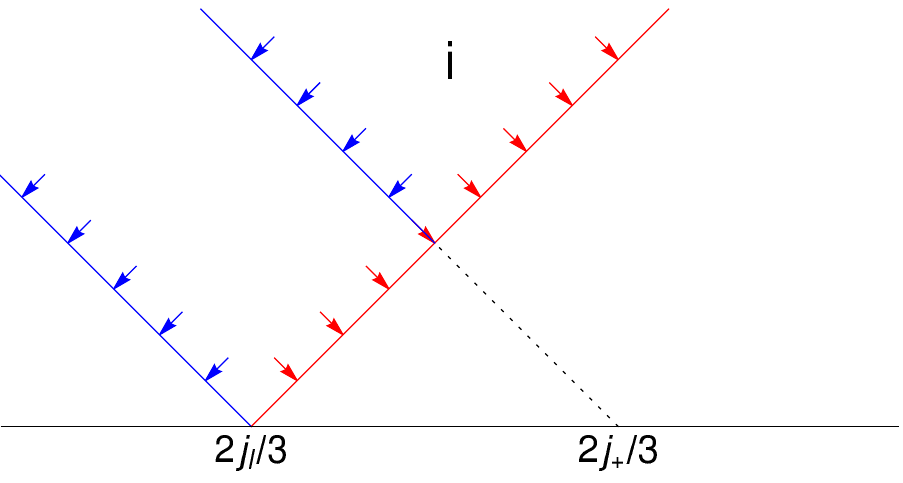}\quad \grf{5}{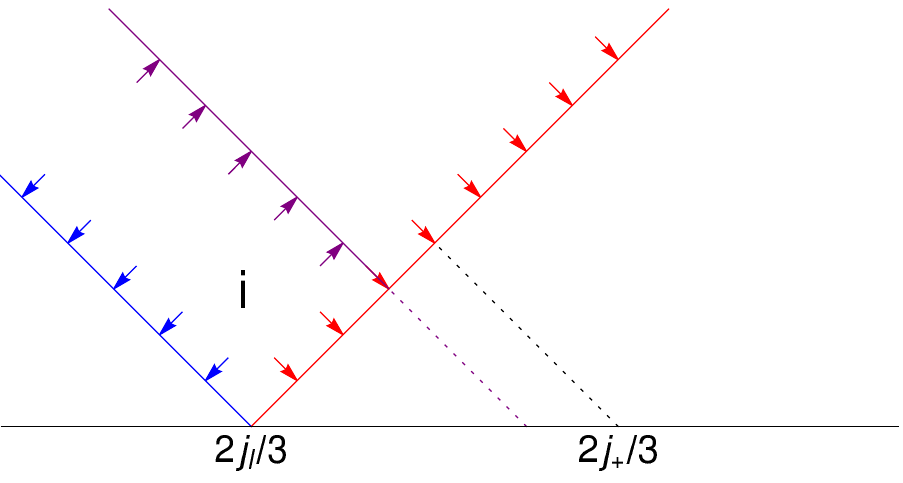}\\
$\Vfu_\n^\psi$
\hspace{4.5cm} $\Wfu_\n^\psi$\\
\grf5{pj1aM}\\
$\Mfu^\psi_\n$
\end{center}
\caption{Structure of $\Vfu_\n^\psi$, $\Wfu_\n^\psi$, and $\Mfu^\psi_\n$
for $\ell>0$, $j_r=j_+$, and $m_0(j_r)= m_0(j_+) <0\leq m_0(j_l)$. See
\S\ref{sect-pj1}.} \label{fig-pj1a}
\end{figure}
The module $\Vfu^\ps_\n$ has an irreducible module of type
$\II_+(j_l,j_l)$, and $\Wfu^\ps_\n$ and $\Mfu^\ps_\n$ contain an
irreducible submodule of type $\FI(j_l,0)$ with parameters $h_0=2j_l$,
$p_0=0$, $A=\nu_+-1$, $B=\infty$, and the condition $j_l \leq -2$,
even.

\medskip\par
\subsubsection{}\emph{One $m_0(j)$ non-negative, $\ell<0$.
}\label{sect-mj1}
The last case is $\ell<0$, $m_0(j_l)\leq m_0(j_+)<0\leq m_0(j_r)$.
Analogously to \S\ref{sect-mj1} we get the situation in
Figures~\ref{fig-mj1} and~\ref{fig-mj1a}.

\begin{figure}[htp]
\begin{center}\grf{5}{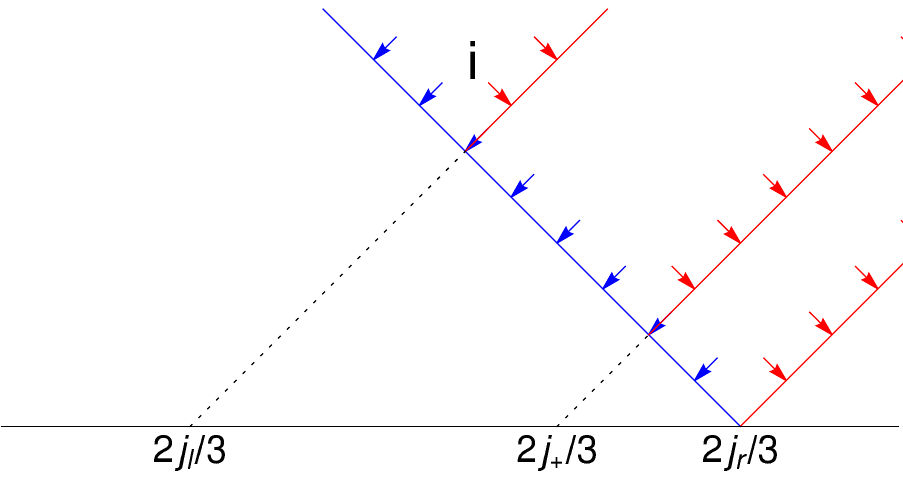}\quad \grf{5}{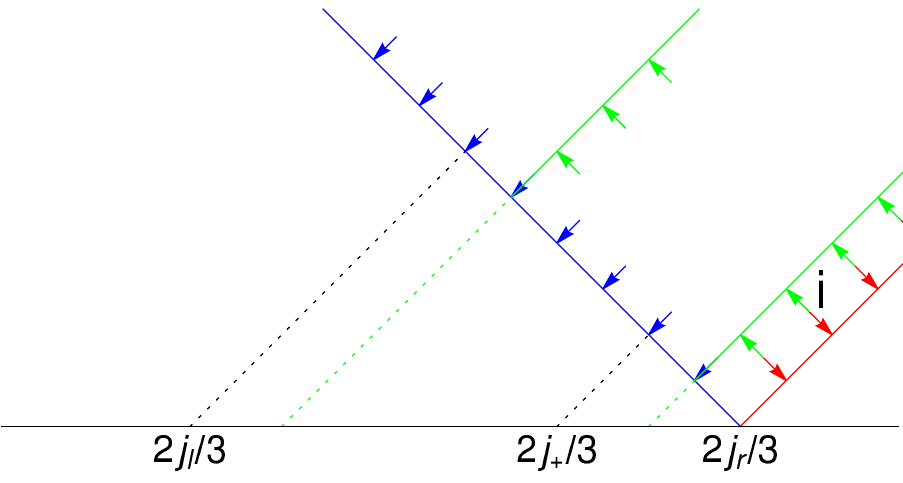}\\
$\Vfu_\n^\psi$
\hspace{4.5cm} $\Wfu_\n^\psi$\\
\grf5{mj1M}\\
$\Mfu^\psi_\n$
\end{center}
\caption{Structure of $\Vfu_\n^\psi$, $\Wfu_\n^\psi$, and $\Mfu^\psi_\n$
for $\ell<0$, $\nu_+\geq |j_+|+2$ and $m_0(j_l)< m_0(j_+) < 0\leq
m_0(j_r)$. See \S\ref{sect-mj1}.} \label{fig-mj1}
\end{figure}
We find an irreducible submodule of type $\II_+(\nu_+,j_+)$ in
$\Vfu^\ps_\n\cap \Mfu^\ps_\n$, and an irreducible submodule of type
$\IF(j_r,\nu_r)$ in $\Wfu_\n^\ps\cap \Mfu_\n^\ps$.
\begin{figure}[htp]
\begin{center}\grf{5}{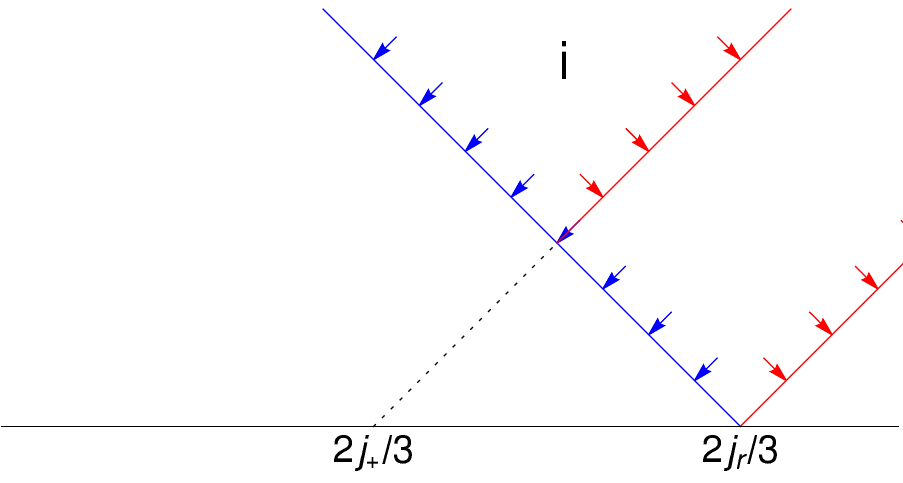}\quad \grf{5}{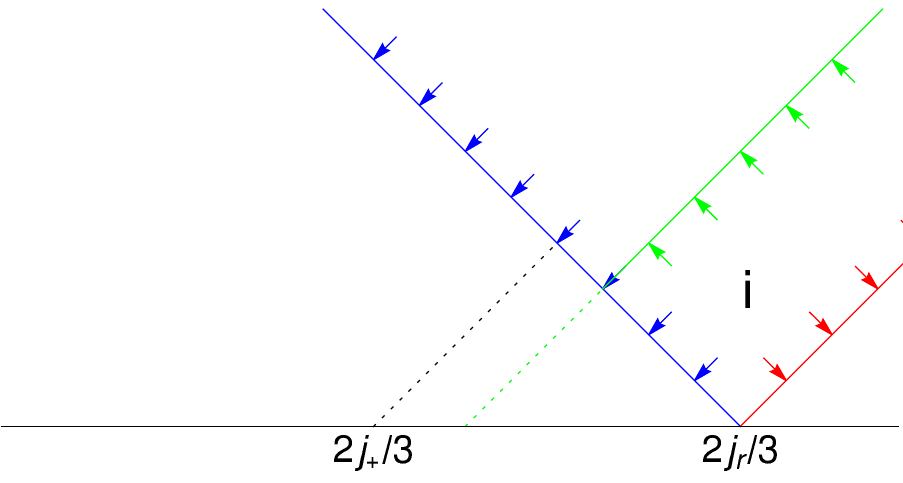}\\
$\Vfu_\n^\psi$
\hspace{4.5cm} $\Wfu_\n^\psi$\\
\grf5{mj1aM}\\
$\Mfu^\psi_\n$
\end{center}
\caption{Structure of $\Vfu_\n^\psi$, $\Wfu_\n^\psi$, and $\Mfu^\psi_\n$
for $\ell<0$, $\nu_+=-j_l$ and $m_0(j_l)= m_0(j_+)<0 \leq m_0(j_r)$.
See \S\ref{sect-mj1}. } \label{fig-mj1a}
\end{figure}
The latter module has $1\leq \nu_r\leq j_r-2$ in Figure~\ref{fig-mj1}
and $\nu_r=0$ in~Figure~\ref{fig-mj1a}.

\clearpage


\def\flnm{rFtm-III-unstr}


\section{Unitary structure}\label{sect-un}\markright{15. UNITARY
STRUCTURE}

A $(\glie,K)$-module $V$ is \emph{unitarizable} if there is a positive
definite sesquilinear invariant form on $V$. If $V$ is irreducible such
sesquilinear forms are unique up to a positive factor. Unitarizability
is a property of the isomorphism class of the module. For each
isomorphism class of irreducible modules, we will determine in this
section whether it is unitarizable.

\subsection{Invariant sesquilinear forms and unitarizability} We will
use \il{sqlf}{sesquilinear form}sesqui\-linear forms $(\cdot,\cdot)$,
which are complex linear in the first variable and conjugate complex
linear in the second variable. Such a form is positive definite if
$(x,x)\geq 0$ for all $x$ in the domain, and, furthermore, if $(x,x)=0$
implies $x=0$.

A sesquilinear form on a $\glie$-module $V$ is \il{invsf}{invariant
sesquilinear form}invariant if
\be \label{sesq}\bigl( \XX v,w\bigr) + \bigl( v,\bar\XX w
\bigr)\=0\qquad \text{for all }v,w\in V, \XX\in \glie_c\,.\ee
By $\bar \XX$ for $\XX\in \glie$ we denote the complex conjugate with
respect to the real Lie algebra $\glie\subset\glie_c$. If $V$ is a
$(\glie,K)$-module we have $\bigl( k v, k w\bigr) = (v,w)$ for all
$k\in K$. If $V$ is irreducible all invariant sesquilinear forms on it
are proportional. A \il{us}{unitarizable} $(\glie,K)$-module $V$ is
unitarizable if it allows a positive definite invariant sesqui\-linear
form on $V$. If $V$ is irreducible all positive definite sesquilinear
forms are related by a positive factor.

In \S\ref{sect-usps} and \S\ref{sect-usdst} we will prove the following
result:
\begin{thm}\label{thm-us}The following isomorphism classes of
irreducible $(\glie,K)$-mod\-ules are unitarizable.
\begin{itemize}
\item \emph{Principal series}
\begin{itemize}
  \item \emph{ Irreducible unitary principal series: }\ $\II(j,\nu)$
  with $\nu \in i\RR$ and $j\in \ZZ$; if $\nu =0$, then
  $j\in \{0\}\cup\left(1+2\ZZ\right)$.
  \item \emph{Complementary series: }\ $\II(j,\nu)$ with either
   $0<\nu<2$, $j=0$, or $0<\nu<1$ and $j\equiv 1\bmod 2$.
\end{itemize}
\item \emph{Discrete series types}
\begin{itemize}
  \item \emph{Large dst: }\ $\II_+(j_+,\nu_+)$ with
  $\nu_+ \equiv j_+\bmod 2$, $\nu_+\geq |j_+|$, $\nu_+\geq 1$.

  \item \emph{Holomorphic dst: }\ $\IF(j_r,\nu_r)$ with
$0\leq \nu_r \leq j_r-2$.

   \item \emph{Antiholomorphic dst: }\ $FI(j_l,\nu_l)$ with
$1\leq  \nu_l \leq -j_l$.

\end{itemize}
\item \emph{Langlands representations}
\begin{itemize}
  \item \emph{Thin representations:}\ $T^+_{-1} = \IF(1,-1)$,
  $T^-_{-1}= \FI(-1,-1)$, $T^+_k = \IF_+(2k+3,-1)$,
  $T^-_k = \FI_+(-2k-3,-1)$ with $k\in \ZZ_{\geq 0}$.\il{Tkpm}{$T^\pm_k$}

\item \emph{Trivial representation: }$\FF(0,-2)$.
\end{itemize}
\end{itemize}
\end{thm}

\rmrk{Remarks} (1) \ These are the isomorphism classes of unitarizable
modules that Wallach gives in~\cite[\S7]{Wal76}. For the thin
representations we have also indicated the notations of that paper.

(2) \il{thrpru}{thin representation}The thin representations turn up in
the continuous cohomology of~$G$; \cite[Lemma 9.2]{Wal76}, and
Section~3 and Theorem~4, ii), in~\cite{Ish00}. See also \cite[Theorem
2.5]{VZ84}.

The term ``thin representation'' is used by Ishikawa \cite{Ish00}. The
$K$-types in the thin representations correspond to a single line in
the $\bigl( \frac h3,p \bigr)$-plane. There are discrete series
representation that have the same property; see Figure~\ref{fig-ps3a}
for an example.

\subsection{Principal series representations}
\label{sect-usps}One knows the positive definite ses\-qui\-lin\-ear forms on
principal series representations  in a much more general
context. See for instance Baldoni Silva and Barbasch \cite{BB83} for
$\RR$-rank one groups.

For $\SU(2,1)$ we have the following.
\begin{prop}\label{prop-pscs}The irreducible principal series
representation $H^{\xi,\nu}_K$ is unitarizable in precisely the
following cases:
\begin{enumerate}
\item[i)] \emph{Unitary principal series. } \il{ups}{unitary principal
series}For $\nu\in i \RR$ and $j\in \ZZ$, with the sesquilinear form
determined by \il{unprs}{$\bigl(\cdot,\cdot \bigr)_\uprs$}
\be \bigl( \kph hprq(\nu), \kph{h'}{p'}{r'}{q'}
(\nu)\bigr)_\uprs \= \dt_{p,p'}\,\dt_{r,r'}\, \dt_{q,q'} \, \bigl\|\Kph
h{p}{r}{q}\bigr\|^2_K\,. \ee
\item[ii)] \emph{Complementary series. } \il{cos}{complementary series}
For $\nu \in \RR$, with $0<|\nu|<2$ if $j=0$, and $|\nu|<1$ if
$j\equiv1\bmod 2$, with the sesquilinear form determined
by\ir{compls}{\bigl(\cdot,\cdot \bigr)_\cmpl}
\badl{compls}& \bigl( \kph h{p}{r}{q}(\nu), \kph{h'}{p'}{r'}{q'}
(\nu)\bigr)_\cmpl \\
&\qquad\=\dt_{p,p'}\,\dt_{r,r'}\, \dt_{q,q'} \, \frac{\Gf\bigl(
1+\frac{j-\nu+p+r}2\bigr) \, \Gf\bigl( 1+\frac{-j-\nu+p-r}2\bigr) }
{\Gf\bigl( 1+\frac{j+\nu+p+r}2\bigr) \, \Gf\bigl(
1+\frac{-j+\nu+p-r}2\bigr) }\, \bigl\|\Kph h{p}{r}{q}\bigr\|^2_K \,.
\eadl
\end{enumerate}
\end{prop}
These sesquilinear forms are determined up to a positive factor. If
$\nu=0$ the form $\bigl( \cdot,\cdot\bigr)_\cmpl$ coincides with the
form $\bigl( \cdot,\cdot\bigr)_\uprs$.

\rmrk{Discussion}The existence and description of the sesquilinear form
of the unitary principal series follows, for instance, from Theorem 2
in \S2, Chap.~III of \cite{La74}. Take
$K_{\mathrm{there}}=\SU(2) \subset K$ and $P_{\mathrm{there}}=NAM$.
Wallach gives a discussion of the principal series representation for
$\SU(2,1)$ in Section 7 of~\cite{Wal76}.

More precisely, the existence of $\bigl( \cdot,\cdot \bigr)_\uprs$ as a
non-degenerate sesquilinear form on $H^{\xi,\nu}_K$ is equivalent to an
identification of $H^{\xi,\nu}_K$ with its conjugate dual, which is
$H^{\xi,-\bar \nu}_K$. So in i) we have to take $\re\nu=0$.

The family $\ii_0$ in~\eqref{ii0} gives an isomorphism
$H^{\xi,\nu} \rightarrow H^{\xi,-\nu}_K$ under general parametrization.
For $\nu \in \RR$ we get the sesquilinear form
\be\label{cmpl1} \bigl( \ph_1,\ph_2\bigr)_\cmpl \= \bigl( \ph_1,\ii_0\,
\ph_2\bigr)_\uprs\,.\ee
Here we check only that $\bigl( \cdot,\cdot \bigr)_\cmpl$ is positive
definite under the conditions in~ii).

\begin{lem}If $\nu\in \RR$ and $(j,\nu)$ corresponds to generic
parametrization the sesquilinear form $\bigl( \cdot,\cdot\bigr)_\cmpl$
is positive definite under the conditions in
Proposition~\ref{prop-pscs}, ii).
\end{lem}
\begin{proof}For given $j\in \ZZ$ the factor
\be\label{cprnu} c(p,r,\nu) \= \frac{\Gf\bigl( 1+\frac{j-\nu+p+r}2\bigr)
\, \Gf\bigl( 1+\frac{-j-\nu+p-r}2\bigr) } {\Gf\bigl(
1+\frac{j+\nu+p+r}2\bigr)
\, \Gf\bigl( 1+\frac{-j+\nu+p-r}2\bigr) }\ee
in \eqref{compls} should have the same sign for all $p\in\ZZ_{\geq 0}$
and $r\equiv p\bmod 2$, $|r|\leq p$.

Writing $A=1+j/2$, $B=1-j/2$, $a=\frac{p+r}2$ and $b=\frac{p-r}2$ we are
in the situation of Lemma~\ref{lem-AB} below. Hence if $j\in 2\ZZ$ we
need $j=0$, and $\frac\nu 2\in (-1,1)$, and if $j\equiv 1\bmod 2$ then
$|x| < \max\bigl( \frac12, 1-\frac{|j|}2 \bigr)= \frac12$.
\end{proof}

\begin{lem}\label{lem-AB}Let $x\in \RR$, $A, B \in \frac12 \ZZ$,
$A\equiv B \bmod 1$. Then
\[ p(a,b,x) \= \frac{\Gf(A+a-x)\, \Gf(B+b-x)} {\Gf(A+a +x)
\,\Gf(B+b+x)}\,>0\, \text{ for all }a,b\in \ZZ_{\geq 0}\]
if and only if
\begin{align*}
&A,B\in \ZZ_{\geq 1} \text{ and } |x| < \min(A,B)\,,\\
\text{or }&A,B\in \frac12+\ZZ \text{ and } |x| < \max\bigl(
\tfrac12,\min(A,B) \bigr)\,.
\end{align*}
\end{lem}
\begin{proof}
First we consider $A,B\in \ZZ$. Then we need to have
$p(a+1,b)/p(a,b)= \frac{A+a-x}{A+a+x}>0$ for all $a,b\in \ZZ_{\geq 0}$.
If $A=0$ this does not hold for $a=0$. So we need $A\geq 1$. Then
$A+a>0$ for all $a\in \ZZ_{\geq 0}$, and $A+a+|x|>0$. Then we need also
$A+a-|x|>0$, hence $|x|< A$. If this condition is satisfied the
quotient $\Gf(A+a-x)/\Gf(A+a-x)$ is indeed positive.

Similarly we arrive at the condition $|x|<B$.\medskip

Now let $A,B\in \frac 12+\ZZ$. Then $A+a$ does not take the value $0$,
and from the requirement that $\Gf(A+a-x)/ \Gf(A+a+x)>0$ we arrive at
the condition that $|x| < |A+a|$ for all $a\geq 0$. If $A>0$ this leads
to the necessary condition $|x|<A$. If $A<0$ we take $a=-A-\frac12$ to
get the necessary condition $|x|<\frac12$. We check that
$\Gf(A+a-x)/ \Gf(A+a+x)$ is indeed positive if this condition holds.

For the other quotient we arrive at the condition
$|x|< \max(\frac12,B)$. Both conditions together give the condition in
the lemma.
\end{proof}

\subsection{Other irreducible modules}
\label{sect-usdst}Irreducible modules occur in principal series
representations $H^{\xi,\nu}_K$, as the whole of $H^{\xi,\nu}_K$ under
general para\-metrization, and as a genuine submodule under integral
parametrization. So for the types $\II_+$, $\IF$, $\FI$ and $\FF$ we
assume that $(j,\nu)\in \WOI$.

Equation \eqref{cmpl1} defines a sesquilinear form with help of the
meromorphic family $\ii_0$ of morphisms
$H^{\xi,\nu}_K \rightarrow H^{\xi,-\nu}_K$ of $(\glie,K)$-modules
defined in~\eqref{ii0}. At values $\nu=\nu_0$ in $\ZZ$ it need not be
an isomorphism. It may even have a singularity. Replacing $\ii_0$ by
$\al(\nu)\,\ii_0$ for a suitable analytic function~$\al$ we may remove
the singularity at~$\nu_0$. For an irreducible module, the resulting
sesquilinear form is unique up to a constant in $\CC^\ast$ if we work
with an irreducible module.

So the outcome of a check whether the form can be made positive definite
on the submodule of $H^{\xi,\nu}_K$ in which we are interested,
determines the unitarizability of this module. For $\nu=0$ we work with
$\bigl( \cdot,\cdot\bigr)_\uprs$ instead of
$\bigl( \cdot,\cdot)_\cmpl$.

\subsubsection{}\emph{Isomorphism types $\II_+$. }The isomorphism class
$\II_+(j,\nu)$ can be represented by a module $V \subset H^{\xi,\nu}_K$
with $(j,\nu)\in \ZZ^2$, $j\equiv\nu\bmod 2$ and $\nu\geq \max(|j|,1)$.
The module $V$ has parameters
$\bigl[ \ld_2(j,\nu); -j,\nu;\infty,\infty\bigr]$. See
Figure~\ref{fig-ps1}.

The basis vectors occurring in $V$ are $\kph h{p}{r}{p}(\nu)$ with
$p=\nu+a+b$, $h=-j+3(a-b)$ and $r=\frac13(h-2j) = a-b-j$ with
$a,b\in \ZZ_{\geq 0}$. The factor $c$ in~\eqref{cprnu} takes the value
\be \label{c-lds} c(p,r,\nu) \= \frac{ a!\, b! }{ (a+\nu)!\, (b+\nu)!
}\,. \ee
This is well defined for $a,b\in \ZZ_{\geq 0}$, and positive. So the
isomorphism class $\II_+(j,\nu)$ is unitarizable. See
\cite[\S25a]{Math} for computations.

\subsubsection{}\emph{Isomorphism types $\IF$ and $\FI$. }In
Table~\ref{tab-isot-ps}, p~\pageref {tab-isot-ps}, we listed the
isomorphism classes $\IF(j,\nu)$ and $\FI(j,\nu)$ and parameters of a
submodule of $H^{\xi,\nu}_K$ representing this class. From this table
we collect, and reformulate with \eqref{jnurels}, the information in
Table~\ref{tab-IFFI}.
\begin{table}[ht]
\[\renewcommand\arraystretch{1.3}
\begin{array}{|c|cc|cccc|}\hline
&\text{isomph.~class}&&h_0&p_0&A&B\\ \hline
\text{a}&\IF(j,\nu)&1\leq \nu \leq j-2&2j& 0&\infty&\frac{j-\nu}2-1\\
\text{b}&\IF(j,0) & j\in 2\ZZ_{\geq 1}&2j&0&\infty&\frac j2-1 \\
\text{c}&\IF_+(j,-\nu)& 1\leq \nu \leq j-2 &\frac{j+3\nu}2&
\frac{j-\nu}2-1&\infty& \nu-1 \\
\text{d}& \IF(j,-j)&j\geq 1 & 2j&0&\infty&j-1\\
\text{e}& \FI (j,\nu)& 1\leq \nu \leq -j-2&2j&0 &
-\frac{j+\nu}2-1 & \infty\\
\text{f}& \FI(j,0)& j\in 2\ZZ_{\leq -1}& 2j& 0 & -\frac j2-1 & \infty\\
\text{g}&\FI_+(j,-\nu) &1\leq \nu \leq |j|-2& \frac{j-3\nu}2&
-\frac{\nu+j}2&\nu-1&\infty\\
\text{h}&\FI(j,j)& j\leq -1 &2j& 0 & -j-1 & \infty
\\ \hline
\end{array}
\]
\caption{Isomorphism classes of types $\IF$ and $\FI$.}\label{tab-IFFI}.
\end{table}

\rmrk{Cases b and f} In cases b and f the irreducible module is
contained in $H^{\xi,0}_K$. By restriction of
$\bigl(\cdot,\cdot\bigr)_\uprs$ we get a positive definite sesquilinear
form.

\rmrk{Case a} We take
\be c_a(p,q,\nu') \=
\frac{\sin\pi\frac{\nu'+j}2}{\sin\pi\frac{\nu'-j}2}\, c(p,r,\nu')\,.\ee
With some relations for gamma functions and goniometrical functions this
meromorphic function in $\nu'$ can be written as
\be c_a(p,q,\nu') \= (-1)^{p-r}\, \frac{\Gf\bigl(
\frac{-\nu'-p+r+j}2\bigr)\, \Gf\bigl( 1+\frac{-\nu'+p+r+j}2\bigr)} {
\Gf\bigl(
\frac{\nu'-p+r+j}2\bigr)\,\Gf(1+\frac{\nu'+p+r+j}2\bigr)}\,.\ee
We write $p=p_0+a+b$, $h=h_0+3(a-b)$, and $r=\frac13(h_0-2j)+a-b$, and
obtain
\bad c_a(p,r,\nu)&\= \frac{ \bigl( a+\frac{j-\nu}2\bigr)!\; \bigl(
\frac{j-\nu}2 -b-1\bigr)! }{ \bigl( a+\frac{j+\nu}2\bigr)!\; \bigl(
\frac{j+\nu}2 -b-1\bigr)! }\\
&\=\frac{(1+a+B)!\;(B-b)!}{(B+a+\nu+1)!\;(\nu+B-b)!} \,. \ead
The factors depending on $a$ are positive for all $a\in \ZZ_{\geq 0}$;
for the factors with $b$ we need $0\leq b \leq \frac{j-\nu}2-1=B$. This
means that $c_a$ induces a positive definite sesquilinear form.

\rmrk{Case c}
\bad c_c(p,&r,\nu')\=\bigl(\sin \pi\tfrac{j-\nu'}2\bigr)^{-1} \,
c(p,r,\nu')\\
&\= (-1)^{(p-r)/2}\,\pi^{-1} \,\frac{\Gf\bigl(
1+\frac{-\nu'+p-r-j}2\bigr)\, \Gf\bigl(\frac{-\nu'-p+r+j}2\bigr)\,
\Gf\bigl(1+\frac{-\nu'+p+r+j}2\bigr)
}{ \Gf\bigl( 1+\frac{-\nu'+p+r+j}2\bigr)
}\,, \ead
which specializes to
\be c_c(p,r,-\nu) \= (-1)^{b+1} \, i^{j-\nu} \frac{ b!\; (B-b)!\;
\bigl(a+(\nu+j)/2\bigr)!}{ \pi \, \bigl( a+(j-\nu)/2\bigr))!}\,. \,.\ee
Only if $\nu=1$ this determines a positive definite sesquilinear form.

The irreducibility of the module implies that the sesquilinear form is
unique up to a multiple. Hence a positive definite sesquilinear form is
impossible if $\nu > 1$. Or, alternatively, suppose that $q$ is an
invariant sesquilinear form on the submodule of type $\IF_+(j,-\nu)$ of
$H^{\xi,\nu}_K$ with $\nu \geq 2$. Then it satisfies for
$\ph_0 = \kph{h_0}{p_0}_{r_0}{p_0}(-\nu)$, and
$\ph_1= \kph{h_0-4}{p_0+1}{r_0-1}{p_0+1}(-\nu)$ the relation
\[ q\bigl( \sh{-3}1 \ph_0, \ph_1 \bigr) + q\bigl(\ph_0,\sh 3 {-1} \ph_1
\bigr)= 0\,.\]
We use that $\sh {-3}1$ is given by $\Z_{23}$, and
$\overline{\Z_{23}}= \Z_{32}$. With Table~\ref{tab-sho},
p~\pageref{tab-sho}, and~\eqref{shps} we obtain explicit factors $c_u$
and $c_d$ such that
\be c_u \bigl( \ph_1,\ph_1\bigr) + c_d \bigl( \ph_0,\ph_0 \bigr) =
0\,.\ee
It turns out that the product $c_u c_d$ is positive for $\nu \geq 2$. So
there cannot be a positive definite invariant sesquilinear form.

\rmrk{The other cases} In \cite[\S25c]{Math} we handle the other cases
in a similar way.

\subsubsection{} \emph{Type $\FF$. }We use $\nu \leq |j|-2$ to describe
the isomorphism class $\FF(j,-\nu)$ with parameters
\[ \bigl[\ld_2(j,-\nu);2j,0;\tfrac{\nu-j}2-1,\frac{j+\nu}2-1 \bigr]\,.\]
See Figure~\ref{fig-ps7}.

With
\[ c_f(p,r,\nu)
\=\frac{c(p,r,\nu)}{\sin\pi\frac{j+\nu}2\;\sin\pi\frac{\nu-j}2 }\]
this leads to
\[ c_f(p,r,\nu) \= \frac{(-1)^{a+b}}{\pi^2}\, \bigl(
a+\tfrac{\nu+j}2\bigr)!\; \bigl( \tfrac{\nu-j}2-1-a\bigr)!\; \bigl(b+
\tfrac{\nu-j}2\bigr)!\; \bigl(\tfrac{\nu+j}2-1-b\bigr)!\,.\]
See \cite[\S25d]{Math}. This is well defined for
$0\leq a \leq \frac{\nu-j}2-1$ and $0\leq b \leq \frac{\nu+j}2-1$. The
factor $(-1)^{a+b}$ shows that the sesquilinear form is
positive-definite only if $A=B=0$, hence $j=0$ and $\nu=-2$. The
isomorphism class $\FF(0,-2)$ contains the trivial representation.


\def\flnm{rFtm-chap-IV}

\setcounter{tabold}{\arabic{table}}
\setcounter{figold}{\arabic{figure}}


\chapter{Fourier expansion of automorphic forms} \label{chap-4}
\setcounter{section}{15}\setcounter{table}{\arabic{tabold}}
\setcounter{figure}{\arabic{figold}}\markboth{IV. FOURIER EXPANSION OF
AUTOMORPHIC FORMS }{IV. FOURIER EXPANSION OF AUTOMORPHIC FORMS } The
first three chapters have given us explicit information concerning
Fourier term modules. Now we apply it to automorphic forms and their
Fourier expansion.

In Section~\ref{sect-af} we discuss the Fourier expansions of
automorphic forms. We concentrate on results that we need for our
paper~\cite{BM2}. The Poincar\'e series, for instance those
in~\cite{MW89}, are meromorphic families of automorphic forms with
exponential growth. Proposition \ref{prop-OmMuUps} gives results that
we need. In the completeness theorem the residues and values of
Poincar\'e series are related to the Fourier coefficients of individual
square integrable automorphic forms. Their Fourier expansion is
described in Propositions \ref{prop-i-pscs}--\ref{prop-i-thr}. These
expansions have the same structure as Ishikawa's Fourier expansions in
\cite{Ish99} and~\cite{Ish00}.

At each cusp, the Fourier expansion of an automorphic forms describes a
function on $G$ modulo a standard lattice in~$N$. In
Section~\ref{sect-Fs} we collect some general results concerning
eigenfunctions of $ZU(\glie)$ in $C^\infty(\Ld_\s \backslash G)_K$.


\def\flnm{rFtm-IV-af}
 

\section{Automorphic forms}\label{sect-af} \markright{16. AUTOMORPHIC
FORMS} We recall the definition of automorphic forms, and describe the
form of their Fourier terms. We consider families of automorphic forms
with moderate exponential growth, and Fourier expansions of square
integrable automorphic forms that generate irreducible
$(\glie,K)$-modules.\medskip

 In \S\ref{sect-cucond} we imposed the $\ZZ[i]$-con\-di\-tion on the
cusps on the cofinite discrete subgroups $\Gm$ that we consider. This
ensures that for any $f\in C^\infty(\Gm\backslash G)_K$ and each cusp
$\c$ of~$\Gm$ we have the translated function $f^\c(g) = f(g_\c g)$,
which is in $C^\infty(\Ld_{\s(\c)}\backslash G)_K$ for some standard
lattice $\Ld_{\s(\c)}$. Proposition~\ref{prop-absconv} gives the
absolute convergence of the Fourier expansion of $f^\c$. It suffices
to
 consider only representatives $\c $ in a
(finite) set of representatives of the $\Gm$-orbits of cusps.
\begin{defn}\label{def-gc}\hspace{.1cm}
\begin{enumerate}
\item[i)] We say that a function $h\in C^\infty(\Ld_\s\backslash G)_K$
satisfies:
\begin{itemize}
\item \emph{Exponential decay at $\infty$} if
$h\bigl( n \am(t)k\bigr) \= \oh(e^{-at^m})$ , as $t\uparrow\infty$, uniformly in $n$ and $k$ for some $a>0$ and some
$m\in \ZZ_{\geq 1}$.\il{ed1}{exponential decay}
 
 This was already
used in Definition~\ref{def-grbb}. Elements of
 $\Wfu_\Nfu^\ps$ provide
examples, with $n=1$ (linear exponential decay)
 in the abelian case,
\eqref{Kae}; and with $n=2$ (quadratic exponential
 decay) in the
non-abelian case, \eqref{Wae}.

\item \emph{Quick decay } if $h\bigl( n \am(t)k\bigr) \= \oh(t^{-a})$ as $t\uparrow\infty$, uniformly in $n$ and $k$, 
for all $a>0$. \il{qd}{quick decay}
 
\item \emph{Polynomial growth } if
$h\bigl( n \am(t)k\bigr) \= \oh(t^{a})$   as $t\uparrow\infty$, uniformly in $n$ and $k$, for some
$a>0$.\il{pg}{polynomial growth}
 
 Elements of the $N$-trivial Fourier
term modules $\Ffu_0^\ps$ satisfy
 this condition.
 
\item \emph{Exponential growth } if
$h\bigl( n \am(t)k\bigr) \= \oh(e^{at^m})$  as $t\uparrow\infty$, uniformly in $n$ and $k$, for some $a>0$ and some
$m\in \ZZ_{\geq 1}$.
 
 Elements of $\Mfu^\ps_\bt$ satisfy this
condition with $n=1$ (linear
 exponential growth). For elements of
$\Vfu^\ps_\n$ we need $n=2$
(quadratic exponential growth).\il{eg}{exponential growth}
\end{itemize}
\item[ii)]We call a function $f\in C^\infty(\Ld_\s\backslash G)_K$
\il{sqiinf}{square integrable at $\infty$}\emph{square integrable near
$\infty$} if
\[ \int_{n\in \Ld_\c\backslash N} \int_{t=t_0}^\infty \int_{k\in K}
\bigl|h\bigl(n\am(t)k\bigr)\bigr|^2 \; dn \;\frac{dt}{t^5}\; dk
\,<\;\infty\]
 for some $t_0>0$. (The measure
$dg  =  dn\, t^{-5}\, dt\, dk$ is a Haar
 measure on $G$.)

 If the
condition holds for some $t_0$, then it holds for all $t_0>0$.
Elements of $H^{\xi,\nu}_K$ satisfy this condition if and only if
$\re\nu<0$.
\end{enumerate}
\end{defn}

The use of $O$-statements in part i) of the definition shows that these growth conditions  are understood with ``at most'' added implicitly.

\begin{defn}\label{def-af}
 Let $\ps\in \WO$. An \il{af}{automorphic
form}\emph{automorphic form} on
 $\Gm\backslash G$ with character $\ps$
is a $K$-finite function
 $f\in C^\infty( \Gm\backslash G)_K$ that
satisfies $u f = \ps(u) f$ for
 all $u\in ZU(\glie)$, such that $f^\c $
has polynomial growth for each
 cusp~$\c$. By $\A(\Gm;\ps)$ we denote
the space of automorphic forms on
 $\Gm$ with character
$\ps$.\il{AGmps}{$\A(\Gm,\ps)$}
 
 The space of \il{cf}{cusp
form}\emph{cusp forms}
 $\Au0(\Gm;\ps) \subset \A(\Gm;\ps)$ is
determined by the condition that
 $\Four_0 f^\c=0$ for all
cusps~$\c$.\il{A0ps}{$\Au0(\Gm,\ps)$}
 
 By
$\Au{(2)}(\Gm,\ps)\subset \A(\Gm,\ps)$ we denote the subspace of
\il{sqiaf}{square integrable automorphic form}\il{afsqi}{automorphic
form, square integrable}\emph{square integrable automorphic forms},
determined by
$\int_{\Gm\backslash G} |f(g)|^2\, dg <\infty$.\il{A2ps}{$\Au{(2)}(\Gm,\ps)$}
\end{defn}
 
 In this definition we follow Harish-Chandra \cite[pp 7,
12--14]{HCh68},
 but working only with scalar-valued functions and
letting $\Gm$ act on
 the left. The spaces $\mathcal{A}(G/\Gm,\s,\ch)$
in \cite{HCh68} are
 spaces of automorphic forms of a given $K$-type.
Theorem 1,
 \cite[p~8]{HCh68}, states that the spaces
$\A(\Gm;\ps)_{h,p}$ with a
 fixed $K$-type have finite dimension.
 
Theorem~\ref{thm-qd} implies for each cusp form $f\in \Au0(\ps)$ that
$f^\c$ has quick decay for each cusp~$\c$.
 
\rmrk{Module structure} Proposition~\ref{prop-inv} shows that
exponential decay at~$\infty$ is preserved under the right action of
$\glie$ and $K$. To show that, we used the convolution representation
theorem of Harish Chandra, which describes eigenfunctions $f$ of
$ZU(\glie)$ as a convolution with a suitable smooth function $\al$
with
 compact support
\[ f(g) = \int_G f(g_1)\, \al( g^{-1}_1g) \, dg_1\,.\]
 The analysis in
the proof of Proposition~\ref{prop-inv} can be
 generalized, and gives
the following result.
\begin{prop}The spaces $\A(\Gm;\ps) \supset \Au0(\Gm;\ps)$ are
$(\glie,K)$-modules for the action by right translations and
differentiation.
\end{prop}
 
 Each Fourier term operator gives an intertwining operator
$\Four_\Nfu :\A(\Gm;\ps) \rightarrow \Ffu^\ps_\Nfu$. See
Proposition~\ref{prop-Ftit}. The integrals defining the Fourier term
operators preserve the various growth conditions in
Definition~\ref{def-gc}. Hence we have for all $f\in \A(\Gm;\ps)$:
\be \Four_\Nfu f \in \Wfu^\ps_\Nfu \qquad\text{ for all
$\Nfu\neq\Nfu_0$}\,. \ee

 The $N$-trivial Fourier term $\Four_0 f$ can be any element of
$\Ffu^\ps_0$ if $f\in \A(\Gm;\ps)$. We get for
$f\in \A^{(2)}(\Gm;\ps)$
\be \Four_0 f \in \bigoplus_{(j,\nu) \in \wo(\ps), \,\re\nu<0}
H^{\xi_j,\nu}\,.\ee
This implies that $\Au{(2)}(\Gm;\ps)$ is a $(\glie,K)$-module as well.

\rmrk{Examples} \il{Es}{Eisenstein series}\emph{Eisenstein series} are
  examples of automorphic forms on general semisimple Lie groups, see
  Langlands \cite{Lal76}, and Harish Chandra \cite{HCh68}. Sections 2--4
of~\cite{BKNPP10} discuss Eisenstein series on $\SU(2,1)$.

Eisenstein series in the domain of absolute convergence are given by
an
 infinite sum. Cusp forms tend to be more elusive. Reznikov
\cite[\S5]{Rez93} shows, on the basis of a Kuznetsov formula, that
there are infinitely many generic cusp forms with $K$-type $\tau^0_0$
on $\Gm\backslash \SU(2,1)$ for suitable discrete subgroups $\Gm$.
Generic means that the automorphic form has some non-zero Fourier
terms
 $\Four_\bt f^\c$ for some $\bt\neq 0$.
 
\rmrk{Moderate exponential growth} Instead of sums of elements of
principal series representations, one may form sums of elements of
$\Mfu^{\xi,\nu}_\bt$ for $\bt\in \ZZ[i]\setminus\{0\}$. That leads to
the \il{Ps}{Poincar\'e series}Poincar\'e series studied
in~\cite{MW89},
 actually for all Lie groups with real rank one. The
functions on
 $\SU(2,1)$ that one obtains in this way are in the
following space,
 which is larger than $\A(\Gm;\ps)$.

\begin{defn}\label{def-afmeg}Let $\ps\in \WO$. An automorphic form with
\il{meaf}{moderate exponential growth, automorphic form
with}\il{egm}{exponential growth, moderate}\il{afme}{automorphic form
with moderate exponential growth}\emph{moderate exponential growth} on
$\Gm\backslash G$ with character $\ps$ is a $K$-finite function
$f\in C^\infty(\Gm\backslash G)_K$ that satisfies $u \, f= \ps(y) f$
for all $u\in ZU(\glie)$, and for which there is for each cusp $\c$ a
finite set of Fourier terms orders $E(\c)$ such that
$f^\c - \sum_{\Nfu\in E(\c)} \Four_\Nfu f^\c$ has polynomial growth.
 
We denote the space of such automorphic forms by
$\Au\me(\Gm;\ps)$.\il{Aume}{$\Au\me(\Gm;\ps)$}
\end{defn}
 
 Since the Fourier term operators are intertwining
operators
(Proposition~\ref{prop-Ftit}), the space $\Au!(\Gm;\ps)$ is a
$(\glie,K)$-module.

\rmrk{Remarks}
(1)
 With the use of Poincar\'e series like those in~\cite{MW89} one can
show
 that the spaces $\Au!(\ps)_{h,p,q}$ have infinite dimension.
 
(2)
 On $\SL_2(\RR)$ automorphic forms with moderate exponential growth
turn
 up, often under the name ``weak Maass forms''. These automorphic
forms
 are the harmonic completions of mock modular forms. See the
overview~\cite{Za07}.

(3) The example in \eqref{exU} shows that there are elements of
$C^\infty(\Ld_\s\backslash G)_K^\ps$ that have infinitely many Fourier
terms with exponential growth. Estimate~\eqref{uest0} enables us to
apply the approach of \cite[Lemma 2.1]{MW89} to build a Poincar\'e
series that satisfies the properties of an automorphic form, but has
infinitely many exponentially increasing Fourier terms. For
$\SL_2(\RR)$, a space of automorphic forms of this type occurs in
Theorem C in \cite{BLZ15}, which shows that the space of all
$\Gm$-invariant eigenfunctions of the hyperbolic Laplace operator is
in
 a bijective correspondence with a mixed parabolic cohomology
group.
 
\subsection{Families}\label{sect-famaf}
\begin{defn}\label{def-holfam} Let $U$ be a connected open set
in~$\CC$,
 let $j\in \ZZ$ and let $\tau^h_p$ be a $K$-type such that
$|h-2j|\leq 3p$.
 
 A \il{hfaf}{holomorphic family of automorphic
forms}holomorphic family
 of automorphic forms for $(j,h,p)$ on $U$ is
an element
 $f\in C^\infty( U\times G)$ such that $g\mapsto f(\nu,g)$
is an element
 of $\Au!\bigl(\Gm,\ps[j,\nu]\bigr)_{h,p,p}$ for each
$\nu \in U$, and
 such that $\nu \mapsto f(\nu,g)$ is holomorphic on
$U$ for each
 $g\in G$.
 
 A \il{mfaf}{meromorphic family of
automorphic forms}meromorphic family
 of automorphic forms has the
form
 $(\nu,g) \mapsto \frac1{\ph(\nu)} \, f(\nu,g)$ where $f$ is a
holomorphic family on $U$ and $\ph$ is a non-zero holomorphic function
on~$U$.
\end{defn}
 
\rmrk{Remarks}(1) The restriction to automorphic forms with moderate
exponential growth is non-essential. Prescribing the first spectral
parameter $j$, the $K$-type $\tau^h_p$, and the highest weight $p$ in
the $K$-type is practical. One obtains more general families as a
$U(\klie)$-linear combination of families of this type.

(2) Meromorphically continued Eisenstein series are examples.

\rmrk{Fourier terms}The Fourier term operators are given by integration
over compact sets. Hence if $f$ is a holomorphic family of automorphic
forms, then $\nu \mapsto \Four_\Nfu f(\nu)$ is a holomorphic family of
elements of the modules $\Ffu^{\ps[j,\nu]}_\Nfu$.
 
 If $f$ is a
meromorphic family of automorphic forms on~$U$ for
 $(j,h,p)$, then all
its Fourier terms can be written in the form
\be \label{mFt} \Four_\Nfu f \= \begin{cases}
 d_0(\nu) \kph h p r
p(\nu) + c_0(\nu) \kph h p r p (-\nu) &\text{ if
 }\Nfu=\Nfu_0\,,\\
d_\Nfu(\nu) \mu^{a,b}_\Nfu(j,\nu) + c_\Nfu(\nu)
\om^{a,b}_\Nfu(j,\nu)&\text{ otherwise}\,,
\end{cases}\ee
with meromorphic functions $c_\Nfu$ and $d_\Nfu$ on $U$. We take
$r=\frac13(h-2j)$, and $a,b$ such that $h=2j+3(a-b)$, $p=a+b$, and use
the families of Fourier terms in \eqref{HK-def} and \eqref{xab}.
 
 We
define families of Fourier terms analogously to
Definition~\ref{def-holfam}; we need only replace
$\Au!(\Gm;\ps)_{h,p,p}$ by
$\Ffu^{\ps[j,nu]}_{\Nfu;h,p,p}$.\il{hfFt}{holomorphic family of
Fourier
 terms}\il{mfFt}{meromorphic family of Fourier terms}
 
 If $f$
is a holomorphic family of automorphic forms, then the
 coefficients
$c_\Nfu$ and $d_\Nfu$ are holomorphic on
 $U \setminus \ZZ$. The
families used in \eqref{mFt} do not always form
 a basis of
$\Ffu^{\ps[j,\nu]}_{\Nfu;h,p,p}$. That may cause
 singularities of the
coefficients that are not due to singularities of
 the family~$f$. When
dealing with meromorphic families this is often no
 problem. However if
we are interested in values or residues at integral
 points it is
better to use an adapted basis in the case that
$\Nfu\neq \Nfu_0$, for instance the basis in the following
proposition.
\begin{prop}\label{prop-OmMuUps}
 Let $(j,h,p)$ be as above, and let
$\Nfu=\Nfu_\bt$,
 $\bt\in \ZZ[i]\setminus \{0\}$, or
$\Nfu=\Nfu_{\ell,c,d}$ with
 $m_0(j) \geq 0$.
 
\begin{itemize}
\item[i)]\il{Mu}{$\Mu_{\Nfu;h,p}$}\il{Om}{$\Om_{\Nfu;h,p}$}
There is a holomorphic family $\nu\mapsto \Om_{\Nfu;h,p}(j,\nu)$ of
Fourier terms such that for each $\nu \in \CC$ we have $
\CC\,\Om_{\Nfu;h,p}(j,\nu)\= \Wfu^{\xi_j,\nu}_{\Nfu;h,p,p}$.

There is also a meromorphic family $\nu\mapsto \Mu_{\Nfu;h,p}(j,\nu)$
with at most first order singularities in $\ZZ_{\leq -1}$ such that
$ \Mu_{\Nfu;h,p}(j,\nu)$ spans $\Mfu^{\xi_j,\nu}_{\Nfu;h,p,p}$ if
$\nu \not\in \ZZ_{\leq -1}$. At $\nu\in \ZZ_{\leq -1}$, the values and
residues are elements of $\Ffu^{\ps[j,\nu]}_{\Nfu;h,p,p}$.

\item[ii)]The elements $\Om_{\Nfu;h,p}(j,\nu)$ and
$\Mu_{\Nfu;h,p}(j,\nu)$ form a basis of
$\Ffu^{\ps[j,\nu]}_{\Nfu;h,p,p}$, except if $\pm \ell >0\,,$
$\pm j \leq -1\,,$ $\nu\geq 0$, $\nu \equiv j\bmod 2$,
$ 0\leq m_0(j) < \frac{|j|-\nu}2\,,$ $ h \pm 3p < \mp 3\nu - j\,.$
{\rm(For given $(j,h,p)$ and $m_0(j)\geq 0$ this describes a finite
set
 of values of $\nu$.)}

\item[iii)] If $\Nfu=\Nfu_{\ell,c,d}$, then there is a holomorphic
family of Fourier terms
\il{Ups}{$\Ups_{\Nfu;h,p}$}$\nu \mapsto \Ups_{\Nfu;h,p}(j,\nu)$ on $U$
such that for each $\nu \in U$
\[ \CC\, \Ups_{\Nfu;h,p}(j,\nu) \= \Vfu^{\ps[j,\nu]}_{\Nfu;h,p,p}\,,\]
and such that $\Ups_{\Nfu;h,p}(j,\nu)$ and $\Om_{\Nfu;h,p}(j,\nu)$
form
 a basis of $\Ffu^{\ps[j,\nu]}_{\Nfu;h,p,p}$.
\end{itemize}
\end{prop}
 
 The property in i) fixes the families up to multiplication
by a
 holomorphic function on $U$ that has no zeros. In the proofs in
\S\ref{sect-invef} convenient choices are specified. In many cases
these families are holomorphic on all of~$\CC$.
 
 The exception in ii)
of the proposition forces us to consider the family
 $\Ups_{\n;h,p}$ as
well. In the non-abelian case the family
 $\Mu_{\Nfu;h,p}$ is a linear
combination of $\Ups_{\Nfu;h,p}$ and
 $\Om_{\Nfu;h,p}$ with
coefficients that are holomorphic on~$U$. The
 exact form of this
relation depends on the choice of the families.
 
 The families of
Fourier terms in the proposition, and also the families
 used in
\eqref{mFt}, were defined by repeated application of shift
 operators,
and it is hard to give explicit formulas, except in special
 cases.
 
\subsection{Square integrability}
\label{sect-sqiaf}
 The space $L^2( \Gm\backslash G)$ is a unitary
representation of~$G$ for
 the action by right translation. It has a
closed subspace
 $L^{2,\mathrm{discr}}(\Gm\backslash G)$ generated by
all irreducible
 subspaces. The $K$-finite vectors in each of these
irreducible
 components form an irreducible $(\glie,K)$-module in
$\Au{(2)}(\Gm;\ps)$ for some character $\ps$ of $ZU(\glie)$. These
submodules inherit a unitary structure from $L^2(\Gm\backslash G)$.
Each such modules is generated by its subspace of minimal $K$-type
$\tau^{h_0}_{p_0}$. The finite-dimensional space
$\Au{(2)}(\ps)_{h_0,p_0}$ spans a $(\glie,K)$-submodule of
$\Au{(2)}(\ps)$ that is the direct sum of a finite number of
irreducible $(\glie,K)$-modules in one of the isomorphism classes in
Theorem~\ref{thm-us}; see also Proposition~\ref{prop-Ktp}. For given
non-zero $f\in \Au{(2)}(\ps)_{h_0,p_0,p_0}$, the Fourier expansions of
the automorphic forms in $U(\glie) f$ are determined by the Fourier
expansions of~$f$.
 
 In our next results, we discuss the Fourier
expansions of $f^\c$ for
 square integrable automorphic forms
$f\in \Au{(2)}(\Gm;\ps)_{h_0,p_0,p_0}$ for all isomorphism classes of
unitarizable irreducible $(\glie,K)$-modules. The Fourier terms can be
expressed in explicitly given functions. For $\Nfu\neq \Nfu_0$ these
functions are non-zero multiples of $\Om_{\Nfu;h_0,p_0}$ in
Proposition~\ref{prop-OmMuUps}. We formulate the result in four
propositions, in which we combine isomorphism classes for which the
Fourier expansion has a similar structure. See \S\ref{sect-list-irr}
and Theorem~\ref{thm-us} for the list of isomorphism classes.

\begin{prop}\label{prop-i-pscs}{\rm Unitary irreducible principal
series
 and complementary series.} Isomorphism type $\II(j,\nu)$ ,
where
 $(j,\nu)$ satisfies $\nu\in i\RR$, $j\in \ZZ$,
$j\not\in 2\ZZ_{\neq 0}$
 if $\nu=0$; or satisfies
$\nu\in (-2,2) \setminus \{0\}$, $j=0$, or
$\nu\in(-1,1)\setminus \{0\}$, $j\equiv 1 \bmod 2$. The minimal
$K$-type is $\tau^{2j}_0$.
 
 If
$f\in \Au{(2)}\bigl( \ps[j,\nu]\bigr)_{2j,0,0}$ generates a
$(\glie,K)$-module of type $\II(j,\nu)$, then each $f^\c$ has a
pointwise absolutely convergent Fourier expansion of the following
form.
\begin{align*}
 f^\c &\= a_\c(0) \, \kph{2j}000\bigl(
-|\nu|\bigr)\qquad\bigl(\text{only
 for the complementary
series}\bigr)\\
&\quad\hbox{} + \sum_{\bt\in \ZZ[i]\setminus\{0\}} a_\c(\bt)
\,\om^{0,0}_\bt(j,\nu)
+ \sum_{(\ell,c,d)} a_\c(\ell,c,d) \, \om^{0,0}_{\ell,c,d}(j,\nu)\,.
\end{align*}
 The summation in the non-abelian term is over
$\ell\in \frac{\s(\c)}2\ZZ_{\neq 0}$, $c\bmod 2\ell$, and
$d\in 1+2\ZZ$
 such that
$m_0(j)= \frac{\sign(\ell)}6(d-2j)-\frac12 \in \ZZ_{\geq 0}$.
 
 The
basis functions are given in \eqref{HK-def}, \eqref{om00ab} and
\eqref{om00nab}.
\end{prop}

\begin{prop}\label{prop-i-hadst}{\rm Holomorphic and antiholomorphic
discrete series type.} Isomorphism types $\IF(j,\nu)$,
$\nu \equiv j \bmod 2$, $j\in \ZZ_{\geq 2}$, $0\leq \nu\leq j-2$, and
$\FI(j,\nu)$ $\nu \equiv j \bmod 2$ $j\in \ZZ_{\leq -2}$,
$0\leq \nu \leq -j-2$. The minimal $K$-type is $\tau^{2j}_0$.
 
 If
$f\in \Au{(2)}\bigl( \ps[j,\nu]\bigr)_{2j,0,0}$ generates a
$(\glie,K)$-module of one of these types, then each $f^\c$ has a
pointwise absolutely convergent Fourier expansion of the following
form.
\[ f^\c \= \sum_\n a_\c(\ell,c,d) \, \om_{\ell,c,d}^{0,0}(j,\nu)\,,\]
with the basis function in~\eqref{om00nab}. The summation is over
$(\ell,c,d)$ such that
\bad
&\ell\in \frac{\s(\c)}2\ZZ_{\leq -1}\,,&& c\bmod 2\ell\,,
&&0\leq m_0(j) < \frac{j-\nu}2\,,&
&\text{ for }\IF(j,\nu)\,,\\
&\ell\in \frac{\s(\c)}2 \ZZ_{\geq 1}\,,&&c\bmod 2\ell\,,
&& 0\leq m_0(j) < - \frac{j+\nu}2\,,&
&\text{ for }\FI(j,\nu)\,.
\ead
\end{prop}

\begin{prop}\label{prop-i-ldst} {\rm Large discrete series type.}
Isomorphism type $\II_+(j,\nu)$, $\nu\equiv j\bmod 2$, $\nu\geq |j|$,
$\nu \neq 0$. The minimal $K$-type is $\tau^{-j}_\nu$.
 
 If
$f\in \Au{(2)}\bigl( \ps[j,\nu]\bigr)_{-j,\nu,\nu}$ generates a
$(\glie,K)$-module of type $\II_+(j,\nu)$, then each $f^\c$ has a
pointwise absolutely convergent Fourier expansion of the following
form.
\begin{align*}
 f^\c &\= \sum_{\bt\in \ZZ[i]\setminus\{0\}} a_\c(\bt)\,
\kk^K_{\bt;-j,\nu}
+ \sum_{(\ell,c,d)} a_\c(\ell,c,d) \, \kk^W_{\ell,c,d;-j,\nu}\,,
\end{align*}
 with basis functions as in \eqref{expl-kdso} and
\eqref{kkVW}. In the
 non-abelian term the summation is over
$(\ell,c,d)$ such that
 $\ell \in \frac{\s(\c)}2\ZZ_{\neq 0}$,
$c\bmod 2\ell$, and
$m_0(j) \geq \frac12\bigl(\nu- j\sign(\ell) \bigr)$.

\end{prop}

\begin{prop}\label{prop-i-thr} {\rm Thin representations.} Isomorphism
types
\[ \begin{array}{rlc}
 \IF(1,-1)
& =T^+_{-1} &
\\
 \IF_+(2k+3,-1)& =T^+_k & k \in \ZZ_{\geq 0}
\\
 \FI(-1,-1)&= T^-_{-1} &
\\
 \FI_+(-2k-3,-1)&= T^-_k & k \in \ZZ_{\geq 0}
\end{array}\]
 with minimal $K$-type $\tau^{\pm(2k+3)}_{k+1}$ for
$T^\pm_k$.
 
 If $f\in \Au{(2)}\bigl( \ps[j,\nu]\bigr)_{-j,\nu,\nu}$
generates a
 $(\glie,K)$-module of one of these types, then each $f^\c$
has a
 pointwise absolutely convergent Fourier expansion of the
following
 form.
\begin{align*}
 f^\c &\= a_\c(0) \,
\kph{\pm(k+3)}{k+1}{\mp(k+1)}{k+1}(-1)
+ \sum_{(\ell,c,d)} a_\c(\ell,c,d)\, \kk^W_{\ell,c,d;\pm(k+3),k+1}\,,
\end{align*}
 with basis functions as in \eqref{HK-def}
and~\eqref{kkVW}. In the
 non-abelian term the summation is over
$\ell \in \mp \frac{\s(\c)}2\ZZ_{\geq 0}$, $c\bmod 2\ell$,
$m_0(j) = k+1$.
 
 The expression in~\eqref{kkVW} leads to
\badl{kkth-expl} &\kk^W_{\ell,c,d;\pm(k+3),k+1}\bigl( n \am(t)
k_1\bigr)\= \sum_{r\equiv k+1, \, r\equiv k+1(2)} i^{(k+1\mp r)/2}
(2\pi|\ell|)^{(-k-1\pm r)/4}
\\
&\quad\hbox{} \cdot
\sqrt{\frac{k+1\mp r}2 !} \;\Th_{\ell,c}\bigl( h_{\ell,(k+1\mp
r)/2};n\bigr)\, t^{(k+3\pm r)/2}\, e^{-\pi |\ell|t^2}\;
\Kph{\pm(k+3)}{k+1}r{k+1}(k_1)\,.
\eadl
\end{prop}

\rmrk{Trivial representation}For type $\FF(0,-2)$, only the $N$-trivial
Fourier term can be non-zero. It is a multiple of
$\kph 0 0 0 0(-2) = 1$.
 
\begin{proof}[Proofs of Propositions
\ref{prop-i-pscs}--\ref{prop-i-thr}] The conditions on the spectral
parameters $(j,\nu)$ are in Theorem~\ref{thm-us}. The $K$-types of
minimal dimension can be found in Table~\ref{tab-isot-ps},
p~\pageref{tab-isot-ps}.
 
 For the $N$-trivial Fourier terms we need
that the isomorphism type is
 represented in a principal series
representation with $\re\nu<0$. This
 occurs only in the complementary
series and the thin representations.
 For the other Fourier term orders
$\Nfu$, the square integrability
 implies that we have an element of
$\Wfu^\ps_\Nfu$.
 
 In generic abelian Fourier term modules the upward
shift operators are
 injective, by Proposition~\ref{prop-uso-ga}. Hence
only the isomorphism
 types $\II$ and $\II_+$ may have generic abelian
Fourier terms.
 
 The non-abelian Fourier terms can occur under
conditions on $m_0(j)$.
 For the irreducible principal series this is
just the condition
 $m_0(j) \in \ZZ_{\geq 0}$. For the other
isomorphism types we consult
 Table~\ref{tab-isoWna},
p~\pageref{tab-isoWna}.
 
 For the thin representations we can go
further. With the notations of
 Proposition~\ref{prop-ik} and the
relations in
 Table~\ref{tab-parms-na}, p~\pageref{tab-parms-na}, we
have the
 following results for~$T^\pm_k$:
\bad \e &\= \sign(\ell) \= \mp 1&\quad j&\= \pm(2k+3)\,,\\
 h&\= \pm
(k+3) \,, & m_0(j) &\= k+1\,,\\
 \k(r) &\= \frac{-k-1\pm r}4\,,& s(r)
&\= \frac{ \pm k \pm 3-r}4\,.
\ead
So $\k(r) \pm s(r) = \frac12$. This means that the $W$-Whittaker
functions in~\eqref{kkVW} can be expressed in terms of exponential
functions, with~\eqref{specWV}. We have $r_0 = \pm(k+1)$, so the sum
in
 \eqref{kkVW} ranges over $r\equiv k+1\bmod 2$, $|r|\leq k+1$.
Carrying
 out the computation gives~\eqref{kkth-expl}.
\end{proof}

\rmrk{Remarks} (1) We see that for the discrete series types and for
the
 irreducible unitary principal series the $N$-trivial Fourier term
has
 to vanish. This is in accordance with a more general result of
Wallach
 \cite[Theorem 4.3]{Wal84}, stating that tempered
representations occur
 in $L^2(\Gm\backslash G)$ as spaces of cusp
forms.
 
(2) The square integrable automorphic forms of holomorphic discrete
series type on $\SU(2,1)$ have their counterpart in the automorphic
forms of holomorphic and antiholomorphic discrete series type on
$\SL_2(\RR)$, although there the Fourier expansion has of course only
abelian Fourier terms. New in the comparison with square integrable
automorphic forms on $\SL_2(\RR)$ is the appearance of the large
discrete series type and the thin representations. The large discrete
series type allows non-zero abelian as well as non-abelian Fourier
terms. The thin representations are infinite-dimensional Langlands
representations that allow a unitary structure. They have non-abelian
Fourier terms and possibly an $N$-trivial Fourier term.
 
\rmrk{Ishikawa's Fourier expansions} Theorem 5.3.1 in \cite{Ish99}
gives
 Fourier expansions of automorphic forms with at most polynomial
growth.
 The expansions i-1) and ii)
 in that theorem have the same
overall structure as the expansions that
 we obtained. A difference is
caused by the use, in \cite{Ish99}, of
 unscaled Hermite functions in
the construction of the theta-functions,
 and the corresponding
different expression of the parameter $\k$ in the
 non-abelian terms.

 Ishikawa's expansions in i-2) and i-3) for the holomorphic and
antiholomorphic discrete series do not concern a vector of minimal
$K$-type in the $(\glie,K)$-module, but concern a corner type (of
higher dimension) on which one upward and one downward shift operator
vanish. These expansions are similar to the expansions in our
Propositions \ref{prop-nsefa} and~\ref{prop-ik}.

 In Proposition~4
of the later paper \cite{Ish00} Ishikawa gives modified
 versions of
these expansions, and also expansions for thin
 representations that
agree with the results in
 Proposition~\ref{prop-i-thr}.


\def\flnm{rFtm-IV-Fs}


\section{Invariant eigenfunctions} \label{sect-Fs} \markright{17.
INVARIANT EIGENFUNCTIONS}\label{sect-invef}

In the previous sections we studied the Fourier expansions of the
translated functions $f^\c\in C^\infty(\Ld_{\s(\c)}\backslash G)_K^\ps$
for automorphic forms~$f$. In this section we discuss general
$\Ld_\s$-invariant functions.

We mention a general result for semisimple Lie groups that explains why
the condition of polynomial growth plays a role in the definition of
automorphic forms. We give the proofs of the propositions in
\S\ref{sect-sqiaf}. Finally we give two examples of elements of
$C^\infty(\Ld_\s\backslash G)_K^\ps$.\medskip

\begin{thm}\label{thm-qd}Let $\ps$ be a character of $ZU(\glie)$. If an
element $f\in C^\infty(\Ld_\s\backslash G)^\ps_K$ has polynomial growth
and $\Four_0 f = 0$, then $f$ has quick decay.
\end{thm}
Harish Chandra \cite{HCh66} gives this result for a general semisimple
Lie group. One may also consult Gan \cite[p 84--89]{Gan}. This result
is a consequence of Lemma 10, \cite[p 11]{HCh66}. The proof is in~\S7,
see specially Lemma~20. Harish Chandra uses several spaces of functions
$\mathcal{L}_F(\ld)$ depending on a linear form $\ld $ on $\alie$,
which has dimension $1$ for $\SU(2,1)$. Going through the definitions
in \cite[Ch I, \S3]{HCh66}, we identify
$\mathcal{L}_\emptyset(\ld_0) =\mathcal{L}$ as a space containing those
$f\in C^\infty(\Ld_\s\backslash G)_K$ for which $u f$ has polynomial
growth of an order specified by $\ld$ for each $u\in U(\glie)$, and
$\mathcal{L}_{\al}(\ld_0)= \mathcal{L}$ as a space containing the
elements of $C^\infty(\Ld_\s\backslash G)_K$ for which all $u f$ have
quick decay. (The latter space does not depend on $\ld_0$.) Lemma 10
then tells that the map $ f\mapsto  f -  \Four_0 f$ sends
$\mathcal{L}(\ld)$ to~$\mathcal{L}$.\medskip

\rmrk{Remark} Definition~\ref{def-af} imposes the condition of
polynomial growth on automorphic forms. Theorem~\ref{thm-qd} then
implies that cusp forms have quick decay. The use of the condition of
polynomial growth in Definition~\ref{def-afmeg} ensures that the growth
of an automorphic form with moderate exponential growth $f$ can be read
off from the Fourier expansions of the functions $f^\c$, modulo a
contribution with quick decay.

For $\SL_2(\RR)$ we do not need Theorem~\ref{thm-qd}. There, exponential
decay holds for any convergent Fourier series with exponentially
decreasing terms. We do not know whether the same can be shown for
$\SU(2,1)$.
\medskip

\begin{proof}[Proof of Proposition~\ref{prop-OmMuUps}]In Chapters
\ref{chap-2} and~\ref{chap-3} we have discussed various families of
bases for Fourier term modules, which are holomorphic in the spectral
parameter $\nu$. There is no best  choice of a basis. We specify
  choices satisfying the requirements in Proposition~\ref{prop-OmMuUps}.

We restrict ourselves to highest weight families of type $(j,h,p)$ with
a $K$-type $\tau^h_p$ and $j\in \ZZ$ such that
$\bigl|h-2j\bigr|\leq p$.

\rmrk{$N$-trivial case}(Not included in the statement of the proposition.)   In this case the choice of the families
$\kph h p r p (\nu)$ and $\kph h p r p(-\nu)$ with $r=\frac13(h-2j)$ is
comfortable in most cases. Near $\nu=0$ it is better to use
$\kph h p r p(\nu)$ and $\ldph h p r p(\nu)$. Explicit descriptions in
\eqref{HK-def}, \eqref{ldphg}, and \eqref{ldph!=0}.

\rmrk{Other Fourier term orders} Now let $\Nfu \neq \Nfu_0$. If $p=0$ we
have $h=2j$, and the obvious explicit choices are
$\Om_{\Nfu;2j,0}(j,\nu) = \om^{0,0}_\Nfu(j,\nu)$,
$\Mu_{\Nfu;2j,0}(j,\nu) = \mu^{0,0}_\Nfu(j,\nu)$,
$\Ups_{\n;2j,0}(j,\nu) = \ups^{0,0}_\n(j,\nu)$ as in
\eqref{mu00a}--\eqref{om00nab}, \eqref{upsab}.

For any $K$-type $\tau^h_p$ in $\sect(j)$ there are unique
$a,b\in \ZZ_{\geq 0}$ such that $h=2j+3(a-b)$, $p=a+b$. Then
\eqref{xab} and \eqref{upsab} give holomorphic families
$\mu^{a,b}_\Nfu(j,\nu)$, $\om^{a,b}_\Nfu(j,\nu)$, and
$\ups^{a,b}_\n(j,\nu)$. In the generic abelian case the upward shift
operators are injective. Then the families $\om^{a,b}_\bt(j,\nu)$ and
$\mu^{a,b}_\bt(j,\nu)$ have no zeros, and can be taken as
$\Om_{\bt;h,p}(j,\nu)$ and $\Mu_{\bt;h,p}(j,\nu)$.
Lemma~\ref{lem-strab} shows the desired properties in i) and ii) of
Proposition~\ref{prop-OmMuUps}.

Actually, this works as well for the families $\ups^{a,b}_\n(j,\nu)$ in
the non-abelian case. (See Lemma~\ref{lem-usho-upsom}.)
The families $\mu^{a,b}_\n(j,\nu)$ and $\om^{a,b}_\n(j,\nu)$ really can
have zeros. In Proposition~\ref{prop-extr.na} we form families
$\tilde x^{a,0}_\n$ and $x^{0,b}_\n$ by dividing out zeros and
normalizing the maximal or minimal component. In \eqref{tuod} and
\eqref{tumu} we extend this to families $\tilde x^{a,b}_\n(j,\nu)$ for
$x=\om$, $\mu$, or $\ups$, to get families that span the corresponding
subspace of $\Ffu^{\ps[j,\nu]}_{\n;h,p,p}$. We use the dimension
results in Lemmas \ref{lem-VW-nab} and~\ref{lem-Mfud1}. Taking
$\Om_{\n;h,p}(j,\nu)= \tilde \om^{a,b}_\n(j,\nu) $,
$\Mu_{\n;h,p}(j,\nu)= \tilde \mu^{a,b}_\n(j,\nu)$ and
$\Ups_{\n;h,p,}(j,\nu)= \ups^{a,b}_\n(j,\nu)$ we obtain families
satisfying i) and iii) in Proposition~\ref{prop-OmMuUps}.

For the exceptions in ii) in Proposition~\ref{prop-OmMuUps} we use ii)
of Proposition~\ref{prop-M-VW}. It gives
$\Mfu_{\n;h,p}^{\ps[j,\nu]} = \Wfu_{\n;h,p}^{\ps[j,\nu]}$ in the
following cases:
\bad \ell&>0\,,&(j,\nu) &\=(j_l,\nu_l)\,,&m_0(j_+)&<0 \leq m_0(j) \,,&
h+3p &< 2j_+\,,
\\
\ell&<0\,,& (j,\nu) &\=(j_r,\nu_r) \,,&m_0(j_+)&< 0 \leq m_0(j) \,,&
h-3p& >2j_+\,,
\ead
with the conventions in~\eqref{jnurels}. This describes for $\ell>0$ the
$K$-types in the isomorphism class $\FI(j,\nu)$ (antiholomorphic
discrete series type), and for $\ell<0$ the $K$-types in the class
$\IF(j,\nu)$ (holomorphic discrete series type). See Figures
\ref{fig-pj1VW}--\ref{fig-mj1a}. In Table~\ref{tab-isoWna} we see that
$0 \leq m_0(j) <\frac12\bigl( |j|-\nu\bigr)$ and
$-\sign(\ell) j \in \ZZ_{\geq 1}$. This leads to the formulation of the
exceptional cases in ii) of Proposition~\ref{prop-OmMuUps}.\end{proof}

\rmrk{Explicit expressions}The families have explicit expressions if
$p=0$; \eqref{mu00a}--\eqref{om00nab}, \eqref{upsab}. Furthermore, the
families are proportional to explicit sums at values of $\nu$ for which the intersection of the kernels of both downward shift
operators is non-zero. See Propositions \ref{prop-nsefa}, \ref{prop-ik},
\ref{prop-idkk}, and \ref{prop-kkM}, and Corollary~\ref{cor-jl+r-ab}.

\subsection{Two examples}\label{sect-exmp}We give two simple examples of
absolute convergent Fourier series representing elements of
$C^\infty(\Ld_\s\backslash G)^\ps_K$. Example $h$ shows that with
suitable choices of the parameters one can relate elements of
$C^\infty(\Ld_\s\backslash G)^\ps_K$ to $\Ld_\s$-invariant holomorphic
functions on the symmetric space $\X=G/K$. Example $u$ illustrates that
there are more functions in $C^\infty(\Ld_\s\backslash G)^\ps_K$ than
those that satisfy the condition of moderate exponential growth.

\rmrk{Example $h$}We pick $w\in \ZZ_{\geq 2}$, and for each
$\ell\in \frac \s 2\ZZ_{\geq 1}$ a coefficient $c_\ell$, such that
$c_\ell=\oh(\ell^A)$ for some $A>0$, and a triple
$\n(\ell)= (\ell,0,3-2w)$. Then the series
\be\label{exH} h \= \sum_{\ell>0} c_\ell \om^{0,0}_{\n(\ell)}(-w,w-2)\ee
defines an element of $C^\infty(\Ld_\s\backslash G)^{\ps[-w,w-2]}_K$.
Explicitly, we have
\begin{align*} h\bigl( n \am(t) k\bigr) &\= \sum_{\ell>0} c_\ell\,
\Th_{\ell,0}\bigl( h_{\ell,0} ; n) \, t\, W_{(1-w)/2, \pm
w/2\mp1}(2\pi|\ell|t^2) \,\Kph{-2w}000(k)\\
&\= t^w \sum_{\ell>0}(2\pi\ell)^{(w-1)/2} c_\ell \,
\Th_{\ell,0}(h_{\ell,0};n) \, e^{-\pi\ell t^2} \,\Kph{-2w}000(k)\,.
\end{align*}
If we divide by the common factor $t^w$ and express $n\am(t)$ in
$(z,u) = n \am(t)\,(i,0) \in \X$, and write out the theta functions, we
get a holomorphic function on the symmetric space:
\be H(z,u) \= \pi^{-1/2}(2\pi)^{w/2} \sum_\ell \ell^{(2w-1)/4}\, c_\ell
\, e^{\pi \ell (i z+u^2)} \sum_{k\in \ZZ} e^{-2\pi \ell(u+k)^2}\,. \ee

We note that the \emph{holomorphic} function $H$ arises from the minimal
vectors $\om^{0,0}_{\n(\ell)}(-w,w-2)$ in Fourier term modules of
\emph{antiholomorphic} discrete series type $\FI(-w,w-2)$.

\rmrk{Example $u$}Here we use abelian Fourier terms in
$\Mfu^{ \xi_j,\nu}_\bt(j,j)$ with fixed $j\in \ZZ_{\geq 3}$. We take
positive coefficients $c_m = e^{-m^2}$ for $m\in \ZZ_{\geq 1}$, and
$c_\bt= 0$ for all other $\bt\in \ZZ[i]\setminus \{0\}$.
\be \label{exU} u \= \sum_{m \geq 1} c_m \, \mu^{0,0}_m(j,j)\,.\ee
Explicitly
\[ u\bigl( n \am(t) k \bigr) \= \ch_m(n)\, \Kph{2j}000(k) \sum_{m\geq 1}
e^{-m^2}\, t^2 \, I_j(2\pi m t) \,.\]
The dependence on $n\in N$ and $k\in K$ is by a non-zero factor. For the
absolute convergence we use \eqref{Iae} to get a bound by
$\sum_{m\geq 1} (2\pi m)^{-1} e^{-m^2}\, t^{2-m}\, e^{2\pi m t} <\infty$.
All terms with $m t \geq t_0$ for a suitable $t_0>0$ are positive.
Hence for $m\geq t_0$ the term has growth larger than
$c\, t^{2-m}\, e^{2\pi m t}$. This implies that $u\bigl( \am(t) \bigr)$
has more than finitely many exponentially growing Fourier terms. This
is an example of an element of $\I_\s(\ps)$ that has no moderate
exponential growth.

In Remark (3) to Definition~\ref{def-afmeg} we used the function $u$ to
form a Poincar\'e series. For this purpose we determine a bound for
$u\bigl( \am(t)\bigr)$ as $t\downarrow 0$. Let $t \leq 1$. From the
series in \eqref{Inu} we get:
\badl{ue1} \sum_{1\leq m \leq 1/2\pi t} &e^{-m^2} t^2 I_j(2\pi m t)
\;\ll_j\; \sum_{1\leq m \leq 1/2\pi t} m^j t^{2+j} \, e^{-m^2}\\
& \;\ll\; t^{2+j}\int_{x=0}^\infty x^j e^{-x^2}\, dx \=
\oh_j(t^{2+j})\,.
\eadl
For the remaining part we use the estimate at~$\infty$.
\badl{ue2} \sum_{ m \geq 1/2\pi t} &e^{-m^2} t^2 I_j(2\pi m t)
\;\ll_j\; \sum_{m\geq 1/2\pi t} t^{3/2} e^{-m^2+\oh(1)} m^{-1/2} \\
&\;\ll\; t^{3/2} \int_{x=1/(2\pi t)}^\infty e^{-x^2}\, x^{-1/2}\,
\frac{dx}{x^{1/2}} \;\ll\; t^3 e^{-1/(2\pi t)^2}\,.
\eadl
This is much smaller than the estimate~\eqref{ue1}, and we get
\be\label{uest0}
u\bigl( n \am(t) k \bigr) \= \oh\bigl(t^{2+j}\bigr)\qquad(t\downarrow
0)\,.\ee\medskip


\def\flnm{rFtm-appendix}

\clearpage 
\chapter*{Appendix}\markboth{APPENDIX}{APPENDIX}

\setcounter{section}{0}
\renewcommand\thesection{\Alph{section}}


\def\flnm{rFtm-spf}


\section{Special functions}\label{app-spf}\markright{A. SPECIAL
FUNCTIONS}

In the description of Fourier term modules we use modified Bessel
functions and Whittaker functions. Here we collect some facts
concerning these special functions.

\subsection{Modified Bessel functions}\label{sect-mBf}
The \il{mBde}{modified Bessel differential equation}modified Bessel
differential equation is
\be \label{mBd} x^2 \, j''(x) + x\, j'(x)
- (x^2+\nu^2) j(x) \=0\,, \ee
for functions $j$ on $(0,\infty)$. See, eg, \cite{Wats}.

The exponents near $t=0$ are $\nu$ and $-\nu$. The exponent $\nu$ leads
to the following modified Bessel function\il{mBfs}{modified Bessel
function} \ir{Inu}{I_\nu}
\be \label{Inu} I_\nu(x)
= \sum_{m\geq0} \frac{(x/2)^{\nu+2m}} {m!\; \Gamma(\nu+m+1)}\,. \ee
So $I_\nu(x) = (x/2)^\nu \, h(x)$, where $h$ is the restriction of an
entire function with value $1$ at $x=0$. For $\nu\in \CC\setminus \ZZ$
the functions $I_\nu$ and $I_{-\nu}$ span the solution space. This is
not the case if $\nu \in \ZZ$:
\be \label{Ievint} I_n(x) \= I_{-n}(x) \qquad \text{ for }n\in \ZZ\,.\ee

The solution \ir{Knu}{K_\nu}
\be\label{Knu} K_\nu(x)
\= \frac\pi 2 \,\frac{I_{-\nu}(x)-I_\nu(x)} {\sin\pi\nu}\ee
extends holomorphically to a function of $\nu \in \CC.$ It satisfies
$K_{-\nu}=K_\nu.$ It is linearly independent of $I_\nu$. This
independence can be seen in the behavior near $x=0$. The expansion of
$I_\nu(x)$ near zero starts with a non-zero multiple of $x^\nu$, or a
multiple of $x^{-\nu}$ if $\nu \in \ZZ_{\leq -1}$. The expansion of
$K_\nu(x)$ has always non-zero multiples of $x^\nu$ and of $x^{-\nu}$
if $\nu \not\in \ZZ$, and a logarithmic term if $\nu \in \ZZ$.

The linear independence is also visible in the asymptotic behavior as
$x\uparrow\infty$. The function $K_\nu$ is characterized by its
exponential decay as $x\uparrow\infty;$ in fact it has an asymptotic
expansion
\badl{Kae} K_\nu(x) & \sim \sqrt{\frac \pi{2x}} \, e^{-x} \sum_{m\geq 0}
\frac{(-1)^m\, \bigl(\frac12-\nu\bigr)_m\, \bigl( \frac12+\nu\bigr)_m}
{m!\;(2x)^m}\,, \eadl
whereas
\be\label{Iae} I_\nu(x) \;\sim\; \frac{e^x}{\sqrt{2\pi x}} \sum_{m\geq
0} \frac{\bigl( \frac12-\nu\bigr)_m\, \bigl( \frac12+\nu\bigr)_m}{
m!\;(2x)^m} \,. \ee
See \cite[7.23]{Wats}.

\rmrk{Contiguous relations} Section 3.71 in \cite{Wats} gives relations
for $K_\nu$ and $I_\nu$:\il{cr1}{contiguous relations}
\badl{cr1} K_{\nu-1}(x) - K_{\nu+1}(x)
&\= -\frac{2\nu}x \, K_\nu(x)\,,& I_{\nu-1}(x) - I_{\nu+1}(x)
&\= \frac{2\nu}x \, I_\nu(x)\,,\\
K_{\nu-1}(x) + K_{\nu+1}(x)
&\= - 2 K_\nu'(x)\,,\\
I_{\nu-1}(x) + I_{\nu+1}(x)
&\= 2 I_\nu'(x)\,. \eadl
See \cite[\S A1]{Math}.

\subsection{Whittaker functions}\label{sect-Whit}

The \il{Wde}{Whittaker differential equation}Whittaker differential
equation for functions on $(0,\infty)$ is
\be\label{Wh-deq} y''(\tau) = \Bigl( \frac14-\frac\kappa \tau
+ \frac{s^2-1/4}{\tau^2}\Bigr)
y(\tau)\,.\ee
It has parameters $\kappa,s\in \CC.$ See eg \cite[(1.6.2)]{Sla}.

The exponents at $\tau=0$ are $\frac12+s$ and $\frac12-s$. The exponent
$\frac12+s$ leads to the solution \ir{Mkps}{M_{\k,s}}
\be\label{Mkps} M_{\kappa,s}(\tau) = \tau^{s+1/2}\, e^{-\tau/2}\,
\sum_{n\geq 0} \frac{ \bigl(\frac12+s-\kappa\bigr)_n}
   {\bigl(1+2s\bigr)_n}\, \frac{\tau^n}{n!}\,.\ee
It is of the form $\tau \mapsto \tau^{s+1/2}\, h(\tau)$ with an entire
function $h$ with value $1$ at~$0$. If $s\in \CC\setminus \frac12\ZZ$
the functions $M_{\k,s}$ and $M_{\k,-s}$ span the solution space. At
values $s\in \frac12\ZZ_{\leq -1}$ the function $M_{\k,s}$ may have a
first order singularity.
If a singularity occurs at $\nu=-\nu_0\in \ZZ_{\leq -1}$, then the
residue is
\be \label{Mres}\frac{ (-1)^{\nu_0-1} \, \bigl( \frac{1-\nu_0}2-\k\bigr)_{\nu_0} }
{2\; \nu_0!\;(\nu_0-1)!}\, M_{\nu_0/2,\k}\,.\ee

The solution \ir{Wkps}{W_{\k,s}} given for $s\not\in \frac12\ZZ$ by
\be\label{Wkps} W_{\kappa,s}(\tau)
= \frac \pi{\sin2\pi s} \sum_\pm \frac{\mp M_{\kappa,\pm s}(\tau)}
{\Gamma(1/2\mp s-\kappa)
\,\Gamma(1\pm 2s) }\ee
extends as a holomorphic function of $s$, and satisfies
$W_{\k,-s}=W_{\k,s}$. This solution is characterized by its exponential
decay as $\tau \uparrow\infty$.

It is convenient to have another solution that is invariant under
$s\leftrightarrow -s$. We make the choice to use\ir{Vkps}{V_{\k,s}}
\badl{Vkps} V_{\kappa,s}(\tau) &\= W_{-\kappa,s}(-\tau)\\
&\qquad \text{(this implies a choice of a branch)}\\
&\= \frac{\pi i}{\sin2\pi s}\sum_\pm \frac{\pm e^{\pm \pi i s} \,
M_{\kappa,\pm s}(\tau)}{\Gamma(1/2\mp s+\kappa)
\,\Gamma(1\pm 2s)}\,.
\eadl
Unlike $M_{\k,s}$ and $W_{\k,s}$, this is not a commonly used notation.
The expression in \eqref{Vkps} gives $V_{\k,s}$ as a meromorphic linear
combination of $M_{\k,s}$ and $M_{\k,s}$ and even in~$s$. In \cite[\S
A2e]{Math} we carry out a check that it is actually holomorphic in~$s$.

The functions $W_{\k,s}$ and $V_{\k,s}$ form a basis of the solution
space for all $\pm s\in \CC$. We have the following meromorphic
relation with $M_{\k,s}$. (See \cite[\S A2a]{Math}.)
\badl{MinWV} M_{\kappa,s}(\tau) &\= e^{\pi i \kappa}\,\Gamma(1+2s)\,
\biggl( \frac{-i\,e^{-\pi i s}}{\Gamma(1/2+s+\kappa)}
W_{\kappa,s}(\tau)
\\
&\quad\hbox{} - \frac 1{\Gamma(1/2+s-\kappa)} \, V_{\kappa,s}(\tau)
\biggr)\,.\eadl

\rmrk{Exponential decay and growth} We have as $\tau\uparrow\infty$
\begin{align}
\nonumber W_{\kappa,s}(\tau) &\;\sim\; \tau^\kappa \, e^{-\tau/2} \Bigl(
1 +\frac{s^2 -(\k-1/2)^2}\tau\\
\label{Wae}
&\qquad\frac{ \bigl(s^2 -(\k-1/2)^2\bigr)(s^2-(\k-3/2)^2\bigr)
}{8\tau^2}
+ \cdots\Bigr)
\,,
\displaybreak[0]\\
\nonumber
V_{\kappa,s}(\tau)&\;\sim\;
- e^{-\pi i\kappa} \, \tau^{-\kappa} \, e^{\tau/2} \Bigl(
1+\frac{(\k+1/2)^2-s^2}\tau\\
\label{Vae}
&\qquad\hbox{} + \frac{\bigl((\k+1/2)^2-s^2\bigr)\bigl((\k+3/2)^2-s^2
\bigr)}{8\tau^2} + \cdots
\Bigr)\,. \end{align}
We use (4.2.22) in \cite{Sla} for $W_{\kappa,s}$, and
(4.1.21)
to get the asymptotic behavior of $V_{\kappa,s}.$ Check in \cite[\S
A2b]{Math}.

The families $W_{\k,s}$ and $V_{\k,s}$ are linearly independent for all
choices of the parameters.

\begin{lem}\label{lem-VI}Let $f$ be a linear combination of functions
$t\mapsto t^{1+c}\, V_{\k+k,s}(2\pi u t^2)$, where $c$ runs over a
finite subset of $\ZZ_{\geq 0}$ and $k$ over a finite subset of $\ZZ$.
The quantities $u>0$, $\k \in \RR$ and $s\in \CC$ are fixed.

If $f(t) = \oh(1)$ as $t\uparrow \infty$, then it is zero.
\end{lem}
\begin{proof}The contiguous relations in \eqref{crW1} allow us to
express $V_{\k,s}$ as a linear combination of $V_{\k_1,s}$ with
$\k_1$ running over $\k-1, \k,\k+1$. Using this repeatedly we arrive at
a finite sum
\[ \sum_j c_j \, V_{\k+j,s}(2\pi u t^2) \= \oh(1)\,.\]
Let $j_c$ be the minimal value of $j$ for which $c_j\neq 0$. The
estimate \eqref{Vae} implies that $c_{j_c}$ is zero. Proceeding in this
way we arrive at $f=0$.
\end{proof}

\rmrk{Linear dependence} If $s\not\in \frac12\ZZ_{\leq -1}$ we obtain
from~\eqref{MinWV}:
\badl{MWV} M_{\k,s} \in \CC\, W_{\k,s} &\Leftrightarrow \frac12-\k+s\in
\ZZ_{\leq 0} \,,\\
M_{\k,s}\in \CC\, V_{\k,s}&\Leftrightarrow \frac12+\k+s\in \ZZ_{\leq 0}
\,. \eadl

For $s_0\in \CC\setminus \ZZ_{\leq -1}$ the function $M_{\k,s_0}(\tau)$
spans the space of solutions of the Whittaker differential equation
with parameters $\k $ and $s_0$ that are of the form $\tau^{s_0+1/2}$
times an entire function.

\rmrk{Behavior at zero} Let $\re s_0\in \frac12 \ZZ_{\leq 0}$ and
$\frac12+s_0+\k \in \ZZ_{\geq 1}$. Then the leading term in the
expansion of $ W_{\k,s_0}(\tau)$ as $\tau\downarrow0$ is a non-zero
multiple of $\tau^{s_0+1/2}$ if $s_0>0$, and a non-zero multiple of
$\tau^{1/2}\log\tau$ if $s_0=\frac12$. The expansion of
$V_{\k,s_0}(\tau)$ as $\tau\downarrow 0$ has the same properties under
the conditions $\re s_0\in \frac12\ZZ_{\geq 0}$ and
$\frac12+s_0-\k \in\ZZ_{\geq 1}$.

\rmrk{Specializations}For special combinations of the parameters these
Whittaker functions have expressions in simpler functions. (See
\cite[\S A2c]{Math}.)
\badl{specWV} W_{\kappa,\pm (\kappa-1/2)}(\tau) &\= \tau^\kappa\,
e^{-\tau/2} = M_{\kappa,\kappa-1/2}(\tau)\,,\\
V_{\kappa,\pm
(\kappa+1/2)}(\tau) &\=
- e^{-\pi i \kappa}\, \tau^{-\kappa} e^{\tau/2} \\
& \=
-e^{-\pi i \kappa} M_{\kappa,-\kappa-1/2}(\tau)
\,. \eadl
The relations with $M_{\kappa,\kappa\pm 1/2}$ are valid as holomorphic
functions in $\kappa.$ The function $M_{\kappa,s}$ may have a
singularity as a function of $(\k,s)$ at these points.

\rmrk{Contiguous relations}\il{cr2}{contiguous relations} We will need
several of the relations in Section 2.5 of \cite{Sla}. See also
\cite[\S A2d]{Math}.
\badl{crWd} W_{\kappa,s}'(\tau)&\= \Bigl( \frac12-\frac\kappa\tau\Bigr)
W_{\kappa,s}(\tau)
  - \frac 1 \tau W_{\kappa+1,s}(\tau)\,,
  \\
V_{\kappa,s}'(\tau) & \= \Bigl( \frac12-\frac\kappa\tau\Bigr)
V_{\kappa,s}(\tau)
+\frac{(\kappa+1/2)^2-s^2}\tau V_{\kappa+1,s}(\tau)\,,
  \\
M_{\kappa,s}'(\tau) &\= \bigl( \frac12-\frac\kappa\tau\Bigr)
M_{\kappa,s}(\tau)
+\bigl(\frac12+\kappa+s\bigr)\tau^{-1} M_{\kappa+1,s}(\tau)\,;
\eadl

\badl{crW1}
(\tau-2\kappa) W_{\kappa, s}(\tau)&\= W_{\kappa+1,s}(\tau)
+\bigl(
(\kappa-1/2)^2-s^2) W_{\kappa-1,s}(\tau)\,,
  \\
(\tau-2\kappa) V_{\kappa,s}(\tau)&\=
(s^2-(\kappa+1/2)^2)V_{\kappa+1,s}(\tau)
- V_{\kappa-1,s}(\tau)\,,
  \\
(\tau-2\kappa)\, M_{\kappa,s}(\tau)
&\= \bigl(\frac12-\kappa+s\bigr)
M_{\kappa-1,s}(\tau)
-\bigl(\frac12+\kappa+s)
M_{\kappa+1,s}(\tau)\,;\eadl

\badl{crWh1} W_{\kappa+1/2,s}(\tau) &\=
(s-\kappa) W_{\kappa-1/2,s}(\tau)
+\tau^{1/2}\, W_{\kappa,s-1/2}(\tau)\,,
  \\
V_{\kappa+1/2,s}(\tau) & \=
(\kappa+s)^{-1} V_{\kappa-1/2,s}(\tau)
-\frac i{\kappa+s} \tau^{1/2} V_{\kappa,s-1/2}(\tau)\,,
  \\
(s+\kappa)\, M_{\kappa+1/2,s}(\tau)
&\= (\kappa-s)\, M_{\kappa-1/2,s}(\tau)
+2 s \tau^{1/2} \, M_{\kappa,s-1/2}(\tau)\,;\eadl
\badl{crWh2} W_{\kappa+1/2,s}(\tau)
&\= -(\kappa+s)\, W_{\kappa-1/2,s}(\tau)
+\tau^{1/2}\, W_{\kappa,s+1/2}(\tau)\,,
  \\
V_{\kappa+1/2,s}(\tau) &\= \frac1{\kappa-s} V_{\kappa-1/2,s}(\tau)
+\frac i{s-\kappa}\,\tau^{1/2} V_{\kappa,s+1/2}(\tau)\,,
  \\
(2s+1)\, M_{\kappa+1/2,s}(\tau)
&\=
(2s+1)\, M_{\kappa-1/2,s}(\tau)
-\tau^{1/2}\, M_{\kappa,s+1/2}(\tau)\,.
\eadl

\newpage 

\def\flnm{rFtm-lists}

\section{Irreducible submodules of special Fourier term
modules}\label{app-tabs}
\markright{B. IRREDUCIBLE SUBMODULES}

Table~\ref{tab-isot-ps}, p~\pageref{tab-isot-ps}, lists the irreducible
$(\glie,K)$-modules under the condition of integral parametrization
that occur in principal series modules. In Table~\ref{tab-irr-Ftm}
below, we add the occurrences in other special Fourier term modules.
(The occurrences in non-abelian Fourier term modules have conditions on
the quantities $m_0(j)$ that we do not give in this table.)

\newcommand\tl[2]{$ #1 $&$ #2 $&}
\newcommand\ml[2]{\multicolumn{1}{|r}{in}&$ #1 $& #2 }

\begin{center}
\begin{longtable}{|cll|}
\caption{Irreducible modules}\label{tab-irr-Ftm}\\\hline
\textrm{type}& \textrm{parameters} & \textrm{reference} \\ \hline
\hline
\endfirsthead \multicolumn{3}{c} {\tablename\ \thetable\ --
\textit{Continued from previous page} }\\\hline
\textrm{type}& \textrm{parameters} & \textrm{reference}\\
\hline
\endhead \tl{\II_+(j_+,\nu_+)}{\nu_+\geq |j_+|+2}\\
\ml{H^{\xi_+,\nu_+}_K} { Fig.~\ref{fig-ps1}}\\
\ml{H^{\xi_r,\nu_r}_K}{ Fig.~\ref{fig-ps3}} \\
\ml{H^{\xi_l,\nu_l}_K}{ Fig.~\ref{fig-ps11}}\\
\ml{ \Wfu^{\psi[j_+,\nu_+]}_\bt \text{ and } \Mfu^{\psi[j,\nu]}_\bt}
{Thm.~\ref{mnthm-ab-ip} iv)}\\
\ml{ \Mfu^{\psi[j_+,\nu_+]}_\n }{Rem.~\ref{rmk-inab}}\\
\ml{ \Wfu^{\psi[j_+,\nu_+]}_\n }{Fig.~\ref{fig-j3}}
\\\hline
\tl{\II_+(j_+,j_+) }{ j_+= \nu_+\in \ZZ_{\geq 1}}\\
\ml{H^{\xi_+,j_+}_K}{ Fig.~\ref{fig-ps2}} \\
\ml{H^{\xi_l,-0}_K}{ Fig.~\ref{fig-ps10}}\\
\ml{ \Wfu^{\psi[j_+,\nu_+]}_\bt \text{ and } \Mfu^{\psi[j_+,\nu_+]}_\bt}
{Thm.~\ref{mnthm-ab-ip} iv)}\\
\ml{ \Mfu^{\psi[j_+,\nu_+]}_\n }{Rem.~\ref{rmk-inab}}\\
\ml{ \Wfu^{\psi[j_+,\nu_+]}_\n }{Fig.~\ref{fig-j3a}}
\\\hline
\tl{\II_+(j_+,-j_+) }{ j_+=-\nu_+\in \ZZ_{\geq 0}}\\
\ml{H^{\xi_r,0}_K, }{ Fig.~\ref{fig-ps4}} \\
\ml{H^{\xi_l,-j_l}_K}{ Fig.~\ref{fig-ps12}}\\
\ml{ \Wfu^{\psi[j_+,\nu_+]}_\bt \text{ and } \Mfu^{\psi[j_+,\nu_+]}_\bt}
{Thm.~\ref{mnthm-ab-ip} iv)}\\
\ml{ \Mfu^{\psi[j_+,\nu_+]}_\n }{Rem.~\ref{rmk-inab}}\\
\ml{ \Wfu^{\psi[j_+,\nu_+]}_\n }{Fig.~\ref{fig-j3a}}
\\\hline
\tl{\IF(j_r,\nu_r) }{ 1\leq \nu_r \leq j_r-2}\\
\ml{H^{\xi_r,\nu_r}_K}{ Fig.~\ref{fig-ps3}}\\
\ml{ \Wfu_\n^{\psi[\xi_+,\nu_+]} }{ Fig.~\ref{fig-mj1}}
\\ \hline
\tl{\IF(j_r,0) }{ j_r\in 2\ZZ_{\geq 0}}\\
\ml{H^{\xi_r,0}_K}{ Fig.~\ref{fig-ps4}} \\
\ml{\Wfu^{\psi[\xi_+,\nu_+]}_\n}{Fig.~\ref{fig-mj1a}}
\\ \hline
\tl{\IF_+( j_r,-\nu_r) }{ 1\leq \nu_r\leq j_r-2}\\
\ml{H^{\xi_r,-\nu_r}_K}{ Fig.~\ref{fig-ps5}} \\
\ml{\Wfu_\n^{\psi[j_+,\nu_+]}}{Fig.~\ref{fig-mj2}}\\\hline
\newpage
\tl{\IF(j_r,-j_r)}{ j_r\in \ZZ_{\geq 1} }\\
\ml{H^{\xi_r,-\nu_r}_K}{ Fig.~\ref{fig-ps6}}\\
\ml{\Wfu_\n^{\psi[j_+,\nu_+]}} {Fig.~\ref{fig-mj2b}} \\ \hline
\tl{\FI(j_l,\nu_l) }{ 1\leq \nu_l \leq -j_l-2 } \\
\ml{H^{\xi_l,\nu_l}_K}{ Fig.~\ref{fig-ps11}}\\
\ml{\Wfu_\n^{\psi[j_+,\nu_+]}}{Fig.~\ref{fig-pj1VW}}
\\ \hline
\tl{\FI(j_l,0) }{j_l = -2 j_r=-j_+\in 2\Z_{\leq 0}}\\
\ml{H^{\xi_l,-0}_K}{ Fig.~\ref{fig-ps10}}\\
\ml{\Wfu_\n^{\psi[j_+,\nu_+]}}{Fig.~\ref{fig-pj1a}}
\\ \hline
\tl{\FI_+(j_l,-\nu_l)}{j_l+2 \leq -\nu_l -1}\\
\ml{H^{\xi_l,-\nu_l}_K}{ Fig.~\ref{fig-ps9}}\\
\ml{ \Wfu^{\psi[j_+,\nu_+]}_\n }{Fig.~\ref{fig-pj2VW}}
\\ \hline
\tl{\FI(j_l,j_l)}{j_l \in \ZZ_{\leq -1}}\\
\ml{H^{\xi_l,-j_l}_K}{ Fig.~\ref{fig-ps8}}\\
\ml{ \Wfu^{\psi[j_+,\nu_+]} _\n} {Fig.~\ref{fig-pj2b}}
\\ \hline
\tl{\FF(j_+,-\nu_+)}{\nu_+ \geq |j_+|+2}\\
\ml{H^{\xi_+,-\nu_+}_K}{ Fig.~\ref{fig-ps7}}
\\ \hline
\end{longtable}
\end{center}


\def\flnm{rFtm-ds}


\section{Some discrete subgroups}\label{app-ds}\markright{C. SOME DISCRETE
SUBGROUPS}

In this note we study the $(\glie,K)$-mod\-ules realized in spaces of
Fourier terms of elements of $C^\infty(\Ld_\s\backslash G)_K$ for a
standard lattice $\Ld_\s\subset N$; see Definition~\ref{def-stlatt}.
Although all lattices in $N$ are isomorphic to some standard lattices,
this forms a genuine restriction as soon as we want to apply it to
modules occurring in Fourier expansions of functions in
$C^\infty(\Gm\backslash G)$ for a cofinite discrete group
$\Gm\subset G$. The reason is that the isomorphism of lattices $\Ld$
and $\Ld'$ in $N$ does not lead, in general, to morphisms of
$(\glie,K)$-modules
$C^\infty( \Ld\backslash G)_K \rightarrow C^\infty(\Ld'\backslash G)_K$.
By imposing the $\ZZ[i]$-condition on the cusps in
Definition~\ref{cucond} we restrict ourselves to discrete groups $\Gm$
for which we have to deal only with isomorphisms between lattices in
unipotent groups $N_\c$ given by conjugation by elements of~$G$. This
leads to the \il{Zccs}{$\ZZ[i]$-condition on the
cusps}\emph{$\ZZ[i]$-condition on the cusps} in
Definition~\ref{cucond}.

For a given cusp, the group $\Gm\cap N_\c=g_c\Ld_{\s(\c)} g_\c^{-1}$ is
normalized by $\Gm \cap P_\c$. The group
$(\Gm\cap N_\c) \bigm\backslash (\Gm\cap P_\c)$ is isomorphic to a
subgroup $U_\c$ of the groups of $12$-th roots of unity. This gives
linear relations between generic abelian Fourier terms $\Four_\bt $ and
$\Four_{\z^3\bt}$ with $\z\in U_\c$. The relations between non-abelian
Fourier terms involve all $\Four_{\ell,c,d}$ with $c\bmod 2\ell$ for
given $\ell\in \frac \s 2\ZZ_{\neq 0}$ and $d\in 1+2\ZZ$.

\rmrk{Two examples}We look closer at the two examples in
\S\ref{sect-cucond}. The group
\il{Gm0}{$\Gm_0$}$\Gm_0 = G \cap \SL_3\bigl( \ZZ[i]\bigr)$ is by
definition a subgroup of the standard realization $G$ of $\SU(2,1)$. We
checked in \S\ref{sect-cucond} with use of~\eqref{Ndef} that
$\Gm_0\cap N = h(i+1) \Ld_4 h(i+1)^{-1}$.

The example of Francsics and Lax \cite{FL}, conjugated into $G$, is
\ir{GmFL}{\Gm_\fl}
\be\label{GmFL} \Gm_\fl := U_\fl \SL_3\bigl(\ZZ[i]\bigr)U_\fl^{-1} \cap
G\,.\ee
Since
\be\label{GmflN} U_\fl^{-1} \nm(b,r) U_\fl = \begin{pmatrix}1& i\sqrt 2
b& -2 r - i |b|^2\\
0&1&-\sqrt 2\bar b\\
0&0&1
\end{pmatrix}
\ee
is integral if and only if $b\in \frac{1+i}{\sqrt 2}\ZZ[i]$ and
$r\in \frac12\ZZ$, we have
\[\Gm_\fl\cap N = \mm(e^{\pi i/4}) \Ld_4 \mm(e^{\pi i/4})^{-1}\,.\]
The fundamental domain in~\cite{FL} shows that $\Gm_\fl$ has only one
$\Gm_\fl$-orbit of cusps. So $\Gm_\fl$ satisfies the $\ZZ[i]$-condition
on the cusps.

To see that $\Gm_0$ satisfies the $\ZZ[i]$-condition on the cusps as
well, it suffices to prove the following result.
\begin{prop}The cusps of $\Gm_0$ form one $\Gm_0$-orbit represented by~$\infty$.
\end{prop}
\begin{proof}[Sketch of a proof] We show that if $\tau_0>0$ is
sufficiently small the set
\newcommand\Sie{\mathfrak S}
\badl{Sie} \Sie &\= \Bigl\{ (z,u)\in \X\;:\; u= \tfrac{1+i}2\xi +
\tfrac{1-i}2\eta\,,\xi,\eta\in [0,1]\,,\\
&\qquad\quad \re z \in [-1,1], \; \im z-|u|^2 \geq \tau_0 \Bigr\}
\eadl
is a Siegel domain for $\Gm_0$, i.e., $\X = \Gm\, \Sie$. Since $\infty $
is the sole element of the closure of $\Sie$ intersecting $\partial\X$,
this implies that all cusps of $\Gm_0$ are equivalent to~$\infty$.

By left multiplication by elements of $\Gm_0 \cap N$ we can bring each
$(z,u)\in \X$ into the form $(z',u')$ with $\re z'\in [-1,1]$ and
$u= \frac{1+i}2 \xi+\frac{1-i}2\eta$ with $\xi,\eta\in [-1,1]$, without
a change of $\tau=\im z-|u|^2$. By application of
\[ \mm(i) = \begin{pmatrix}i&0&0\\
0&-1&0\\
0&0&i\end{pmatrix} \in \Gm_0\cap M,\]
we bring $(z,u)$ into the set
\be S \= \Bigl\{ (z,u)\in \X\;:\; u= \tfrac{1+i}2\xi +
\tfrac{1-i}2\eta\,,\xi,\eta\in [0,1],\, \re z\in [-1,1]\Bigr\}\ee
without changing the value of $\tau=\im(z) - |u|^2$.

Next we use suitable elements $\gm\in\Gm_0$ such that for $(z,u)\in S$
with $\tau(z,u) < \tau_0$ the image $(z',u')=\gm^{-1}(z,u)$ satisfies
$\tau'=\im(z')-|u'|^2 > \th \tau$ for a fixed number $\th$ slightly
larger than~$1$. Repeating this process we arrive in~$\Sie$.

The choice of the elements $\gm$ and of suitable values of $\tau_0$
takes some work. We partition the part of $S$ with $\tau<\tau_0$ in the
following way, with $\frac12\leq \eta_0\leq 1$ and $0\leq \xi_0\leq 1$.
\[\setlength\unitlength{1cm}
\begin{picture}(5,5)(0,-2.5)
\put(.8,1.7){$u$}
\put(-.2,0){$0$}
\put(5.1,0){$1$}
\put(2.7,-2.5){$(1-i)/2$}
\put(2.7,2.3){$(1+i)/2$}
\put(1.5,.4){$A_0$}
\put(3.75,.4){$A_1$}
\put(3.1,0){$\eta_0$}
\thicklines
\put(3,0){\circle*{.1}}
\put(3,-2){\line(0,1){4}}
\put(0,0){\line(1,1){2.5}}
\put(0,0){\line(1,-1){2.5}}
\put(5,0){\line(-1,1){2.5}}
\put(5,0){\line(-1,-1){2.5}}
\end{picture}
\qquad\qquad\qquad
\begin{picture}(1,5)(-.5,-2.5)
\put(-1.7,1.7){$x=\re z$}
\put(.1,2.5){$1$}
\put(.1,0){$0$}
\put(.1,-2.5){$-1$}
\put(-.6,0){$B_0$}
\put(.2,2){$B_+$}
\put(.2,-2.1){$B_-$}
\put(.2,1.5){$\xi_0$}
\put(.2,-1.7){$-\xi_0$}
\thicklines
\put(0,1.6){\circle*{.1}}
\put(0,-1.6){\circle*{.1}}
\put(0,2.5){\circle*{.1}}
\put(0,-2.5){\circle*{.1}}
\put(0,-2.5){\line(0,1){5}}
\end{picture}
\]
So $A_0$ is determined by $\re u\leq \eta_0$, and $A_1$ by
$\re u\geq 1-\eta_0$. Further, $B_0=[-\xi_0,\xi_0]$, $B_+ = [\xi_0,1]$,
and $B_-=[-1,-\xi_0]$. We write $z=x+iy$.

We take, for each of the six resulting parts of the set $S$, an element
$\gm \in \Gm$ such that the cusp $\gm \,\infty$ is near to that part.
For instance for $u\in A_0$ and $z\in B_0$ we are near to the cusp
$(0,0)=\wm\infty$, with
\[ \wm\=\wm^{-1}\= \begin{pmatrix}-1&0&0\\0&-1&0\\0&0&1\end{pmatrix} :
(z,u)\mapsto \Bigl( \frac{-1}z,\frac{-iu}z\Bigr)\,.\]
For such $(z,u)$
\[ |z|^2\= x^2+\bigl(\tau+|u|^2\bigr)^2 \leq
\xi_0^2+\bigl(\tau_0+m(\eta_0)\bigr)^2\,, \]
with
\be m(\eta_0) \= \eta_0^2+(1-\eta_0)^2\,.\ee
(In the set in the $u$plane depicted above the point in $A_0$ with
maximal distance to $0$ is $\bigl( \eta_0,1-\eta_0)$.)
This leads to
\be \tau' \geq \tau\bigm/ \bigl( \xi_0^2+(\tau_0+m(\eta_0)^2\bigr)\,.\ee
Hence $ (\xi_0^2+(\tau_0+m(\eta_0)^2)^{-1}$ is a lower bound for
$\eta_0$.

In \cite[6b]{Math} we carry out a similar computation in five more
cases:
\begin{align*}
&u\in A_0,\, z\in B_+& \gm&\=
\begin{pmatrix}-1&0&0\\0&-i&0\\0&0&-i\end{pmatrix}:& \infty&\mapsto
(1,0)\,,
\displaybreak[0]\\
&u\in A_0 ,\,z\in B_-&
\gm&\=\begin{pmatrix}i&0&0\\0&-i&0\\0&0&1\end{pmatrix}:& \infty
&\mapsto (-1.0)\,,\displaybreak[0]\\
&u\in A_1,\, z\in B_0&\gm&\= \begin{pmatrix} 0&-i&0\\
-i&0&0\\0&0&1
\end{pmatrix}:& \infty&\mapsto (i,1)\,,
\displaybreak[0]\\
&u\in A_1,\,z\in B_+& \gm&\=
\begin{pmatrix}1&1-i&-1-i\\
1-i&1&-1-i\\
1+i&1+i&1-2i\end{pmatrix}:& \infty &\mapsto (i+1,1)\,,
\displaybreak[0]\\
&u\in A_1,\, z\in B_-& \gm&\= \begin{pmatrix}0&-1-i&i\\
1-i&1&-1-i\\
i&-1+i&2
\end{pmatrix}: & \infty &\mapsto (i-1,1)\,.
\end{align*}

In all cases we consider $(z',u')=\gm^{-1}(z,u)$ and check that
$ \tau'= \im(z')-|u'|$ has the form
\be \tau' \= \frac\tau{|D(z,u)|^2}\ee
with a polynomial $D$ of degree one. In the first case, carried out
above, we have $D(z,u)=z$. The polynomial $D$ vanishes at the point of
$\partial\X$ indicated for each case.

Then we give an estimate of $|D(z,u)|^2$ by a quantity depending on the
parameters $\tau_0$, $\xi_0$ and $\eta_0$ that is valid for all $(z,u)$
with $\im(z)-|u|^2=\tau$ for $0<\tau<\tau_0$ for which $z $ and $u$ are
in the region $B_\ast$ and $A_\ast$ indicated above for each case.

Having obtained these six upper bounds we look by trial and error for
values for $\tau_0$, $\xi_0$ and $\eta_0$ for which all six quantities
have a value in $[0,1)$. Since that search succeeds easily, there is a
factor $\eta>1$ such that $\tau'>\eta \tau$.

Iterating $(z,u) \mapsto (z',u')$ a finite number of times, we arrive at
$\tau(z',u')\geq \tau_0$. This means that for the value of $\tau_0$
that we found the set $\Sie$ in~\eqref{Sie} is indeed a Siegel domain
for $\Gm\backslash \X$. This implies that $\Gm$ has only on orbit of
cusps.
\end{proof}

\rmrk{Other number fields}Instead of the ring of integers $\ZZ[i]$ in
$\QQ(i)$ we may try the ring of integers in other number fields. For
example, let us take
\be \Gm \= \SL_3\bigl( \sqrt{-2}\bigr) \cap G\,.\ee
In \eqref{Ndef} we see that $\Gm\cap N$ is generated by $\nm(2,0),$
$\nm\bigl( i\sqrt{-2},0\bigr)$ and $\nm(0,\sqrt 2).$ The automorphism
$A\in \Aut(N)$ associated to
$\begin{pmatrix}2 & 0\\0&\sqrt 2\end{pmatrix}
\in \GL_2(\RR),$ as in \eqref{oautN}, satisfies $ A \Ld_4 = \Gm\cap N.$
This automorphism cannot be obtained by conjugation by an element of
$NAM.$



\def\flnm{rFtm-lit}

\iflitnum
\newcommand\bibit[4]{
\bibitem {#1}#2: {\em #3;\/ } #4}
\newcommand\bibitq[4]{
\bibitem {#1}#2: {\em #3\/ } #4}
\else
\newcommand\bibit[4]{
\bibitem[#1] {#1}#2: {\em #3;\/ } #4.}
\newcommand\bibitq[4]{
\bibitem[#1] {#1}#2: {\em #3\/ } #4.}
\fi
\newcommand\bibitnot[4]{} 

{\raggedright
}



\def\flnm{rFtm-ind}

\newpage


\markboth{INDEX}{INDEX}

\newcommand\ind[2]{\item #1\quad\ #2}
\renewcommand\il[1]{\pageref{i-#1}}
\renewcommand\ir[1]{\pageref{#1}}{}

\section*{Index}
\begin{multicols}{2}
\raggedright
\begin{trivlist}\footnotesize
\ind{antiholomorphic discrete series type}{\il{adst}}
\ind{automorphic form}{\il{af}}
\ind{---, square integrable}{\il{afsqi}}
\ind{--- with moderate exponential growth}{\il{afme}}
\ind{automorphism group of $N$}{\il{aN}}
\indexspace
\ind{big cell}{\il{bc}}
\ind{boundary of symmetric space}{\il{bdss}}
\ind{Bruhat decomposition}{\il{Bd}}
\indexspace
\ind{Cartan subalgebra}{\il{Ca}}
\ind{Casimir element}{\il{Cas-K}, \il{Cas} }
\ind{characters of $M\subset K$}{\il{chM}}
\ind{cohomological representation}{\il{cohrepr}}
\ind{combinations 1, 2 and 3}{\il{comb}}
\ind{complementary series}{\il{cs}, \il{cos}}
\ind{component function}{\il{comp}, \il{co} }
\ind{composition diagram}{\il{cd}}
\ind{$\ZZ[i]$-condition on the cusps}{\il{cuc}, \il{Zccs}}
\ind{contiguous relations}{\il{cr1}, \il{cr2}}
\ind{cusp}{\il{cu}}
\ind{cusp form}{\il{cf}}
\indexspace
\ind{determining component}{\il{dc}}
\ind{discrete series}{\il{ds}}
\ind{downward shift operator}{\il{dso}, \il{dso1} }
\indexspace
\ind{eigenfunction equations}{\il{efeq}}
\ind{Eisenstein series}{\il{Es}}
\ind{evaluation at zero}{\il{eval0}}
\ind{exponential decay}{\il{ed}, \il{ed1}}
\ind{exponential growth}{\il{eg}}
\ind{---, moderate}{\il{egm}}
\ind{exponentials}{\ir{LieKexp}, \il{expnlie} }
\indexspace
\ind{Fourier expansion}{\il{Fe}, \ir{Fe1}}
\ind{---, absolute convergence}{\il{Fsac}}
\ind{Fourier term module}{\il{Ftmo}. \il{Ftm}}
\ind{---, abelian}{\il{Ftma}}
\ind{---, generic abelian}{\il{Ftmgab}}
\ind{---, large}{\il{Ftml}}
\ind{---, $N$-trivial}{\il{Ftm-Ntri}, \il{Ntefm}}
\ind{---, non-abelian}{\il{Ftmn}, \il{Ftmnab}}
\ind{Fourier term operator}{\il{Fto}, \il{Fto1} }
\indexspace
\ind{generic abelian Fourier term module}{\il{gabef}, \il{gabFtm} }
\ind{generic parametrization}{\il{gp}}
\indexspace
\ind{Haar measure}{\il{HmK}, \il{Hm2} }
\ind{Heisenberg group}{\il{Hg}}
\ind{Hermite polynomial}{\il{Hpol}}
\ind{highest weight space in a $K$-type}{\il{hwsp}}
\ind{holomorphic discrete series type}{\il{hdst}}
\ind{holomorphis family of}{}
\ind{--- automorphic forms}{\il{hfaf}}
\ind{--- Fourier terms}{\il{hfFt}}
\indexspace
\ind{integral parametrization}{\il{ip}}
\ind{interior differentiation}{\il{idiff}}
\ind{invariant sesquilinear form}{\il{invsf}}
\ind{Iwasawa decomposition}{\il{Iwd}, \il{Id1}}
\indexspace
\ind{kernel relations}{\il{krel}, \il{krn} }
\ind{Kunze-Stein operator}{\il{KSo}}
\indexspace
\ind{Langlands repesentation}{\il{Llr}}
\ind{large discrete series type}{\il{ldst-intro}, \il{ldst}}
\ind{large Fourier term module}{\il{lFtm}}
\ind{large abelian Fourier term module}{\il{labFtm}}
\ind{large non-abelian Fourier term module}{\il{nablFtm}}
\ind{lattice}{\il{latt}}
\ind{left differentiation}{\il{ld}}
\ind{Lie algebra}{\il{La}, \il{Lie1} }
\ind{limits of discrete series}{\il{lds}}
\ind{logarithmic module}{\il{lm}, \il{ls1} }
\indexspace
\ind{maximal compact subgroup}{\il{mcs}}
\ind{maximal component}{\il{maxc}}
\ind{meromorphic family of}{}
\ind{--- automorphic forms}{\il{mfaf}}
\ind{--- Fourier terms}{\il{mfFt}}
\ind{metaplectic parameter}{\il{mpprm-i}, \il{mtplprm}}
\ind{minimal component}{\il{minc}}
\ind{minimal vector}{\il{mv}}
\ind{moderate exponential growth}{}
\ind{---, automorphic form with}{\il{meaf}}
\ind{modified Bessel differential equation}{\ir{mBd0}, \il{mBde} }
\ind{modified Bessel function}{\il{mBfs}}
\ind{$(\glie,K)$-module}{\il{gKm}}
\ind{multiple parametrization}{\il{mp}}
\ind{multiplication relations for functions on $K$}{\il{mlt}}
\indexspace
\ind{non-abelian Fourier term module}{\il{nabef}}
\ind{norm in $L^2(K)$}{\ir{Kph-norm}}
\ind{normalized Hermite function}{\il{nHf}}
\indexspace
\ind{parameter set of a special cyclic module}{\il{tp}}
\ind{Poincar\'e series}{\il{Ps}}
\ind{polynomial growth}{\il{polgr0}, \il{pg}}
\ind{positive Weyl chamber}{\il{pWch}}
\ind{principal series}{\il{pssum}, \il{ps} }
\ind{---, irreducible}{\il{pdsi}}
\ind{---, unitary irreducible}{\il{ps-ui}}
\indexspace
\ind{quick decay}{\il{qd}}
\indexspace
\ind{realization of $\SU(2,1)$}{\il{remn}}
\ind{$\nu$-regular behavior at $0$}{\il{regbeh}}
\ind{right differentiation}{\il{rd}}
\indexspace
\ind{Schr\"odinger representation}{\il{Schrep}}
\ind{Schwartz function}{\il{Sf}}
\ind{sector of lattice points in $(h/3,p)$-plane}{\il{slp}}
\ind{sesquilinear form}{\il{sqlf}}
\ind{shift operator}{\il{sho}, \il{shob} }
\ind{---, downward}{\il{sowd}}
\ind{---, in abelian case}{\il{shopab}}
\ind{---, in non-abelian case}{\il{shopnab}}
\ind{---, upward}{\il{souw}}
\ind{shift parameter for theta functions}{\il{c-intro}}
\ind{simple parametrization}{\il{sp}}
\ind{simple positive roots}{\il{spr}}
\ind{special module}{\il{spcm}}
\ind{spectral parameters}{\il{spprm}}
\ind{split torus}{\il{sto}}
\ind{square integrable at~$\infty$}{\il{sqiinf}}
\ind{square integrable automorphic form}{\il{sqiaf}}
\ind{standard lattice}{\il{stl}}
\ind{Stone-von Neumann representation}{\il{StvNr}}
\ind{subquotient theorem}{\il{sqth}}
\ind{symmetric space}{\il{syms}}
\indexspace
\ind{theta function}{\il{Thf}}
\ind{thin representation}{\il{thrpru}}
\ind{type of irreducible representation}{\il{tpirr}}
\indexspace
\ind{unique embedding in principal series}{\il{ueb}}
\ind{unipotent subgroup}{\il{upsg}}
\ind{unitary principal series}{\il{ups}}
\ind{unitarizable}{\il{us}}
\ind{universal enveloping algebra}{\il{uea}}
\ind{upper half-plane model of symmetric space}{\il{uhpm}}
\ind{upward shift operator}{\il{uwso}, \il{uwso1} }
\indexspace
\ind{walls of positive chamber}{\il{walls}}
\ind{weight in a $K$-type}{\il{Kwt}}
\ind{weight space in a $K$-type}{\il{wsp}}
\ind{Weil restriction}{\il{Wr}}
\ind{Weyl group}{\il{wg}}
\ind{Whittaker differential equation}{\ir{Whd0}, \il{Wde} }
\end{trivlist}
\end{multicols}


\section*{List of notations}\markright{LIST OF NOTATIONS}
\renewcommand\ind[2]{\item $#1$\quad\ #2}
\begin{multicols}{3}
\raggedright
\begin{trivlist}\footnotesize
\ind{\A(\Gm;\ps)}{\il{AGmps}}
\ind{\Au0(\Gm;\ps)}{\il{A0ps}}
\ind{\Au{(2)}(\Gm;\ps)}{\il{A2ps}}
\ind{\Au\me(\Gm;\ps))}{\il{Aume}}
\ind{\aut(N)}{\il{AutN}}
\ind{A\subset G}{\il{Adef}}
\ind{\alie}{\il{alie}}
\ind{\am(t)\in A}{\ir{Adef}}
\indexspace
\indexspace
\ind{\CK_i\in \klie}{\il{CKi}}
\ind{C\in ZU(\glie)}{\ir{tab-casdt}}
\ind{C_K}{\ir{CasK}}
\ind{C^\infty(G)_K}{\il{CiGK}}
\ind{C^\infty(\Ld_\s\backslash G)_K}{\il{CLGK}}
\ind{C^\infty(\Ld_\s\backslash G)_K^\ps}{\il{CiLdGps}}
\ind{c \text{ shift parameter}}{\il{c-theta}}
\ind{c^M(r)}{\il{coM}}
\ind{c^V(r), \; c^W(r)}{\ir{cWVexpl}}
\indexspace
\ind{d \in 1+2\ZZ \text{ metaplectic param.}}{\il{d-inftyo}, \ir{mudef}}
\indexspace
\ind{E_0}{\ir{ev0}}
\indexspace
\ind{\Ffu_\bt}{\il{Ffubt}}
\ind{\Ffu_{\ell,c}}{\il{Ffuellc}}
\ind{\Ffu_{\ell,c,d}=\Ffu_\n}{\il{Fn}}
\ind{\Ffu_\bt^\ps,\; \Ffu_\n^\ps}{\il{FfuNfu-intro}, \il{Ffups}}
\ind{\Four_\bt}{\ir{Frbtdef}}
\ind{\Four_{\ell,c}}{\ir{Frlcdef}}
\ind{\Four_{\ell,c,d}=\Four_\n}{\ir{Fourmu}}
\ind{\FF,\, \FI}{\il{FFFI}}
\ind{f^\c}{\il{fc}}
\ind{f_{\ell,c,m}}{\il{fellcm}}
\ind{f_\bt}{\il{fbt}}
\indexspace
\ind{G=\SU(2,1)}{\ir{Gdef}}
\ind{\glie}{\il{glie}}
\ind{\glie_c}{\il{gliec}}
\ind{g_\c}{\il{gc}}
\indexspace
\ind{\HH_i}{\ir{HHi}}
\ind{\HH_r}{\il{HHr}}
\ind{H^{\xi,\nu}_K}{\il{HK-intro}, \ir{HK-def}}
\ind{\hK}{\il{hK1}}
\ind{\hm(c)}{\ir{hmdef}}
\ind{h_{\ell,m}}{\ir{hlmdef}}
\indexspace
\ind{I_\nu}{\ir{Inu}}
\ind{I_{2,1}}{\ir{I21}}
\ind{\II,\; \IF}{\il{IIIF}}
\ind{\ii_0}{\ir{ii0}}
\indexspace
\ind{j=j_\xi}{\il{j}}
\ind{j_+, \; j_r,\; j_l}{\il{jprl}}
\indexspace
\ind{K \subset G}{\il{Kdef}, \il{Kdefa} }
\ind{K_0\subset K}{\ir{K0def}}
\ind{K_\nu}{\ir{Knu}}
\ind{ K_{0;h,p}, K_{\bt;h,p},\; K_{\n;h,p}}{\il{K0hp}, \il{Khp},
\il{Khpn} }
\ind{\klie}{\il{klie}, \il{klie1} }
\ind{\klie_c}{\il{kliec}}
\ind{\kk^I_{\bt;h,p},\; \kk^K_{\bt;h,p}}{\ir{expl-kdso}}
\ind{\kk^W_{\n;h,p},\; \kk^V_{\n;h,p}}{\ir{kkVW}}
\ind{\kk^M_{\n;h,p}}{\ir{kkMdef}}
\ind{\km(\eta,\al,\bt)\in K}{\ir{Kdef}}
\ind{dk}{\il{dk}}
\indexspace
\ind{\Lfu_0^{\xi,\nu}}{\il{Lfu0}}
\ind{\tilde\Lfu_0^{j,0}}{\ir{tLfuip}}
\ind{L,\; L^+}{\ir{Ldef}}
\ind{L(\XX)}{\ir{Ldiff}}
\ind{L_h}{\il{Lh}}
\ind{\ell\text{ parametr.~char.~of }Z(N)}{\il{ell-intro}, \il{ell}}
\indexspace
\ind{\Mfu_\Nfu^\ps}{\il{Mfudef}, 
}
\ind{\Mfu_\Nfu^{\xi,\nu}}{\il{Mfu-intro}, \il{MFUximudef},
\il{Mfudefia}, \il{Mfuxinu} }
\ind{M\subset K\subset G}{\il{Mdef}, \il{M1} }
\ind{M(\XX)f}{ \ir{Mdiff1}}
\ind{M_{\k,s}}{\ir{Mkps}}
\ind{\mm(\z)\in M}{\il{mztdef}}
\ind{m(h,r)}{\ir{mhrdef}, \il{mhrn} }
\ind{m_0=m_0(j)}{\il{m0}}
\indexspace
\ind{N}{\il{Ndef}}
\ind{N_\c}{\il{Nc}}
\ind{\Norm_P(\Ld_\s)}{\il{NP}}
\ind{\nlie}{\il{nlie}, \il{nlie1} }
\ind{\n=(\ell,c,d)}{\il{n-intr}, \il{n}}
\ind{\nm(x+iy,r)=\nm(x,y,r)\in N}{\ir{Ndef}}
\ind{dn}{\il{dn}}
\indexspace
\ind{\wo(\ps), \wo^1(\ps)}{\il{Wops}, \ir{wo}}
\ind{\wo(\ps)_\n,\; \wo^1(\ps)_\n}{\il{won}}
\ind{\wo(\ps)^+}{\ir{Wops+}}
\ind{ \wo(\ps)^+_\n}{\ir{Wops+n}, \ir{WO+}}
\indexspace
\ind{P_c}{\il{Pc}}
\indexspace
\ind{R(\XX) f = \XX f}{\ir{Rdiff}, \ir{Rdiff1} }
\ind{r_0(h)}{\il{r0def}, \il{r0n} }
\indexspace
\ind{\Schw(\RR)}{\il{Schsp}}
\ind{\SU(2,1)}{\il{su21}}
\ind{S_{\!\Nfu}(\eta)}{\il{Scho}}
\ind{\sect(j)}{\il{sect}}
\ind{\Ws1,\; \Ws2}{\ir{S12}}
\ind{\sh{\pm 3}{\pm1}}{\il{sopS}}
\indexspace
\ind{T^+_k,\; T^-_k}{\il{Tkpm}}
\indexspace
\ind{U(\glie)}{\il{Uglie}}
\indexspace
\ind{\Vfu_\n^\ps}{\ir{Vps}}
\ind{\Vfu^{\xi,\nu}_\n}{\il{Vfu-intro}, \il{Vfuxinu} }
\ind{V_{\k,s}}{\ir{Vkps}}
\ind{v(a,b)\text{ in Subsection~\ref{sect-scm}}}{\il{vab}}
\indexspace
\ind{\WO}{\il{WO}}
\ind{\WOS,\,\WOM,\, \cdots}{\il{WSMGI}}
\ind{\Wfu_\Nfu^\ps}{\ir{Wfudef}, 
}
\ind{\Wfu_\Nfu^{\xi,\nu}}{\il{Wfu-intro}, \il{WFUximudef},
\il{Wfudefia}, \il{Wfuxinu}, }
\ind{\WW_0,\,\WW_1,\,\WW_2\in \klie}{\il{WW}}
\ind{W}{\il{W}}
\ind{W_{\k,s}}{\ir{Wkps}}
\ind{\wm\in G}{\il{wmdef}}
\indexspace
\ind{\X}{\ir{Xdef}}
\ind{\XX_0,\,\XX_1,\,\XX_2\in \nlie}{\ir{XX012}}
\indexspace
\ind{\Z_{ij}\in \glie_c}{\il{Z12}, \il{tab-rootv} }
\ind{ZU(\glie)}{\il{ZUg}}
\ind{Z(N)}{\il{ZN}}
\indexspace
\indexspace
\ind{\al_1,\;\al_2}{\ir{al12}}
\indexspace
\ind{\Gm_0}{\il{Gmst}, \il{Gm0}}
\ind{\Gm_\fl}{\ir{GmFL}}
\indexspace
\ind{\Dt_3\in ZU(\glie)}{\il{Dt3}}
\ind{\dt(\cdot)}{\il{dtt}}
\indexspace
\ind{\e=\sign(\ell)}{\ir{epsdtt}, \il{eps}}
\indexspace
\ind{\Th_{\ell,c}}{\ir{Thdef}}
\ind{\th_m}{\il{thm}, \il{th-m}}
\indexspace
\ind{\k_0=\k_0(j)}{\il{kap0}}
\indexspace
\ind{\Ld_\s}{\il{Ldsg}}
\ind{\ld_2(j,\nu),\; \ld_3(j,\nu)}{\ir{ld23}}
\ind{\ldph h{p}{r}{q}}{\ir{ldphg}, \il{ldph}}
\indexspace
\ind{\Mu_{\Nfu;h,p}}{\il{Mu}}
\ind{\mu^{a,b}_\Nfu(j,\nu)}{\ir{mu00a}, \il{muab}}
\ind{\tilde \mu^{p,0}_\n(j,\nu),\; \tilde\mu^{p,0}_\n(j,\nu),\;
\tilde\mu^{a,b}_\n}{\il{tmu}, \ir{tumu}}
\indexspace
\ind{\nu\text{ spectral parameter}}{\il{ld23}}
\ind{\nu_+,\; \nu_r,\; \nu_l}{\il{nuprl}}
\indexspace
\ind{\xi=\xi_j}{\il{xi}}
\indexspace
\ind{\pi_\ld}{\ir{pild}}
\ind{d\pi_\ld}{\ir{dpild}}
\indexspace
\ind{\s(\c)}{\il{sgc}}
\indexspace
\ind{\tau_p}{\il{taup}}
\ind{\tau^h_p}{\il{tauhp-intro}, \ir{tauhp}}
\ind{\tau^h_{r,p}}{\il{tauhrp}}
\indexspace
\ind{\Ups_{\Nfu;h,p}}{\il{Ups}}
\ind{\ups^{a,b}_\n(j,\nu)}{\ir{upsab}}
\ind{\tilde \ups^{p,0}_\n(j,\nu),\; \tilde\ups^{p,0}_\n(j,\nu), \;
\tilde\ups^{a,b}_\n(j,\nu)}{\il{tups}, \ir{tuod}}
\indexspace
\ind{\Phi^{p}_{r,q}}{\ir{PhiSU}}
\ind{\Kph h{p}{r}{p}}{\ir{Kphd}, \ir{Phi-exmp}}
\ind{\kph h{p}{r}{q}}{\il{kph}}
\indexspace
\ind{\ch_\bt}{\ir{chbt}}
\indexspace
\ind{\ps[\xi,\nu],\;\ps[j,\nu]}{\il{jnu}}
\indexspace
\ind{\Om_{\Nfu;h,p}}{\il{Om}}
\ind{\om^{a,b}_\Nfu(j,\nu)}{\ir{om00ab}, \il{omab}}
\ind{\tilde \om^{p,0}_\n(j,\nu),\; \tilde\om^{p,0}_\n(j,\nu),\;
\tilde\om^{a,b}_\n(j,\nu) }{\il{tom}, \il{tod}}
\indexspace
\indexspace
\ind{\infty\in \partial\X}{\il{iuhp}}
\ind{(\cdot,\cdot)_\uprs}{\il{unprs}}
\ind{(\cdot,\cdot)_{\cmpl}}{\ir{compls}}
\ind{\dis}{\il{.=}}
\ind{\sum_r}{\ir{sumr}, \il{sumrnot}}
\ind{V_{h,p} \text{ subspace with $K$-type $\tau^h_p$}}{\il{Ktp-intro},
\il{Vhp}}
\ind{V_{h,p,q}\subset V_{h,p} \text{ subspace with weight }
q}{\il{Vhpq}}
\end{trivlist}
\end{multicols}

\end{document}